\documentclass[12pt]{amsart}
\usepackage{amssymb}
\usepackage{amsbsy}
\usepackage{amscd}
\usepackage[mathscr]{eucal}
\usepackage{nicefrac}

\DeclareSymbolFont{rsfs}{U}{rsfs}{m}{n}
\DeclareSymbolFontAlphabet{\mathrsfs}{rsfs}
\usepackage{a4wide}
\usepackage{tikz}
\usetikzlibrary{arrows,decorations.pathmorphing,backgrounds,positioning,fit,petri}
\usetikzlibrary{decorations.markings}

\usepackage{verbatim}
\usepackage{version}
\usepackage{color}
\definecolor{darkspringgreen}{rgb}{0.09, 0.45, 0.27}
\definecolor{deepjunglegreen}{rgb}{0.0, 0.29, 0.29}
\newenvironment{NB}{
\color{red}{\bf NB}. \footnotesize
}{}
\newenvironment{NB2}{
\color{blue}{\bf NB2}. \footnotesize
}{}
\excludeversion{NB}
\excludeversion{NB2}
\makeatletter

\usepackage{url}
\usepackage{ifmtarg}

\usepackage{xr-hyper}
\usepackage[
pdftex,
bookmarks=true,
colorlinks=true,
citecolor=deepjunglegreen,
filecolor=deepjunglegreen,
unicode,
pdfnewwindow=true]{hyperref}
\usepackage{cleveref}
\let\RR\relax
\let\bN\relax
\let\SL\relax

[http://arxiv.org/pdf/1706.02112.pdf]
[http://arxiv.org/pdf/1601.03586.pdf]

\usepackage{xspace}

\AtBeginDocument{%
\makeatletter
\let\oldref=\ref
\renewcommand{\ref}[1]{%
  \def\@mystring{affine_pre-#1}%
  \@ifundefined{r@\@mystring}{%
    \def\@mystring{Coulomb2-#1}%
    \@ifundefined{r@\@mystring}{%
      \oldref{#1}}{%
      \namecref{Coulomb2-#1}II.\oldref{Coulomb2-#1}%
    }%
  }%
  {\cite[\namecref{affine_pre-#1}\oldref{affine_pre-#1}]{affine}%
  }%
}
\renewcommand{\eqref}[1]{%
  \def\@mystring{affine_pre-#1}%
  \@ifundefined{r@\@mystring}{%
    \def\@mystring{Coulomb2-#1}%
    \@ifundefined{r@\@mystring}{%
      \textup{\tagform@{\oldref{#1}}}}{%
      \text{(II.\oldref{Coulomb2-#1})}%
    }%
  }%
  {\text{\cite[(\oldref{affine_pre-#1}]{affine}}%
}%
}
\makeatother
}

\crefname{Theorem}{Theorem\xspace}{Theorems}
\crefformat{Theorem}{Theorem~#2#1#3}
\crefname{section}{\S}{\S\S}
\crefformat{section}{\S#2#1#3} % to remove unnecessary space
\crefmultiformat{section}{\S\S#2#1#3}{, #2#1#3}{, #2#1#3}{, #2#1#3}
\crefname{Lemma}{Lemma\xspace}{Lemmas\xspace}
\crefformat{Lemma}{Lemma~#2#1#3}
\crefname{Proposition}{Proposition\xspace}{Propositions\xspace}
\crefformat{Proposition}{Proposition~#2#1#3}
\crefname{Corollary}{Corollary\xspace}{Corollaries\xspace}
\crefformat{Corollary}{Corollary~#2#1#3}
\crefname{Definition}{Definition}{Definitions}
\crefformat{Definition}{Definition~#2#1#3}
\crefname{Remark}{Remark\xspace}{Remarks\xspace}
\crefformat{Remark}{Remark~#2#1#3}
\crefname{Remarks}{Remark\xspace}{Remarks\xspace}
\crefformat{Remarks}{Remark~#2#1#3}
\crefname{Conjecture}{Conjecture\xspace}{Conjectures\xspace}
\crefformat{Conjecture}{Conjecture~#2#1#3}
\crefname{figure}{Figure\xspace}{Figure\xspace}
\crefformat{figure}{Figure~#2#1#3}

\crefformat{appendix}{\S#2#1#3}
\crefformat{enumi}{#2#1#3}

\usepackage{stmaryrd}
\SetSymbolFont{stmry}{bold}{U}{stmry}{m}{n}
\usepackage{pigpen} % for pigpenfont used for fiber squares
\usepackage{enumitem}
\usepackage{calc}

\usepackage{scalerel}

\usepackage{mathtools}

\renewcommand{\thesubsection}{\thesection(\@roman\c@subsection)}
\newenvironment{aenume}{%
  \begin{enumerate}%
  }{\end{enumerate}}
\newcounter{number}
\makeatother
\newtheorem{Theorem}[equation]{Theorem}
\newtheorem{Corollary}[equation]{Corollary}
\newtheorem{Lemma}[equation]{Lemma}
\newtheorem{Proposition}[equation]{Proposition}

\theoremstyle{definition}
\newtheorem{Definition}[equation]{Definition}

\newtheorem{Conjecture}[equation]{Conjecture}

\theoremstyle{remark}
\newtheorem{Remark}[equation]{Remark}
\newtheorem{Remarks}[equation]{Remarks}

\newtheorem{Question}[equation]{Question}
\numberwithin{equation}{section}
\newcommand{\thmref}[1]{Theorem~\ref{#1}}
\newcommand{\secref}[1]{\S\ref{#1}}
\newcommand{\lemref}[1]{Lemma~\ref{#1}}
\newcommand{\propref}[1]{Proposition~\ref{#1}}
\newcommand{\corref}[1]{Corollary~\ref{#1}}
\newcommand{\subsecref}[1]{\S\ref{#1}}
\newcommand{\remref}[1]{Remark~\ref{#1}}
\newcommand{\defeq}{\overset{\operatorname{\scriptstyle def.}}{=}}
\newcommand{\CC}{{\mathbb C}}
\newcommand{\ZZ}{{\mathbb Z}}

\newcommand{\RR}{{\mathbb R}}
\newcommand{\proj}{{\mathbb P}}
\newcommand{\CP}{\proj}

\newcommand{\SL}{\operatorname{\rm SL}}
\newcommand{\SU}{\operatorname{\rm SU}}
\newcommand{\GL}{\operatorname{GL}}

\renewcommand{\U}{\operatorname{\rm U}}

\newcommand{\gl}{\operatorname{\mathfrak{gl}}}

\newcommand{\g}{{\mathfrak g}}

\newcommand{\Spec}{\operatorname{Spec}\nolimits}

\newcommand{\End}{\operatorname{End}}
\newcommand{\Hom}{\operatorname{Hom}}

\renewcommand{\MR}[1]{}

\newcommand{\dslash}{/\!\!/}
\newcommand{\vin}[1]{{%\color{red}
\operatorname{i}(#1)}} % incoming vertex
\newcommand{\vout}[1]{{%\color{red}
\operatorname{o}(#1)}} % outgoing vertex
\newcommand{\bM}{\mathbf M}
\newcommand{\bN}{\mathbf N}

\newcommand{\shfO}{\mathcal O}

\newcommand{\tslabar}{\mathbin{
\setbox0=\hbox{/\!\!/\!\!/}\rule[0.4\ht0]{\wd0}{.3\dp0}\kern-\wd0\box0}}

\newcommand{\bmu}{\boldsymbol\mu}

\newcommand{\la}{\lambda}
\newcommand{\aff}{\mathrm{aff}}

\newcommand{\Gr}{\mathrm{Gr}}

\newcommand{\cR}{\mathcal R}

\newcommand{\cT}{\mathcal T}

\newcommand{\cO}{\mathcal O}
\newcommand{\cW}{\mathcal W}

\newcommand{\scP}{\mathscr P}

\newcommand{\br}{\mathbf r}

\makeatletter
\newcommand{\cA}[1][{}]{%
  \@ifmtarg{#1}%
  {\mathcal A}% if #1 is empty
  {\mathcal A(#1)}% if #1 is not empty
}
\newcommand{\cAh}[1][{}]{%
  \@ifmtarg{#1}%
  {\mathcal A_\hbar}% if #1 is empty
  {\mathcal A_\hbar(#1)}% if #1 is not empty
}
\makeatother

\newcommand{\Stab}{\operatorname{Stab}}

\newcommand{\ft}{\mathfrak t}
\newcommand{\gr}{\operatorname{gr}}

\newcommand{\bNT}{\bN_T}

\newcommand{\calO}{\cO}

\newcommand{\po}{\ar@{}[dr]|{\text{\pigpenfont R}}}
\newcommand{\pb}{\ar@{}[dr]|{\text{\pigpenfont J}}}
\newcommand{\pp}{\ar@{}[dr]|{\text{\pigpenfont P}}}

\newcommand{\uS}{\mathscr S}
\newcommand{\uQ}{\mathscr Q}

\newcommand{\GV}{\GL(V)}
\newcommand{\TV}{T(V)}
\newcommand{\TW}{T(W)}
\newcommand{\Weyl}{\mathbb W}
\newcommand{\SDAHA}{\mathbf S\mathbf H}

\newcommand{\cM}{\mathcal M}

\newcommand{\BA}{{\mathbb{A}}}
\newcommand{\BC}{{\mathbb{C}}}
\newcommand{\BF}{{\mathbb{F}}}
\newcommand{\BG}{{\mathbb{G}}}
\newcommand{\BN}{{\mathbb{N}}}

\newcommand{\BP}{{\mathbb{P}}}

\newcommand{\BZ}{{\mathbb{Z}}}

\newcommand{\bp}{{\mathbf{p}}}
\newcommand{\bq}{{\mathbf{q}}}
\newcommand{\bfr}{{\mathbf{r}}}

\newcommand{\bt}{{\mathbf{t}}}
\newcommand{\bz}{{\mathbf{z}}}

\newcommand{\CB}{{\mathcal{B}}}
\newcommand{\CF}{{\mathcal{F}}}
\newcommand{\CG}{{\mathscr{G}}}

\newcommand{\CI}{{\mathcal{I}}}
\newcommand{\CJ}{{\mathcal{J}}}
\newcommand{\CK}{{\mathcal{K}}}
\newcommand{\CL}{{\mathcal{L}}}
\newcommand{\CM}{{\mathcal{M}}}
\newcommand{\CO}{{\mathcal{O}}}
\newcommand{\CR}{{\mathcal{R}}}
\newcommand{\cS}{{\mathcal{S}}}
\newcommand{\CT}{{\mathcal{T}}}
\newcommand{\CU}{{\mathcal{U}}}
\newcommand{\CV}{{\mathcal{V}}}
\newcommand{\CW}{{\mathcal{W}}}
\newcommand{\oW}{\overline{\mathcal{W}}{}}
\newcommand{\iso}{\overset{\sim}{\longrightarrow}}

\newcommand{\fg}{{\mathfrak{g}}}
\newcommand{\fri}{{\mathfrak{i}}}

\newcommand{\fz}{{\mathfrak{z}}}
\newcommand{\fB}{{\mathfrak{B}}}
\newcommand{\fC}{{\mathfrak{C}}}

\newcommand{\fM}{{\mathfrak{M}}}
\newcommand{\fU}{{\mathfrak{U}}}
\newcommand{\sft}{{\mathsf{t}}}
\newcommand{\sy}{{\mathsf{y}}}
\newcommand{\sx}{{\mathsf{x}}}
\newcommand{\sfu}{{\mathsf{u}}}
\newcommand{\sw}{{\mathsf{w}}}
\newcommand{\sfM}{{\mathsf{M}}}

\newcommand{\unl}{\underline}
\newcommand{\ol}{\overline}
\newcommand{\on}{\operatorname}

\newcommand{\lambdavee}{\lambda^{\!\scriptscriptstyle\vee}}
\newcommand{\alphavee}{\alpha^{\!\scriptscriptstyle\vee}}

\newcommand\calL{\mathcal L}
\newcommand\calX{\mathcal X}
\newcommand\Pic{\operatorname{Pic}}
\newcommand\ra{\rangle}
\newcommand\ten{\otimes}
\newcommand\x{\times}
\newcommand\Bunt{\operatorname{Bun}_T}
\newcommand\Bunb{\operatorname{Bun}_B}
\newcommand\Bung{\operatorname{Bun}_G}
\newcommand\grp{\mathfrak p}
\newcommand\grq{\mathfrak q}

\newcommand{\II}{{%\color{red}
Q_0}}
\newcommand{\Qo}{{%\color{red}
Q_1}}
\newcommand{\La}{\mathfrak L}
\newcommand{\fA}{\mathfrak A}

\makeatletter
\let\@wraptoccontribs\wraptoccontribs
\def\@setcontribs{%
  \@xcontribs
  \textsf{\xcontribs}
}
\makeatother

\usepackage{accents}
\makeatletter
\newcommand{\sbullet}{%
  \hbox{\fontfamily{lmr}\fontsize{.4\dimexpr(\f@size pt)}{0}\selectfont\textbullet}}
\DeclareRobustCommand{\mathbullet}{\accentset{\sbullet}}
\makeatother

\newcommand{\CZ}{{\mathscr{Z}}}
\newcommand{\bW}{\mathbullet{\CW}}
\newcommand{\uoW}{{\overline{\underline{\mathcal{W}}}}{}}
\newcommand{\bGZ}{\mathbullet{\CG\CZ}}
\newcommand{\oGZ}{\mathring{\CG\CZ}}

\newcommand{\butimes}{\mathbullet{\times}}

\newcommand{\oU}{\mathring{\fU}^\alpha_{G,B}}

\newcommand{\oZ}{\mathring{Z}}
\newcommand{\bZ}{\mathbullet{Z}}
\newcommand{\oE}{\mathring{E}}
\newcommand{\bE}{\mathbullet{E}}
\newcommand{\uZ}{\underline{Z}}
\newcommand{\uoZ}{\mathring{\underline Z}}
\newcommand{\uB}{\underline{\mathfrak B}}
\newcommand{\oB}{\mathring{\mathfrak B}}
\newcommand{\uoB}{\mathring{\underline{\mathfrak B}}}
\newcommand{\oA}{\mathring{\mathbb A}}
\newcommand{\oG}{\mathring{\mathbb G}}
\newcommand{\bA}{\mathbullet{\mathbb A}}

\begin{document}
\title[Quiver gauge theories and slices in the affine
Grassmannian]{{C}oulomb branches of $3d$ $\mathcal N=4$ quiver gauge
  theories and slices in the affine Grassmannian
%{\rm Preliminary Version (\today)}
}
\author[A.~Braverman]{Alexander Braverman}
\address{
Department of Mathematics,
University of Toronto and Perimeter Institute of Theoretical Physics,
Waterloo, Ontario, Canada, N2L 2Y5
}
\email{braval@math.toronto.edu}
\author[M.~Finkelberg]{Michael Finkelberg}
\address{National Research University Higher School of Economics,
Russian Federation,
Department of Mathematics, 6 Usacheva st, Moscow 119048;
Skolkovo Institute of Science and Technology}
\email{fnklberg@gmail.com}
\author[H.~Nakajima]{Hiraku Nakajima}
\address{Research Institute for Mathematical Sciences,
Kyoto University, Kyoto 606-8502,
Japan}
\email{nakajima@kurims.kyoto-u.ac.jp}
\contrib[Two appendices by]{Alexander~Braverman
   Michael~Finkelberg,
   Joel~Kamnitzer,
   Ryosuke~Kodera,
   Hiraku~Nakajima,
   Ben~Webster,
and    
   Alex~Weekes
}

\subjclass[2000]{Primary 81T13; Secondary 14D21, 16G20, 20G45}

\begin{abstract}
    This is a companion paper of \cite{main}. We study Coulomb
    branches of unframed and framed quiver gauge theories of type
    $ADE$. In the unframed case they are isomorphic to the moduli
    space of based rational maps from $\CP^1$ to the flag variety. In
    the framed case they are slices in the affine Grassmannian and
    their generalization. In the appendix,
    written jointly with
Joel~Kamnitzer,
Ryosuke~Kodera,
Ben~Webster,
and    
Alex~Weekes,
we identify the quantized Coulomb branch with the truncated shifted
Yangian.
\end{abstract}

\maketitle

\setcounter{tocdepth}{2}

%\tableofcontents

% !TEX root = blowup_pre.tex
\section{Introduction}\label{blowup-intro}

In \cite{2015arXiv150303676N} the third named author proposed an
approach to define the Coulomb branch $\mathcal M_C$ of a $3d$
$\mathcal N=4$ SUSY gauge theory in a mathematically rigorous way.
% In particular, a definition of its coordinate ring as a graded vector
%space was given there. A multiplication was introduced in 
The subsequent paper \cite{main} by the present authors gives
%, hence we have
a mathematically rigorous definition of $\mathcal M_C$ as an affine
algebraic variety. The purpose of this companion paper is to determine
$\mathcal M_C$ for a framed/unframed quiver gauge theory of type
$ADE$.

Let us first consider the unframed case. We are given a quiver
$Q = (\II,\Qo)$ of type $ADE$ and a $\II$-graded vector space
$V=\bigoplus V_i$. Here $\II$ (resp.\ $\Qo$) is the set of vertices
(resp.\ oriented arrows) of $Q$.
We consider the corresponding quiver gauge theory: the gauge group is
$\GL(V) = \prod \GL(V_i)$, its representation is
$\bN = \bigoplus_{h\in \Qo}\Hom(V_{\vout{h}},V_{\vin{h}})$, where
$\vout{h}$ (resp.\ $\vin{h}$) is the outgoing (resp.\ incoming) vertex
of an arrow $h\in\Qo$.
Our first main result (\thmref{pestun}) is $\mathcal M_C \cong
\oZ^{\alpha}$, where $\oZ^\alpha$ is the moduli space of based
rational maps from $\CP^1$ to the flag variety $\CB = G/B$ of degree
$\alpha$, where the group $G$ is determined by the $ADE$ type of $Q$,
and $\alpha$ is given by the dimension vector $\dim V = (\dim
V_i)_{i\in \II}$.

In physics literature, it is argued that $\mathcal M_C$ is the moduli
space of monopoles on $\RR^3$. See \cite{MR1451054} for type $A$,
\cite{MR1677752} in general. Here the gauge group $G_{c} = G_{ADE,c}$
is the maximal compact subgroup of $G$, hence is determined by the
type of the quiver as above. It is expected that two moduli spaces are
in fact isomorphic as complex manifolds. See \remref{rem:bijection}
below.

We also generalize this result to the case of an affine quiver
(\thmref{conditional}). We show that $\mathcal M_C$ is a partial
compactification of the moduli space of parabolic framed $G$-bundles over
$\CP^1\times\CP^1$. In fact, we prove this under an assumption that
the latter space is normal, which is known only for type $A$. This
space is expected to be isomorphic to the Uhlenbeck partial
compactification of the moduli space of $G_c$-calorons (= $G_c$-instantons
on $\RR^3\times S^1$).

In order to show $\mathcal M_C\cong\oZ^\alpha$, we use the criterion
in \ref{prop:flat}: Suppose that we have $\varPi\colon\mathcal M\to
\ft(V)/\Weyl$ such that $\mathcal M$ is a Cohen-Macaulay affine
variety and $\varPi$ is flat. Here $\ft(V)$ is the Lie algebra of a
maximal torus $T(V)$ of $\GL(V)$, and $\Weyl$ is the Weyl group of
$\GL(V)$.
If we have a birational isomorphism $\Xi^\circ\colon \mathcal M
\overset{\approx}{\dasharrow} \mathcal M_C$ over $\ft(V)/\Weyl$, which
is biregular up to codimension $2$, it extends to the whole spaces.
In order to apply this for $\mathcal M = \oZ^\alpha$, we shall check
those required properties.
To construct $\Xi^\circ$, let us recall that $\mathcal M_C$ is
birational to $T^*T(V)^\vee/\Weyl$ over the complement of union of
hyperplanes (generalized root hyperplanes) in $\ft(V)/\Weyl$ (see
\ref{cor:coordinate}). The same is true by the `factorization'
property of $\oZ^\alpha$. This defines $\Xi^\circ$ up to codimension
$1$, hence it is enough to check that $\Xi^\circ$ extends at a generic
point in each hyperplane. This last check will be done again by using
the factorization.
The factorization property is well-known in the context of zastava
space $Z^\alpha$, which is a natural partial compactification of
$\oZ^\alpha$. (See \cite{bdf} and references therein.)

\begin{Remark}\label{rem:bijection}
    Consider $\oZ^\alpha$ the moduli space of based rational maps
    $\CP^1\to G/B$ of degree $\alpha$, and the moduli space of
    $G_c$-monopoles with monopole charge $\alpha$.
    It is known that there is a bijection between two moduli spaces
    (given in \cite{MR709461,MR769355} for $A_1$,
    \cite{MR994495,MR987771} for classical groups and \cite{MR1625475}
    for general groups).
    It is quite likely that this is an isomorphism of complex
    manifolds, but not clear to authors whether the proofs give this
    stronger statement. For $A_1$, one can check it by using
    \cite{MR1215288}, as it is easy to check that the bijection
    between the moduli space of solutions of Nahm's equation and the
    moduli space of based maps is an isomorphism.

    For affine type $A$, a bijection is given by \cite{Takayama} based
    on earlier works
    \cite{Nye-thesis,MR1789949,MR2395473,MR2681686}. The same question
    above arises here also.
\end{Remark}

Let us turn to the framed case. We take an additional $\II$-graded vector space $W = \bigoplus W_i$ and add $\bigoplus \Hom(W_i,V_i)$ to $\bN$.
The answer has been known in the physics
literature~\cite{MR1451054,MR1454291,MR1454292,MR1636383} (at least
for type $A$): $\mathcal M_C$ is a moduli space of singular
$G_c$-monopoles on $\RR^3$. Two coweights $\la,\mu\colon S^1\to G_{c}$
attached at $0$, $\infty$ of $\RR^3$ are given by dimension vectors
for framed and ordinary vertices respectively. 
We will not use the moduli space of singular monopoles, hence we will
not explain how $\la$ and $\mu$ arise in this paper. (See the summary
in \cite[App.~A]{2015arXiv150304817B}.)
But let us emphasize that we need to take the Uhlenbeck partial
compactification.\footnote{Uhlenbeck partial compactification is
  necessary, due to bubbling at $0$. This is naturally understood by
  considering singular monopoles as $S^1$-equivariant instantons on
  the Taub-NUT space. See \remref{rem:Takayama} below. This bubbling
  is called \emph{monopole bubbling} in physics
  literature. e.g.,\cite{MR2306566,2015arXiv150304817B}.} This point
will be important as explained below.

On the other hand, it was conjectured that $\mathcal M_C$ is a framed
moduli space of $S^1$-equivariant instantons on $\RR^4$ in
\cite[\S3.2]{2015arXiv150303676N} when $\mu$ is dominant. There is a
subtle difference between $S^1$-equivariant instantons and singular
monopoles. (The third named author learned it after reading
\cite{2015arXiv150304817B}. See \cite[\S5(iii)]{2015arXiv151003908N}.)
The former makes sense only when $\mu$ is dominant, but is expected
to be isomorphic to the latter as complex manifolds.

In order to identify $\mathcal M_C$, we use the criterion as above. In particular,
we need a candidate $\mathcal M$ as an affine algebraic variety, or at least
a complex analytic space. 
For this purpose, the moduli space of singular monopoles has a defect,
as a complex
%an algebraic 
structure on its Uhlenbeck partial compactification is
not constructed in the literature except of type $A$.
(See \remref{rem:Takayama} below for type $A$.)

By \cite{braverman-2007} the Uhlenbeck partial compactification of the framed moduli space of $S^1$-equivariant instantons on $\RR^4$ is isomorphic to a slice $\oW^\lambda_{\mu}$ of a $G[[z]]$-orbit in the affine Grassmannian $\Gr_{G}$ in the closure of another orbit.
This is a reasonable alternative, as it has a close connection to the
zastava space $Z^\alpha$, and lots of things are known in the
literature.
\begin{NB}
   If you want to make the latter more specific, please edit it.
\end{NB}%

The first half of this paper is devoted to the construction of 
a generalization of the slice 
$\oW^\lambda_{\mu}$,
which makes sense even when $\mu$ is \emph{not} dominant. We call it a \emph{generalized slice}, and denote it by the same notation. There are
several requirements for the generalized slice. It must be possible for us to check
properties required in the criterion above. The most important one is the factorization. Also we should have a dominant birational morphism 
$\oZ^{\alpha^*}\to\oW^\la_{\mu}$, as a property of the Coulomb branch (see \ref{rem:further} and \cref{open_piece} below).
These properties naturally led the authors to our definition of
generalized slices.
The heart of the first half is \cref{prop:blow} showing
$\oW^\la_{\mu}$ is a certain affine blowup of the zastava space
$Z^{\alpha^*}$ up to codimension $2$.
\begin{NB}
    Please edit this paragraph, if necessary.
\end{NB}%

%By its nature of the criterion we use, we do not need to \emph{derive}
%the candidate of $\mathcal M_C$ logically. In particular, an
%identification (of its open subset) with moduli of singular monopoles
%is not logically necessary, and is not established in this paper. It
%remains as a conjecture.
We introduce $\oW^\lambda_\mu$ as a moduli space of $G$-bundles over $\BP^1$
with trivialization outside $0$ and $B$-structure. This definition originally
appeared in~\cite{mf}. We also observe that it has an embedding into $G(z)$
so that its image coincides with the space of scattering matrices of singular
monopoles appearing in~\cite{2015arXiv150304817B}.\footnote{This was explained
to the authors by D.~Gaiotto and J.~Kamnitzer during the preparation of the
manuscript.}

We conjecture that Coulomb branches of framed affine quiver gauge theories
are Uhlenbeck partial compactifications of moduli spaces of instantons on the
Taub-NUT space invariant under a cyclic group action. This is {\em not}
precise yet as 1) we do not endow them with the structure of affine algebraic
varieties, and 2) we do not specify the cyclic group action. Also we should
recover the moduli spaces of singular monopoles by replacing the cyclic group
with $S^1$. Therefore they must be isomorphic to the generalized slice
$\oW^\la_{\mu}$, but we do not give a proof of this simple version of the 
conjecture.

\begin{Remark}\label{rem:Takayama}
    Singular monopoles are $S^1$-equivariant instantons on the
    Taub-NUT space by \cite{Kronheimer-msc}. See also
    \cite{MR2748801}. A moduli space of instantons on the Taub-NUT
    space is described as Cherkis bow variety
    \cite{MR2525636,MR2721657}, though mathematically rigorous proof
    of the completeness is still lacking as far as the authors
    know. Its $S^1$-fixed locus is also Cherkis bow variety of a
    different type. A bow variety is technically more tractable than
    the moduli space of singular monopoles.
    In \cite{2016arXiv160602002N}, Takayama and the third named author will identify
    $\mathcal M_C$ for a framed quiver gauge theory of type $A$ with
    Cherkis bow variety. This method is applicable for affine type $A$
    case, which conjecturally corresponds to a moduli space of
    $\ZZ/k\ZZ$-equivariant instantons on the Taub-NUT space.
\begin{NB}
    Relation between generalized slices and $S^1$-equivariant
    instantons on the Taub-NUT ?
\end{NB}%
\end{Remark}

The paper is organized as follows. In \secref{ZS} we introduce a
generalized slice as explained above.
In \secref{QGT} we identify the Coulomb branch of a quiver gauge
theory of type $ADE$ with generalized slices. We also treat the case
of affine type, but without framing. Then we identify the Coulomb
branch of a framed Jordan quiver gauge theory with a symmetric power
of the surface $xy = w^l$ ($l\ge 0$).
In \secref{sec:twist} we study the folding of a quiver gauge
theory. We show that the character of the coordinate ring of the fixed
point loci of the Coulomb branch is given by the twisted monopole
formula in \cite{Cremonesi:2014xha}.
In the appendices \S\S\ref{sec:minuscule},\ref{sec:quantization}
written jointly with
Joel~Kamnitzer,
Ryosuke~Kodera,
Ben~Webster,
and    
Alex~Weekes,
we study the embedding of the quantized Coulomb branch into the ring
of difference operators. We find various difference operators known in
the literature, such as Macdonald operators and those in
representations of Yangian (\cite{GKLO,kwy}). In particular, we show
that the quantized Coulomb branch of a framed quiver gauge theory of
type $ADE$ is isomorphic to the truncated shifted Yangian introduced
in \cite{kwy} under the dominance condition in \secref{sec:quantization}.

%%% Local Variables:
%%% mode: latex
%%% TeX-master: "blowup_pre"
%%% End:

\subsection*{Notation}

We basically follow the notation in \cite{main}. However a group $G$
is used for a flag variety $\CB = G/B$, while the group for the gauge
theory is almost always a product of general linear groups and denoted
by $\GL(V)$. Exceptions are \propref{prop:ad} and
\subsecref{subsec:adjoint}, where the gauge theory for the adjoint
representation of a reductive group is considered. We use $W$ for a
vector space for a quiver, while the Weyl group of $G$ is denoted by
$\Weyl$.

Sections, equations, Theorems, etc in \cite{main} will be referred
with `II.' plus the numbering, such as \ref{prop:flat}.

\subsection*{Acknowledgments}

We thank
%S.~Arkhipov,
R.~Bezrukavnikov,
%T.~Braden,
M.~Bullimore,
S.~Cherkis,
T.~Dimofte,
P.~Etingof,
B.~Feigin,
D.~Gaiotto,
D.~Gaitsgory,
V.~Ginzburg,
A.~Hanany,
%K.~Hori,
%J.~Kamnitzer,
%R.~Kodera,
A.~Kuznetsov,
A.~Oblomkov,
V.~Pestun,
%N.~Proudfoot,
L.~Rybnikov,
%Y.~Tachikawa,
Y.~Takayama,
M.~Temkin
and
A.~Tsymbaliuk
for the useful discussions.
We also thank
J.~Kamnitzer,
R.~Kodera,
B.~Webster,
and
A.~Weekes
for the collaboration.

A.B.\  was partially supported by the NSF grant DMS-1501047.
M.F.\ was partially supported by
% The paper was prepared within the framework of
a subsidy granted to the HSE by
the Government of the Russian Federation for the implementation of the Global
Competitiveness Program.
The research of H.N.\ is supported by JSPS Kakenhi Grant Numbers
%22244003, % Moriwaki's A
23224002, % Fukaya's S
23340005, % my B
24224001, % Saito's S
25220701. % Mukai's S

% !TEX root = blowup_pre.tex
\section{Zastava and slices}
\label{ZS}

\subsection{Zastava}
\label{zastava}
Let $G$ be an adjoint simple
%almost simple simply connected
simply laced complex algebraic group.
We fix a Borel and a Cartan subgroup $G\supset B\supset T$. Let $\Lambda$
be the coweight lattice, and let $\Lambda_+\subset\Lambda$ be the submonoid
spanned by the simple coroots $\alpha_i,\ i\in \II$.
Here the index set of simple coroots is identified with the set of
vertices of the Dynkin diagram, i.e.\ $\II$.
The involution
$\alpha\mapsto-w_0\alpha$ of $\Lambda$ restricts to an involution of
$\Lambda_+$ and induces an involution $\alpha_i\mapsto\alpha_{i^*}$ of the
set of the simple coroots. We will sometimes write $\alpha^*:=-w_0\alpha$ for
short. For
$\alpha=\sum_{i\in \II}a_i\alpha_i$ let $Z^\alpha\supset\oZ^\alpha$ be the
corresponding zastava space (moduli space of based quasimaps $\phi$ from $\BP^1$
to the flag variety $\CB=G/B$ such that $\phi$ has no defect at $\infty\in\BP^1$
and $\phi(\infty)=B_-\in\CB$, the opposite Borel subgroup to $B$ sharing the
same Cartan torus $T$) and its open moduli subspace of based maps.
Recall the factorization map $\pi_\alpha\colon Z^\alpha\to\BA^\alpha$ and its
section $s_\alpha\colon \BA^\alpha\hookrightarrow Z^\alpha$, see
e.g.~\cite{bdf}: the restriction of $\pi_\alpha$ to $\oZ^\alpha\subset Z^\alpha$
takes a based map $\phi\colon \BP^1\to\CB$ to the pullback $\phi^*{\mathfrak S}$
of the $\II$-colored Schubert divisor (the boundary of the open $B$-orbit in
$\CB$). Recall that
$\BA^\alpha=\BA^{|\alpha|}/S_\alpha$ where $\BA^{|\alpha|}=\prod_{i\in \II}\BA^{a_i}$,
and $S_\alpha$ is the product of the symmetric groups $\prod_{i\in \II}S_{a_i}$.
We define $\uZ^\alpha:=Z^\alpha\times_{\BA^\alpha}\BA^{|\alpha|},\
\uoZ^\alpha:=\oZ^\alpha\times_{\BA^\alpha}\BA^{|\alpha|}$. Clearly, $S_\alpha$ acts
on both $\uZ^\alpha$ and $\uoZ^\alpha$, and we have $Z^\alpha=\uZ^\alpha/S_\alpha,\
\oZ^\alpha=\uoZ^\alpha/S_\alpha$.

\begin{NB}
We consider the affine space $\BA^{|\alpha|}\times\BA^{|\alpha|}$
with the coordinates $w_{i,r}$ on the first factor and the corresponding
coordinates $y_{i,r}$ on the second factor. The group $S_\alpha$ acts on
$\BA^{|\alpha|}\times\BA^{|\alpha|}$ diagonally. We define the affine blowup
$\uB^\alpha$ as the spectrum of the following algebra
$$\BC[\uB^\alpha]:=\BC[\BA^{|\alpha|}\times\BA^{|\alpha|}]
\left[\frac{y_{i,r}-y_{i,s}}{w_{i,r}-w_{i,s}}\right]_{i\in \II,\ 1\leq r\ne s\leq a_i}
\left[\frac{y_{i,r}^{d_i}y_{j,s}}{(w_{i,r}-w_{j,s})^{d_i}}\right]_{\alpha_i\cdot\alpha_j<0,\
d_i\geq d_j}$$
Here $?\cdot?$ is the unique $W$-invariant scalar product on $\Lambda$ such that
the square length of a short coroot is 2, and $d_i=\alpha_i\cdot\alpha_i/2$.
Note that the action of $S_\alpha$ on $\BA^{|\alpha|}\times\BA^{|\alpha|}$ extends
uniquely to its action on $\uB^\alpha$, and we define $\fB^\alpha$ as
$\uB^\alpha/S_\alpha$.

We consider the following regular function on $\BA^{|\alpha|}\times\BA^{|\alpha|}$:
$$F_\alpha:=\prod_{i,r}y_{i,r}^{d_i}
\prod_{j\ne i}^{1\leq s\leq a_j}(w_{i,r}-w_{j,s})^{\alpha_j\cdot\alpha_i/2}.$$
Finally, we define $\uoB^\alpha$ as the spectrum of the algebra
$\BC[\uoB^\alpha]:=\BC[\uB^\alpha][F_\alpha^{-1}]$, and
$\oB^\alpha:=\uoB^\alpha/S_\alpha$. The natural projections to the first factor
$\uoB^\alpha\subset\uB^\alpha\to\BA^{|\alpha|}$ and
$\oB^\alpha\subset\fB^\alpha\to\BA^\alpha$ will be denoted $\varpi_\alpha$.

Unfortunately, these blowups are bad for the following reasons
(though we always have a birational projection $Z^\alpha\to\fB^\alpha$):

1) For non simply laced case, already for $G$ of type $B_2,G_2$ and
degree $\alpha=\alpha_i+\alpha_j$, we have $\oZ^\alpha\ne\oB^\alpha$.

2) Even for simply laced case, already for $G$ of type $A_3$ and degree
$\alpha=\alpha_i+\alpha_j+\alpha_k$, the projection $\varpi$ is not flat.
\end{NB}%

We denote by $w_{i,r},\ i\in \II,\ 1\leq r\leq a_i$ the natural coordinates on
$\BA^{|\alpha|}$.
We define an open subset
$\oA^{|\alpha|}\subset\BA^{|\alpha|}$ as the complement to all the diagonals
$w_{i,r}=w_{j,s}$, and also $\oA^\alpha:=\oA^{|\alpha|}/S_\alpha\subset\BA^\alpha$.
%We have an evident identification
%$\pi_\alpha^{-1}(\oA^{|\alpha|})\cong\varpi_\alpha^{-1}(\oA^{|\alpha|})$
%identical on the functions $w_{i,r},y_{i,r}$.
We also define a bigger open subset
$\oA^{|\alpha|}\subset\bA^{|\alpha|}\subset\BA^{|\alpha|}$ as the complement to
all the pairwise intersections of diagonals. We set
$\bA^\alpha:=\bA^{|\alpha|}/S_\alpha\subset\BA^\alpha$.

Recall that $\pi_\alpha\colon Z^\alpha\to\BA^\alpha$ is flat (since $\BA^\alpha$
is smooth, $Z^\alpha$ has rational singularities and hence it is
Cohen-Macaulay~\cite[Proposition~5.2]{brfi}, and all the fibers of $\pi_\alpha$
have the same dimension $|\alpha|$~\cite[Propositions~2.6,~6.4, and the line
right after~6.4]{bfgm}). Recall the regular functions
$(w_{i,r},y_{i,r})_{i\in \II,\ 1\leq r\leq a_i}$ on $\uZ^\alpha$, see~\cite[2.2]{bdf}.
Note that $\pi_\alpha(w_{i,r},y_{i,r})=(w_{i,r})$. We have
$\pi_\alpha^{-1}(\oA^{|\alpha|})\cong\oA^{|\alpha|}\times\BA^{|\alpha|}$ with
coordinates $(y_{i,r})$ on the second factor. Recall that the boundary
$\partial Z^\alpha=Z^\alpha\setminus\oZ^\alpha$ is
the zero divisor of a regular function $F_\alpha\in\BC[Z^\alpha]$ defined uniquely
up to a multiplicative scalar; in terms of $(w_{i,r},y_{i,r})$ coordinates we have
$$F_\alpha=\prod_{i,r}y_{i,r}
\prod_{\substack{h:\Qo\sqcup\overline\Qo \\ \vout{h}=i}}^{1\leq s\leq a_{\vin{h}}}(w_{i,r}-w_{\vin{h},s})^{-1/2},$$
(the inner product over all arrows $h\in\Qo$ or in the opposite orientation $\overline\Qo$ connected to $i$), see~\cite[Theorem~1.6.(2)]{bdf}. It follows that
$\pi_\alpha^{-1}(\oA^{|\alpha|})\cap\uoZ^\alpha\cong\oA^{|\alpha|}\times\BG_m^{|\alpha|}$,
and $\pi_\alpha^{-1}(\oA^\alpha)\cap\oZ^\alpha\cong
(\oA^{|\alpha|}\times\BG_m^{|\alpha|})/S_\alpha$
(with respect to the diagonal action).

In case $\alpha=\beta+\gamma,\ \beta=\sum_{i\in \II}b_i\alpha_i,\
\gamma=\sum_{i\in \II}c_i\alpha_i$, according to~\cite[Theorem~1.6.(3)]{bdf},
the factorization isomorphism ${\mathfrak f}_{\beta,\gamma}\colon
\uoZ^\alpha|_{(\BA^{|\beta|}\times\BA^{|\gamma|})_{\on{disj}}}\iso
(\uoZ^\beta\times\uoZ^\gamma)|_{(\BA^{|\beta|}\times\BA^{|\gamma|})_{\on{disj}}}$ takes
$(w_{i,r},y_{i,r})_{i\in \II}^{1\leq r\leq a_i}$ to
\begin{equation} \label{149}
\left((w_{i,r},y_{i,r}
\prod_{b_i+1\leq s\leq a_i}(w_{i,r}-w_{i,s}))_{i\in \II}^{1\leq r\leq b_i},(w_{i,r},y_{i,r}
\prod_{1\leq s\leq b_i}(w_{i,r}-w_{i,s}))_{i\in \II}^{b_i+1\leq r\leq a_i}\right).
\end{equation}
Here $(\BA^{|\beta|}\times\BA^{|\gamma|})_{\on{disj}}$ is the open subset of
$\BA^{|\beta|}\times\BA^{|\gamma|}$ formed by all the configurations where none of the
first $|\beta|$ points meets any of the last $|\gamma|$ points.

\begin{Remark}
\label{4dim}
For a future use we recall the examples of $\uoZ^\gamma$ for $|\gamma|=2$,
see~\cite[5.5,~5.6]{bdf}. In case $\gamma=\alpha_i+\alpha_j$ and $i,j$ are
not connected by an edge of the Dynkin diagram of $G$, we have
$\BC[\uoZ^\gamma]=\BC[w_i,w_j,y_i^{\pm1},y_j^{\pm1}]$.
In case $\gamma=\alpha_i+\alpha_j$ and $i,j$ are connected by an edge, we have
$\BC[\uoZ^\gamma]=\BC[w_i,w_j,y_i,y_j,y_{ij}^{\pm1}]/(y_iy_j-y_{ij}(w_j-w_i))$.
In case $\gamma=2\alpha_i$, we have
$\BC[\uoZ^\gamma]=\BC[w_{i,1},w_{i,2},y_{i,1}^{\pm1},y_{i,2}^{\pm1},\xi]/
(y_{i,1}-y_{i,2}-\xi(w_{i,1}-w_{i,2}))$.
\end{Remark}

\subsection{Generalized transversal slices}
\label{nondom}
In this subsection $\lambda$ is a dominant coweight of $G$, and $\mu\leq\lambda$
is an {\em arbitrary} coweight of $G$, not necessarily dominant, such that
$\alpha:=\lambda-\mu=\sum_{i\in \II}a_i\alpha_i,\ a_i\in\BN$.
We will define the analogues of slices
$\oW^\lambda_{G,\mu}$ of~\cite[Section~2]{bf14}
and prove that they are the Coulomb branches of the
corresponding quiver gauge theories.

Recall the convolution diagram
$\ol{\Gr}{}^\lambda_G\stackrel{\bp}{\longleftarrow}\CG\CZ^{-\mu}_\lambda
\stackrel{\bq}{\longrightarrow}Z^{\alpha^*}$ of~\cite[11.7]{mf}.
Here $\CG\CZ^{-\mu}_\lambda$ is the moduli space of the following data:\par
(a) a $G$-bundle $\scP$ on $\BP^1$. \par
(b) A trivialization $\sigma\colon
\scP_{\on{triv}}|_{\BP^1\setminus\{0\}}\iso\scP|_{\BP^1\setminus\{0\}}$ having a pole
of degree $\leq\lambda$ at $0\in\BP^1$. This means that for an irreducible
$G$-module $V^{\lambdavee}$ and the associated vector bundle
$\CV^{\lambdavee}_\scP$ on $\BP^1$ we have
$V^{\lambdavee}\otimes\CO_{\BP^1}(-\langle\lambda,\lambdavee\rangle\cdot0)
\subset\CV^{\lambdavee}_\scP\subset
V^{\lambdavee}\otimes\CO_{\BP^1}(-\langle w_0\lambda,\lambdavee\rangle\cdot0)$.
\par
(c) a generalized $B$-structure $\phi$
on $\scP$ of degree $w_0\mu$ having no defect at $\infty\in\BP^1$ and having
fiber $B_-\subset G$ at $\infty\in\BP^1$ (with respect to the trivialization
$\sigma$ of $\scP$ at $\infty\in\BP^1$). This means in particular that for an irreducible
$G$-module $V^{\lambdavee}$ and the associated vector bundle
$\CV^{\lambdavee}_\scP$ on $\BP^1$ we are given an invertible subsheaf
$\CL_{\lambdavee}\subset\CV^{\lambdavee}_\scP$ of degree
$-\langle w_0\mu,\lambdavee\rangle$.\par
Now $\bp$ forgets $\phi$, while $\bq$ sends $(\scP,\sigma,\phi)$ to a
collection of invertible subsheaves
$\CL_{\lambdavee}(\langle w_0\lambda,\lambdavee\rangle\cdot0)\subset
V^{\lambdavee}\otimes\CO_{\BP^1}$. This collection will be denoted by
$\sigma^{-1}\phi(w_0\lambda\cdot0)$ for short. Clearly, $\deg\sigma^{-1}\phi(w_0\lambda\cdot0)=\alpha^*$, i.e.
$\deg\CL_{\lambdavee}(\langle w_0\lambda,\lambdavee\rangle\cdot0)=
\langle w_0\lambda-w_0\mu,\lambdavee\rangle=
-\langle\alpha^*,\lambdavee\rangle$. Note that the range of $\bq$ in~\cite[11.7]{mf}
is erroneously claimed to be $Z^\alpha$ as opposed to $Z^{\alpha^*}$.

We have an open subvariety $\oGZ^{-\mu}_\lambda\subset\CG\CZ^{-\mu}_\lambda$
formed by all the triples
$(\scP,\sigma,\phi)$ such that $\phi$ has no defects (i.e.\ is a genuine
$B$-structure). We define the generalized slice
$\oW^\lambda_{\mu}:=\oGZ^{-\mu}_\lambda$.
To avoid a misunderstanding about possible nilpotents in the structure sheaf,
let us rephrase the definition. Let $'\!\on{Bun}_G(\BP^1)$ be the moduli stack of
$G$-bundles on $\BP^1$ with a $B$-structure at $\infty\in\BP^1$.
Let $'\ol{\on{Bun}}_B^{w_0\mu}(\BP^1)$
be the moduli stack of degree $w_0\mu$ generalized $B$-bundles on
$\BP^1$ having no defect at $\infty\in\BP^1$. Let $\on{Bun}^{w_0\mu}_B(\BP^1)$
be its open substack formed by the genuine $B$-bundles.
Finally, we equip $\ol\Gr{}^\lambda_G$ with the reduced scheme structure. Then
$\CG\CZ^{-\mu}_\lambda:=\ol\Gr{}^\lambda_G\times_{'\!\on{Bun}_G(\BP^1)}{}'\ol{\on{Bun}}_B^{w_0\mu}(\BP^1)$,
and $\oW^\lambda_{\mu}=\oGZ^{-\mu}_\lambda:=
\ol\Gr{}^\lambda_G\times_{'\!\on{Bun}_G(\BP^1)}\on{Bun}_B^{w_0\mu}(\BP^1)$.
Note that $\oW^\lambda_{\mu}$ is reduced since it is generically reduced and
Cohen-Macaulay (see~\lemref{lem:CMBD} below).

We denote by $s^\lambda_\mu\colon \oW^\lambda_{\mu}\to Z^{\alpha^*}$ the restriction
of $\bq\colon \CG\CZ^{-\mu}_\lambda\to Z^{\alpha^*}$ to
$\oW^\lambda_{\mu}=\oGZ^{-\mu}_\lambda\subset\CG\CZ^{-\mu}_\lambda$.
Note that when $\mu$ is dominant,
$\bp\colon \oGZ^{-\mu}_\lambda\to\ol{\Gr}{}^\lambda_G$ is a locally closed
embedding~\cite[Remark~2.9]{bf14}, and the image coincides with the
transversal slice $\oW^\lambda_{G,\mu}$ in the affine Grassmannian
$\Gr_G$~\cite[Section~2]{bf14}, hence the name and notation.
However, when $\mu$ is nondominant, the restriction of
$\bp\colon \oGZ^{-\mu}_\lambda\to\ol{\Gr}{}^\lambda_G$ is not a locally closed
embedding.

\subsection{Determinant line bundles and Hecke correspondences}
\label{Determinant}
We recall that
given a family $f\colon \calX\to S$ of smooth projective curves and two line bundles
$\calL_1$ and $\calL_2$ on $\calX$ Deligne defines a line bundle
$\langle\calL_1,\calL_2\ra$ on $S$~\cite[Section~7]{Deligne}.
In terms of determinant bundles the definition is simply
\begin{equation}\label{deligne}
\langle\calL_1,\calL_2\ra=\det Rf_*(\calL_1\ten\calL_2)\ten\det Rf_*(\calO_{\calX})\ten
(\det Rf_*(\calL_1)\ten \det Rf_*(\calL_2))^{-1}.
\end{equation}
Deligne shows that the resulting pairing
$\Pic(\calX)\x\Pic(\calX)\to \Pic(S)$ is symmetric (obvious) and bilinear (not obvious).
%------------------------------------------------------------------------------------------------

Let $Y$ (resp.\ $Y^{\vee}$) denote
the coweight (resp.\ weight) lattice of $T$.
Let $(\cdot,\cdot)$ be an even pairing on $Y$.
Let also $X$ be a smooth projective curve
and let $\Bunt$ denote the moduli stack of
$T$-bundles on $X$. Then to the above data one associates a line bundle
$\calL_T$ on $\Bunt$ in the following way. Let $e_1,...,e_n$ be a basis of $Y$ and let
$f_1,...,f_n$ be the dual basis (of the dual lattice). For every $i=1,...,n$ let $\calL_i$ denote the
line bundle on $\Bunt\x X$ associated to the weight $f_i$.
Let also $a_{ij}=(e_i,e_j)\in\ZZ$. Then we define
\begin{equation}\label{detformula}
\calL_T=(\bigotimes\limits_{i=1}^n \langle\calL_i,\calL_i\ra^{\ten \frac{a_{ii}}{2}})\otimes
(\bigotimes\limits_{1\leq i<j\leq n} \langle\calL_i,\calL_j\ra^{\ten a_{ij}}).
\end{equation}
It is easy to see that $\calL_T$ does not depend on the choice of the basis (here, of course, we have to use
the statement that Deligne's pairing is bilinear).
%----------------------------------------------------------------------------------

We have natural
maps $\grp\colon \Bunb\to \Bunt,\ \grq\colon \Bunb\to \Bung$.
Let $(\cdot,\cdot)$ be a pairing as above. Let us in addition assume that it is
$W$-invariant.
Let $\calL_T$ be the corresponding determinant bundle. Faltings~\cite{Faltings}
shows that the pullback $\grp^*\calL_T$
descends naturally to $\Bung$, i.e.\ there exists a (canonically defined) line
bundle $\calL_G$ on $\Bung$ with an isomorphism
$\grp^*\calL_T\simeq \grq^*\calL_G$. The pullback of $\CL_G$ under the natural
morphism $\Gr_G\to\Bung$ is the determinant line bundle on the affine
Grassmannian; it will be also denoted $\CL_G$ or even simply $\CL$ when no
confusion is likely.

\subsection{Basic properties of generalized transversal slices}
\label{basic_proper}
The convolution diagram $\CG\CZ^{-\mu}_\lambda$ is
equipped with the tautological morphism $\bfr$ to the stack
$\overline{\on{Bun}}_B(\BP^1)$ (see~\cite[Section~1]{bfgm} for notation).
The boundary $\partial\overline{\on{Bun}}_B(\BP^1):=
\overline{\on{Bun}}_B(\BP^1)\setminus\on{Bun}_B(\BP^1)$ is a Cartier divisor, and
$\CO_{\CG\CZ^{-\mu}_\lambda}(\bfr^{-1}(\partial\overline{\on{Bun}}_B(\BP^1)))=\bp^*\CL$
~\cite[Proof of~Theorem~11.6]{BFG} where $\CL$ is the very ample determinant
line bundle on $\Gr_G$.

\begin{Lemma}
\label{lem:affine}
$\oW^\lambda_{\mu}$ is an affine variety.
\end{Lemma}

\begin{proof}
The morphism $(\bp,\bq)\colon \CG\CZ^{-\mu}_\lambda\to\ol{\Gr}{}^\lambda_G\times
Z^{\alpha^*}$ is a closed embedding. Since $Z^{\alpha^*}$ is affine, we conclude
that $\bp\colon \CG\CZ^{-\mu}_\lambda\to\ol{\Gr}{}^\lambda_G$ is affine.
The complement $\CG\CZ^{-\mu}_\lambda\setminus\oGZ^{-\mu}_\lambda=
\br^{-1}(\partial\overline{\on{Bun}}_B(\BP^1))$, but
$\CO_{\CG\CZ^{-\mu}_\lambda}(\br^{-1}(\partial\overline{\on{Bun}}_B(\BP^1)))=\bp^*\CL$
is the very ample determinant line bundle on $\CG\CZ^{-\mu}_\lambda$ since
$\bp$ is affine. Hence the complement
$\CG\CZ^{-\mu}_\lambda\setminus\br^{-1}(\partial\overline{\on{Bun}}_B(\BP^1))=
\oW^\lambda_{\mu}$ is affine.
\end{proof}

The proof of the following lemma is contained in a more general proof
of~\lemref{lem:CMBD} below:
\begin{Lemma}
\label{lem:CMslice}
$\oW^\lambda_{\mu}$ is Cohen-Macaulay. \qed
\end{Lemma}

\begin{Lemma}
\label{semismall}
The composition
$\pi_{\alpha^*}\circ s^\lambda_\mu\colon \oW^\lambda_{\mu}\to\BA^{\alpha^*}$ is flat.
\end{Lemma}

\begin{proof}
Since $\oW^\lambda_{\mu}$ is Cohen-Macaulay,
%(end even Gorenstein~\cite[Theorem~2.5.(1)]{BrKa}),
it suffices to prove that all the fibers of $\pi_{\alpha^*}\circ s^\lambda_\mu$
have the same dimension $|\alpha|$.
\begin{NB} and then apply~\cite[Theorem~23.1]{matsumura}.
\end{NB}
To this end, for $\beta\leq\alpha$, we consider a locally closed subvariety
$Z^{\alpha^*}_{\beta^*}\subset Z^{\alpha^*}$ formed by all the based quasimaps
whose defect at $0\in\BA^1$ has degree precisely $\beta^*$. Note that
$Z^{\alpha^*}_{\beta^*}$ is isomorphic to an open subvariety in $Z^{\alpha^*-\beta^*}$.
It is enough to prove that for $\varphi\in Z^{\alpha^*}_{\beta^*}$, we have
$\dim(s^\lambda_\mu)^{-1}(\varphi)\leq|\beta|$.
%According to~\cite[Remark~2.9]{bf14}, $s^\lambda_\mu$ coincides with
%the convolution morphism $\bq$ in the convolution diagram
%$\ol{\Gr}{}^\lambda_G\stackrel{\bp}{\longleftarrow}\CG\CZ^{-\mu}_\lambda
%\stackrel{\bq}{\longrightarrow}Z^\alpha$ (restricted to an open subset in
%$\CG\CZ^{-\mu}_\lambda$).
Now the desired dimension estimate follows from
the semismallness of $\bq$~\cite[Lemma~12.9.1]{mf}.
\end{proof}

\subsection{A symmetric definition of generalized slices}
\label{symmetric}
We slightly modify our definition of the transversal slices.

Given arbitrary coweights $\mu_-,\mu_+$ such that $\mu_-+\mu_+=\mu$
we consider the moduli space $\oW^\lambda_{\mu_-,\mu_+}$ of the following data:
(a) $G$-bundles $\scP_-,\scP_+$ on $\BP^1$;
(b) an isomorphism
$\sigma\colon \scP_-|_{\BP^1\setminus\{0\}}\iso\scP_+|_{\BP^1\setminus\{0\}}$ having a
pole of degree $\leq\lambda$ at $0\in\BP^1$; (c) a trivialization of
$\scP_-=\scP_+$ at $\infty\in\BP^1$; (d) a reduction $\phi_-$ of $\scP_-$ to a
$B_-$-bundle (a $B_-$-structure on $\scP_-$) such that the induced $T$-bundle
has degree $-w_0\mu_-$, and the fiber of $\phi_-$ at $\infty\in\BP^1$ is
$B\subset G$; (e) a reduction $\phi_+$ of $\scP_+$ to a $B$-bundle (a
$B$-structure on $\scP_+$) such that the induced $T$-bundle has degree $w_0\mu_+$,
and the fiber of $\phi_+$ at $\infty\in\BP^1$ is $B_-\subset G$.

Note that the trivial $G$-bundle on $\BP^1$ has a unique $B_-$-reduction
of degree 0 with fiber $B$ at $\infty$. Conversely, a $G$-bundle $\scP_-$
with a $B_-$-structure of degree 0 is necessarily trivial, and its
trivialization at $\infty$ uniquely extends to the whole of $\BP^1$.
Hence $\oW^\lambda_{0,\mu}=\oW^\lambda_{\mu}$.

For arbitrary $\oW^\lambda_{\mu_-,\mu_+}$,
the $G$-bundles $\scP_-,\scP_+$ are identified via $\sigma$ on
$\BP^1\setminus\{0\}$, so they are both equipped with $B$ and $B_-$-structures
transversal around $\infty\in\BP^1$, that is they are both equipped with a
reduction to a $T$-bundle around $\infty\in\BP^1$. So $\scP_\pm=\scP_\pm^T\times^TG$
for certain $T$-bundles $\scP_\pm^T$ around $\infty\in\BP^1$, trivialized at
$\infty\in\BP^1$. The modified $T$-bundles
$'\scP_\pm^T:=\scP_\pm^T(w_0\mu_-\cdot\infty)$
are also trivialized at $\infty\in\BP^1$ and canonically isomorphic to
$\scP_\pm^T$ off $\infty\in\BP^1$. We define $'\scP_\pm$ as the
result of gluing $\scP_\pm$ and $'\scP_\pm^T\times^TG$ in the punctured
neighbourhood of $\infty\in\BP^1$. Then the isomorphism
$\sigma\colon '\scP_-|_{\BP^1\setminus\{0,\infty\}}\iso\
'\scP_+|_{\BP^1\setminus\{0,\infty\}}$ extends to $\BP^1\setminus\{0\}$,
and $\phi_\pm$ also extend from $\BP^1\setminus\{\infty\}$ to a
$B$-structure $'\phi_+$ in $'\scP_+$ of degree $w_0\mu$
(resp.\ a $B_-$-structure $'\phi_-$ in $'\scP_-$ of degree 0).

This defines an isomorphism $\oW^\lambda_{\mu_-,\mu_+}\simeq\oW^\lambda_{\mu}$.

\subsection{Multiplication of slices}
\label{multiplication}
Given $\lambda_1\geq\mu_2$ and $\lambda_2\geq\mu_2$ with $\lambda_1,\lambda_2$
dominant, we think of $\oW^{\lambda_1}_{\mu_1}$ (resp.\ $\oW^{\lambda_2}_{\mu_2}$)
in the incarnation $\oW^{\lambda_1}_{\mu_1,0}$ (resp.\ $\oW^{\lambda_2}_{0,\mu_2}$).
Given $(\scP^1_\pm,\sigma_1,\phi^1_\pm)\in\oW^{\lambda_1}_{\mu_1,0}$ and
$(\scP^2_\pm,\sigma_2,\phi^2_\pm)\in\oW^{\lambda_1}_{0,\mu_2}$,
we consider $(\scP^1_-,\scP^2_+,\sigma_2\circ\sigma_1,\phi^1_-,\phi^2_+)\in
\oW^{\lambda_1+\lambda_2}_{\mu_1,\mu_2}=\oW^{\lambda_1+\lambda_2}_{\mu_1+\mu_2}$
(note that $\scP^2_-$ is canonically trivialized as in~\secref{symmetric},
and $\scP^1_+$ is canonically trivialized for the same reason, so that
$\scP^1_+=\scP^2_-$).
This defines a multiplication morphism $\oW^{\lambda_1}_{\mu_1}\times
\oW^{\lambda_2}_{\mu_2}\to\oW^{\lambda_1+\lambda_2}_{\mu_1+\mu_2}$.\footnote{We learnt
of this multiplication from J.~Kamnitzer, D.~Gaiotto and T.~Dimofte.}

In particular, taking $\mu_2=\lambda_2$ so that $\oW^{\lambda_2}_{\lambda_2}$ is a
point and $\oW^{\lambda_1}_{\mu_1}\times\oW^{\lambda_2}_{\lambda_2}=\oW^{\lambda_1}_{\mu_1}$,
we get a stabilization morphism
$\oW^{\lambda_1}_{\mu_1}\to\oW^{\lambda_1+\lambda_2}_{\mu_1+\lambda_2}$.

\subsection{Involution}
\label{involution}
For the same reason as in~\secref{symmetric},
$\scP_+$ in $\oW^\lambda_{\mu,0}$ is canonically trivialized,
so we obtain a morphism $'\bp\colon \oW^\lambda_{\mu,0}\to\ol\Gr{}^{-w_0\lambda}_G$,
sending the data of $(\scP_\pm,\sigma,\phi_\pm)$ to
$(\scP_+=\scP_{\on{triv}}\stackrel{\sigma^{-1}}{\longrightarrow}\scP_-)$.
Also, recalling the ``symmetric'' definition of zastava~\cite[2.6]{bdf},
we obtain a morphism $'s^\lambda_\mu\colon \oW^\lambda_{\mu,0}\to Z^{\alpha^*}$.
Namely, it takes a collection $(\CL_{\lambdavee}^+\subset\CV^{\lambdavee}_{\scP_+}=
V^{\lambdavee}\otimes\CO_{\BP^1})$
to a collection of invertible subsheaves
$\CL_{\lambdavee}^+(\langle-\lambda,\lambdavee\rangle\cdot0)\subset
\CV^{\lambdavee}_{\scP_-}$. This transformed
generalized $B$-structure will be denoted by $\sigma^{-1}\phi_+(-\lambda\cdot0)$ for short.
Finally, we have an isomorphism $\iota^\lambda_\mu\colon \oW^\lambda_{\mu}=
\oW^\lambda_{0,\mu}\iso\oW^\lambda_{\mu,0}$ obtained by an application of the
Cartan involution $\fC$ of $G$ (interchanging $B$ and $B_-$, and acting
on $T$ as $t\mapsto t^{-1}$): replacing $(\scP_-,\scP_+,\phi_-,\phi_+)$ by
$(\fC\scP_+,\fC\scP_-,\fC\phi_+,\fC\phi_-)$, and $\sigma$ by $\fC\sigma^{-1}$.
Clearly, $\pi_{\alpha^*}\circ\ 's^\lambda_\mu\circ\iota^\lambda_\mu=
\pi_{\alpha^*}\circ s^\lambda_\mu\colon \oW^\lambda_{\mu}\to\BA^{\alpha^*}$.
Indeed, for $(\scP_\pm,\sigma,\phi_\pm)\in\oW^\lambda_{0,\mu}=\oW^\lambda_{\mu}$,
the $\II$-colored divisor $\pi_{\alpha^*}\circ s^\lambda_\mu(\scP_\pm,\sigma,\phi_\pm)$
on $\BP^1$ measures the nontransversality of $\phi_-$ and $\sigma^{-1}\phi_+(-\lambda\cdot0)$,
while
$\pi_{\alpha^*}\circ\ 's^\lambda_\mu\circ\iota^\lambda_\mu(\scP_\pm,\sigma,\phi_\pm)$
measures the nontransversality of $\sigma\phi_-(-\lambda\cdot0)$ and $\phi_+$, and these
two measures coincide manifestly.

Under the identification $\oW^\lambda_{\mu_-,\mu_+}\simeq\oW^\lambda_{\mu}$
of~\secref{symmetric}, the isomorphism $\iota^\lambda_\mu$
becomes an involution\footnote{We thank J.~Kamnitzer who has convinced us
such an involution should exist.}
$\iota^\lambda_\mu\colon \oW^\lambda_{\mu}\iso \oW^\lambda_{\mu}$.

\subsection{Divisors in the convolution diagram}
\label{divisors}
For a future use we describe certain divisors in the convolution diagram.
We define a divisor $\oE_i\subset\oW^\lambda_{0,\mu}=\oGZ^{-\mu}_\lambda$
as the subvariety formed by the data $(\scP,\sigma,\phi)$ such that the
transformed $B$-structure $\bq(\scP,\sigma,\phi)=\sigma^{-1}\phi(w_0\lambda\cdot0)$
in the trivial bundle
$\scP_{\on{triv}}$ acquires the defect of color $i$ at $0\in\BP^1$ (the defect
may be possibly bigger than $\alpha_i\cdot0$). We define a divisor
$E_i\subset\CG\CZ^{-\mu}_\lambda$ as the closure of $\oE_i$.
Thus $E:=\bigcup_{i\in \II}E_i$ is the exceptional divisor of
$\bq\colon \CG\CZ^{-\mu}_\lambda\to Z^{\alpha^*}$, and
$s^\lambda_\mu\colon \oGZ^{-\mu}_\lambda\to Z^{\alpha^*}$
restricted to $\oGZ^{-\mu}_\lambda\setminus\oE$
(where $\oE:=\bigcup_{i\in \II}\oE_i$) induces an isomorphism
$\oGZ^{-\mu}_\lambda\setminus\oE\iso\oZ^{\alpha^*}$.
Note that $E_i$ can be empty if $\lambda$ is nonregular.

Similarly, we define a divisor $\oE'_i\subset\oW^\lambda_{\mu,0}$ as the
subvariety formed by the data $(\scP_\pm,\sigma,\phi_\pm)$ such that the
transformed $B$-structure
$'s^\lambda_\mu(\scP_\pm,\sigma,\phi_\pm)=\sigma^{-1}\phi_+(-\lambda\cdot0)$ in
$\scP_-$ acquires the defect of color $i$ at $0\in\BP^1$.

\begin{Lemma}
\label{lem:full}
The full preimage $(s^\lambda_\mu)^*(\pi_{\alpha^*}^*(\BA^{\alpha^*}_i))=
\oE_i\cup(\iota^\lambda_\mu)^{-1}(\oE'_i)$.
\end{Lemma}

\begin{proof}
At a general
point of $(\iota^\lambda_\mu)^{-1}(\oE'_i)\subset\oW^\lambda_{\mu}$ the
transformed $B$-structure $\bq(\scP,\sigma,\phi)=\sigma^{-1}\phi(w_0\lambda\cdot0)$
in the trivial bundle
$\scP_{\on{triv}}$ has no defect but at $0\in\BP^1$ is not transversal
to $B$: it lies in position $s_i$ with respect to $B$.
Indeed, if $\BA^{\alpha^*}_i$ denotes the divisor formed by the configurations where
at least on point of color $i$ meets $0\in\BA^1$, then
$\pi_{\alpha^*}\circ\ 's^\lambda_\mu(\oE'_i)\subset\BA^{\alpha^*}_i$, and hence
$\pi_{\alpha^*}\circ s^\lambda_\mu(\iota^\lambda_\mu)^{-1}(\oE'_i)\subset\BA^{\alpha^*}_i$.
Now the full preimage $(\pi_{\alpha^*}\circ s^\lambda_\mu)^*(\BA^{\alpha^*}_i)$
a priori lies in the union of the exceptional divisor $\oE$ and the strict
transform $(s^\lambda_\mu)^{-1}_*(\pi_{\alpha^*}^*\BA^{\alpha^*}_i)$ of the divisor
$\pi_{\alpha^*}^*\BA^{\alpha^*}_i$. At a general point of the component
$\oE_j,\ j\ne i$, the degree of the defect of $s^\lambda_\mu(\scP,\sigma,\phi)$
at $0$ is exactly $\alpha_j$, hence the intersection of $\oE_j$ with
the full preimage of $\BA^{\alpha^*}_i$ is not a divisor. Thus
$(\pi_{\alpha^*}\circ s^\lambda_\mu)^*(\BA^{\alpha^*}_i)=\oE_i\cup
(s^\lambda_\mu)^{-1}_*(\pi_{\alpha^*}^*\BA^{\alpha^*}_i)$,
and the strict transform $(s^\lambda_\mu)^{-1}_*(\pi_{\alpha^*}^*\BA^{\alpha^*}_i)$
must coincide with $(\iota^\lambda_\mu)^{-1}(\oE'_i)$.
We conclude that the full preimage
$(s^\lambda_\mu)^*(\pi_{\alpha^*}^*(\BA^{\alpha^*}_i))=
\oE_i\cup(\iota^\lambda_\mu)^{-1}(\oE'_i)$.
\end{proof}

\subsection{Factorization}
\label{factorization}
For $n\in\BN$, let $\cS_n$ stand for a hypersurface in $\BA^3$ with
coordinates $x,y,w$ cut out by an equation $xy=w^n$ (in particular,
$\cS_0\simeq\BG_m\times\BA^1$). Let $\varPi\colon \cS_n\to\BA^1$ stand for the projection
onto the line with $w$ coordinate. Given $i\in \II$ such that $a_i\geq1$
(recall that $\alpha=\sum_{i\in \II}a_i\alpha_i$) we identify $\BA^{\alpha_{i^*}}$
with $\BA^1$, and we set $\beta:=\alpha-\alpha_i$. We denote by
$\BG_m^{\beta^*}\subset\BA^{\beta^*}$ the open subset formed by all the colored
configurations such that none of the points equals $0\in\BA^1$.
We denote by $(\BG_m^{\beta^*}\times\BA^1)_{\on{disj}}\subset\BG_m^{\beta^*}\times\BA^1$
the open subset equal to the intersection $(\BA^{\beta^*}\times\BA^1)_{\on{disj}}\cap
\BG_m^{\beta^*}\times\BA^1$.

Let $s_n\colon \cS_n\to\BA^1\times\BA^1\simeq Z^{\alpha_{i^*}}$
be the birational isomorphism sending $(x,y,w)$ to $(y,w)$.
Then $s^\lambda_\mu$ gives rise to the birational isomorphism
$$(\BG_m^{\beta^*}\times\BA^1)_{\on{disj}}\times_{\BA^{\alpha^*}}\oW^\lambda_{\mu}\to
(\BG_m^{\beta^*}\times\BA^1)_{\on{disj}}\times_{\BA^{\alpha^*}}Z^{\alpha^*},$$
and $s_{\langle\lambda,\alphavee_i\rangle}$
gives rise to the birational isomorphism
$$(\BG_m^{\beta^*}\times\BA^1)_{\on{disj}}\times_{\BA^{\alpha^*}}(\oZ^{\beta^*}\times
\cS_{\langle\lambda,\alphavee_i\rangle})\to
(\BG_m^{\beta^*}\times\BA^1)_{\on{disj}}\times_{\BA^{\alpha^*}}(\oZ^{\beta^*}\times
Z^{\alpha_{i^*}}).$$
Composing the above birational isomorphisms with the factorization isomorphism
for zastava
$$(\BG_m^{\beta^*}\times\BA^1)_{\on{disj}}\times_{\BA^{\alpha^*}}Z^{\alpha^*}\iso
(\BG_m^{\beta^*}\times\BA^1)_{\on{disj}}\times_{\BA^{\beta^*}\times\BA^1}(Z^{\beta^*}\times
Z^{\alpha_{i^*}})$$ we obtain a birational isomorphism $$\varphi\colon
(\BG_m^{\beta^*}\times\BA^1)_{\on{disj}}\times_{\BA^{\alpha^*}}\oW^\lambda_{\mu}\dasharrow
(\BG_m^{\beta^*}\times\BA^1)_{\on{disj}}\times_{\BA^{\beta^*}\times\BA^1}(\oZ^{\beta^*}\times
\cS_{\langle\lambda,\alphavee_i\rangle}).$$

The aim of this section is the following

\begin{Proposition}
\label{prop:factor}
The birational isomorphism $\varphi$ extends to a regular isomorphism
of the varieties over $(\BG_m^{\beta^*}\times\BA^1)_{\on{disj}}$:
$$(\BG_m^{\beta^*}\times\BA^1)_{\on{disj}}\times_{\BA^{\alpha^*}}\oW^\lambda_{\mu}\iso
(\BG_m^{\beta^*}\times\BA^1)_{\on{disj}}\times_{\BA^{\beta^*}\times\BA^1}(\oZ^{\beta^*}\times
\cS_{\langle\lambda,\alphavee_i\rangle}).$$
\end{Proposition}

The proof will be given after a certain preparation.

\begin{NB}
If $\langle\lambda,\alphavee_i\rangle=0$, the desired isomorphism follows
from $\cS_0\simeq\BG_m\times\BA^1\simeq\oZ^{\alpha_i}$, the usual factorization
for $\oZ^\alpha$, and the observation that the image of $s^\lambda_\mu\colon
(\BG_m^\beta\times\BA^1)_{\on{disj}}\times_{\BA^\alpha}\oW^\lambda_{\mu}\to Z^\alpha$
lands into $\oZ^\alpha\subset Z^\alpha$.

Let now $n:=\langle\lambda,\alphavee_i\rangle>0$, and
$\phi\in\pi_\beta^{-1}(\BG_m^\beta)\subset\oZ^\beta$. We define a finite subscheme
$D_\phi\subset\BG_m$ as $D_\phi:=\pi_\beta(\phi)$. The factorization isomorphism
for zastava defines a locally closed embedding $Z^{\alpha_i}|_{\BA^1\setminus D_\phi}=
\phi\times Z^{\alpha_i}|_{\BA^1\setminus D_\phi}\hookrightarrow Z^\alpha$. We define
$\cS_\phi:=\oW^\lambda_{\mu}\times_{Z^\alpha}(Z^{\alpha_i}|_{\BA^1\setminus D_\phi})$.
We have to prove that the birational isomorphism
$\cS_\phi\to Z^{\alpha_i}|_{\BA^1\setminus D_\phi}\dasharrow \cS_n|_{\BA^1\setminus D_\phi}$
extends to a regular isomorphism $\cS_\phi\iso \cS_n|_{\BA^1\setminus D_\phi}$
of varieties over $\BA^1\setminus D_\phi$. We know this over
$\BG_m\setminus D_\phi$. We also know that the fiber of $\cS_\phi$ over
$0\in\BA^1$ is the spectrum of $\BC[x,y]/(xy)$. Finally, we know that in case
$n=1$, the surface $\cS_\phi$ is smooth, and in case $n>1$ the surface $\cS_\phi$
has a unique singular point with Kleinian singularity of type $A_{n-1}$
(see e.g.~\cite[Theorem~(5.2.1)]{mov}). This completes the proof.
\end{NB}

Let $\bZ^{\alpha^*}\subset Z^{\alpha^*}$ be an open subset
formed by all the based quasimaps $\phi$ satisfying the following two
conditions: (i) the defect $\on{def}\phi$ is at most a simple coroot;
(ii) the multiplicity of the origin $0\in\BP^1$ in the divisor
$\pi_{\alpha^*}(\phi)$ is at most a simple coroot.

Note the three properties: a) The codimension of the complement
$Z^{\alpha^*}\setminus\bZ^{\alpha^*}$ in $Z^{\alpha^*}$ is 2; b) $\bZ^{\alpha^*}$ inherits the
factorization property from $Z^{\alpha^*}$; c) $\bZ^{\alpha^*}$ is smooth.
We consider the open subset
$\bW^\lambda_{\mu}:=(s^\lambda_\mu)^{-1}(\bZ^{\alpha^*})\subset\oW^\lambda_{\mu}$, and
$\bGZ^{-\mu}_\lambda:=\bq^{-1}(\bZ^{\alpha^*})\subset\CG\CZ^{-\mu}_\lambda$.
We set $\bE_i=E_i\cap\bGZ^{-\mu}_\lambda$.
The codimension of the complement $\oW^\lambda_{\mu}\setminus\bW^\lambda_{\mu}$ in
$\oW^\lambda_{\mu}$ is 2. The open embedding
$(\BG_m^{\beta^*}\times\BA^1)_{\on{disj}}\times_{\BA^{\alpha^*}}\bW^\lambda_{\mu}\hookrightarrow
(\BG_m^{\beta^*}\times\BA^1)_{\on{disj}}\times_{\BA^{\alpha^*}}\oW^\lambda_{\mu}$ is an
isomorphism, so we have to prove that $\varphi$ extends to a regular
isomorphism of the varieties over $(\BG_m^{\beta^*}\times\BA^1)_{\on{disj}}$:
$$(\BG_m^{\beta^*}\times\BA^1)_{\on{disj}}\times_{\BA^{\alpha^*}}\bW^\lambda_{\mu}\iso
(\BG_m^{\beta^*}\times\BA^1)_{\on{disj}}\times_{\BA^{\beta^*}\times\BA^1}(\oZ^{\beta^*}\times
\cS_{\langle\lambda,\alphavee_i\rangle})$$
To this end we will identify $\bW^\lambda_{\mu}$ with a certain affine blowup of
$\bZ^{\alpha^*}$. We consider the smooth connected components
$\partial_i\bZ^{\alpha^*},\ i\in \II$, of the boundary divisor
$\bZ^{\alpha^*}\setminus\oZ^{\alpha^*}$.
Recall the divisors $\BA^{\alpha^*}_i\subset\BA^{\alpha^*}$ formed by all the
colored configurations such that at least one point of color $i\in \II$ meets
$0\in\BA^1$. Let $f_i\in\BC[\BA^{\alpha^*}]$ be an equation of $\BA^{\alpha^*}_i$.
Let $\CI_i\subset\CO_{\bZ^{\alpha^*}}$ (resp.\ $\CJ_i\subset\CO_{\bZ^{\alpha^*}}$)
be the ideal of functions vanishing at $\pi_{\alpha^*}^{-1}(\BA^{\alpha^*}_i)$
(resp.\ at $\partial_i\bZ^{\alpha^*}$). We define an ideal
$\CK_i:=\CI_i^{\langle\lambda,\alphavee_{i^*}\rangle}+\CJ_i$, and
$\CK:=\bigcap_{i\in \II}\CK_i$. We define
$\on{Bl}_\CK\bZ^{\alpha^*}$ as the blowup of $\bZ^{\alpha^*}$ at the ideal $\CK$, and
$\on{Bl}_\CK^{\on{aff}}\bZ^{\alpha^*}$ as the complement in $\on{Bl}_\CK\bZ^{\alpha^*}$
to the union of the strict transforms of the divisors
$\partial_i\bZ^{\alpha^*},\ i\in \II$. A crucial step towards~\propref{prop:factor}
is the following
\begin{Proposition}
\label{prop:blow}
The identity isomorphism over $\pi_{\alpha^*}^{-1}(\BG_m^{\alpha^*})$ extends to a
regular isomorphism of the varieties over
$\bZ^{\alpha^*}\colon \bW^\lambda_{\mu}\iso\on{Bl}_\CK^{\on{aff}}\bZ^{\alpha^*}$.
\end{Proposition}

The proof will be given after a couple of lemmas. Recall that
$\CO_{\CG\CZ^{-\mu}_\lambda}(\bfr^{-1}(\partial\overline{\on{Bun}}_B(\BP^1)))=\bp^*\CL$,
and $\bfr^{-1}(\partial\overline{\on{Bun}}_B(\BP^1))$ is the strict transform
$\sum_{i\in \II}\bq^{-1}_*(\partial_iZ^{\alpha^*})$. The pullback of the zastava boundary
divisor will be denoted by $\sum_{i\in \II}\bq^*(\partial_i\bZ^{\alpha^*})$.

\begin{Lemma}
\label{lem:dete}
\textup{(1)} $\on{div}F_{\alpha^*}=\sum_{i\in \II}\partial_i\bZ^{\alpha^*}$;

\textup{(2)} $\on{div}\bq^*F_{\alpha^*}=\sum_{i\in \II}\bq^{-1}_*(\partial_i\bZ^{\alpha^*})+
\sum_{i\in \II}\langle\lambda,\alphavee_{i^*}\rangle\bE_i$;

\textup{(3)} $\bp^*\CL=\CO_{\bGZ^{-\mu}_\lambda}(\sum_{i\in \II}\bq^*(\partial_i\bZ^{\alpha^*})-
\sum_{i\in \II}\langle\lambda,\alphavee_{i^*}\rangle\bE_i)\simeq
\CO_{\bGZ^{-\mu}_\lambda}(-\sum_{i\in \II}\langle\lambda,\alphavee_{i^*}\rangle\bE_i)$.
\end{Lemma}

\begin{proof}
The first assertion is already known. Let us prove the second and
third assertions that are equivalent by the remark preceding the lemma.
Consider the moduli space $\mathcal X^{\lambda}_{\mu}$ of the following data:

1) Two $G$-bundles $\scP_+,\scP_-$ on $\BP^1$.

2) An isomorphism $\sigma\colon\scP_-\to\scP_+$ away from $0\in \BP^1$, which
lies in $G_\CO\backslash\ol\Gr{}_G^{\lambda}$.

3) A $B$-structure $\phi_+$ on the bundle $\scP_+$ of degree $w_0\mu$ such that
the transformed $B$-structure $\sigma^{-1}\phi_+(w_0\lambda\cdot0)$ on $\scP_-$ has no defects.

4) A trivialization of the $B$-bundle $\phi_+$ at $\infty\in\BP^1$.

\noindent
%We have natural maps $\bq_+,\bq_-:\mathcal X^{\lambda}_{\mu}\to \on{Bun}_B(\BP^1)$.
Note that the open subspace of $\mathcal X^{\lambda}_{\mu}$ given by the condition
of triviality of $\scP_-$ is an open subspace of $\oGZ^{-\mu}_\lambda$. We will later introduce
a larger space $\widetilde{\overline{\mathcal X}}{}^{\lambda}_{\mu}$ that is a $B_\CO$-torsor over the whole
of $\oGZ^{-\mu}_\lambda$. We have natural maps
$\pi_+,\pi_-:\mathcal X^{\lambda}_{\mu}\to \on{Bun}_G(\BP^1)$.

Let $\CL_G$ denote the determinant bundle on $\on{Bun}_G(\BP^1)$. Then
the pull-back $\pi_-^*\CL_G$ acquires a natural trivialization coming
from the $B$-structure on $\scP_-$ (note that the associated $T$-bundle has degree $\mu$ and is trivialized
at $\infty$; hence it is canonically isomorphism to the $T$-bundle $\mathcal O(\mu)$.
In fact, the above trivialization is well-defined up to (one) multiplicative scalar;
the scalar is fixed if we trivialize the determinant of the $T$-bundle
$\mathcal O(\mu)$ on $\BP^1$).

On the other hand, consider a bigger moduli space
$\widetilde{\mathcal X}{}^{\lambda}_{\mu}$ of the data 1--3 above together with a
trivialization of
$\scP_+$ in the formal neighbourhood of $0$ compatible with the $B$-structure
(this is a $B_\CO$-torsor over $\mathcal X^{\lambda}_{\mu}$). Then it
acquires a natural map $\bp_+$ to $\Gr_G$; moreover, it follows easily
from~2 and~3 that $\bp_+$ actually lands in the open subvariety
$\ol\Gr{}_G^{-w_0(\lambda)}\cap S_{-w_0(\lambda)}$ (intersection with
a semiinfinite orbit). Indeed, the open subvariety
$\ol\Gr{}_G^{-w_0(\lambda)}\cap S_{-w_0(\lambda)}\subset\ol\Gr{}_G^{-w_0(\lambda)}$
is the moduli space of data $(\scP_-\stackrel{\sigma}{\longrightarrow}\scP_+)$
where $\scP_+$ is trivial on the formal disc, $\sigma$ has a pole of degree
$\leq\lambda$ at $0$, and the transformation
$\sigma^{-1}\phi_+(w_0\lambda\cdot0)$ of the
standard $B$-structure in $\scP_+$ has no defect at $0$. In effect, the latter
condition is satisfied for the torus fixed point $-w_0\lambda\in\Gr_G$, and
since the condition is $N(\CO)$-invariant, the intersection
$\ol\Gr{}_G^{-w_0(\lambda)}\cap S_{-w_0(\lambda)}$ lies in the above moduli space.
However, for the other torus fixed points
$-w_0\lambda\ne\nu\in\ol\Gr{}_G^{-w_0(\lambda)}$ the condition is not satisfied,
and hence the intersection of the above moduli space with $S_\nu$ is empty.

Let us denote by $f$ the projection
$\widetilde{\mathcal X}{}^{\lambda}_{\mu}\to {\mathcal X}^{\lambda}_{\mu}$;
let $\widetilde{\pi}_-=\pi_-\circ f$.
Let us also recall that we denote by $\CL$ the determinant bundle on
$\Gr_G$. We have a canonical isomorphism
$\bp_+^*\CL={\widetilde\pi}{}_-^*\CL_G$. This is so because
$\bp_+^*\CL$ is naturally isomorphic to the ratio of
${\widetilde\pi}{}_-^*\CL_G$ and ${\widetilde\pi}{}_+^*\CL_G$ and the
latter is canonically trivial, since $\scP_+$ is equipped with a $B$-structure
with a fixed reduction to $T$).\footnote{Of course, the fiber of the determinant
bundle at $\scP_-$ can be trivialized as well (for the same reason), but we want
to ignore this here, since a little later we are going to work with a larger space
where $\scP_-$ will only be endowed with a generalized $B$-structure.}

Since the restriction of $\mathcal L$ to
$\ol\Gr{}_G^{-w_0(\lambda)}\cap S_{-w_0(\lambda)}$
acquires a canonical trivialization, we get a trivialization of
$\bp_+^*\mathcal L$. This trivialization is equal to the
pullback under $f$ of the trivialization of $\pi_-^*\CL_G$ discussed above
(since both come from the same reduction of $\scP_-$ to $B$).

Let us now consider a variant of this situation. Namely, we consider a
moduli space $\overline{\mathcal X}{}^{\lambda}_{\mu}$ of the same data as above,
except that in 3) we do not require that the transformed $B$-structure has no
defect. Then  ${\mathcal X}^{\lambda}_{\mu}$ is an open subset of
$\overline{\mathcal X}{}^{\lambda}_{\mu}$.

Similarly, we have the corresponding space
$\widetilde{\overline{\mathcal X}}{}^{\lambda}_{\mu}$.
We will denote the extension of $\widetilde{\pi}_-$
to $\widetilde{\overline{\mathcal X}}{}^{\lambda}_{\mu}$ by
$\widetilde{\overline{\pi}}_-$. Similarly, we have
$\overline{\bp}_+\colon \widetilde{\overline{\mathcal X}}{}^{\lambda}_{\mu}\to \Gr_G$.
The line bundles $\overline{\bp}{}_+^*\mathcal L$ and
$\widetilde{\overline{\pi}}{}_-^*\CL_G$ are again canonically isomorphic,
so we can regard them as the same line bundle.

The above trivialization of this bundle extends to a section (without poles
but with zeroes). We are interested in the divisor of this section.
Namely, let ${\mathcal E}_i$ denote the divisor in
$\overline{\mathcal X}{}^{\lambda}_{\mu}$ corresponding
to the condition that the transformed $B$-structure in $\scP_-$ acquires
the defect of degree at least $\alpha_i$. Then we claim that the corresponding
section of $\widetilde{\overline{\pi}}{}_-^*\CL_G$ vanishes to the order
$\langle-w_0(\lambda),\alphavee_i\rangle=
\langle\lambda,\alphavee_{i^*}\rangle$ on ${\mathcal E}_i$.
This immediately follows from the above, since a similar statement is true on
$\ol\Gr{}_G^{-w_0(\lambda)}$.%\cap S_{-w_0(\lambda)}$.

In effect, assume $\lambda$ regular (the argument in the general case is
similar but requires introducing more notations). Then we have a canonical
projection $\on{pr}\colon \Gr_G^{-w_0\lambda}\to\CB$.
The preimage under $\on{pr}$ of the open $B$-orbit in $\CB$ is nothing but
$\Gr_G^{-w_0(\lambda)}\cap S_{-w_0(\lambda)}$. The complement to the open
$B$-orbit in $\CB$ is the union of Schubert divisors $D_i\subset\CB,\ i\in \II$,
and we have $\CL|_{\Gr_G^{-w_0\lambda}}\cong\CO_{\Gr_G^{-w_0\lambda}}
(\sum_{i\in \II}\langle\-w_0\lambda,\alphavee_i\rangle\on{pr}^*D_i)$ as can be seen
by comparing the $T$-weights in the fibers of both sides at the $T$-fixed
points, see e.g.\ the proof of~\cite[Proposition~3.1]{MV2}.

Finally it remains to note that when $\scP_-$ is trivialized, its
determinant is trivialized as well, and the above section of
$\widetilde{\overline{\pi}}{}_-^*\CL_G$ is a function which coincides with
$\bq^*F_{\alpha^*}$, by its construction in~\cite[Section~4]{bf14}.
\end{proof}

\begin{Lemma}
\label{lem:invo}
The divisor $\on{div}\bq^*\pi_{\alpha^*}^*f_i$ is the sum of $E_i$ and
the strict transform $\bq^{-1}_*(\pi_{\alpha^*}^*\BA^{\alpha^*}_i)$.
\end{Lemma}

\begin{proof}
We must prove that the multiplicity of the exceptional divisor
$E_i$ in $\on{div}\bq^*\pi_{\alpha^*}^*f_i$ equals 1, or equivalently,
the multiplicity of $\oE_i$ in $\on{div}(s^\lambda_\mu)^*\pi_{\alpha^*}^*f_i$ equals 1.
But according to~\lemref{lem:full}
$\on{div}(s^\lambda_\mu)^*\pi_{\alpha^*}^*f_i$ is a sum of multiples of
$\oE_i$ and $(\iota^\lambda_\mu)^{-1}(\oE'_i)$,
and the multiplicities of the summands are equal. The latter divisor
coincides with the strict transform
$(s^\lambda_\mu)^{-1}_*(\pi_{\alpha^*}^*\BA^{\alpha^*}_i)$
and has multiplicity one, hence the former also has multiplicity one.
\end{proof}

\begin{proof}[Proof of \propref{prop:blow}] It suffices
to prove that the identity isomorphism over $\pi_{\alpha^*}^{-1}(\BG_m^{\alpha^*})$
extends to a regular isomorphism of the varieties over
$\bZ^{\alpha^*}\colon \bGZ^{-\mu}_\lambda\iso\on{Bl}_\CK\bZ^{\alpha^*}$.
Indeed, removing the strict transform $\bq_*^{-1}(\partial\bZ^{\alpha^*})=
\bGZ^{-\mu}_\lambda\cap(\CG\CZ^{-\mu}_\lambda\setminus\oGZ^{-\mu}_\lambda)$ we then
obtain the desired isomorphism
$\bW^\lambda_{\mu}\iso\on{Bl}^{\on{aff}}_\CK\bZ^{\alpha^*}$.
We first prove that
$\bq^{-1}\CK\cdot\CO_{\bGZ^{-\mu}_\lambda}\subset\CO_{\bGZ^{-\mu}_\lambda}$
is an {\em invertible} sheaf of ideals. More precisely, we will prove that
$\bq^{-1}\CK\cdot\CO_{\bGZ^{-\mu}_\lambda}\simeq\bp^*\CL$ where
$\CL$ is the very ample determinant line bundle on $\Gr_G$.
It follows from~\lemref{lem:dete}
that $\bq^{-1}(\bigcap_{i\in \II}\CJ_i)\cdot\CO_{\bGZ^{-\mu}_\lambda}$ may be viewed
as the sheaf of sections of $\bp^*\CL$ vanishing at the strict transform
$\sum_{i\in \II}\bq^{-1}_*(\partial_i\bZ^{\alpha^*})$. Also, it follows
from~\lemref{lem:dete}(3) and~\lemref{lem:invo} that
$\bq^{-1}(\bigcap_{i\in \II}\CI_i^{\langle\lambda,\alphavee_{i^*}\rangle})\cdot
\CO_{\bGZ^{-\mu}_\lambda}$ may be viewed as the sheaf of sections of
$\bp^*\CL(-\sum_{i\in \II}\langle\lambda,\alphavee_{i^*}\rangle
\bq^{-1}_*(\pi_{\alpha^*}^*\BA^{\alpha^*}_i))$.
%$=\bp^*\CL^{-1}(-\sum_{i\in \II}\langle\lambda,\alphavee_i\rangle
%\bp^*\overline{p^*D_i^-})$.
The strict transforms
$\bq^{-1}_*(\partial_i\bZ^{\alpha^*})$ and $\bq^{-1}_*(\pi_{\alpha^*}^*\BA^{\alpha^*}_j)$
do not intersect for any $i,j$ (including the case $i=j$). Indeed, for
$(\scP,\sigma,\phi)\in\bq^{-1}_*(\partial_{i^*}\bZ^{\alpha^*})$ the generalized
$B$-structure $\phi$ has defect of order exactly $\alpha_{i^*}$, and the saturated
(nongeneralized) $B$-structure $\widetilde\phi$ is well defined, so that
$(\scP,\sigma,\widetilde\phi)\in\bGZ_\lambda^{-\mu-\alpha_i}$.
For $(\scP,\sigma,\phi)\in\bq^{-1}_*(\partial_{i^*}\bZ^{\alpha^*})\cap
\bq^{-1}_*(\pi_{\alpha^*}^*\BA^{\alpha^*}_j)$ we have
$(\scP,\sigma,\widetilde\phi)\in\bq^{-1}_*(\pi_{\beta^*}^*\BA^{\beta^*}_j)$,
and $\pi_{\alpha^*}\bq(\scP,\sigma,\phi)$ contains the origin with multiplicity
more than a simple coroot. So the above intersection must be empty.
Hence, the sum of subsheaves
$\bp^*\CL(-\sum_{i\in \II}\langle\lambda,\alphavee_{i^*}\rangle
\bq^{-1}_*(\pi_{\alpha^*}^*\BA^{\alpha^*}_i))$ and
$\bp^*\CL(-\sum_{i\in \II}\bq^{-1}_*(\partial_i\bZ^{\alpha^*}))$ in $\bp^*\CL$ is
the whole of $\bp^*\CL$.

Now by the universal property of blowup we obtain a projective morphism
$\Upsilon\colon \bGZ^{-\mu}_\lambda\to\on{Bl}_\CK\bZ^{\alpha^*}$. From the above,
$\Upsilon$ is an isomorphism away from the closed subvariety of codimension 2,
namely the intersection of the exceptional divisor with the strict transform
of the boundary $\partial Z^{\alpha^*}$ and of the divisor $\prod_{i\in \II}f_i=0$.
Hence $\Upsilon$ induces an isomorphism of the Picard groups, and the relative
Picard group of $\Upsilon$ is trivial. Hence $\Upsilon$ is an isomorphism.
\propref{prop:blow} is proved.
\end{proof}

\begin{proof}[Proof of \propref{prop:factor}]
If $\langle\lambda,\alphavee_i\rangle=0$, the desired isomorphism follows
from $\cS_0\simeq\BG_m\times\BA^1\simeq\oZ^{\alpha_{i^*}}$, the usual factorization
for $\oZ^{\alpha^*}$, and the observation that the image of $s^\lambda_\mu\colon
(\BG_m^{\beta^*}\times\BA^1)_{\on{disj}}\times_{\BA^{\alpha^*}}\oW^\lambda_{\mu}\to
Z^{\alpha^*}$ lands into $\oZ^{\alpha^*}\subset Z^{\alpha^*}$.

For arbitrary $\lambda$, we have the isomorphisms
\begin{multline}
(\BG_m^{\beta^*}\times\BA^1)_{\on{disj}}\times_{\BA^{\alpha^*}}\bW^\lambda_{\mu}\iso
(\BG_m^{\beta^*}\times\BA^1)_{\on{disj}}\times_{\BA^{\alpha^*}}\on{Bl}_\CK^{\on{aff}}\bZ^{\alpha^*}\iso\\
\iso(\BG_m^{\beta^*}\times\BA^1)_{\on{disj}}\times_{\BA^{\beta^*}\times\BA^1}\on{Bl}_\CK^{\on{aff}}
(\oZ^{\beta^*}\times Z^{\alpha_{i^*}})\iso
(\BG_m^{\beta^*}\times\BA^1)_{\on{disj}}\times_{\BA^{\beta^*}\times\BA^1}(\oZ^{\beta^*}\times
\cS_{\langle\lambda,\alphavee_i\rangle})
\end{multline}
\propref{prop:factor} is proved.
\end{proof}

\begin{Remark}
\label{non sl}
\propref{prop:factor} along with its proof holds for an arbitrary almost simple
simply connected complex algebraic group $G$, not necessarily simply laced.
\end{Remark}

\subsection{BD slices}\label{subsec:bd-slices}
Recall the definition of Beilinson-Drinfeld slices $\oW^{\unl\lambda}_{\mu}$
from~\cite[2.4]{kwy}. Here $\lambda\geq\mu$ are dominant coweights of $G$,
and $\unl\lambda=(\omega_{i_1},\ldots,\omega_{i_N})$ is a sequence of fundamental
coweights of $G$ such that $\sum_{s=1}^N\omega_{i_s}=\lambda$.
Namely, $\oW^{\unl\lambda}_{\mu}$ is the moduli space of the following data:

(a) a collection of points $(z_1,\ldots,z_N)\in\BA^N$;

(b) a $G$-bundle $\scP$ on $\BP^1$ of isomorphism type $\mu$;

(c) a trivialization (a section) $\sigma$ of $\scP$ on
$\BP^1\setminus\{z_1,\ldots,z_N\}$ with a pole of degree
$\leq\sum_{s=1}^N\omega_{i_s}\cdot z_s$ on the complement,
such that the value of the Harder-Narasimhan flag of $\scP$ at $\infty\in\BP^1$
(where $\scP$ is trivialized via $\sigma$) is compatible with $B_-\subset G$.

Note that the Harder-Narasimhan flag above can be uniquely refined to a full
flag of degree $w_0\mu$ with value $B_-\subset G$ at $\infty\in\BP^1$, and this
flag is the unique flag of degree $w_0\mu$ with the prescribed value at $\infty$.
Hence the above definition of $\oW^{\unl\lambda}_{\mu}$ can be extended to the
case when $\mu\leq\lambda$ is not necessarily dominant (but
$\lambda=\sum_{s=1}^N\omega_{i_s}$ is still dominant) as follows:
$\oW^{\unl\lambda}_{\mu}$ is the moduli space of the following data:

(a) a collection of points $(z_1,\ldots,z_N)\in\BA^N$;

(b) a $G$-bundle $\scP$ on $\BP^1$;

(c) a trivialization (a section) $\sigma$ of $\scP$ on
$\BP^1\setminus\{z_1,\ldots,z_N\}$ with a pole of degree
$\leq\sum_{s=1}^N\omega_{i_s}\cdot z_s$ on the complement;

(d) a $B$-structure $\phi$ on $\scP$ of degree $w_0\mu$ having fiber
$B_-\subset G$
at $\infty\in\BP^1$ (with respect to the trivialization $\sigma$).

If in this definition we allow $B$-structure in (d) to be generalized (but
with no defects at $\infty\in\BP^1$), then we obtain a partial compactification
$\CG\CZ^{-\mu}_{\unl\lambda}\supset\oW^{\unl\lambda}_{\mu}$.
As in~\subsecref{nondom}, let us rephrase the definition to avoid a possible
misunderstanding about nilpotents in the structure sheaf.
We equip $\ol\Gr{}^{\unl\lambda}_{G,BD}$ with the reduced scheme structure. Then
$\CG\CZ^{-\mu}_{\unl\lambda}:=\ol\Gr{}^{\unl\lambda}_{G,BD}\times_{'\!\on{Bun}_G(\BP^1)}{}'\ol{\on{Bun}}_B^{w_0\mu}(\BP^1)$,
and $\oW^\lambda_{\mu}:=
\ol\Gr{}^{\unl\lambda}_{G,BD}\times_{'\!\on{Bun}_G(\BP^1)}\on{Bun}_B^{w_0\mu}(\BP^1)$.

As in~\lemref{lem:affine}
one can prove that $\oW^{\unl\lambda}_{\mu}$ is an affine algebraic variety.
For $\alpha=\lambda-\mu$, we have the convolution diagram
$\ol{\Gr}{}^{\unl\lambda}_{G,BD}\stackrel{\bp}{\longleftarrow}\CG\CZ^{-\mu}_{\unl\lambda}
\stackrel{\bq}{\longrightarrow}Z^{\alpha^*}\times\BA^N$ defined similarly
to~\subsecref{nondom}. %\cite[11.7]{mf}.
In particular, $\bq$ sends $(z_1,\ldots,z_N,\scP,\sigma,\phi)$ to a
collection of invertible subsheaves
$\CL_{\lambdavee}(\sum_{1\leq s\leq N}\langle w_0\omega_{i_s},\lambdavee\rangle\cdot z_s)\subset
V^{\lambdavee}\otimes\CO_{\BP^1}$.
%This collection will be denoted by
%$\sigma^{-1}\phi(w_0\lambda\cdot0)$ for short.
The restriction of $\bq$ to
$\oW^{\unl\lambda}_{\mu}\subset\CG\CZ^{-\mu}_{\unl\lambda}$ is denoted by
$s^{\unl\lambda}_\mu\colon \oW^{\unl\lambda}_{\mu}\to Z^{\alpha^*}\times\BA^N$.
We also have a morphism
$\bfr\colon \CG\CZ^{-\mu}_{\unl\lambda}\to\ol{\on{Bun}}_B(\BP^1)$ forgetting the
data (a,c) above.
Let $f_{i,\unl\lambda}\in\BC[\BA^{\alpha^*}\times\BA^N]$ be
defined as $f_{i,\unl\lambda}(\unl{w},\unl{z})=
\displaystyle{\prod_{\substack{1\leq r\leq a_i\\
1\leq s\leq N\colon i_s={i^*}}}(w_{i,r}-z_s)}$.
By an abuse of notation we will keep the name $f_{i,\unl\lambda}$ for
$\pi_{\alpha^*}^*f_{i,\unl\lambda}\in\BC[Z^{\alpha^*}\times\BA^N]$.
Let $Z^{\alpha^*}\butimes\BA^N\subset Z^{\alpha^*}\times\BA^N$ be an open subset
formed by all the pairs $(\phi,\unl{z})$ of the based quasimaps and
configurations satisfying the following two conditions: (i) the defect of
$\phi$ is at most a simple coroot; (ii) the multiplicity of $z_i$ in the
divisor $\pi_{\alpha^*}(\phi)$ is at most a simple coroot for any $i=1,\ldots,N$.
We define an open subvariety
$\bW^{\unl\lambda}_{\mu}\subset\oW^{\unl\lambda}_{\mu}$
(resp.\ $\bGZ^{-\mu}_{\unl\lambda}\subset\CG\CZ^{-\mu}_{\unl\lambda}$) as
$(s^{\unl\lambda}_\mu)^{-1}(Z^{\alpha^*}\butimes\BA^N)$
(resp.\ $\bq^{-1}(Z^{\alpha^*}\butimes\BA^N)$)

Let $\CI_i^{\unl\lambda}\subset\CO_{Z^{\alpha^*}\butimes\BA^N}$ (resp.\
$\CJ_i\subset\CO_{Z^{\alpha^*}\butimes\BA^N}$) be the ideal generated by $f_{i,\unl\lambda}$
(resp.\ the ideal of functions vanishing at $\partial_i Z^{\alpha^*}\butimes\BA^N$).
We define an ideal $\CK_i:=\CI_i^{\unl\lambda}+\CJ_i$, and
$\CK:=\bigcap_{i\in \II}\CK_i$. We define $\on{Bl}_\CK(Z^{\alpha^*}\butimes\BA^N)$
as the blowup of $Z^{\alpha^*}\butimes\BA^N$ at the ideal $\CK$, and
$\on{Bl}_\CK^{\on{aff}}(Z^{\alpha^*}\butimes\BA^N)$ as the complement in
$\on{Bl}_\CK(Z^{\alpha^*}\butimes\BA^N)$ to the union of the strict transforms of
the divisors $\partial_i Z^{\alpha^*}\butimes\BA^N,\ i\in \II$.

The proof of the following proposition is parallel to the one
of~\propref{prop:blow}.

\begin{Proposition}
\label{prop:blowBD}
The identity isomorphism over $\bigcap_{i\in \II}f_{i,\unl\lambda}^{-1}(\BG_m)$ extends
to the regular isomorphisms of the varieties over $Z^{\alpha^*}\butimes\BA^N\colon
\bGZ^{-\mu}_{\unl\lambda}\iso\on{Bl}_\CK(Z^{\alpha^*}\butimes\BA^N)$
and $\bW^{\unl\lambda}_{\mu}\iso\on{Bl}_\CK^{\on{aff}}(Z^{\alpha^*}\butimes\BA^N)$.
\qed
\end{Proposition}

\begin{Lemma}
\label{lem:CMBD}
$\oW^{\unl\lambda}_{\mu}$ is Cohen-Macaulay.
\end{Lemma}

\begin{proof}
%According to~\cite[Proposition~1.2.4]{zhu},
The natural morphism
$\ol{\Gr}^{\unl\lambda}_{G,BD}\to\BA^N$ is flat with fibers isomorphic to the
products of Schubert varieties in $\Gr_G$, as a consequence
of~\cite[Theorem~1]{fl}. Since these Schubert varieties are Cohen-Macaulay,
we deduce from~\cite[Corollary of~Theorem~23.3]{matsumura} that
$\ol{\Gr}^{\unl\lambda}_{G,BD}$ is Cohen-Macaulay as well.

The morphisms
$\ol{\Gr}^{\unl\lambda}_{G,BD}\stackrel{p}{\to}{}'\!\on{Bun}_G(\BP^1)\leftarrow\on{Bun}_B^{w_0\mu}(\BP^1)$
are Tor-independent since the left morphism is a product locally in the smooth
topology. In effect, let
$'\!\on{Bun}_G(\BP^1)\leftarrow{\mathcal H}^{\unl\lambda}_{G,BD}\to{}'\!\on{Bun}_G(\BP^1)$
be the (Beilinson-Drinfeld-)Hecke correspondence. Then the right projection
is a product locally in the smooth topology. Let
$'\!\on{Bun}_G^{\on{triv}}(\BP^1)\subset{}'\!\on{Bun}_G(\BP^1)$ be the open substack
of trivial $G$-bundles. Then its preimage in ${\mathcal H}^{\unl\lambda}_{G,BD}$
under the left projection to $'\!\on{Bun}_G(\BP^1)$ is
$G\backslash\ol{\Gr}^{\unl\lambda}_{G,BD}$, and the restriction of the right projection to the
preimage is $G\backslash p$.

Now the morphism $'\!\on{Bun}_G(\BP^1)\leftarrow\on{Bun}_B^{w_0\mu}(\BP^1)$
is locally complete intersection (lci) since
both the target and the source are smooth. Hence its base change
$\ol{\Gr}^{\unl\lambda}_{G,BD}\leftarrow\oW^{\unl\lambda}_{\mu}$ is also lci
(see~\cite[Corollary~2.2.3(i)]{illusie}).
Hence the Cohen-Macaulay property of $\ol{\Gr}^{\unl\lambda}_{G,BD}$ implies the
one of the fiber product
$\ol{\Gr}^{\unl\lambda}_{G,BD}\times_{'\!\on{Bun}_G(\BP^1)}\on{Bun}_B^{w_0\mu}(\BP^1)=
\oW^{\unl\lambda}_{\mu}$.\footnote{The last part of the proof is due to
M.~Temkin.}
\end{proof}

\subsection{An embedding into \texorpdfstring{$G(z)$}{G(z)}}
\label{G(z)}
Given a collection $(z_1,\ldots,z_N)\in\BA^N$ we define
$P_{\unl{z}}(z):=\prod_{s=1}^N(z-z_s)\in\BC[z]$. We also define a closed subvariety
$\oW^{\unl{\lambda},\unl{z}}_\mu\subset\oW^{\unl{\lambda}}_\mu$ as the fiber of the
latter over $\unl{z}=(z_1,\ldots,z_N)$.
We construct a locally closed embedding
$\Psi\colon \oW^{\unl{\lambda},\unl{z}}_\mu\hookrightarrow G[z,P^{-1}]$
into an ind-affine scheme as follows.
Similarly to~\subsecref{symmetric}, we have a symmetric definition of
BD slices and an isomorphism $\zeta\colon \oW^{\unl{\lambda},\unl{z}}_\mu=
\oW^{\unl{\lambda},\unl{z}}_{0,\mu}\iso\oW^{\unl{\lambda},\unl{z}}_{\mu,0}$.
We denote $\zeta(\scP_\pm,\sigma,\phi_\pm)$ by
$(\scP'_\pm,\sigma',\phi'_\pm)$. Note that $\scP_-$ and $\scP'_+$
are trivialized, and $\scP'_+$ is obtained from $\scP_+$ by an application of
a certain Hecke transformation at $\infty\in\BP^1$. In particular, we obtain
an isomorphism $\scP_+|_{\BA^1}\iso\scP'_+|_{\BA^1}=\scP_{\on{triv}}|_{\BA^1}$.
Composing it with $\sigma\colon \scP_{\on{triv}}|_{\BA^1\setminus\unl{z}}=
\scP_-|_{\BA^1\setminus\unl{z}}\iso\scP_+|_{\BA^1\setminus\unl{z}}$ we obtain an
isomorphism $\scP_{\on{triv}}|_{\BA^1\setminus\unl{z}}\iso
\scP_{\on{triv}}|_{\BA^1\setminus\unl{z}}$ i.e.\ an element of $G[z,P_{\unl{z}}^{-1}]$.

Note that if $N=0$, then $P_{\unl{z}}=1$, and
$\oW^{\unl{\lambda},\unl{z}}_\mu=\oZ^\alpha$ where
$\alpha=w_0\mu$. Thus we obtain an embedding $\Psi\colon
\oZ^\alpha\hookrightarrow G[z]$, which should be the same as the one
in~\cite[4.2]{MR1625475}.

Here is an equivalent construction of the above embedding due to J.~Kamnitzer.
Given $(\scP_\pm,\sigma,\phi_\pm)\in\oW^{\unl{\lambda},\unl{z}}_{\mu_-,\mu_+}$, we choose
a trivialization of the $B$-bundle $\phi_+|_{\BA^1}$ (resp.\ of the
$B_-$-bundle $\phi_-|_{\BA^1}$); two choices of such a trivialization differ
by the action of an element of $B[z]$ (resp.\ $B_-[z]$). This trivialization
gives rise to a trivialization of the $G$-bundle $\scP_+|_{\BA^1}$
(resp.\ of $\scP_-|_{\BA^1}$), so that $\sigma$ becomes an element of $G(z)$
well-defined up to the left multiplication by an element of $B[z]$ and the
right multiplication by an element of $B_-[z]$, i.e.\ a well defined element
of $B[z]\backslash G(z)/B_-[z]$. Clearly, this element of $G(z)$ lies in
the closure of the double coset $\ol{G[z]z^{\unl{\lambda},\unl{z}}G[z]}$ where
$z^{\unl{\lambda},\unl{z}}:=\prod_{s=1}^N(z-z_s)^{\omega_{i_s}}$. Thus we have constructed
an embedding $\Psi'\colon \oW^{\unl{\lambda},\unl{z}}_{\mu_-,\mu_+}\to B[z]\backslash
\ol{G[z]z^{\unl{\lambda},\unl{z}}G[z]}/B_-[z]$. If we compose with an embedding
$G(z)\hookrightarrow G((z^{-1}))$, then the image of $\Psi'$ lies in
$B[z]\backslash B_1[[z^{-1}]]z^\mu B_{-,1}[[z^{-1}]]/B_-[z]$ where
$B_1[[z^{-1}]]\subset B[[z^{-1}]]$ (resp.\ $B_{-,1}[[z^{-1}]]\subset B_-[[z^{-1}]]$)
stands for the kernel of evaluation at $\infty\in\BP^1$. However, the projection
$B_1[[z^{-1}]]z^\mu B_{-,1}[[z^{-1}]]\to
B[z]\backslash B_1[[z^{-1}]]z^\mu B_{-,1}[[z^{-1}]]/B_-[z]$ is clearly one-to-one.
Summing up, we obtain an embedding
$$\Psi\colon \oW^{\unl{\lambda},\unl{z}}_{\mu_-,\mu_+}\to
B_1[[z^{-1}]]z^\mu B_{-,1}[[z^{-1}]]\bigcap\ol{G[z]z^{\unl{\lambda},\unl{z}}G[z]}.$$
We claim that $\Psi$ is an isomorphism. To see it, we construct the inverse
map to $\oW^{\unl{\lambda},\unl{z}}_{0,\mu}$: given $g(z)\in
B_1[[z^{-1}]]z^\mu B_{-,1}[[z^{-1}]]\bigcap\ol{G[z]z^{\unl{\lambda},\unl{z}}G[z]}$,
we use it to glue $\scP_+$ together with a rational isomorphism
$\sigma\colon \scP_{\on{triv}}=\scP_-\to\scP_+$, and define $\phi_+$ as the
image of the standard $B$-structure in $\scP_{\on{triv}}$ under $\sigma$.

Note that the same space of scattering matrices appears
in~\cite[6.4.1]{2015arXiv150304817B}.

\subsection{An example}
\label{an example}
Let $G=\GL(2)=\GL(V)$ where $V=\BC e_1\oplus\BC e_2$.
Let $N,m\in\BN;\ \unl{\lambda}$ be an $N$-tuple of fundamental coweights
$(1,0)$, and $\mu=(N-m,m)$, so that $w_0\mu=(m,N-m)$.
Let $\CO:=\CO_{\BP^1}$. We fix a collection $(z_1,\ldots,z_N)\in\BA^N$ and define
$P_{\unl{z}}(z):=\prod_{s=1}^N(z-z_s)\in\BC[z]$.
Then $\oW^{\unl{\lambda},\unl{z}}_\mu$ is the moduli space of flags
$(\CO\otimes V\supset\CV\supset\CL)$, where

\textup{(a)} $\CV$ is a 2-dimensional
locally free subsheaf in $\CO\otimes V$ coinciding with $\CO\otimes V$
around $\infty\in\BP^1$ and such that on $\BA^1\subset\BP^1$ the global
sections of $\det\CV$ coincide with $P_{\unl{z}}\BC[z]e_1\wedge e_2$
as a $\BC[z]$-submodule of
$\Gamma(\BA^1,\det(\CO_{\BA^1}\otimes V))=\BC[z]e_1\wedge e_2$.

\textup{(b)} $\CL$ is a line subbundle in $\CV$ of degree
$-m$, assuming the value $\BC e_1$ at $\infty\in\BP^1$.
In particular, $\deg\CV/\CL=m-N$.

On the other hand, let us introduce a closed subvariety
$\CM^{\unl{\lambda},\unl{z}}_\mu$ in $\on{Mat}_2[z]$
formed by all the matrices $\sfM=\begin{pmatrix}A&B\\ C&D\end{pmatrix}$
such that $A$ is a monic polynomial of degree $m$, while the degrees of
$B$ and $C$ are strictly less than $m$, and $\det\sfM=P_{\unl{z}}(z)$.

Finally, let $\on{inv}\colon \on{Mat}_2^*(z)\to\on{Mat}_2^*(z)$ denote the
inversion operation on matrices with nonzero determinant.

\begin{Proposition}[J.~Kamnitzer]
\label{prop:Joel}
The composition $\Phi:=\on{inv}\circ\Psi$ establishes an isomorphism
$\oW^{\unl{\lambda},\unl{z}}_\mu\iso\CM^{\unl{\lambda},\unl{z}}_\mu$.
\end{Proposition}

\begin{proof}
First note that a morphism between two line bundles on $\BP^1$ trivialized
at $\infty\in\BP^1$, viewed as a polynomial in $z$, has a leading term $1$
if and only if the morphism preserves the trivializations at $\infty$.

Let us denote $\Phi(\CO\otimes V\supset\CV\supset\CL)$ by
$\sfM=\begin{pmatrix}A&B\\ C&D\end{pmatrix}\in\on{Mat}_2[z]$.
By construction $\det\sfM$ is proportional to $P_{\unl{z}}(z)$.
If we view $\det\sfM$ as a rational morphism from $\det\CV$ to $\CO$
compatible with trivializations at $\infty$, we deduce that the leading
coefficient of $\det\sfM=1$, i.e.\ $\det\sfM=P_{\unl{z}}(z)$.

Furthermore, the pole of the first column of $\sfM$ at $\infty\in\BP^1$
has order exactly $m$; more precisely, the leading term of $A$ is
$az^m,\ a\in\BC^\times$, while $C$ has a smaller degree. If we view $A$
as a morphism $\CL\to\CO$ compatible with trivializations at $\infty$,
we obtain $a=1$.

Let us consider the involution $\iota^{\unl{\lambda},\unl{z}}_\mu\colon
\oW^{\unl{\lambda},\unl{z}}_\mu\iso\oW^{\unl{\lambda},\unl{z}}_\mu$ defined as
in~\subsecref{involution}. Then by construction,
$\Phi\circ\iota^{\unl{\lambda},\unl{z}}_\mu$ equals the composition of transposition
and $\Phi$. Hence we obtain that $\deg B<m$ as well, so that the image
of $\Phi$ lies in $\CM^{\unl{\lambda},\unl{z}}_\mu$.

Now let us describe the inverse morphism $\mho\colon
\CM^{\unl{\lambda},\unl{z}}_\mu\iso\oW^{\unl{\lambda},\unl{z}}_\mu$.
Given $\sfM\in\CM^{\unl{\lambda},\unl{z}}_\mu$ we view it as a transition matrix
in a punctured neighbourhood of $\infty\in\BP^1$ to glue a vector bundle
$\CV$ which embeds, by construction, as a locally free subsheaf into
$\CO\otimes V$. The morphism
$\sfM\CO_{\BA^1}e_1\hookrightarrow\CO_{\BA^1}\otimes V$ naturally extends to
$\infty\in\BP^1$ with a pole of degree $m$, hence it extends to an embedding
of $\CO(-m\cdot\infty)$ into $\CV\subset\CO\otimes V$. The image of this
embedding is the desired line subbundle $\CL\subset\CV$.

Finally, one can check that $\Phi$ and $\mho$ are inverse to each other.
\end{proof}
Note that this argument is just a special case of the one in~\subsecref{G(z)}.
Indeed, $z^\mu=\on{diag}(z^{N-m},z^m)$, and $$B_1[[z^{-1}]]z^\mu B_{-,1}[[z^{-1}]]=
\left\{\begin{pmatrix}a_{11}&a_{12}\\ a_{21}&a_{22}\end{pmatrix} \mid
\deg(a_{22})=m>\deg(a_{21}),\deg(a_{12})\right\}.$$ Furthermore,
$z^{\unl{\lambda},\unl{z}}=\on{diag}(P(z)^{-1},1)$, so that
$\on{inv}\left(\ol{G[z]z^{\unl{\lambda},\unl{z}}G[z]}\right)$
consists of matrices with entries in
$\BC[z]$ and determinant $P(z)$ up to a scalar multiple.

\subsection{Scattering matrix}\label{subsec:scatter}

The isomorphism between moduli spaces of $G_c$-monopoles and rational
maps is given by the scattering matrix
\cite{MR804459,MR1625475}. Although we do not use this fact, let us
briefly review it following \cite{MR934202}, as it seems closely
related to a version of definition of the zastava due to Drinfeld (see
\cite[2.12]{BFG}).

Let $(A,\Phi)$ be a monopole on a $G_c$-bundle $P$ over $\RR^3$. Let
us assume $G_c = \SU(2)$ for brevity. Let $k$ be the monopole
charge. Therefore the Higgs field has the asymptotic behaviour
\begin{equation*}
    \Phi = \sqrt{-1}\operatorname{diag}(1, -1) -
    \frac{\sqrt{-1}}{2r}\operatorname{diag}(k,-k) + O(r^{-2}).
\end{equation*}

We fix an isomorphism $\RR^3 = \RR\times \CC$ and consider rays
$(t,z)$ ($t\to\pm \infty$). We solve $(\nabla_A - \sqrt{-1}\Phi)s = 0$ along
rays for the associated rank $2$ vector bundle
$P\times_{\SU(2)}\CC^2$. We have two sections $s_0$, $s_1$ along
$t\to\infty$ and $s'_0$, $s'_1$ along $t\to -\infty$.
Here $s_0$ and $s'_0$ are exponentially decaying while $s_1$ and
$s'_1$ are exponentially growing. The scattering matrix is defined as
the transition between $(s_0,s_1)$ and $(s'_0,s'_1)$.
We consider the framed moduli space, i.e., we choose an eigenvector
for $\Phi$ at $+\infty$ with eigenvalue $\sqrt{-1}$. Then $s_0$ is
uniquely determined, while $s_1$ is well-defined up to the addition of
a multiple of $s_0$.
On the other hand, $s'_0$ is well-defined up to a multiple of scalar
as we do not take the framing at $-\infty$.
Therefore the scattering matrix is naturally a map from $\CC$ to the
quotient stack $B\backslash G/U$ where $G = \SL(2)$, $B$ (resp.\ $U$) is the
group of upper triangular (resp.\ uni-triangular) matrices in $G$.
\begin{NB}
    It may be natural to consider $(s_1',s_0')$ instead of
    $(s_0',s_1')$, as $s_0$ and $s_1'$ are eigenvectors for eigenvalue
    $\sqrt{-1}$.
\end{NB}%
This is nothing but a description of the zastava due to Drinfeld, as
explained in \cite[\S2.12]{BFG}.
Moreover these maps make sense for a Riemann surface $X$, not only
$\CC$. As we have learned from Gaiotto, we expect that the scattering
matrix for a monopole on $\RR\times X$ is a map from $X$ to
$B\backslash G/U$.

We write $s_0'(z) = g(z) s_0(z) + f(z) s_1(z)$. Then $f$ and $g$ have
no common zeroes and they are well-defined up to
\begin{enumerate}
      \item multiplying both by an invertible function on $X$,
      \item adding a multiple of $f$ to $g$.
\end{enumerate}
For $X = \CC$, we can uniquely bring it to a pair $(f,g)$ such that
\begin{aenume}
    \item $g$ is a monic polynomial of some degree $k$,
    \item $f$ is a polynomial of degree $<k$.
\end{aenume}
Thus we have a based map $z\mapsto g(z)/f(z)$.

See \cite[App.~A]{2015arXiv150304817B} for the consideration for
singular monopoles.

\section{Quiver gauge theories}
\label{QGT}

We choose an orientation of the Dynkin graph of $G$, and denote the set of oriented arrows by $\Qo$.
\begin{NB}
We denote an element of $\Qo$ by $i\to j$ where $i$ (resp.\ $j$) is
the outgoing (resp.\ incoming) vertex. Even though we do \emph{not}
assume that the graph is simply-laced in general (e.g.\
\secref{affine}), we hope that this does not cause confusion.
\end{NB}%

\subsection{\texorpdfstring{$W=0$}{W=0} cases}
\label{Co_bra}
We set
$V_i=\BC^{a_i}$, and $\GL(V):=\prod_{i\in \II}\GL(V_i)$. The group $\GL(V)$
acts naturally on $\bN=\bN^\alpha:=\bigoplus_{h\in\Qo}\Hom(V_{\vout{h}},V_{\vin{h}})$.
This representation gives rise to the variety of triples
$\CR\to\Gr=\Gr_{\GL(V)}$. The equivariant Borel-Moore homology
$H_*^{\GL(V)_\CO}(\CR)$ equipped the convolution product forms a commutative algebra,
and its spectrum is the Coulomb branch $\cM_C=\cM_C(\GL(V),\bN)$.
We choose a maximal torus $T(V)\subset\GL(V)$
and its identification with $\prod_{i\in \II}\BG_m^{a_i}$.
The basic characters of $T(V)$ are denoted
$\sw_{i,r},\ i\in \II,\ 1\leq r\leq a_i$; their differentials are
$w_{i,r}\in\ft^\vee(V)$.
The generalized roots are $w_{i,r}-w_{i,s},\ r\ne s$, and $w_{i,r}-w_{j,s}$ for
$i\ne j$ vertices connected in the Dynkin diagram.

We consider the algebra homomorphism
$\iota_*\colon H_*^{T(V)_\CO}(\CR_{T(V),\bN_{T(V)}})\to
H_*^{\GL(V)_\CO}(\CR)\otimes_{H^*_{\GL(V)}(\on{pt})}H^*_{T(V)}(\on{pt})$
of~\ref{lem:iotahom}. According to {\em loc.\ cit.}, $\iota_*$ becomes
an isomorphism over ${\mathfrak t}^\circ(V)$ (note that
${\mathfrak t}^\circ(V)/S_\alpha=\oA^\alpha$). We denote by
$\sy_{i,r}\in H_*^{T(V)_\CO}(\CR_{T(V),\bN_{T(V)}})$ the fundamental class of the fiber
of $\CR_{T(V),\bN_{T(V)}}$ over the point of $\Gr_{T(V)}$ equal to the cocharacter
$w^*_{i,r}$ of $T(V)$: an element of the dual basis to
$\{w_{i,r}\}_{i\in \II,\ 1\leq r\leq a_i}$. Finally, we denote
$\sfu_{i,r}\in H_*^{T(V)_\CO}(\Gr_{T(V)})$ the fundamental class of the point
$w^*_{i,r}$. According to~\ref{prop:integrable}, $H_*^{T(V)_\CO}(\Gr_{T(V)})\cong
\BC[{\mathfrak t}(V)\times T^\vee(V)]=\BC[w_{i,r},\sfu_{i,r}^{\pm1}\colon i\in \II,\
1\leq r\leq a_i]$. We define an isomorphism
$\Xi\colon \BC[\oZ^{\alpha}]\otimes_{\BC[\BA^{\alpha}]}\BC[\oA^{\alpha}]\iso
\BC[{\mathfrak t}(V)\times T^\vee(V)]\otimes_{\BC[\BA^\alpha]}\BC[\oA^\alpha]$
identical on  $w_{i,r}$ and sending $y_{i,r}$ to
$\sfu_{i,r}\cdot\prod_{h\in\Qo:\vout{h}=i}\prod_{1\leq s\leq a_{\vin{h}}}(w_{\vin{h},s}-w_{i,r})$.
In notations of~\ref{prop:flat} this defines a generic isomorphism
$\Xi^\circ\colon \BC[\oZ^{\alpha}]\otimes_{\BC[\BA^{\alpha}]}\BC[\oA^{\alpha}]\iso
H_*^{\GL(V)_\CO}(\CR)\otimes_{\BC[\BA^\alpha]}\BC[\oA^\alpha]$.
According to~\ref{sec:chang-repr}, the homomorphism
$\bz^*\colon H_*^{T(V)_\CO}(\CR_{T(V),\bN_{T(V)}})\to H_*^{T(V)_\CO}(\Gr_{T(V)})$ takes
$\sy_{i,r}$ to $\sfu_{i,r}\cdot\prod_{h\in\Qo:\vout{h}=i}\prod_{1\leq s\leq a_{\vin{h}}}(w_{\vin{h},s}-w_{i,r})$.
Thus in notations of~\ref{prop:flat}, $\Xi^\circ$ takes $y_{i,r}$ to
$\iota_*\sy_{i,r}$.
%\begin{NB} This defines an isomorphism $\Psi$ between
%$\on{Spec}H_*^{T(V)_\CO}(\CR_{T(V),\bN_{T(V)}}|_{\oA^\alpha}$ and
%$(\BA^{|\alpha|}\times\BG_m^{|\alpha|})|_{\oA^\alpha}$ identical on
%$w_{i,r}$ and sending $\sy_{i,r}$ to $y_{i,r}$. Note that this is {\em not}
%the isomorphism of~\propref{prop:integrable}.\end{NB}

\begin{NB}
The convolution
algebra $H_*^{\GL(V)_\CO}(\CR)$ contains the equivariant cohomology
of a point $H^*_{\GL(V)_\CO}(pt)=\BC[\BA^\alpha]$. Similarly,
$H_*^{T(V)_\CO}(\CR)$ is a module over $H^*_{T(V)_\CO}(pt)=\BC[\BA^{|\alpha|}]$.

\begin{Lemma}
\label{BFM_6.2}
The $H^*_{T(V)_\CO}(pt)$-module $H_*^{T(V)_\CO}(\CR)$ is flat, and the
$H^*_{\GL(V)_\CO}(pt)$-module $H_*^{\GL(V)_\CO}(\CR)$ is flat. Moreover, the
natural map $H^*_{T(V)_\CO}(pt)\otimes_{H^*_{\GL(V)_\CO}(pt)}H_*^{\GL(V)_\CO}(\CR)\to
H_*^{T(V)_\CO}(\CR)$ is an isomorphism, and
$H_*^{\GL(V)_\CO}(\CR)=(H_*^{T(V)_\CO}(\CR))^{S_\alpha}$.
\end{Lemma}

\begin{proof}
Same as the one of~\cite[Lemma~6.2]{MR2135527}
\end{proof}

\subsection{Bimodules}
\label{bimodules}
Recall that we have a natural embedding $\Gr_{T(V)}\subset\Gr_{\GL(V)}$, and
if we denote by $\bN_T$ the restriction of $\GL(V)$-module $\bN$ to
$T(V)\subset\GL(V)$, then $\CR_{\bN_T}$ is nothing but the preimage in $\CR$
of $\Gr_{T(V)}\subset\Gr_{\GL(V)}$ under the natural projection
$\CR\to\Gr_{\GL(V)}$. Hence $H_*^{T(V)_\CO}(\CR)$ is a
$H_*^{T(V)_\CO}(\CR_{\bN_T})-H_*^{\GL(V)_\CO}(\CR)$-bimodule. Also, $S_\alpha$
acts on $H_*^{T(V)_\CO}(\CR)$ commuting with the right action of
$H_*^{\GL(V)_\CO}(\CR)$ and normalizing the left action of
$H_*^{T(V)_\CO}(\CR_{\bN_T})$.

Recall the closed embedding $\jmath\colon \CR\hookrightarrow\CT=\CT_{\GL(V),\bN}$
(see~Section~\ref{triples}).
By an abuse of notation we will also denote by $\jmath$ the closed
embedding $\CR_{\bN_T}\hookrightarrow\CT_{T(V),\bN}$. We will denote by
$\imath$ the projections $\CT_{\GL(V),\bN}\to\Gr_{\GL(V)}$ and
$\CT_{T(V),\bN}\to\Gr_{T(V)}$.
Then $\imath^*\colon H_*^?(\Gr_{\GL(V)})\to H_*^?(\CT_{\GL(V),\bN})$ is an isomorphism
(where ? stands for equivariance with respect to $T(V)_\CO$ or $\GL(V)_\CO$),
and similarly $\imath^*\colon H_*^{T(V)_\CO}(\Gr_{T(V)})\to H_*^{T(V)_{\CO}}(\CT_{T(V),\bN})$
is an isomorphism. We will denote $(\imath^*)^{-1}$ the inverse isomorphism.

\begin{Lemma}
\label{tilde_H}
The composition $(\imath^*)^{-1}\circ\jmath_*\colon H_*^{T(V)_\CO}(\CR_{\bN_T})\to
H_*^{T(V)_\CO}(\Gr_{T(V)})$ (resp.\\ $H_*^{\GL(V)_\CO}(\CR_{\bN})\to
H_*^{\GL(V)_\CO}(\Gr_{\GL(V)})$) is a homomorphism of convolution algebras.
Furthermore, the composition
$(\imath^*)^{-1}\circ\jmath_*\colon H_*^{T(V)_\CO}(\CR_{\bN})\to
H_*^{T(V)_\CO}(\Gr_{\GL(V)})$ is a homomorphism of the bimodules over the above
algebras.
\end{Lemma}

\begin{proof}
We can only prove the first statement. We use the notations (and result)
of~\ref{abel}. Note that the target of $(\imath^*)^{-1}\circ\jmath_*$
is $H_*^{T(V)_\CO}(\CR_{\bN'_T})$ for $\bN'=0$. Let us denote the corresponding
generators of the convolution ring by $\fz^\lambda,\ \lambda\in\Lambda$.
Then $(\imath^*)^{-1}\circ\jmath_*$ takes $z^\lambda$ to
$\fz^\lambda\prod_{i=1}^n(-\chi_i)^{\max(0,\chi_i(\lambda))}$. From the relations
of~\ref{abel} it is clear that $(\imath^*)^{-1}\circ\jmath_*$ is
a homomorphism.

\end{proof}

\subsection{Localization}
\label{BFM_6.3}
Let $t$ be a general element of $\on{Lie}T(V)$ (more precisely, such that
$t\in\oA^{|\alpha|}\subset\BA^{|\alpha|}=\on{Lie}T(V)$). Then the fixed point set
$\CR^t$ coincides with $\Gr_{T(V)}\subset\Gr_{\GL(V)}\subset\CR$.
According to the Localization theorem the localized
$\left(H_*^{T(V)_\CO}(\CR_{\bN_T})\right)_t$-module
$\left(H_*^{T(V)_\CO}(\CR)\right)_t$ is free of rank one. Lemma~\ref{tilde_H}
provides an isomorphism $$\left(H_*^{T(V)_\CO}(\CR_{\bN_T})\right)_t\iso
\left(H_*^{T(V)_\CO}(\Gr_{T(V)})\right)_t=
\left(\BC[\BA^{|\alpha|}\times\BG_m^{|\alpha|}]\right)_t.$$
Hence after restriction to
$\oA^{|\alpha|}\subset\BA^{|\alpha|}$ we have an isomorphism
$H_*^{T(V)_\CO}(\CR)|_{\oA^{|\alpha|}}\simeq\oA^{|\alpha|}\times\BG_m^{|\alpha|}$
compatible with the natural $S_\alpha$-actions (diagonal on the RHS).
The localized algebra $H_*^{\GL(V)_\CO}(\CR)|_{\oA^\alpha}$ is embedded into
$\left(\End_{H_*^{T(V)_\CO}(\CR_{\bN_T})|_{\oA^{|\alpha|}}}
(H_*^{T(V)_\CO}(\CR)|_{\oA^{|\alpha|}})\right)^{S_\alpha}$.
By~Lemma~\ref{BFM_6.2}, $H_*^{\GL(V)_\CO}(\CR)=(H_*^{T(V)_\CO}(\CR))^{S_\alpha}$.
Hence the above embedding is an isomorphism, and we have
$H_*^{\GL(V)_\CO}(\CR)|_{\oA^{\alpha}}\simeq
(\BC[\BA^{|\alpha|}\times\BG_m^{|\alpha|}])^{S_\alpha}|_{\oA^\alpha}$.
\end{NB}

\begin{Theorem}
\label{pestun}
The isomorphism
$\Xi^\circ\colon \BC[\oZ^{\alpha}]\otimes_{\BC[\BA^{\alpha}]}\BC[\oA^{\alpha}]\iso
H_*^{\GL(V)_\CO}(\CR)\otimes_{\BC[\BA^\alpha]}\BC[\oA^\alpha]$
extends to a biregular isomorphism
$\BC[\oZ^{\alpha}]\simeq H_*^{\GL(V)_\CO}(\CR)$.
\end{Theorem}

\begin{proof}
%Both $\BC[\oZ^\alpha]$ and $H_*^{\GL(V)_\CO}(\CR)$ are flat
%$\BC[\BA^\alpha]$-modules embedded into
%$(\BC[\oA^{|\alpha|}\times\BG_m^{|\alpha|}])^{S_\alpha}$.
%It suffices to prove that the
%identification of $H_*^{\GL(V)_\CO}(\CR)|_{\oA^{\alpha}}$ and
%$\BC[\oZ^\alpha]|_{\oA^\alpha}$ extends over $\bA^\alpha$ (notations
%of~Section~\ref{zastava}). In effect, if $j$
%stands for an open embedding of $\bA^\alpha$ into $\BA^\alpha$, then the
%codimension in $\BA^\alpha$ of the complement $\BA^\alpha\setminus\bA^\alpha$
%is 2, so for a flat $\BC[\BA^\alpha]$-module $M$ the natural map
%$M\to j_*j^*M$ is an isomorphism.
%According to~\subsecref{sec:chang-repr}, the homomorphism
%$\bz^*\colon H_*^{T(V)_\CO}(\CR_{T(V),\bN_{T(V)}})\to
%H_*^{T(V)_\CO}(\Gr_{T(V)})$ takes
%$\sy_{i,r}$ to $\sfu_{i,r}\cdot\prod_{j\leftarrow i}
%\prod_{1\leq s\leq a_j}(w_{i,r}-w_{j,s})$.
%where $\sfu_{i,r}$ is the fundamental class of the point of $\Gr_{T(V)}$
%corresponding to the cocharacter $w_{i,r}^*$ of $T(V)$. In notations
%of~\propref{prop:flat}(1) we set ${\mathcal M}=\oZ^\alpha$ and define a
%generic isomorphism $\Xi$ identical on $w_{i,r}$ and sending $y_{i,r}$ to
%$\iota_*\sy_{i,r}$; equivalently, sending
%$y_{i,r}\cdot\prod_{j\leftarrow i}
%\prod_{1\leq s\leq a_j}(w_{i,r}-w_{j,s})^{-1}$ to
%$\iota_*\sfu_{i,r}$.
%Thus in notations of~\propref{prop:flat}, $\Xi^\circ$ takes $y_{i,r}$ to
%$\iota_*\sy_{i,r}$.
Recall the coordinates $\sw_{i,r},\ i\in \II,\ 1\leq r\leq a_i$ (the
characters) on $T(V)$ whose differentials are $w_{i,r}\in\ft^\vee(V)$.
%Recall also the isomorphism $\ft(V)=\BA^{|\alpha|}\iso\BA^{|\alpha|}$ sending
%$w_{i,r}$ to $w_{i,r}$. We will denote this isomorphism by $t\mapsto t$.
Let $t\in\bA^{|\alpha|}\setminus\oA^{|\alpha|}\subset\ft(V)$.
According to~\ref{prop:flat} and~\ref{rem:conditions}, it suffices
to identify the localizations $\BC[\uoZ^{\alpha}]_{t}$ and
$\left(H_*^{T(V)_\CO}(\CR)\right)_t$ as
$\BC[\BA^{|\alpha|}]_{t}=\BC[\ft(V)]_t$-modules. Our $t$ lies on a diagonal
divisor. We will consider two possibilities.

First we can have $(w_{i,r}-w_{j,s})(t)=0$ for $i\ne j$. Then
%by the Localization theorem,
%$z'''{}^*\iota'''_*{}^{-1}\colon \left(H_*^{T(V)_\CO}(\CR)\right)_t\iso
%\left(H_*^{T(V)_\CO}(\CR^t_\bN)\right)_t$ (see~\eqref{eq:22}), and
the fixed point set $\CR^t_\bN$ is isomorphic to the product
$\Gr_{T_1}\times\CR_{T_2,\bN'}$. Here $T_2$ is a 2-dimensional torus with coordinates
$\sw_{i,r},\sw_{j,s}$, and $T_1$ is an $(|\alpha|-2)$-dimensional torus with
coordinates $\{\sw_{k,p}\colon (i,r)\ne(k,p)\ne(j,s)\}$, so that $T(V)=T_1\times T_2$.
Furthermore, $\bN'$ is the following representation of $T_2$: if $i,j$ are
not connected by an edge of the Dynkin diagram, then $\bN'=0$;
and if there is an arrow $h$ from
$i$ to $j$ in our orientation $\Omega$, then $\bN'$ is a character
$\sw_{i,r}^{-1}\sw_{j,s}$ with differential $w_{j,s}-w_{i,r}$.
In case $i,j$ are not connected we conclude that in notations
of~\ref{prop:flat}(2), $G'=T(V)$, and
$$\left(H_*^{G'_\CO}(\CR_{G',\bN'})\right)_t=
\left(H_*^{T(V)_\CO}(\Gr_{T(V)})\right)_t=(\BC[{\mathfrak t}(V)\times T^\vee(V)])_t
=(\BC[\BA^{|\alpha|}\times\BG_m^{|\alpha|}])_t$$
We define
$\Xi^t\colon (\BC[\oZ^{\alpha}]\otimes_{\BC[\BA^{\alpha}]}\BC[\BA^{|\alpha|}])_{t}\to
\left(H_*^{G'_\CO}(\CR_{G',\bN'})\right)_t$ identical on $w_{k,p}$ and sending
$y_{k,p}$ to $\sfu_{k,p}\cdot\prod_{h\in\Qo:\vout{h}=k}\prod_{1\leq q\leq a_{\vin{h}}}(w_{\vin{h},q}-w_{k,p})$.
Note that at the moment $\Xi^t$ is defined only as a rational morphism.
The condition of~\ref{prop:flat}(2) is trivially satisfied.
Also, $\Xi^t$ is a regular isomorphism due to the factorization property of
zastava~\textup{(\ref{149})}
and e.g.~\cite[5.5]{bdf}, and the fact that the factors $(w_{\vin{h},q}-w_{k,p})$
in the formula for $\Xi^t$ are all invertible at $t$.

In case $h\in\Qo$ with $\vout{h}=i$, $\vin{h}=j$, in notations of~\ref{prop:flat}(2), $G'=T(V)$,
and $$\left(H_*^{G'_\CO}(\CR_{G',\bN'})\right)_t=
\left(H_*^{T(V)_\CO}(\Gr_{T_1}\times\CR_{T_2,\bN'})\right)_t=
(\BC[{\mathfrak t}_1\times T_1^\vee]\otimes A)_t$$
where $A$ is generated by $w_{i,r},w_{j,s}$ and the fundamental classes
$\sy'_{i,r},\sy'_{j,s},\sy'_{irjs},\sy'^{-1}_{irjs}$
of the fibers of $\CR_{T_2,\bN'}$ over the points
$w_{i,r}^*,w_{j,s}^*,w_{i,r}^*+w_{j,s}^*,-w_{i,r}^*-w_{j,s}^*\in\Gr_{T_2}$ respectively.
According to~\ref{abel}, the relations in $A$ are as follows:
$\sy'_{irjs}\cdot\sy'^{-1}_{irjs}=1;\
\sy'_{i,r}\cdot\sy'_{j,s}=\sy'_{irjs}\cdot(w_{j,s}-w_{i,r})$.
%according to~Theorem~\ref{abel},
%$$\left(H_*^{T(V)_\CO}(\CR^t_\bN)\right)_t\simeq
%(\BC[{\mathfrak t}_1\times T_1^\vee]\otimes A)_t$$ where $A$ is the following
%algebra. It has generators $y_{i,r},y_{j,s},w_{i,r},w_{j,s},u$ and relation
%$y_{i,r}y_{j,s}=(w_{i,r}-w_{j,s})u$, subject to the condition that $u$ is
%invertible.

According to~\ref{sec:chang-repr}, the homomorphism
$\bz'^*\colon H_*^{T(V)_\CO}(\Gr_{T_1}\times\CR_{T_2,\bN'})\to H_*^{T(V)_\CO}(\Gr_{T(V)})$
takes $\sy'_{i,r}$ to $(w_{j,s}-w_{i,r})\sfu_{i,r}$, while
$\bz'^*\sy'_{j,s}=\sfu_{j,s},\ \bz'^*\sy'_{irjs}=\sfu_{i,r}\sfu_{j,s},\
\bz'^*\sy'^{-1}_{irjs}=\sfu_{i,r}^{-1}\sfu_{j,s}^{-1}$.
We define
$\Xi^t\colon (\BC[\oZ^{\alpha}]\otimes_{\BC[\BA^{\alpha}]}\BC[\BA^{|\alpha|}])_{t}\to
\left(H_*^{G'_\CO}(\CR_{G',\bN'})\right)_t$ identical on $w_{k,p}$ and sending
$y_{k,p}$ to $\sfu_{k,p}\cdot\prod_{h\in\Qo:\vout{h}=k}\prod_{1\leq q\leq a_{\vin{h}}}(w_{\vin{h},q}-w_{k,p})$.
%unless $(k,p)=(i,r)$ in which case $y_{i,r}$ goes to
%$\sfu_{i,r}\cdot(w_{i,r}-w_{j,s})^{-1}
%\prod_{l\leftarrow i}\prod_{1\leq q\leq a_l}(w_{i,r}-w_{l,q})$.
In particular, $y_{i,r}$ goes to $\sy'_{i,r}\cdot(w_{j,s}-w_{i,r})^{-1}
\prod_{h\in\Qo:\vout{h}=i}\prod_{1\leq q\leq a_{\vin{h}}}(w_{\vin{h},q}-w_{i,r})$.
Note that at the moment $\Xi^t$ is defined only as a rational morphism.
The condition of~\ref{prop:flat}(2) is trivially satisfied.
Also, $\Xi^t$ is a regular isomorphism
due to the factorization property of zastava and e.g.~\cite[5.6]{bdf},
and the fact that the factors $(w_{\vin{h},q}-w_{k,p})$
in the formula for $\Xi^t$ (note that the factor $(w_{j,s}-w_{i,r})$ is excluded)
are all invertible at $t$. In particular,
$\Xi^t$ sends $y_{irjs}$ of~Remark~\ref{4dim} to $\sy'_{irjs}$ up to a product
of invertible factors $(w_{\vin{h},q}-w_{k,p})$.

%Now the localization $\BC[\uoZ^\alpha]_t$ is isomorphic to
%$(\BC[\BA^{|\alpha|}\times\BG_m^{|\alpha|}])_t$ in case $i,j$ are not
%connected, due to the factorization property of zastava. Finally, in case
%$i\to j\in\Omega$, we have
%$$\BC[\uoZ^\alpha]_t\simeq(\BC[{\mathfrak t}_1\times T_1^\vee]\otimes A)_t$$
%due to the factorization property and e.g.~\cite[5.6]{bdf}.

The second possibility is $(w_{i,r}-w_{i,s})(t)=0$.
Then the fixed point set $\CR^t$ is isomorphic to the product
$\Gr_{T_1}\times\Gr_{\GL(V'')}$. Here $V''\subset V_i$ is a 2-dimensional
subspace whose $T(V)$-weights are $\sw_{i,r},\sw_{i,s}$, and $T_1$ is an
$(|\alpha|-2)$-dimensional torus with
coordinates $\{\sw_{k,p}\colon (i,r)\ne(k,p)\ne(i,s)\}$, and $T_2$ is a
2-dimensional torus with coordinates $\sw_{i,r},\sw_{i,s}$.
 Hence in notations
of~\ref{prop:flat}(2), $G'=T_1\times\GL(V''), \bN'=0$, and
$$H_*^{G'_\CO}(\CR_{G',\bN'})\otimes_{H^*_{G'}(\on{pt})}\BC[\ft(V)]=
H_*^{T_{1,\CO}}(\Gr_{T_1})\otimes H_*^{T_{2,\CO}}(\Gr_{\GL(V'')})=
\BC[{\mathfrak t}_1\times T_1^\vee]\otimes B$$
where $B$ is the following algebra. It has generators
$\iota'_*\sfu_{i,r},\iota'_*\sfu_{i,s},w_{i,r},w_{i,s},\eta$ and relation
$\iota'_*\sfu_{i,r}-\iota'_*\sfu_{i,s}=(w_{i,r}-w_{i,s})\eta$, subject to the condition that
$\iota'_*\sfu_{i,r},\iota'_*\sfu_{i,s}$ are invertible. In effect, the isomorphism
$B\iso H_*^{T_2}(\Gr_{\GL(V'')})$ takes
%$w_{i,r},w_{i,s}$ to the generators of
%$H_*^{T_2}(pt)$, and $y_{i,r}$ to the fundamental cycle of the point
%$(1,0)\in\Gr_{T_2}\subset\Gr_{\GL(V'')}$ (where we have identified $\Gr_{T_2}$
%with the lattice $\BZ^2$ of cocharacters of $T_2$); furthermore,
%$y_{i,s}\mapsto[(0,1)],\ y_{i,r}^{-1}\mapsto[(-1,0)],\
%y_{i,s}^{-1}\mapsto[(0,-1)]$. Finally,
$\eta$ to the fundamental cycle
$[\BP^1_1]\in H_2^{T_2}(\Gr_{\GL(V'')})$ where
$\BP^1_1$ is the 1-dimensional $\GL(V'')$-orbit containing $w_{i,r}^*$ and
$w_{i,s}^*$ (use the argument in~\cite[3.10]{MR2135527}).

We define
$\Xi^t\colon (\BC[\oZ^{\alpha}]\otimes_{\BC[\BA^{\alpha}]}\BC[\BA^{|\alpha|}])_t\to
\left(H_*^{T(V)_\CO}(\CR_{G',\bN'})\right)_t$ identical on $w_{k,p}$ and sending
$y_{k,p}$ to $\iota'_*\sfu_{k,p}\cdot\prod_{h\in\Qo:\vout{h}=k}\prod_{1\leq q\leq a_{\vin{h}}}(w_{\vin{h},q}-w_{k,p})$.
Note that at the moment $\Xi^t$ is defined only as a rational morphism.
The condition of~\ref{prop:flat}(2) is trivially satisfied.
Also, $\Xi^t$ is a regular isomorphism
due to the factorization property of zastava and e.g.~\cite[5.5]{bdf},
and the fact that the factors $(w_{l,q}-w_{k,p})$
in the formula for $\Xi^t$ are all invertible at $t$.
%Now the localization $\BC[\uoZ^\alpha]_t$ is isomorphic to
%$(\BC[{\mathfrak t}_1\times T_1^\vee]\otimes B)_t$ due to the factorization
%property
%and e.g.~\cite[5.5]{bdf}. The last condition of~\propref{prop:flat}(2) holds
%due to the associativity of factorization.

The theorem is proved.
\end{proof}

\begin{Remark}
\label{cartan_grading}
$H_*^{\GL(V)_\CO}(\CR)$ is naturally graded by $\pi_1\GL(V)=\BZ^\II$.
Under the isomorphism of~Theorem~\ref{pestun} this grading
becomes the grading of $\BC[\oZ^{\alpha}]$ by the root lattice of the
Cartan torus $T\subset G\colon \BZ^\II=\BZ\langle\alphavee_i\rangle_{i\in \II}$
corresponding to the natural action of $T$ on $\oZ^{\alpha}$.
Indeed, the weight of $w_{i,r}$ is 0, while the weight of $y_{i,r}$ is
$\alphavee_{i}$.
\end{Remark}

\begin{Remark}
\label{compare_degrees_zas}
The LHS of~Theorem~\ref{pestun} is naturally
graded by half the homological degree $\deg_h$, while the RHS is naturally graded by the action of loop rotations, $\deg_r$. These gradings are different.
Let $x$ be a homogeneous homology class supported at the connected component
$\nu=(n_i)\in\BZ^\II=\pi_0\Gr_{\GL(V)}$. Then one can check that
$\deg_r(x)=\deg_h(x)-\nu^t\cdot\sqrt{\det\bN}+
\frac{1}{2}\nu^t\cdot{\mathbf C}\cdot\alpha$. Here ${\mathbf C}$ is the Cartan
matrix of $G$, and we view $\sqrt{\det\bN}$ as a (rational) character of
$\GL(V)$, i.e.\ an element of $\frac{1}{2}\BZ^\II$. Note that
$\deg_h(x)-\nu^t\cdot\sqrt{\det\bN}$ coincides with the monopole formula
exponent $\Delta(x)$ of~\eqref{eq:18}, see~\ref{discrepancy}(2).
\end{Remark}

\subsection{Positive part of an affine Grassmannian}
\label{positive}
Given a vector space $U$ we define\footnote{The second named
author thanks Joel Kamnitzer for correcting his mistake.}
$\Gr_{\GL(U)}^+\subset\Gr_{\GL(U)}$ as
the moduli space of vector bundles $\CU$ on the formal
disc $D$ equipped with trivialization $\sigma\colon \CU|_{D^*}\iso U\otimes\CO_{D^*}$
on the punctured disc such that $\sigma$ extends through the puncture
as an embedding $\sigma\colon \CU\hookrightarrow U\otimes\CO_D$.

Now since $\GL(V)=\prod_{i\in \II}\GL(V_i)$ (notations of \cref{Co_bra}),
$\Gr_{\GL(V)}=\prod_{i\in \II}\Gr_{\GL(V_i)}$, and we define
$\Gr_{\GL(V)}^+=\prod_{i\in \II}\Gr_{\GL(V_i)}^+$. We define $\CR^+$ as the preimage
of $\Gr_{\GL(V)}^+\subset\Gr_{\GL(V)}$ under $\CR\to\Gr_{\GL(V)}$. Then
$H_*^{\GL(V)_\CO}(\CR^+)$ forms a convolution subalgebra of
$H_*^{\GL(V)_\CO}(\CR)$. Note that $\CR^+_0$ (the preimage in $\CR$ of the base
point in $\Gr_{\GL(V)}$) is a connected component of $\CR^+$, and the union
of the remaining connected components supports an ``augmentation'' ideal
of $H_*^{\GL(V)_\CO}(\CR^+)$.
Hence we have an algebra homomorphism  $H_*^{\GL(V)_\CO}(\CR^+)\to
H_*^{\GL(V)_\CO}(\CR^+_0)=H^*_{\GL(V)}(\on{pt})$.
The proof of~Theorem~\ref{pestun} repeated essentially
word for word gives a proof (using the fact that $Z^\alpha$ is normal)
of the following

\begin{Corollary}
\label{Pestun}
The pushforward with respect to the closed embedding $\CR^+\hookrightarrow\CR$
induces an injective algebra homomorphism $H_*^{\GL(V)_\CO}(\CR^+)\hookrightarrow
H_*^{\GL(V)_\CO}(\CR)$. The isomorphism $\BC[\oZ^{\alpha}]\iso H_*^{\GL(V)_\CO}(\CR)$
of~Theorem~\ref{pestun} takes $\BC[Z^{\alpha}]\subset\BC[\oZ^{\alpha}]$ onto
$H_*^{\GL(V)_\CO}(\CR^+)\subset H_*^{\GL(V)_\CO}(\CR)$, and so induces an isomorphism
$\BC[Z^{\alpha}]\iso H_*^{\GL(V)_\CO}(\CR^+)$. The above homomorphism
$\BC[Z^{\alpha}]=H_*^{\GL(V)_\CO}(\CR^+)\to H^*_{\GL(V)}(\on{pt})=\BC[\BA^\alpha]$
corresponds to the section $s_\alpha\colon \BA^\alpha\hookrightarrow Z^\alpha$
of~\ref{zastava}.
%The isomorphism $H_*^{\GL(V)_\CO}(\CR^+)|_{\oA^{\alpha}}\simeq
%(\BC[\BA^{|\alpha|}\times\BA^{|\alpha|}])^{S_\alpha}|_{\oA^\alpha}\simeq
%\BC[Z^\alpha]|_{\oA^\alpha}$ extends to a biregular isomorphism
%$H_*^{\GL(V)_\CO}(\CR^+)\simeq\BC[Z^\alpha]$. The natural embedding
%$\BC[Z^\alpha]\hookrightarrow\BC[\oZ^\alpha]$ corresponds to the pushforward
%homomorphism $H_*^{\GL(V)_\CO}(\CR^+)\hookrightarrow H_*^{\GL(V)_\CO}(\CR)$.
\qed
\end{Corollary}

\begin{Remark}
\label{Ober}
The center of $\GL(V)$ is canonically identified with $\BG_m^\II$, and we have
the diagonal embedding
$\Delta\colon \BG_m\hookrightarrow\BG_m^\II\hookrightarrow\GL(V)$. We can view
$\Delta$ as a cocharacter of $\GL(V)$, and hence a point of
$\Gr_{\GL(V)}^+\subset\Gr_{\GL(V)}$. Note that this point is a $\GL(V)_\CO$-orbit.
We denote the fundamental class of its preimage in $\CR^+$ by
$F_\Delta\in H_*^{\GL(V)_\CO}(\CR^+)$. Then under the isomorphism
$H_*^{\GL(V)_\CO}(\CR^+)\simeq\BC[Z^{\alpha}]$ the class $F_\Delta$ goes to the
boundary equation $F_{\alpha}$ of~\cite[Section~4]{bf14}. Indeed, $F_\Delta$
viewed as an element in $H_*^{\GL(V)_\CO}(\CR)=\BC[\oZ^\alpha]$
is clearly invertible, but all the invertible regular functions on
$\oZ^\alpha$ are of the form
$cF_\alpha^k,\ k\in\BZ,\ c\in\BC^\times$~\cite[Lemma~5.4]{bdf}.
Now for degree reasons, $F_\Delta$ must coincide with $cF_\alpha$.
\end{Remark}

\begin{Remark}
\label{Brav}
We consider the following isomorphism $\fri\colon \Gr_{\GL(V)}\iso\Gr_{\GL(V^*)}$:
it takes
$(\scP,\sigma)$ to $(\scP^\vee,\ ^t\sigma^{-1})$ where $\scP$ is a $\GL(V)$-bundle
on the formal disc, $\sigma$ is a morphism from the trivial $\GL(V)$-bundle
on the punctured disc, $\scP^\vee$ is the dual $\GL(V^*)$-bundle,
and $^t\sigma$ is the
transposed morphism from $\scP^\vee$ to the (dual) trivial bundle on the
punctured disc. Let $\overline\Qo$ be the opposite orientation of our quiver,
and let $\bN$ be the corresponding representation of $\GL(V^*)$
(note that $\Hom(V_i,V_j)=\Hom(V_j^*,V_i^*)$). Then $\fri$ lifts to the
same named isomorphism $\CR_{\GL(V),\bN}\iso\CR_{\GL(V^*),\bN}$. Together with
an isomorphism $\GL(V)\iso\GL(V^*),\ g\mapsto\ ^tg^{-1}$, it gives rise to a
convolution algebra isomorphism $\fri_*\colon H_*^{\GL(V)_\CO}(\CR_{\GL(V),\bN})\iso
H_*^{\GL(V^*)_\CO}(\CR_{\GL(V^*),\bN})$. The composition
$\BC[\oZ^{\alpha}]\simeq H_*^{\GL(V)_\CO}(\CR_{\GL(V),\bN})
\stackrel{\fri_*}{\longrightarrow}H_*^{\GL(V^*)_\CO}(\CR_{\GL(V^*),\bN})
\simeq\BC[\oZ^{\alpha}]$ is an involution of the algebra $\BC[\oZ^{\alpha}]$.
This involution arises from the Cartan involution $\iota$ of
$\oZ^{\alpha}$~\cite[Section~4]{bdf} composed with the involution $\varkappa_{-1}$ of
$\oZ^{\alpha}$ induced by an automorphism $z\mapsto-z\colon\BP^1\iso\BP^1$ and
finally composed with the action $a(h)$ of a certain element of the Cartan
torus $h=\beta(-1)\in T$.\footnote{The second named author thanks Joel
Kamnitzer for correcting his mistake.}
Here $\beta$ is a cocharacter of $T$ equal to
$\sum_{i\in \II}b_i\omega_i$ where $b_i=a_i-\sum_{h\in\Qo:\vout{h}=i}a_{\vin{h}}$.

In effect, $\displaystyle{\iota(w_{i,r},y_{i,r})=(w_{i,r},y_{i,r}^{-1}
\prod_{\substack{h\in\Qo\sqcup\overline\Qo\\ \vout{h}=i}}^{1\leq s\leq a_{\vin{h}}}(w_{i,r}-w_{\vin{h},s}))}$~\cite[Proposition~4.2]{bdf}
(product over \emph{unoriented} edges of the Dynkin diagram connected to $i$),
$\varkappa_{-1}(w_{i,r},y_{i,r})=(-w_{i,r},(-1)^{a_i}y_{i,r})$, and
$a(\beta(-1))(w_{i,r},y_{i,r})=(w_{i,r},(-1)^{a_i-\sum_{h\in\Qo:\vout{h}=i}a_{\vin{h}}}y_{i,r})$.
One checks explicitly that
%in the two examples considered in the proof
%of~Theorem~\ref{pestun} (algebras $A$ and $B$), we have
$\Xi^\circ$ intertwines
$a(\beta(-1))\circ\varkappa_{-1}\circ\iota$ and $\fri_*$.
The desired claim follows.
\end{Remark}

\begin{Remark}
\label{compact_zastava}
The $\GL(V)_\CO$-orbits in $\Gr_{\GL(V)}^+$ are numbered by $\II$-multipartitions
$(\lambda^{(i)})_{i\in \II},\ \lambda^{(i)}=(\lambda^{(i)}_1\geq\lambda^{(i)}_2\geq\ldots)$,
such that the number of parts $l(\lambda^{(i)})\leq a_i$. Given a positive roots combination
$\alphavee=\sum_{i\in\II} m_i\alphavee_i$,
we define a closed $\GL(V)_\CO$-invariant subvariety
$\ol\Gr_{\GL(V)}^{+,\alphavee}\subset\Gr_{\GL(V)}^+$ as the union of orbits
$\Gr_{\GL(V)}^{(\lambda^{(i)})}$ such that $\lambda^{(i)}_1\leq m_i\ \forall i\in \II$.
We define $\CR^+_{\leq\alphavee}\subset\CR^+$ as the preimage of
$\ol\Gr_{\GL(V)}^{+,\alphavee}\subset\Gr_{\GL(V)}^+$ under $\CR^+\to\Gr_{\GL(V)}^+$.
This filtration is the intersection of a certain coarsening of the one
of~\ref{subsec:equiv-borel-moore} and~\ref{sec:degeneration}
with $\CR^+\subset\CR$.
We consider an increasing multifiltration
$H_*^{\GL(V)_\CO}(\CR^+)=\bigcup_{\alphavee}H_*^{\GL(V)_\CO}(\CR^+_{\leq\alphavee})$,
and its Rees algebra ${\mathsf{Rees}}F_\bullet H_*^{\GL(V)_\CO}(\CR^+)$. This is
a multigraded algebra, and we take its multiprojective spectrum
$\on{Proj}{\mathsf{Rees}}F_\bullet H_*^{\GL(V)_\CO}(\CR^+)$. It contains
$\on{Spec}H_*^{\GL(V)_\CO}(\CR^+)\simeq Z^{\alpha}$ as an open dense subvariety.
The relative compactification
$\on{Proj}{\mathsf{Rees}}F_\bullet H_*^{\GL(V)_\CO}(\CR^+)\to\BA^\alpha$ of
$\on{Spec}H_*^{\GL(V)_\CO}(\CR^+)\simeq Z^{\alpha}\to\BA^{\alpha}$ is nothing but the
``two-sided'' compactified zastava
$\overline{Z}{}^{\alpha}\to\BA^{\alpha}$~\cite[7.2]{gait} where
e.g.\ in the ``symmetric'' definition of~\cite[2.6]{bdf} we allow both $B$ and
$B_-$-structures to be generalized (cf.~\cite[1.4]{mirk}). Note that the
Cartan involution $\iota$ of $\oZ^{\alpha}$ (Remark~\ref{Brav}) extends to the
same named involution of $\overline{Z}{}^{\alpha}$.

In effect, it suffices to check that the multifiltration $F_\bullet H_*^{\GL(V)_\CO}(\CR^+)$ of $H_*^{\GL(V)_\CO}(\CR^+)\simeq\BC[Z^{\alpha}]$
coincides with the multifiltration of $\BC[Z^{\alpha}]$ by the order of the pole
at the components of the boundary $\overline{Z}{}^{\alpha}\setminus Z^{\alpha}$.
Due to~Remark~\ref{Brav}, it suffices to check that the multifiltration of
$\BC[\oZ^{\alpha}]\simeq H_*^{\GL(V)_\CO}(\CR)$ by the order of the pole at the
components of the boundary $\partial Z^{\alpha}$ coincides with the following
multifiltration $\BF_\bullet H_*^{\GL(V)_\CO}(\CR)$.
The $\GL(V)_\CO$-orbits in $\Gr_{\GL(V)}$ are numbered by generalized
$\II$-multipartitions $(\lambda^{(i)})_{i\in \II},\ \lambda^{(i)}=
(\lambda^{(i)}_1\geq\lambda^{(i)}_2\geq\ldots\geq\lambda^{(i)}_{a_i}),\
\lambda^{(i)}_r\in\BZ$. Given a positive roots combination
$\alphavee=\sum_{i\in\II} m_i\alphavee_i$,
we define a closed $\GL(V)_\CO$-invariant ind-subvariety
$\Gr_{\GL(V)}^+\subset\Gr_{\GL(V)}^{\geq-\alphavee}\subset\Gr_{\GL(V)}$
as the union of orbits $\Gr_{\GL(V)}^{(\lambda^{(i)})}$
such that $\lambda^{(i)}_{a_i}\geq-m_i\ \forall i\in \II$.
In particular, $\Gr_{\GL(V)}^{\geq-0}=\Gr_{\GL(V)}^+$.
We define $\CR_{\geq-\alphavee}\subset\CR$ as the preimage of
$\Gr_{\GL(V)}^{\geq-\alphavee}\subset\Gr_{\GL(V)}$ under $\CR\to\Gr_{\GL(V)}$.
The desired increasing multifiltration is
$\BF_{\alphavee}H_*^{\GL(V)_\CO}(\CR)=H_*^{\GL(V)_\CO}(\CR_{\geq-\alphavee})$.
It coincides with the multifiltration by the order of the pole at
$\partial Z^{\alpha}$ by~Remark~\ref{Ober}.
\end{Remark}

\begin{Remark}
\label{Simon}
We consider the Rees algebra ${\mathsf{Rees}}\BF_\bullet H_*^{\GL(V)_\CO}(\CR)$
of the multifiltration $\BF_\bullet$ of~Remark~\ref{compact_zastava}.
This is an algebra over $\BC[\overline{T}_{\on{ad}}]$ where
$\overline{T}_{\on{ad}}\supset T_{\on{ad}}$
is a partial closure of the adjoint Cartan torus $T_{\on{ad}}$ determined by the cone of
positive combinations of simple roots in the weight lattice of $T$.
It seems likely that the spectrum of the Rees
algebra $\widetilde{Z}{}^{\alpha}\to\overline{T}_{\on{ad}}$ is nothing but the {\em local
model} of the Drinfeld-Lafforgue-Vinberg degeneration of~\cite[6.1]{sch}.

In view of these Remarks, it is interesting to consider the Rees algebra of
the filtration in~\ref{sec:degeneration} in general. We do not know what
it is in general.
\end{Remark}

\begin{Remarks}
\label{trig}
(1) According to~\ref{rem:Steinberg}(3), we can consider the convolution
algebra $K^{\GL(V)_\CO}(\CR)$. Then, similarly to~Theorem~\ref{pestun}, one
can construct an isomorphism $\BC[^\dagger\!\oZ^\alpha]\simeq K^{\GL(V)_\CO}(\CR)$
where $^\dagger\!\oZ^\alpha\subset\oZ^\alpha$ stands for the trigonometric open
zastava of~\cite{fkr}, that is the preimage of $\BG_m^\alpha\subset\BA^\alpha$
under the factorization projection $\oZ^\alpha\to\BA^\alpha$. Note that the embedding
$^\dagger\!\oZ^\alpha\subset\oZ^\alpha$ is \emph{not} compatible with the symplectic
structure by the argument in \subsecref{symplectic}.

(2) Here is what we have learned from Gaiotto:

Let us consider the $K$-theoretic Coulomb branch
$\cM_C^K = \Spec K^{G_\cO}(\cR)$. Then $\cM_C^K$ is supposed to be
isomorphic to the Coulomb branch of the corresponding $4$-dimensional
gauge theory with a generic complex structure. (Recall that the latter
is a hyper-K\"ahler manifold, which shares many common properties with
Hitchin's moduli spaces of solutions of the self-duality equation over
a Riemann surface. Among the $S^2$-family of complex structures, two
are special and others are isomorphic.) For an unframed quiver gauge theory, the
latter is known to be the moduli space of $G_{ADE,c}$-monopoles on
$\RR^2\times S^1$ \cite{2012arXiv1211.2240N}. Moreover the isomorphism
is given by the scattering matrix: we identify $\RR^2\times S^1$ with
$\RR\times\CC^\times$ and consider scattering at $t\to\pm\infty$ in
the first factor. Then for $G_c=\operatorname{SU}(2)$ as
in~\subsecref{subsec:scatter}, we transform $f$, $g$ uniquely to polynomials
(instead of Laurent polynomials) such that
\begin{aenume}
    \item $g$ is a monic polynomial of degree $k$,
    \item $f$ is a polynomial of degree $<k$.
    \item $g(0)\neq 0$.
\end{aenume}
Thus we recover the trigonometric open zastava.

We do not have further evidences of this conjecture. For example, we
cannot see the remaining two special complex structures. We also
remark that our definition of $K$-theoretic Coulomb branches makes
sense for any $(G,\bN)$, while $4$-dimensional Coulomb branches are
usually considered only when $\bN$ is `smaller' than $G$ (conformal or asymptotically free in physics terminology).
\end{Remarks}

\subsection{General cases}
\label{quivar}
Recall the setup of \cref{Co_bra}. We write down a dominant coweight
$\lambda$ of $G$ as a linear combination of fundamental coweights
$\lambda=\sum_{i\in \II}l_i\omega_i$. Given another coweight
$\mu\leq\lambda$ such that $\lambda-\mu=\alpha=\sum_{i\in \II}a_i\alpha_i$,
we set $W_i=\BC^{l_i}$, and consider the
natural action of $\GL(V)$ on $\bN=\bN^\lambda_\mu:=\bigoplus_{h\in\Qo}
\on{Hom}(V_{\vout{h}},V_{\vin{h}})\oplus\bigoplus_{i\in \II}\on{Hom}(W_i,V_i)$.
The corresponding variety of triples will be denoted $\CR^\lambda_\mu$.
Our goal is to
describe the convolution algebra $H_*^{\GL(V)_\CO}(\CR^\lambda_\mu)$. To this end we
introduce $\lambda^*:=-w_0\lambda,\ \mu^*:=-w_0\mu$. Note that
$\lambda^*$ is dominant, and $(\lambda^*-\mu^*)^*=\alpha$.
We consider an open subset $\oG_m^{|\alpha|}\subset\oA^{|\alpha|}$ defined
as the complement in $\BG_m^{|\alpha|}$ to all the diagonals
$w_{i,r}=w_{j,s}$, and also $\oG_m^\alpha:=\oG_m^{|\alpha|}/S_\alpha\subset\BA^\alpha$
(notations of \cref{zastava}). The generalized roots are
$w_{i,r}-w_{j,s}$ for $i\ne j$ connected in the Dynkin diagram; $w_{i,r}-w_{i,s}$,
and finally $w_{i,r}$. The isomorphism of~\eqref{eq:21}
and~\ref{prop:integrable} identifies
$H_*^{\GL(V)_\CO}(\CR^\lambda_\mu)|_{\oG_m^\alpha}$ and $H_*^{\GL(V)_\CO}(\CR)|_{\oG_m^\alpha}$
(with $(\BC[\BG_m^{|\alpha|}\times\BG_m^{|\alpha|}])^{S_\alpha}|_{\oG_m^\alpha}$).
Furthermore, $H_*^{\GL(V)_\CO}(\CR)$ is identified with $\BC[Z^{\alpha}]$
by~Theorem~\ref{pestun}, and $\BC[Z^{\alpha}]|_{\oG_m^{\alpha}}$ is identified with
$\BC[\oW^{\lambda^*}_{\mu^*}]|_{\oG_m^{\alpha}}$ via $s^{\lambda^*}_{\mu^*}$.
The composition of the above identifications gives us a generic isomorphism
$\Xi^\circ\colon \BC[\oW^{\lambda^*}_{\mu^*}]|_{\oG_m^{\alpha}}\iso
H_*^{\GL(V)_\CO}(\CR^\lambda_\mu)|_{\oG_m^\alpha}$. Equivalently, as
in \cref{Co_bra}, we denote by
$\bar\sy_{i,r}\in H_*^{T(V)_\CO}(\CR^\lambda_{\mu,T(V),\bN_{T(V)}})$ the fundamental class
of the fiber of $\CR^\lambda_{\mu,T(V),\bN_{T(V)}}$ over the point
$w^*_{i,r}\in\Gr_{T(V)}$. We have the algebra homomorphism
$\iota_*\colon H_*^{T(V)_\CO}(\CR^\lambda_{\mu,T(V),\bN_{T(V)}})\to
H_*^{\GL(V)_\CO}(\CR^\lambda_\mu)\otimes_{H^*_{\GL(V)}(\on{pt})}H^*_{T(V)}(\on{pt})$, and
$\Xi^\circ$ sends $y_{i,r}$ to $\iota_*\bar\sy_{i,r}$
(and is identical on $w_{i,r}$).

\begin{Theorem}
\label{Coulomb_quivar}
The isomorphism
$\Xi^\circ\colon \BC[\oW^{\lambda^*}_{\mu^*}]|_{\oG_m^{\alpha}}\iso
H_*^{\GL(V)_\CO}(\CR^\lambda_\mu)|_{\oG_m^\alpha}$
extends to a biregular isomorphism
$\BC[\oW^{\lambda^*}_{\mu^*}]\iso H_*^{\GL(V)_\CO}(\CR^\lambda_\mu)$.
\end{Theorem}

\begin{proof}
We repeat the argument in the proof of~Theorem~\ref{pestun}, and use
its notations. We introduce $\uoW^{\lambda^*}_{\mu^*}:=
\oW^{\lambda^*}_{\mu^*}\times_{\BA^{\alpha}}\BA^{|\alpha|}$.
We consider a general
$t\in\bA^{|\alpha|}\setminus\oG_m^{|\alpha|}\subset\ft(V)$.
According to~\ref{prop:flat} and~\ref{rem:conditions}, we have to identify the localizations
$\left(\BC[\uoW^{\lambda^*}_{\mu^*}]\right)_{t}$ and
$\left(H_*^{T(V)_\CO}(\CR^\lambda_\mu)\right)_t$ as
$\BC[\BA^{|\alpha|}]_{t}=\BC[\ft(V)]_t$-modules. There are two
possibilities: either $t$ lies on a diagonal divisor,
or $t$ is a general point of a coordinate hyperplane $w_{i,r}(t)=0$.
The former case having been dealt with in the proof of~Theorem~\ref{pestun},
it is enough to treat the latter case.

%By the Localization theorem,
%$\left(H_*^{T(V)_\CO}(\CR^\lambda_\mu)\right)_t\simeq
%\left(H_*^{T(V)_\CO}((\CR^\lambda_\mu)^t)\right)_t$,
Then the fixed point set $(\CR^\lambda_\mu)^t$ is isomorphic to the product
$\Gr_{T_1}\times\CR_{T_2,\bN'}$. Here $T_2$ is a 1-dimensional torus with coordinate
$\sw_{i,r}$, and $T_1$ is an $(|\alpha|-1)$-dimensional torus with coordinates
$\{\sw_{j,s}\colon (j,s)\ne(i,r)\}$, so that $T(V)=T_1\times T_2$. Furthermore,
$\bN'$ is an $l_i$-dimensional representation of $T_2$ equal to the direct sum
of $l_i$ copies of the character $\sw_{i,r}$ with differential $w_{i,r}$.
In notations of~\ref{prop:flat}(2), $G'=T(V)$, and
$$\left(H_*^{G'_\CO}(\CR_{G',\bN'})\right)_t=
\left(H_*^{T(V)_\CO}(\Gr_{T_1}\times\CR_{T_2,\bN'})\right)_t=
(\BC[{\mathfrak t}_1\times T_1^\vee]\otimes C)_t$$ where $C$ is
generated by $w_{i,r}$ and the fundamental classes $\sx_{i,r},\bar\sy'_{i,r}$
of the fibers of $\CR_{T_2,\bN'}$ over the points $-w^*_{i,r},w^*_{i,r}\in\Gr_{T_2}$
respectively. According to~\ref{abel}, the relations in $C$ are
as follows: $\sx_{i,r}\bar\sy'_{i,r}=w_{i,r}^{l_i}$.

According to~\ref{sec:chang-repr}, the homomorphism
$\bz'^*\colon H_*^{T(V)_\CO}(\Gr_{T_1}\times\CR_{T_2,\bN'})\to H_*^{T(V)_\CO}(\Gr_{T(V)})$
takes $\bar\sy'_{i,r}$ to $\sfu_{i,r}$, while $\bz'^*\sx_{i,r}=w_{i,r}^{l_i}\sfu_{i,r}^{-1}$.
We define $\Xi^t\colon \left(\BC[\uoW^\lambda_{\mu}]\right)_{t^*}\to
\left(H_*^{G'_\CO}(\CR_{G',\bN'})\right)_t$ identical on $w_{k,p}$ and sending
$y_{k,p}$ to $\sfu_{k,p}\cdot\prod_{h\in\Qo:\vout{h}=k}\prod_{1\leq q\leq a_{\vin{h}}}(w_{\vin{h},q}-w_{k,p})$.
In particular, $x_{i,r}=y_{i,r}^{-1}w_{i,r}^{l_i}$ goes to
$\sx_{i,r}\prod_{h\in\Qo:\vout{h}=i}\prod_{1\leq q\leq a_{\vin{h}}}(w_{\vin{h},q}-w_{i,r})$.
Note that at the moment $\Xi^t$ is defined only as a rational morphism.
The condition of~\ref{prop:flat}(2) is trivially satisfied. Also, $\Xi^t$
is a regular isomorphism due to the factorization property of
$\oW^\lambda_{\mu}$ and the fact that the factors $(w_{\vin{h},q}-w_{k,p})$ in the
formula for $\Xi^t$ are all invertible at $t$.
%Now the localization $\BC[\uoW^\lambda_{\mu}]_t$ is isomorphic to
%$(\BC[{\mathfrak t}_1\times T_1^\vee]\otimes A)_t$ due to the factorization
%property~\propref{prop:factor}.
The theorem is proved.
\end{proof}

\begin{Remark}
\label{open_piece}
It follows from~\propref{prop:blow} that the restriction of
$s^{\lambda^*}_{\mu^*}\colon \oW^{\lambda^*}_{\mu^*}\to Z^{\alpha}$ to
$\oZ^{\alpha}\subset Z^{\alpha}$
is an isomorphism $(s^{\lambda^*}_{\mu^*})^{-1}(\oZ^{\alpha})\iso\oZ^{\alpha}$, and thus
we have a canonical localization embedding
$\BC[\oW^{\lambda^*}_{\mu^*}]\hookrightarrow\BC[\oZ^{\alpha}]$. Under the isomorphisms
of~Theorem~\ref{Coulomb_quivar} and~Theorem~\ref{pestun} this embedding
is nothing but the one of~\ref{rem:further} corresponding to
$\bN_{\on{hor}}\hookrightarrow\bN_{\on{hor}}\oplus\bN_{\on{vert}}$. Here
$\bN_{\on{hor}}:=\bigoplus_{h\in\Qo}\on{Hom}(V_{\vout{h}},V_{\vin{h}})$
(resp.\ $\bN_{\on{vert}}:=\bigoplus_{i\in \II}\on{Hom}(W_i,V_i)$) is a direct
summand of $\bN^\lambda_\mu=\bigoplus_{h\in\Qo}
\on{Hom}(V_{\vout{h}},V_{\vin{h}})\oplus\bigoplus_{i\in \II}\on{Hom}(W_i,V_i)$.

In the same way, we have $\BC[\oW^{\lambda^*+\nu^*}_{\mu^*+\nu^*}] \to
\BC[\oW^{\lambda^*}_{\mu^*}]$ by adding $\BC^{n_i}$ to $W_i$ for $\nu
= \sum \nu_i \omega_i$. The corresponding birational morphism
$\oW^{\lambda^*}_{\mu^*}\to\oW^{\lambda^*+\nu^*}_{\mu^*+\nu^*}$ was
%given by twisting $\sigma$ by a trivialization with pole of degree $\nu^*$.
constructed in~\subsecref{multiplication}.
\end{Remark}

\begin{Remark}
\label{Cartan_grading}
$H_*^{\GL(V)_\CO}(\CR^\lambda_\mu)$ is naturally graded by $\pi_1\GL(V)=\BZ^\II$.
Under the isomorphism of~Theorem~\ref{Coulomb_quivar} this grading
becomes the grading of $\BC[\oW^{\lambda^*}_{\mu^*}]$ by the root lattice of the
Cartan torus $T\subset G\colon \BZ^\II=\BZ\langle\alphavee_i\rangle_{i\in \II}$
corresponding to the natural action of $T$ on $\oW^{\lambda^*}_{\mu^*}$.

This abelian group action extends to an action of $\Stab_G(\mu^*)$,
the stabilizer of $\mu^*$ in $G$. This is the expected
property~\cite[\S4(iv)(d)]{2015arXiv150303676N}.

Similarly, if $\alpha$ happens to be a dominant coweight $\alpha=\lambda$,
the action of the Cartan torus $T$ on $\oZ^\alpha$ of~Remark~\ref{cartan_grading}
extends to the action of $\Stab_G(\lambda)$. Indeed, the morphism
$s^{\lambda^*}_0\colon \oW^{\lambda^*}_{0}\to Z^\alpha$ restricts to an isomorphism
of the open subvarieties $\oW^{\lambda^*}_{0}\setminus\bigcup_i\oE_i\iso\oZ^\alpha$
(see~\subsecref{divisors}). The isomorphism $\iota^{\lambda^*}_0$ restricts to
$\oW^{\lambda^*}_{0}\setminus\bigcup_i\oE_i\iso S_\lambda\cap\oW^\lambda_{0}$
(the open intersection with a semiinfinite orbit).
The composition of the above isomorphisms gives an identification
$\oZ^\alpha\iso S_\lambda\cap\oW^\lambda_{0}$, and the latter intersection is
naturally acted upon by $\Stab_G(\lambda)$.

These actions will be realized directly in terms of Coulomb branches
in \cite[App.~A]{affine}.
\end{Remark}

\begin{Remark}
\label{compare_degrees_slice}
The LHS of~Theorem~\ref{Coulomb_quivar} is naturally
graded by half the homological degree $\deg_h$, while the RHS is naturally graded by the action of loop rotations, $\deg_r$. These gradings are different.
Let $x$ be a homogeneous homology class supported at the connected component
$\nu=(n_i)\in\BZ^\II=\pi_0\Gr_{\GL(V)}$. Then one can check that
$\deg_r(x)=\deg_h(x)-\nu^t\cdot\sqrt{\det\bN_{\on{hor}}}+
\frac{1}{2}\nu^t\cdot{\mathbf C}\cdot\alpha=
\Delta(x)+\nu^t\cdot\sqrt{\det\bN_{\on{vert}}}+
\frac{1}{2}\nu^t\cdot{\mathbf C}\cdot\alpha$
(cf.~Remark~\ref{compare_degrees_zas}).
\end{Remark}

\begin{Remark}
\label{twist}
Let $Q$ be a folding of the Dynkin diagram of $G$ with the corresponding
nonsymmetric Cartan matrix ${\mathbf C}_Q$, and the corresponding nonsimply
laced group $G_Q$. Let $\lambda_Q\geq\mu_Q$ be dominant coweights of $G_Q$,
and $\oW^{\lambda_Q}_{G_Q,\mu_Q}$ the corresponding slice. Then the Hilbert series
of ${\mathbb C}[\oW^{\lambda_Q}_{G_Q,\mu_Q}]$ graded by the loop rotations equals
$$\sum_\theta t^{2\Delta(\theta)+\nu^t\cdot\det\bN_{\on{vert}}+
\nu^t\cdot{\mathbf C}_Q\cdot\alpha_Q}P_{\GL(V_Q)}(t;\theta).$$
Here $\theta$ runs through the set of dominant coweights of $\GL(V_Q)$,
and $\nu$ stands for the connected component of $\Gr_{\GL(V_Q)}$ containing
$\theta$,
and $\Delta(\theta)$ is defined in~\cite[(3.2)--(3.4)]{Cremonesi:2014xha}
(see also~\secref{sec:twist} below).
A proof follows from the realization of $\oW^{\lambda_Q}_{G_Q,\mu_Q}$ as the folding
of an appropriate $\oW^\lambda_{\mu}$ and~Remark~\ref{compare_degrees_slice}.
\end{Remark}

\begin{Remark}
\label{Braverm}
Similarly to~\secref{positive}, we define $\CR^{\lambda+}_{\mu}$ as the preimage
of $\Gr_{\GL(V)}^+\subset\Gr_{\GL(V)}$ under $\CR^{\lambda}_{\mu}\to\Gr_{\GL(V)}$.
Then $H_*^{\GL(V)_\CO}(\CR^{\lambda+}_{\mu})$ forms a convolution subalgebra of
$H_*^{\GL(V)_\CO}(\CR^{\lambda}_{\mu})$. On the other hand, the pullback under
$s^{\lambda^*}_{\mu^*}\colon \oW^{\lambda^*}_{\mu^*}\to Z^\alpha$ realizes
$\BC[Z^\alpha]$ as a subalgebra of $\BC[\oW^{\lambda^*}_{\mu^*}]$.
The proof of~Theorem~\ref{Coulomb_quivar} repeated essentially word for word
shows (cf.~\corref{Pestun}) that the isomorphism
$\BC[\oW^{\lambda^*}_{\mu^*}]\iso H_*^{\GL(V)_\CO}(\CR^{\lambda}_{\mu})$ takes
$\BC[Z^\alpha]\subset\BC[\oW^{\lambda^*}_{\mu^*}]$ onto
$H_*^{\GL(V)_\CO}(\CR^{\lambda+}_{\mu})\subset H_*^{\GL(V)_\CO}(\CR^{\lambda}_{\mu})$,
and in particular induces an isomorphism
$\BC[Z^\alpha]\iso H_*^{\GL(V)_\CO}(\CR^{\lambda+}_{\mu})$.
\end{Remark}

\begin{Remark}
\label{Joel}
Let us consider the opposite orientation $\ol\Qo$ of our quiver and
the representation $\bN^\lambda_\mu=\bigoplus_{h\in\ol\Qo}\Hom(V_i^*,V_j^*)
\oplus\bigoplus_{i\in \II}\Hom(V_i^*,W_i^*)$ of $\GL(V^*)$. Similarly
to~Theorem~\ref{Coulomb_quivar}, we have an isomorphism
$\BC[\oW^{\lambda^*}_{\mu^*}]\iso H_*^{\GL(V^*)_\CO}(\CR_{\GL(V^*),\bN^\lambda_\mu})$.
Similarly to~Remark~\ref{Brav}, we have a convolution algebra isomorphism
$$\fri^\lambda_{\mu*}\colon H_*^{\GL(V)_\CO}(\CR_{\GL(V),\bN^\lambda_\mu})\iso
H_*^{\GL(V^*)_\CO}(\CR_{\GL(V^*),\bN^\lambda_\mu}).$$ The composition
$\BC[\oW^{\lambda^*}_{\mu^*}]\simeq H_*^{\GL(V)_\CO}(\CR_{\GL(V),\bN^\lambda_\mu})\stackrel
{\fri^\lambda_{\mu*}}{\longrightarrow}H_*^{\GL(V^*)_\CO}(\CR_{\GL(V^*),\bN^\lambda_\mu})
\simeq\BC[\oW^{\lambda^*}_{\mu^*}]$ is an involution of the algebra
$\BC[\oW^{\lambda^*}_{\mu^*}]$. Similarly to~Remark~\ref{Brav}, this involution
arises from the involution $\iota^{\lambda^*}_{\mu^*}$ of
$\BC[\oW^{\lambda^*}_{\mu^*}]$ (see~\secref{involution}) composed with the
involution $\varkappa_{-1}$ of $\oW^{\lambda^*}_{\mu^*}$ induced by an automorphism
$z\mapsto-z\colon \BP^1\iso\BP^1$ and finally composed with the action $a(h)$
of a certain element of the Cartan torus $h=\beta(-1)\in T$.
Here $\beta$ is a cocharacter of $T$ equal to
$\sum_{i\in \II}b_i\omega_i$ where $b_i=a_i-\sum_{h\in\Qo:\vout{h}=i}a_{\vin{h}}$.
\end{Remark}

\begin{Remark}
\label{trigo}
According to~\ref{rem:Steinberg}(3), we can consider the convolution
algebra $K^{\GL(V)_\CO}(\CR^\lambda_\mu)$. Then, similarly
to~Theorem~\ref{Coulomb_quivar}, one can construct an isomorphism
$\BC[^\dagger\oW^{\lambda^*}_{\mu^*}]\simeq K^{\GL(V)_\CO}(\CR^\lambda_\mu)$
where $^\dagger\oW^{\lambda^*}_{\mu^*}$ stands for the moduli space of the triples
$(\scP,\sigma,\phi)$ where $\scP$ is a $G$-bundle on $\BP^1;\ \sigma$ is a
trivialization of $\scP$ off $1\in\BP^1$ having a pole of degree
$\leq\lambda^*$ at $1\in\BP^1$, and
$\phi$ is a $B$-structure on $\scP$ of degree $-\mu$ having the fiber $B_-$ at
$\infty\in\BP^1$ and transversal to $B$ at $0\in\BP^1$ (a trigonometric slice),
cf.~\cite[1.5]{fkr}.
\end{Remark}

\subsection{Poisson structures}
\label{symplectic}
The convolution algebra $H_*^{\GL(V)_\CO}(\CR)$
(resp.\ $H_*^{\GL(V)_\CO}(\CR^\lambda_\mu))$ carries a Poisson structure $\{,\}_C$
because of the deformation $H_*^{\GL(V)_\CO\rtimes\BC^\times}(\CR)$
(resp.\ $H_*^{\GL(V)_\CO\rtimes\BC^\times}(\CR^\lambda_\mu)$). The algebra
$\BC[\oZ^{\alpha}]$ carries (a nondegenerate, i.e.\ symplectic) Poisson structure $\{,\}_Z$
defined in~\cite{fkmm}. In case $\mu$ is dominant, the algebra
$\BC[\oW^{\lambda^*}_{\mu^*}]$ carries a Poisson structure $\{,\}_\CW$ defined in~\cite{kwy}.

\begin{Proposition}
\label{prop:poi}
\textup{(1)} The isomorphism of~Theorem~\ref{pestun}
takes $\{,\}_Z$ to $-\{,\}_C$.

\textup{(2)} The isomorphism of~Theorem~\ref{Coulomb_quivar}
takes $\{,\}_\CW$ to $-\{,\}_C$.
\end{Proposition}

\begin{proof}
\textup{(1)} It is enough to check the claim generically, over $\oA^{\alpha}$.
We consider
the new coordinates $u_{i,r}:=y_{i,r}\cdot\prod_{j\leftarrow i}\prod_{1\leq s\leq a_j}
(w_{j,s}-w_{i,r})^{-1}$ on $\uoZ^{\alpha}$ (note that the new coordinates depend
on the choice of orientation). It is easy to check that the only
nonvanishing Poisson brackets on $\uoZ^{\alpha}$ are
$\{w_{i,r},u_{j,s}\}=\delta_{i,j}\delta_{r,s}u_{j,s}$
(see~\cite[Proposition~2]{fkmm}). Hence the generic isomorphism
$\Xi\colon \BC[\oZ^{\alpha}]\otimes_{\BC[\BA^{\alpha}]}\BC[\oA^{\alpha}]\iso
\BC[\ft(V)\times T^\vee(V)]\otimes_{\BC[\BA^\alpha]}\BC[\oA^\alpha]$
of \cref{Co_bra} takes the Poisson structure on $\oZ^\alpha$ to the
negative of the standard Poisson structure on $\ft(V)\times T^\vee(V)$. According
to~\ref{cor:coordinate}, $\bz^*\iota_*^{-1}$ takes the latter structure to
the one on $H_*^{\GL(V)_\CO}(\CR)$.

%Recall from the proof of~Theorem~\ref{pestun} that the homomorphism
%$\bz^*\colon H_*^{T(V)_\CO}(\CR_{T(V),\bN_{T(V)}})\to
%H_*^{T(V)_\CO}(\Gr_{T(V)})$ takes
%$\sy_{i,r}$ to $\sfu_{i,r}\cdot\prod_{j\leftarrow i}
%\prod_{1\leq s\leq a_j}(w_{i,r}-w_{j,s})$
%where $\sfu_{i,r}$ is the fundamental class of the point of $\Gr_{T(V)}$
%corresponding to the cocharacter $w_{i,r}^*$ of $T(V)$. Now $\bz^*$ is Poisson
%according to~\lemref{lem:convolution}, and the only nonvanishing Poisson
%brackets on $H_*^{T(V)_\CO}(\Gr_{T(V)})$ are
%$\{w_{i,r},\sfu_{j,s}\}=\delta_{i,j}\delta_{r,s}\sfu_{j,s}$. This completes the
%proof of (a).

\textup{(2)} Again it suffices to check the claim generically where it follows
from (1). In effect, $s^{\lambda^*}_{\mu^*}$ generically is a symplectomorphism
according to~\cite[Theorem~4.9]{fkr}, while the identification of
$H_*^{\GL(V)_\CO}(\CR^\lambda_\mu)|_{\oG_m^{\alpha}}$ and $H_*^{\GL(V)_\CO}(\CR)|_{\oG_m^\alpha}$
is symplectic by~\ref{lem:convolution}.
\end{proof}

\begin{Remark}
\label{stabilization}
Let $\cW^{\lambda^*}_{\mu^*}$ denote the open subvariety of
$\oW^{\lambda^*}_{\mu^*}$ consisting of triples $(\scP,\sigma,\phi)$ such that $\sigma$ has a pole exactly of order $\lambda^*$. It is the preimage of the orbit $\Gr^{\lambda^*}_G$ under $\bp$.
We have the decomposition
\begin{equation*}
    \oW^{\lambda^*}_{\mu^*} =
    \bigsqcup_{\substack{\mu\le\nu\le\lambda\\ \text{$\nu$ : dominant}}}
    \cW^{\nu^*}_{\mu^*}.
\end{equation*}
In \cite[\S5(iii)]{2015arXiv151003908N} the third named author showed
that this is the decomposition into symplectic leaves. However the
argument is based on the description of $\cW^{\lambda^*}_{\mu^*}$ as
the moduli space of singular monopoles, which is not justified yet. In
particular, we do not know how to show that $\cW^{\nu^*}_{\mu^*}$ is
smooth in our current definition. In fact, this is the only problem:
(1) We know that the Poisson structure on the Coulomb branch is
symplectic on its smooth locus by \ref{prop:symplectic}. (2) The
embedding $\oW^{\nu^*}_{\mu^*}\to \oW^{\lambda^*}_{\mu^*}$ is Poisson,
as it is so when $\mu$ is dominant by \cite[Th.~2.5]{kwy} and the
birational isomorphisms $\oW^{\lambda^*}_{\mu^*}\to
\oW^{\lambda^*+\varsigma^*}_{\mu^*+\varsigma^*}$, $\oW^{\nu^*}_{\mu^*}\to
\oW^{\nu^*+\varsigma^*}_{\mu^*+\varsigma^*}$ in \cref{open_piece} is Poisson
by \ref{lem:convolution}.

It was also shown that a transversal slice to $\cW^{\nu^*}_{\mu^*}$ in
$\oW^{\lambda^*}_{\mu^*}$ is isomorphic to $\oW^{\lambda^*}_{\nu^*}$
in \cite[\S5(iii)]{2015arXiv151003908N}. We do not know how to justify
this either.

These two assertions are true if $\nu$ is dominant, as
$\oW^{\lambda^*}_{\mu^*}$ is a transversal slice to $\Gr^{\mu^*}_G$ in
$\overline{\Gr}^{\lambda^*}_G$ in this case.

We also know these under the following condition:
%Let $\overline\mu$
%be either 0, or negative minuscule weight in the coset of $\mu$ modulo
%the coroot lattice. We assume $\mu\le\overline\mu$ (say $\mu$ is
%antidominant enough).
Let $\mu\leq w_0\lambda$ be antidominant.
Then the projection $\bp\colon
\oW^{\lambda^*}_{\mu^*}\to\ol\Gr{}_G^{\lambda^*}$ is smooth and its
image intersects nontrivially all the $G_\CO$-orbits
$\Gr_G^{\nu^*}\subset\ol\Gr{}_G^{\lambda^*}$.
%
%(ADD A REFERENCE, or EXPLANATION !)
%
Indeed, the smoothness of $\bp$ for antidominant $\mu$
follows by the base change from the
smoothness of $\on{Bun}_B^{-\mu}(\BP^1)\to{}'\!\on{Bun}_G(\BP^1)$
(see~\subsecref{nondom}). To check the latter smoothness at a point
$(\scP,\phi)\in\on{Bun}_B^{-\mu}(\BP^1)$ we have to prove the surjectivity
at the level of tangent spaces, and this follows from the long exact
sequence of cohomology and the vanishing of
$H^1(\BP^1,(\fg/{\mathfrak b})_\phi(-1))$. The latter vanishing holds since
the vector bundle $(\fg/{\mathfrak b})_\phi$ is filtered with the associated
graded $\bigoplus_{\alphavee\in R^\vee_+}\CO_{\BP^1}(-\langle\mu,\alphavee\rangle)$:
a direct sum of ample line bundles.
\begin{NB}
An example of nonsmoothness for $G=\SL(3)$: we consider flags of degree
$2\alpha_2$, that is $0\subset\CL_1\subset\CL_2\subset\CV$ such that
$\CL_1\simeq\CO_{\BP^1}$, and $\CL_2/\CL_1\simeq\CO_{\BP^1}(-2)$.
If we take $\CV=\CO^3_{\BP^1}$, then the fiber is the usual space of maps
to the flag variety of $\SL(3)$; it has dimension $3+4=7$.
If we take $\CV=\CO_{\BP^1}(-2)\oplus\CO_{\BP^1}(1)\oplus\CO_{\BP^1}(1)$,
then the fiber has dimension $3+5=8>7$, hence no smoothness.
\end{NB}%
\begin{NB}
  An example of nonsurjectivity for $G=\GL(2)\colon \lambda=(1,0),\ \mu=(0,1)$.
  Then $\ol\Gr_G^\lambda=\Gr^\lambda$ is formed by subsheaves
  $\CV\subset\CO_Ce_1\oplus\CO_Ce_2$ such that the quotient is 1-dimensional
  supported at $0\in C=\BP^1$. Incidentally, $\Gr_G^\lambda$ is isomorhic to $\BP^1$.
  The slice $\ol\CW{}^\lambda_\mu$ is formed by line subbundles
  $\CL\subset\CV\subset\CO_Ce_1\oplus\CO_Ce_2$ with value $\BC e_2$ at
  $\infty\in C$. If $\CV=\CO_C(-0)e_1\oplus\CO_Ce_2$, then $\CL$ is the
  graph of a morphism from $\CO_Ce_1$ to $\CO_Ce_2$ and can not take value
  $\BC e_2$ at $\infty\in C$. So the image of
  $\bp\colon\CW{}^\lambda_\mu\to\ol\Gr{}_G^\lambda$ is
  $\BA^1\subset\BP^1\simeq\ol\Gr{}_G^\lambda$.
\end{NB}%
To see that $\bp(\ol\CW{}^{\lambda^*}_{\mu^*})\cap\Gr_G^{\nu^*}\ne\emptyset$
for $\mu\leq w_0\lambda,\ \nu\leq\lambda$, recall that
$'\!\on{Bun}_G(\BP^1)=\bigsqcup_\text{$\chi$ : dominant}{}'\!\on{Bun}_G(\BP^1)_\chi$
is stratified according to the isomorphism types of $G$-bundles.
The image of $\ol\Gr{}^{\lambda^*}_G$ in $'\!\on{Bun}_G(\BP^1)$ lies in the open
substack $'\!\on{Bun}_G(\BP^1)_{\leq\lambda^*}:=\bigsqcup_{\chi\leq\lambda^*}{}
'\!\on{Bun}_G(\BP^1)_\chi$. Finally, the image of
$\on{Bun}_B^{-\mu}(\BP^1)\to{}'\!\on{Bun}_G(\BP^1)$ contains the open substack
$'\!\on{Bun}_G(\BP^1)_{\leq\lambda^*}$ if $\mu\leq w_0\lambda$.

Hence $\cW^{\nu^*}_{\mu^*}$ is smooth and its transversal slice is
isomorphic to $\oW^{\lambda^*}_{\nu^*}$.

This assumption is not artificial. The above stratification is
\emph{dual} to one for the corresponding Higgs branches (quiver
varieties ${\mathfrak M}_0(V,W)$).
It is known that ${\mathfrak M}_0(V,W)$ stabilizes for sufficiently
large $V$, more precisely if $\mu\le\overline\lambda$, where
$\overline\lambda$ is the minimal dominant weight $\le\lambda$. (See
e.g.,~\cite[Rem.~3.28]{Na-alg}.) This condition is weaker than
$\mu\leq w_0\lambda$ and antidominant,
but we at least see that singularities of both
Higgs and Coulomb branches do not change if $\mu$ is sufficiently
antidominant.
\end{Remark}

\subsection{Deformations}
\label{defo}
Recall the setup of~\ref{subsec:flavor}. We choose a Cartan torus
$T(W)=\prod_{i\in \II}T(W_i)\subset\prod_{i\in \II}\GL(W_i)=\GL(W)$ (notations
of \cref{quivar}). We consider the extended group
$1\to\GL(V)\to\GL(V)\times T(W)\to T(W)\to1$ acting on $\bN^\lambda_\mu$,
so that $T(W)$ is the flavor symmetry group. We choose a basis
$z_1,\ldots,z_N$ of the character lattice of $T(W)$
(compatible with the product decomposition $T(W)=\prod_{i\in \II}T(W_i)$),
and view it as
a basis of ${\mathfrak t}^*(W)$, i.e.\ the coordinate functions on
${\mathfrak t}(W)=\BA^N$. According to~\ref{subsec:flavor},
$H_*^{\GL(V)_\CO\times T(W)_\CO}(\CR^\lambda_\mu)$ is a deformation of
$H_*^{\GL(V)_\CO}(\CR^\lambda_\mu)$ over the base
$\on{Spec}(H^*_{T(W)}(\on{pt}))={\mathfrak t}(W)=\BA^N$.

We denote the intersection of the open subsets
$\oA^{\alpha}\times\BA^N\subset\BA^{\alpha}\times\BA^N$ and
$\bigcap_{i\in \II}f_{i,\unl\lambda^*}^{-1}(\BG_m)\subset\BA^{\alpha}\times\BA^N$ by
$\oA^{\alpha,N}$. We define a generic isomorphism
$\Xi^\circ\colon \BC[\oW^{\unl\lambda^*}_{\mu^*}]|_{\oA^{\alpha,N}}\iso
H_*^{\GL(V)_\CO\times T(W)_\CO}(\CR^\lambda_\mu)|_{\oA^{\alpha,N}}$ as
in \cref{quivar}: identical on $z_s$ and $w_{i,r}$, and sending
$y_{i,r}$ to $\iota_*\bar\sy_{i,r}$. Here we view the fundamental class of the
fiber of $\CR^\lambda_{\mu,T(V),\bN_{T(V)}}$ over the point $w^*_{i,r}\in\Gr_{T(V)}$
as an element $\bar\sy_{i,r}\in H_*^{T(V)_\CO\times T(W)_\CO}(\CR^\lambda_{\mu,T(V),\bN_{T(V)}})$.

\begin{Theorem}
\label{Coulomb_defo}
The isomorphism $\Xi^\circ\colon \BC[\oW^{\unl\lambda^*}_{\mu^*}]|_{\oA^{\alpha,N}}\iso
H_*^{\GL(V)_\CO\times T(W)_\CO}(\CR^\lambda_\mu)|_{\oA^{\alpha,N}}$ extends to a biregular
isomorphism
$\BC[\oW^{\unl\lambda^*}_{\mu^*}]\iso H_*^{\GL(V)_\CO\times T(W)_\CO}(\CR^\lambda_\mu)$.
\end{Theorem}

\begin{proof}
Same as the proof of~Theorem~\ref{Coulomb_quivar}.
\end{proof}

\begin{Remark}
The Poisson structure of $H_*^{\GL(V)_\CO\times T(W)_\CO}(\CR^\lambda_\mu)$
transferred to $\oW^{\unl\lambda^*}_{\mu^*}$ via $\Xi^\circ$ is still given
by the formulas of~\propref{prop:poi}. But we do not know its modular
definition, and we cannot see a priori that these generic formulas extend
regularly to the whole of $\oW^{\unl\lambda^*}_{\mu^*}$.
\end{Remark}

\subsection{Affine case}
\label{affine}
We change the setup of \cref{Co_bra}. The Dynkin graph of $G$ is
replaced by its {\em affinization}, so that $\Omega$ is an orientation of
this affinization; $\bN$ is a representation space of $\Omega$ in the new sense,
and so on. We change the setup of \cref{zastava} accordingly: now
$Z^\alpha$ stands for the zastava space of the affine group $G_\aff$, denoted
by $\fU^\alpha_{G;B}$ in~\cite[9.2]{BFG}, and $\oZ^\alpha\subset Z^\alpha$ stands for
its open subscheme denoted by $\oU$ in~\cite[11.8]{BFG}: it is formed by all
the points of $Z^\alpha$ with defects allowed only in the open subset
$\BA^1_{\operatorname{horizontal}}\times(\BA^1_{\operatorname{vertical}}\setminus\{0\})
\subset\BA^1_{\operatorname{horizontal}}\times\BA^1_{\operatorname{vertical}}=\BA^2$.

Similarly to~Theorem~\ref{pestun} and~Corollary~\ref{Pestun}, we have the
following conditional

\begin{Theorem}
\label{conditional}
Assume $\oZ^\alpha$ is normal. The isomorphism $\Xi^\circ\colon
\BC[\oZ^\alpha]|_{\oA^\alpha_{\operatorname{horizontal}}}\iso H_*^{\GL(V)_\CO}(\CR)|_{\oA^{\alpha}}$
defined as in \cref{Co_bra} extends to a biregular isomorphism
$\BC[\oZ^\alpha]\simeq H_*^{\GL(V)_\CO}(\CR)$ and to
$\BC[Z^\alpha]\simeq H_*^{\GL(V)_\CO}(\CR^+)$.
\end{Theorem}

\begin{proof}
We essentially repeat the proof of~Theorem~\ref{pestun}. The only problem
arises with the application of~\ref{prop:flat} whose assumptions can be
verified only conditionally. Namely, we do not know if $Z^\alpha,\oZ^\alpha$
are Cohen-Macaulay. We do know that all the fibers of horizontal factorization
$\pi_\alpha\colon Z^\alpha\to\BA_{\operatorname{horizontal}}^\alpha$ are of the same
dimension $|\alpha|$~(\cite[Corollary~15.4 of~Conjecture~15.3 proved in~15.6]{BFG}).
Note that for the condition $\pi_*\CO_\cM\iso j_*\pi_*\CO_{\cM^\bullet}$ i.e.
$\CO_\CM\iso j_*\CO_{\CM^\bullet}$
of~\ref{prop:flat} it suffices to use the $S_2$-property
i.e.\ the normality of $Z^\alpha,\oZ^\alpha$
(see~\ref{rem:conditions}).
The normality of $Z^\alpha$ (and in fact the Cohen-Macaulay property and even
the Gorenstein property) is proved in type $A$ in~\cite[Corollary~3.6]{bf14}.

Note that if we redefine $\oZ^\alpha$ as the affinization of the space of degree
$\alpha$ based maps from $\BP^1$ to the Kashiwara flag scheme of $G_\aff$,
then the normality (and hence the first part of the theorem) would follow
unconditionally.

Note also that in the affine case one more possibility for $\uoZ^\gamma,\
|\gamma|=2$, arises; namely, $G_\aff=\SL(2)_\aff,\ \gamma=\alpha_i+\alpha_j$.
Then according to~\cite[Example~2.8.3]{fra},
$\BC[\uoZ^\gamma]=\BC[w_i,w_j,y_i,y_j,y_{ij}^{\pm1}]/(y_iy_j-y_{ij}(w_i-w_j)^2)$.
%Just in case, the normality of $Z^\alpha$ in type $C$ is proved
%in~\cite[Corollary~2.11]{frc}.

If~\cite[Conjecture~15.3]{BFG} holds for a symmetric Kac-Moody Lie algebra $\fg$, then
the above argument shows that the spectrum of $H_*^{\GL(V)_\CO}(\CR)$ is isomorphic
to the affinization of $\oZ^\alpha$.
\end{proof}

\subsection{Jordan quiver}
\label{Jordan}
We start with a general result. For a reductive group $G$ and its adjoint
representation $\bN=\fg$ we consider the variety of triples $\CR\to\Gr_G$.
Its equivariant Borel-Moore homology $H_*^{G_\CO}(\CR)$ equipped with the
convolution product forms a commutative algebra, and its spectrum is the
Coulomb branch $\cM_C(G,\fg)$.

\begin{Proposition}
\label{prop:ad}
The birational isomorphism
$(\ft^\circ\times T^\vee)/\Weyl\simeq\CM_C(G,\fg)|_{\Phi^{-1}(\ft^\circ/\Weyl)}$
of~\ref{cor:coordinate} extends
to a biregular isomorphism $(\ft\times T^\vee)/\Weyl\iso\CM_C(G,\fg)$.
\end{Proposition}

Here we denote the Weyl group by $\Weyl$ in order to avoid a conflict
with the vector space $W$.

This is nothing but \ref{prop:ad2}. We give another proof.

\begin{proof}
It is but a slight variation of the proof of~\cite[Theorem~7.3]{MR2135527}.
We have to replace the equivariant $K$-theory in~\cite[Theorem~7.3]{MR2135527}
by the equivariant Borel-Moore homology. More precisely, in notations
of~\cite[Theorem~7.3]{MR2135527}, replacing the torus $H$ by our Cartan torus
$T$, we have to prove the following analogue of~\cite[Lemma~7.6]{MR2135527}:
In $H^T_n(M)$ we have an equality
$\imath^*\jmath_*(H^T_{n+2\dim M'}(\CL'))=\bigoplus_\mu i_{\mu*}(H^T_n(\mu))$ where
$\imath$ stands for the closed embedding $M\hookrightarrow T^*M$, while
$\jmath$ stands for the closed embedding $\CL'\hookrightarrow T^*M$
Recall that the proof of~\cite[Lemma~7.6]{MR2135527} used a homomorphism
$SS\colon K^T(D_M)\to K^T(T^*M)$ from the Grothendieck group of weakly
$T$-equivariant holonomic $D$-modules on $M$. By definition,
$SS=\jmath_*\circ\on{gr}$ where $\on{gr}$ stands for
the associated graded with respect to a good filtration.
Instead, we will use a homomorphism
$SS'\colon K^T(D_M)\to H^T_{2\dim M'}(T^*M)$ associating to a weakly $T$-equivariant
$D$-module $\CF$ the pushforward $\jmath_*$ of the fundamental
class $CC(\CF)\in H_{2\dim M'}^T(\CL')$ of its characteristic cycle. Note that
$CC=\on{symb}\circ\on{gr}$ where $\on{symb}$ stands for the $(2\dim M')$-th (top)
graded component of the Chern character (in the {\em homological} grading).
With this replacement, the proof of~\cite[Lemma~7.6]{MR2135527} carries over
to our homological situation. E.g.\ an equality
$i^*_\nu SS'(j_{\mu!}\CO_{M_\mu})=0\in H^T_0(T^*_\nu M')$ follows from the similar one
in $K$-theory of the above cited proof and the fact that $i^*_\nu$ commutes with the
Chern character (defined with respect to the smooth ambient variety $T^*M'$)
and shifts the homological grading by $2\dim M'$, while $\jmath_*$ commutes
with the top part of the Chern character.

Also, the proof of~\cite[Lemma~7.8]{MR2135527} carries over to our homological
situation essentially word for word. The proposition is proved.
\end{proof}

Now let $V$ be an $n$-dimensional vector space. We consider the adjoint
action of $G=\GL(V)$ on $\on{End}(V)=\fg$. We choose a base in $V$; it gives
rise to a Cartan torus $T(V)$ along with an identification
$T^\vee(V)\simeq\BG_m^n,\ \ft(V)\simeq\BA^n$. From~\propref{prop:ad} we obtain
an isomorphism $\on{Sym}^n\cS_0\iso\CM_C(\GL(V),\on{End}(V))$ where
$\cS_0=\BG_m\times\BA^1$ (see \cref{factorization}). This is the Coulomb branch
of the pure quiver gauge theory for the Jordan quiver. Now we consider the Coulomb branch
of the Jordan quiver gauge theory with framing $W=\BC^l$. Recall that $\cS_l$ is the
hypersurface in $\BA^3$ given by the equation $xy=w^l$.

\begin{Proposition}
\label{prop:ad_taut}
The birational isomorphism
$(\ft^\circ(V)\times T^\vee(V))/S_n\simeq
\CM_C(\GL(V),\on{End}(V)\oplus V\otimes\BC^l)|_{\Phi^{-1}(\ft^\circ(V)/S_n)}$
of~\ref{cor:coordinate} extends to a biregular isomorphism
$\on{Sym}^n\cS_l\iso\linebreak[3]\CM_C(\GL(V),\on{End}(V)\oplus V\otimes\BC^l)$.
The projection $\cM_C\to\ft(V)/S_n=\on{Sym}^n\BA^1$ is the $n$-th
symmetric power of the projection $\cS_l\to\BA^1\colon (x,y,w)\mapsto w$.
\end{Proposition}

\begin{proof}
Similar to the one of~Theorem~\ref{Coulomb_quivar}. More precisely, the proof is reduced
to a consideration at the generic points of the generalized root hyperplanes.
If a generalized root is a root $w_i-w_j$ of $\GL(V)$, then we are reduced to
the $G=\GL(2)$ case of~\propref{prop:ad}. If a generalized root is $w_i$,
we are in the abelian case.
\end{proof}

%%% Local Variables:
%%% mode: latex
%%% TeX-master: "blowup_pre"
%%% End:

\subsection{Towards geometric Satake correspondence for Kac-Moody Lie algebras}

In this subsection we formulate conjectural geometric Satake
correspondence for Kac-Moody Lie algebras using Coulomb
branches.\footnote{When the quiver is of affine type $A$, the
  conjecture was given in \cite[\S7.8]{2016arXiv160602002N} in terms of
  bow varieties. This subsection is written afterwards, but the origin
  of the conjecture in \cite{2016arXiv160602002N} is what we explain
  here.} See \cite{fnkl_icm} for more thorough historical accounts.

Let us assume that a quiver $Q=(\II,\Qo)$ has no edge loops, but is
not necessarily of finite nor affine type. We have the associated
symmetric Kac-Moody Lie algebra $\fg_{\mathrm{KM}}$.

Taking $\II$-graded vector spaces $V = \bigoplus_i V_i$,
$W = \bigoplus_i W_i$, we define $\bN$ as above. Taking simple roots
$\alpha_i$ and fundamental weights $\omega_i$ ($i\in\II$), we assign
two weights $\lambda = \sum_{i\in Q_0} \dim W_i\omega_i$,
$\mu = \lambda - \sum_{i\in Q_0} \dim V_i\alpha_i$.
Let $\bM = \bN\oplus\bN^*$, and
$\bmu\colon\bM\to \operatorname{Lie}\GL(V_Q)$ be the moment map with
respect to the natural $\GL(V_Q)$-action on $\bM$.
Let $\fM(\lambda,\mu)$, $\fM_0(\lambda,\mu)$ be quiver varieties
defined by the third-named author \cite{Na-quiver,Na-alg}:
\begin{equation*}
  \fM_0(\lambda,\mu) = \bmu^{-1}(0)\dslash \GL(V_Q),\quad
  \fM(\lambda,\mu) = \bmu^{-1}(0)\dslash_\chi \GL(V_Q),
\end{equation*}
where $\chi\colon\GL(V_Q)\to\CC^\times$ is the character given by
$\chi(g) = \prod_{i\in Q_0}\det g_i$, and $\dslash_\chi$ is the
geometric invariant theory quotient with respect to the polarization
$\chi$. By its construction $\fM_0(\lambda,\mu)$ is an affine variety,
and we have a projective morphism
$\pi\colon\fM(\lambda,\mu)\to\fM_0(\lambda,\mu)$.
It is known that $\fM(\lambda,\mu)$ is nonsingular and $\pi$ is
semi-small.
Let $\La(\lambda,\mu)$ be the inverse image of $0$ under $\pi$. It is
known to be a half-dimensional subvariety in $\fM(\lambda,\mu)$.

Then the main result in \cite{Na-quiver,Na-alg} says that
\begin{equation*}
  \bigoplus_\mu H_{\mathrm{top}}(\La(\lambda,\mu))
\end{equation*}
has a structure of an integrable highest weight representation
$V(\lambda)$ of $\fg_{\mathrm{KM}}$ with highest weight
$\lambda$. Moreover the summand $H_{\mathrm{top}}(\La(\lambda,\mu))$
corresponds to the weight space of $V(\lambda)_\mu$ with weight
$\mu$. Here `$\mathrm{top}$' denotes the top degree homology group,
i.e.\ of degree $2\dim\La(\lambda,\mu)$.
Since $\pi$ is semi-small, we can identify
$H_{\mathrm{top}}(\La(\lambda,\mu))$ with the isotypical component of
$\mathrm{IC}(\{0\})$ in the direct image
$\pi_*(\CC_{\fM(\lambda,\mu)}[\dim\fM(\lambda,\mu)])$.

Let us turn to the corresponding Coulomb branch
$\cM_C(\lambda,\mu) = \cM_C(\GL(V),\bN)$. By
\cite[Remark~4.5]{2015arXiv150303676N} and references therein, we
identify $\chi\colon\GL(V_Q)\to\CC^\times$ with the cocharacter
$\chi = \pi_1(\chi)^\wedge\colon \CC^\times = \pi_1(\CC^\times)^\wedge
\to \pi_1(\GL(V_Q))^\wedge$,
where $(\ )^\wedge$ denotes the Pontryagin dual. Recall that
$\pi_1(\GL(V_Q))^\wedge$, which is a torus of dimension $\# \II$, acts
naturally on $\cM_C(\lambda,\mu)$ as in Remarks~\ref{cartan_grading},
\ref{Cartan_grading}. In physics terminology a K\"ahler parameter for Higgs branch is an equivariant parameter for Coulomb branch.

Let us define the corresponding attracting set
\begin{equation*}
  \fA_\chi(\lambda,\mu) \defeq \left\{ x\in\cM_C(\lambda,\mu)
  \middle| \exists \lim_{t\to 0} \chi(t) x \right\},
\end{equation*}
which is a closed subvariety in $\cM_C(\lambda,\mu)$, possibly empty in
general.

These $\cM_C(\lambda,\mu)$, $\fA_\chi(\lambda,\mu)$ are related to
representation theory in many situations:

(a) Recall that $\cM_C(\lambda,\mu)$ is the (usual) transversal slice
$\oW^{\lambda}_{\mu}$ in the affine Grassmannian $\Gr_G$ if $Q$ is of
type $ADE$ and $\mu$ is dominant. (We ignore the diagram
automorphism~$*$.)  Then $\fA_\chi(\lambda,\mu)$ was studied in
\cite{MV2}. It is nonempty if and only if the corresponding weight
space $V(\lambda)_\mu$ is nonzero, it is of pure dimension
$2|\lambda-\mu|$, and $H_{\mathrm{top}}(\fA_\chi(\lambda,\mu))$ is
naturally isomorphic to $V(\lambda)_\mu$. It can be also considered as
a stalk of the hyperbolic restriction \cite{Braden} of the
intersection cohomology complex $\mathrm{IC}(\oW^{\lambda}_{\mu})$ of
$\oW^{\lambda}_{\mu}$ with respect to $\chi$.
Moreover the stalk of $\mathrm{IC}(\oW^{\lambda}_{\mu})$ at $\mu$ is
the associated graded $\gr V(\lambda)_\mu$ with respect to the
Brylinski-Kostant filtration up to shift~\cite{Lus-ast,1995alg.geom.11007G}.

(b) Suppose that $Q$ is of affine type. Then $\cM_C(\lambda,\mu)$ is
conjecturally the Uhlenbeck partial compactification of a moduli space
of instantons on the Taub-NUT space invariant under a cyclic group
action, which is proved in affine type $A$
\cite{2016arXiv160602002N}. When $\mu$ is dominant, it is
conjecturally isomorphic to the Uhlenbeck partial compactification of
a moduli space of instantons on $\RR^4$ invariant under a cyclic group
action, which is again proved in affine type A
\cite{2016arXiv160602002N}. In
\cite{braverman-2007,MR2900549,MR3134906} the first and second-named
authors conjectured that statements as in (a) hold for affine Kac-Moody
Lie algebras,\footnote{Strictly speaking, only instantons for
  simply-connected groups are considered. Correspondingly
  representations descend to the adjoint group.} where the definition
of the affine Brylinski-Kostant filtration was later corrected
in~\cite{MR2855083}.

(c) When $\cM_C(\lambda,\mu)$ is isomorphic to a quiver variety (for
different quiver, and $V$, $W$), the attracting set
$\fA_\chi(\lambda,\mu)$ is the tensor product variety in
\cite{Na-Tensor}. In particular, the number of its irreducible
components is given by the tensor product multiplicities. In this way,
some of conjectures in \cite{braverman-2007,MR2900549,MR3134906} can
be proved for affine type $A$ by being combined with I.~Frenkel's
level-rank duality.

(d) Hyperbolic restrictions of the intersection cohomology complex
$\on{IC}$ of the Uhlenbeck partial compactification of a moduli space
of instantons on $\RR^4$ was studied in
\cite{2014arXiv1406.2381B}. This Uhlenbeck space is conjecturally
isomorphic to $\cM_C(\omega_0, \omega_0 - n\delta)$
($n\in\ZZ_{\ge 0}$) for an affine quiver, where $\omega_0$ is the
fundamental weight corresponding to the special vertex $0$, and
$\delta$ is the primitive positive imaginary root.
It follows that the direct sum (over $n$) of hyperbolic restriction of
$\on{IC}$ is isomorphic to
$\on{Sym}(\bigoplus_{d>0} z^{-d}\otimes\mathfrak h)$, where
$\mathfrak h$ is the Cartan subalgebra of the underlying finite
dimensional Lie algebra. (See also \cite[\S7.5]{2016arXiv160406316N}.)

In above examples, we assume that $\mu$ is dominant. (It is so in all
known examples in (c).) For the original geometric Satake
correspondence, we \emph{do} understand all weight spaces
$V(\lambda)_\mu$ not necessarily dominant.
In (b) there is a candidate for a space which we should consider when
$\mu$ is not necessarily dominant in \cite{MR3134906}, which was later
found out to be close to the quiver description of bow varieties in
\cite{2016arXiv160602002N}, but the conjecture there was not checked
even for affine type $A$.

Since Coulomb branches are defined for any quiver without assumption
$\mu$ dominant, not necessarily of finite or affine types, we propose
the following conjecture, which makes the situation much simpler:

\begin{Conjecture}\label{conj:satake}
    \textup{(1)} $\fA_\chi(\lambda,\mu)$ is empty if and only if
    $V(\lambda)_\mu = 0$. Moreover
    $\cM_C(\lambda,\mu)^{\chi(\CC^\times)}$ is a single point if it is
    nonempty.

    \textup{(2)} The intersection of $\fA_\chi(\lambda,\mu)$ with
    symplectic leaves of $\cM_C(\lambda,\mu)$ are lagrangian. Hence
    the hyperbolic restriction functor $\Phi$ for $\chi$ \cite{Braden}
    is hyperbolic semi-small in the sense of
    \cite[3.5.1]{2014arXiv1406.2381B}. In particular,
    $\Phi(\mathrm{IC}(\cM_C(\lambda,\mu)))$ remains perverse, and is
    isomorphic to $H_{\mathrm{top}}(\cA_\chi(\lambda,\mu))$.

    \textup{(3)} The direct sum
    \begin{equation*}
        \bigoplus_\mu \Phi(\mathrm{IC}(\cM_C(\lambda,\mu))) =
        \bigoplus_\mu H_{\mathrm{top}}(\cA_\chi(\lambda,\mu))
    \end{equation*}
    has a structure of a $\fg_{\mathrm{KM}}$-module, isomorphic to
    $V(\lambda)$ so that each summand is isomorphic to
    $V(\lambda)_\mu$.
\end{Conjecture}

We naively expect that the usual stalk
$\mathrm{IC}(\cM_C(\lambda,\mu))$ at the fixed point is `naturally'
isomorphic to the associated graded of $V(\lambda)_\mu$ with respect
to a certain filtration which has a representation theoretic
origin. But we do not know what we mean `natural' nor how we define
the filtration in general. Also we could take another generic
cocharacter $\chi$ of $\pi_1(\GL(V_Q))^\wedge$, but we do not know how
to relate it to a representation theoretic object.

This conjecture is checked (except (2)) for finite type
\cite{2017arXiv170900391K}.

As another evidence, we consider the following example, which is not
necessarily finite or affine. Let us suppose $\dim V_i = 1$ for any
$i\in Q_0$. The Higgs branch $\fM_0(\lambda,\mu)$ is a quiver variety,
but it is also an example of a Goto-Bielawski-Dancer toric
hyper-K\"ahler manifold. The Coulomb branch $\cM_C(\lambda,\mu)$ is
also. By a recent work of Braden-Mautner
\cite{2017arXiv170903004B},\footnote{The authors of
  \cite{2017arXiv170903004B} call Goto-Bielawski-Dancer toric
  hyper-K\"ahler manifolds as hypertoric manifolds.} we have a Ringel
duality between perverse sheaves on $\cM_C(\lambda,\mu)$ and those of
$\fM_0(\lambda,\mu)$. In particular,
$\Phi(\mathrm{IC}(\cM_C(\lambda,\mu)))$ is isomorphic to
$H_{\mathrm{top}}(\La(\lambda,\mu))$, hence is isomorphic to the
weight space $V(\lambda)_\mu$.

In fact, \cite{2017arXiv170903004B} and the above conjecture both come
from a `meta conjecture' saying the category of perverse sheaves on a
Higgs branch (e.g.\ $\fM_0(\lambda,\mu)$ and one on the corresponding
Coulomb branch (e.g.\ $\cM_C(\lambda,\mu)$) should be dual in an
appropriate way. We do not know how strata of $\cM_C(\lambda,\mu)$
look like in general, and the category is probably not highest weight
as studied briefly in \cite{2015arXiv151003908N}.
Nevertheless it is expected that the pushforward from
$\fM(\lambda,\mu)$ and the hyperbolic restriction for $\chi$ are
exchanged under the duality.
It should be also related to the symplectic duality
\cite{2014arXiv1407.0964B}.

\begin{Remarks}
  (1)
  Let us take just $V = \bigoplus_i V_i$ and consider the
  corresponding Coulomb branch $\cM_C(\alpha) = \cM_C(\GL(V),\bN)$
  with $\alpha = \sum_{i\in\II} \dim V_i\alpha_i$. It is expected that
  $\cM_C(\alpha)$ has no fixed point with respect to
  $\chi(\CC^\times)$, hence the above construction does not
  work. Instead we consider
  $\cM_C^+(\alpha) \defeq H^{\GL(V)_\cO}_*(\cR^+)$ as in
  \cref{positive}. This is supposed to be a Kac-Moody generalization
  of the zastava space. The same construction with $W$ gives the same
  space $\cM_C^+(\alpha)$, hence we have a morphism
  $\cM_C(\lambda,\mu)\to \cM_C^+(\alpha)$ as in~\cref{Braverm}. It is
  expected that $\cM_C^+(\alpha)$ is a limit of $\cM_C(\lambda,\mu)$
  when $\lambda$, $\mu\to\infty$ keeping $\lambda-\mu = \alpha$.

  We define the attracting set $\fA_\chi^+(\alpha)$ as the set of points
  contracted to $s_\alpha(0)$ by the action of $\chi$, where
  $s_\alpha\colon \BA^\alpha\hookrightarrow\cM_C^+(\alpha)$ is the section as
  in~Corollary~\ref{Pestun}. Note that the action of $\chi$ contracts the whole 
  of $\cM_C^+(\alpha)$ to $s_\alpha(\BA^\alpha)=\cM_C^+(\alpha)^{\chi(\BC^\times)}$,
  cf.~Remark~\ref{cartan_grading}, so that for any $\phi\in\cM_C^+(\alpha)$,
  there exists $\lim\limits_{t\to 0} \chi(t)$. The integrable system
  $\varpi_\alpha^+\colon \cM_C^+(\alpha)\to\BA^\alpha$ is 
  $\chi(\BC^\times)$-equivariant, so $\fA_\chi^+(\alpha)$ coincides with the
  fiber $(\varpi_\alpha^+)^{-1}(0)$ over $0\in\BA^\alpha$. Furthermore,
  since we expect $\cM_C(\lambda,\mu)^{\chi(\BC^\times)}$ to consist of one point
  if nonempty, this point must be fixed with respect to another action of 
  $\BC^\times$ corresponding to the cohomological grading of the Coulomb branch.
  Thus the image of the fixed point under the morphism 
  $\cM_C(\lambda,\mu)\to\cM_C^+(\alpha)$ must be 
  $s_\alpha(0)\in s_\alpha(\BA^\alpha)\subset\cM_C^+(\alpha)$. It follows that
  the image of $\fA_\chi(\lambda,\mu)$ lies in $\fA^+_\chi(\alpha)$.

  Then we expect that the corresponding statements in
  Conjecture~\ref{conj:satake} are true. In particular, the direct sum
  \begin{equation*}
    \bigoplus_\alpha \Phi(\mathrm{IC}(\cM_C^+(\alpha)) =
    \bigoplus_\alpha H_{\mathrm{top}}(\fA_\chi^+(\alpha))
  \end{equation*}
  is isomorphic to $U(\mathfrak n_-)$ where $\mathfrak n_-$ is the
  negative half of $\fg_{\mathrm{KM}}$. Moreover the pull-back
  homomorphism
  $H_{\mathrm{top}}(\fA_\chi^+(\alpha))\to
  H_{\mathrm{top}}(\fA_\chi(\lambda,\mu))$ corresponds to the quotient map
  $U(\mathfrak n_-)\to V(\lambda)$. When $Q$ is of finite type, these
  statements will be proved in a forthcoming paper by J.~Kamnitzer, P.~Baumann 
  and A.~Knutson.

  (2) Let $\on{Irr}(\fA_\chi(\lambda,\mu))$,
  $\on{Irr}(\fA_\chi^+(\alpha))$ be the set of irreducible components
  of $\fA_\chi(\lambda,\mu)$, $\fA_\chi^+(\alpha)$ respectively. Then
  we conjecture that
  \begin{equation*}
    \bigsqcup_\mu \on{Irr}(\fA_\chi(\lambda,\mu)), \qquad
    \bigsqcup_\alpha \on{Irr}(\fA_\chi^+(\alpha))
  \end{equation*}
  have structures of Kashiwara crystal, isomorphic to crystals
  $B(\infty)$, $B(\lambda)$ of $U_q(\mathfrak n_-)$, $V_q(\lambda)$ of
  the quantized enveloping algebra $U_q(\fg_{\mathrm{KM}})$
  respectively. Moreover the inclusion
  $\on{Irr}(\fA_\chi(\lambda,\mu))\subset
  \on{Irr}(\fA_\chi^+(\alpha))$ corresponds to the embedding
  $B(\lambda)\subset B(\infty)$. When $Q$ is of finite type, these statements
  follow from the comparison of the crystal structures defined 
  in~\cite[Section~13]{BFG} and in~\cite{bragai,2017arXiv170900391K}, 
  cf.~\cite[Proposition~4.3]{baga}.

  Furthermore, we expect that 
  the zero level $\varpi_\alpha^{-1}(0)$ of the integrable system
  $\varpi_\alpha\colon \cM_C(\alpha)\to\BA^\alpha$ is a dense open subset
  of $\fA_\chi^+(\alpha)$, so we have a canonical bijection
  $\on{Irr}(\fA_\chi^+(\alpha))=\on{Irr}(\varpi_\alpha^{-1}(0))$.
  Now the Cartan involution of $\cM_C(\alpha)$ described in~\ref{Brav}
  induces an involution of $\on{Irr}(\varpi_\alpha^{-1}(0))$ and we conjecture
  that the latter involution corresponds to Kashiwara's involution
  $*\colon B(\infty)\to B(\infty)$~\cite[8.3]{kash-crys}. If $Q$ is of finite 
  type, this conjecture follows from the definition of crystal structure 
  in~\cite[Section~13.5]{BFG}, cf.~\cite[Remark~1.7]{bdf}.

  (3) It is conjectured that there is a natural bijection between
  symplectic leaves of $\cM_C(\lambda,\mu)$ and $\fM_0(\lambda,\mu)$
  \cite{2015arXiv151003908N}. When $Q$ is of finite type, closures of
  strata are of the forms $\cM_C(\nu,\mu)$ and $\fM_0(\lambda,\nu)$
  respectively, where $\nu$ runs through dominant weights between $\mu$ and
  $\lambda$. This is known for quiver varieties
  \cite[Prop.6.7]{Na-quiver}, while it is only conjectural for Coulomb
  branches. See \cref{stabilization}. In this case the bijection is
  given by $\cM_C(\nu,\mu)\leftrightarrow\fM_0(\lambda,\nu)$. In
  particular, $\cM_C(\lambda,\mu)$ corresponds to
  $\fM_0(\lambda,\lambda)$, which is a point.
  For more general $Q$, the description of the strata of
  $\cM_0(\lambda,\mu)$ are given in \cite[\S6]{Na-quiver}, by being
  combined with \cite{CB}. For an affine type $Q$, extra strata come
  from symmetric products of simple singularities, which can be
  checked easily. We do not have any description of strata of
  $\cM_C(\lambda,\mu)$ if $Q$ is neither finite nor affine.

  This bijection should be upgraded to a bijection between pairs of
  strata and simple local systems on them, but it becomes even more
  speculative. Assuming this bijection, we conjecture the following:
  Suppose $(S_C,\phi_C)$ and $(S_H,\phi_H)$ are strata of
  $\cM_C(\lambda,\mu)$ and $\fM_0(\lambda,\mu)$ and simple local
  system on them respectively, corresponding under the conjectural
  bijection. Then the isotypical component of
  $\mathrm{IC}(S_H,\phi_H)$ in
  $\pi_*(\CC_{\fM(\lambda,\mu)}[\dim \fM(\lambda,\mu)])$ is isomorphic
  to $\Phi(\mathrm{IC}(S_C,\phi_C))$. The above conjecture studies the
  case $(S_C,\phi_C) = (\cM_C(\lambda,\mu), \mathrm{triv})$,
  $(S_H,\phi_H) = (\fM_0(\lambda,\lambda), \mathrm{triv})$, where
  $\mathrm{triv}$ denotes the trivial local system.

\end{Remarks}

Next we consider a structure giving tensor products of integrable
modules. For quiver varieties it is a tensor product variety
$\mathfrak Z(\lambda^1;\lambda^2)$ corresponding to a decomposition
$W = W^1\oplus W^2$ with $\lambda^a = \sum_i \dim W^a_i\omega_i$
($a=1,2$). It is defined as an attracting set in
$\bigsqcup_\mu \fM(\lambda,\mu)$ with respect to the cocharacter
$\rho\colon \CC^\times\to \GL(W)$ given by
$\rho(t) = \on{id}_{W^1}\oplus t\on{id}_{W^2}$.
We introduce a smaller subvariety
$\widetilde{\mathfrak Z}(\lambda^1;\lambda^2)$ requiring the limit
$\lim_{t\to 0}$ lies in the lagrangian
$\bigsqcup_{\mu^1,\mu^2}
\La(\lambda^1,\mu^1)\times\La(\lambda^2,\mu^2)$. Then \cite{Na-Tensor} says
\begin{equation*}
  H_{\mathrm{top}}(\widetilde{\mathfrak Z}(\lambda^1;\lambda^2))
\end{equation*}
is isomorphic to the tensor product $V(\lambda^1)\otimes V(\lambda^2)$
under the convolution product. (See \cite{2012arXiv1211.1287M} for a
better conceptual construction.) For tensor products
$V(\lambda^1)\otimes\cdots\otimes V(\lambda^N)$, we just take
$W = W^1\oplus\cdots\oplus W^N$ and repeat the same construction.

\begin{NB}
    In some situation we can consider a larger group action than one
    coming from $W = W^1\oplus\cdots\oplus W^N$. Consider a quiver
    variety associated with Jordan quiver (or of affine type
    $A$). Then we have $\CC^\times$ action
    $(B_1,B_2,a,b)\mapsto (tB_1, t^{-1} B_2, a, b)$.
\end{NB}%

Let us turn to the Coulomb branch side. We take a maximal torus $T(W)$
of $\GL(W)$ and regard $\bN$ as a representation of
$\tilde G\defeq \GL(V)\times T(W)$. This gives a deformation of
$\cM_C(\lambda,\mu)$ parametrized by $\operatorname{Lie}(T(W))$ as
$H^{\tilde G_\cO}_*(\cR_{G,\bN})$ as in \cref{defo}. We restrict it to
the direction of $d\rho$, that is
$H^{G_\cO\times\rho(\CC^\times)}(\cR_{G,\bN})$, and denote it by
$\underline{\cM}_C(\lambda,\mu)$. Thus we have a morphism
$\underline{\cM}_C(\lambda,\mu)\to\CC$.
We also consider the variety of
triples $\cR_{\tilde G,\bN}$ for the larger group $\tilde G$ and the
corresponding Coulomb branch
$\cM_C(\tilde G,\bN) = H^{\tilde G_\cO}_*(\cR_{\tilde G,\bN})$. By
\ref{prop:reduction} the original $\cM_C(\lambda,\mu)$ is the
Hamiltonian reduction of $\cM_C(\tilde G,\bN)$ by
$\pi_1(\CC^\times)^\wedge$. Note that $\pi_1(T(W))^\wedge$ is the dual
torus $T(W)^\vee$ of the original torus $T(W)$. Therefore the
cocharacter $\rho\colon \CC^\times\to T(W)$ can be regarded as a
character $T(W)^\vee\to\CC^\times$. Therefore we can consider the
corresponding geometric invariant theory quotient
\begin{equation*}
    \widetilde{\cM}_C(\lambda,\mu) \defeq \bmu^{-1}(0)\dslash_{\rho} T(W)^\vee,
\end{equation*}
as in \ref{prop:GITquotient}, where $\bmu$ denotes the moment map
$\cM_C(\tilde G,\bN)\to \on{Lie} T(W) = \Spec H^*_{T(W)}(\mathrm{pt})$
for the $T(W)^\vee$ action. It is equipped with a projective morphism
$\pi_C\colon \widetilde{\cM}_C(\lambda,\mu)\to \cM_C(\lambda,\mu)$.
If we replace the equation $\bmu=0$ by $\bmu\in\CC d\rho$, we have a
family version $\widetilde{\underline{\cM}}_C(\lambda,\mu)$ equipped
with a projective morphism
$\widetilde{\underline{\cM}}_C(\lambda,\mu)\to
\underline{\cM}_C(\lambda,\mu)$. We conjecture that this is a small
birational morphism and $\widetilde{\underline{\cM}}_C(\lambda,\mu)$
is a topologically trivial family, as for quiver varieties.
Therefore
$\psi (\mathrm{IC}(\underline{\cM}_C(\lambda,\mu))) =
\pi_{C,*}(\mathrm{IC}(\widetilde{\cM}_C(\lambda,\mu)))$, where
$\psi$ is the nearby cycle functor for
$\underline{\cM}_C(\lambda,\mu)\to\CC$ \cite[\S8.6]{KaSha}.
Moreover it contains $\mathrm{IC}(\cM_C(\lambda,\mu))$ with
multiplicity one.

\begin{NB}
    Even if we consider a finer decomposition
    $W = W^1\oplus\cdots\oplus W^N$ to $1$-dimensional subspaces, this
    construction gives only \emph{partial} resolution of singularities
    for affine type $A$. In this case,
    $\widetilde{\cM}_C(\lambda,\mu)$ is a Uhlenbeck partial
    compactification of $\U(n)$ instantons on the resolution
    $\widetilde{\CC^2/(\ZZ/N)}$ of $\CC^2/(\ZZ/N)$. In order to obtain
    a moduli space of framed torsion free sheaves, we need to add an
    extra $\CC^\times$-action as in the above \color{red}{\bf NB}.
\end{NB}%

\begin{NB}
We consider attracting sets with respect to $\chi$,
$\underline{\fA}_\chi(\lambda,\mu)$,
$\widetilde{\fA}_\chi(\lambda,\mu)$,
$\widetilde{\underline{\fA}}_\chi(\lambda,\mu)$ in
$\underline{\cM}_C(\lambda,\mu)$, $\widetilde{\cM}_C(\lambda,\mu)$,
$\widetilde{\underline{\cM}}_C(\lambda,\mu)$ respectively. Let us
denote all hyperbolic restriction functors simply by $\Phi$. Note that
we have morphisms
$\widetilde{\underline{\fA}}_\chi(\lambda,\mu)\to
\underline{\fA}_\chi(\lambda,\mu)\to\CC$
and also
$\widetilde{\fA}_\chi(\lambda,\mu)\to \fA_\chi(\lambda,\mu) \to \{0\}$
as their restriction.
We conjecture that the attracting set
$\widetilde{\underline{\fA}}_\chi(\lambda,\mu)$ is a topologically
trivial family.
\end{NB}%

\begin{Conjecture}
  \textup{(1)}
  $\widetilde{\underline{\cM}}_C(\lambda,\mu)^{\chi(\CC^\times)}$ is a
  disjoint union of finitely many copies of $\CC$ such that the
  restriction of the morphism
  $\widetilde{\underline{\cM}}_C(\lambda,\mu)^{\chi(\CC^\times)}\to\CC$
  to each summand is the identity map.
  And $\underline{\cM}_C(\lambda,\mu)^{\chi(\CC^\times)}$ is obtained
  from $\widetilde{\underline{\cM}}_C(\lambda,\mu)^{\chi(\CC^\times)}$
  by identifying the origin of each summand.

  \textup{(2)} A summand in \textup{(1)} corresponds, in bijection, to
  a decomposition $\mu = \mu^1+\mu^2$ with $V(\lambda^1)_{\mu^1}$,
  $V(\lambda^2)_{\mu^2}\neq 0$. The hyperbolic restriction of
  $\mathrm{IC}(\widetilde{\underline{\cM}}_C(\lambda,\mu))$ is the
  direct sum
  $\bigoplus \Phi(\mathrm{IC}(\cM_C(\lambda^1,\mu^1))\otimes
  \Phi(\mathrm{IC}(\cM_C(\lambda^2,\mu^2))$, where each summand is
  considered as a trivial local system on $\CC$. Hence
  \[
    \psi \circ \Phi
    (\mathrm{IC}(\underline{\cM}_C(\lambda,\mu))) =
    \Phi\pi_{C,*}(\mathrm{IC}(\widetilde{\cM}_C(\lambda,\mu)))
    \cong \bigoplus_{\mu=\mu^1+\mu^2}
    V(\lambda^1)_{\mu^1}\otimes V(\lambda^2)_{\mu^2}.
  \]
  In the first equality we use the commutativity of the nearby cycle
  and hyperbolic restriction functors (see e.g.,
  \cite[Prop.~5.4.1]{2016arXiv160406316N}).
\begin{NB} cf.\ https://arxiv.org/pdf/1611.01669.pdf
\end{NB}%

  \textup{(3)} The sum of homomorphisms
  $\Phi(\mathrm{IC}(\cM_C(\lambda,\mu)))\to
  \Phi(\pi_{C,*}(\mathrm{IC}(\widetilde{\cM}_C(\lambda,\mu))))$ over
  $\mu$ is the homomorphism
  $V(\lambda)\to V(\lambda^1)\otimes V(\lambda^2)$ of
  $\fg_{\mathrm{KM}}$-modules, sending $v_\lambda$ to
  $v_{\lambda^1}\otimes v_{\lambda^2}$, where $v_\lambda$ is the
  highest weight vector corresponding to the fundamental class of the
  point $\cM_C(\lambda,\lambda)$.
\end{Conjecture}

%%% Local Variables:
%%% mode: latex
%%% TeX-master: "blowup_pre"
%%% End:

\section{Non-simply-laced case}\label{sec:twist}

In order to describe instanton moduli spaces for non-simply-laced
groups as Coulomb branches, Cremonesi, Ferlito, Hanany and Mekareeya
have introduced a modification of the monopole formula
\cite{Cremonesi:2014xha}. See also \cite{Mekareeya2015} for more
examples.

Let us consider the case of $G_2$ $k$-instantons on %$\RR^4$ 
the Taub-NUT space for brevity. (See \cite[\S4]{Cremonesi:2014xha}.)
\begin{NB}
This is no longer necessary:

We suppose that the
reader is familiar with the monopole formula, reviewed in
\cite[\S4]{2015arXiv150303676N}.
\end{NB}%
We suppose that we already know that a quiver gauge theory
associated with a symmetric affine Dynkin diagram ($D_4^{(1)}$ in this
case) has the Coulomb branch isomorphic to an instanton moduli space
of the corresponding group.
%
%In fact, we do not have a proof of this assertion, 
This is a special case of the conjecture mentioned in the introduction.
Moreover, it is also conjectured that moduli spaces of instantons on
$\RR^4$ and on the Taub-NUT spaces are isomorphic as affine algebraic
varieties. We do not have a proof of this assertion either for $\RR^4$
of the Taub-NUT space,
but the following argument works more generally.
\begin{NB}
    Added on Oct. 15 by HN
\end{NB}%

As for simply-laced cases, the mirror of an instanton moduli space is,
roughly, a quiver gauge theory associated with the corresponding
affine Dynkin diagram of type $G_2^{(1)}$ with dimension vectors
$\mathbf v = k\delta$, $\mathbf w = \Lambda_0$. See
Figure~\ref{fig:G2} left, where we put the numbering $0$, $1$, $2$ on
vertices.
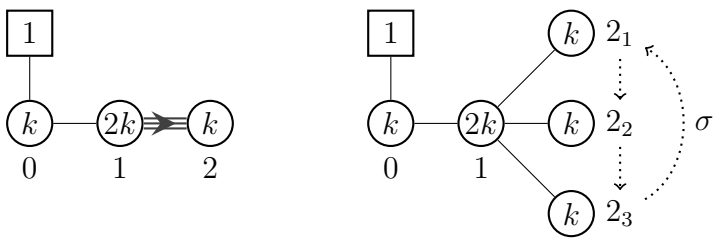
\begin{figure}[htbp]
    \centering
% \setlength{\unitlength}{1mm}
% \begin{picture}(33.5,24)
%     \multiput(3.5,7)(15,0){3}{\circle{7}}
% %    \multiput(3.5,3.5)(15,0){8}{\circle{7}}
%     \multiput(7,7)(15,0){2}{\thicklines\line(1,0){8}}
% %    \multiput(7,3.5)(15,0){7}{\thicklines\line(1,0){8}}
% %    \put(78.5,7){\thicklines\line(0,1){8}}
% %    \put(78.5,18.5){\circle{7}}
%     \put(22,6.5){\thicklines\line(1,0){7.5}}
%     \put(22,7.5){\thicklines\line(1,0){7.5}}
%     \put(30,7){\thicklines\line(-1,-1){2}}
%     \put(30,7){\thicklines\line(-1,1){2}}
%     \put(3.5,10.5){\thicklines\line(0,1){8}}
%     \put(0.5,18.9){\framebox(6,6){$1$}}
%     \put(2.5,5.8){$k$}
%     \put(16.5,5.8){$2k$}
%     \put(32.5,5.8){$k$}
%     \put(2.5,0.3){$0$}
%     \put(17,0.3){$1$}
%     \put(32.5,0.3){$2$}
%     % \put(46.5,2.3){$4k$}
%     % \put(61.5,2.3){$5k$}
%     % \put(76.5,2.3){$6k$}
%     % \put(76.5,17.3){$3k$}
%     % \put(91.5,2.3){$4k$}
%     % \put(106.5,2.3){$2k$}
% \end{picture}
\begin{tikzpicture}[scale=1.2,
circled/.style={circle,draw%=blue!50,fill=blue!20
,thick,inner sep=0pt,minimum size=6mm},
squared/.style={rectangle,draw%=black!50,fill=black!20
,thick,inner sep=0pt,minimum size=6mm},
triplearrow/.style={
  draw=black!75,
  color=black!75,
  thick,
  double distance=3pt, 
%  ->, 
  decoration={markings,mark=at position .75 with {\arrow[scale=.7]{>}}},
  postaction={decorate},
  >=stealth}, 
thirdline/.style={draw=black!75, color=black!75, thick, -%, >=stealth
}
]
\node[circled] (v0) at ( 0,0)  {$k$};
\node[circled] (v1) at ( 1,0)  {$2k$};
\node[circled] (v2) at ( 2,0)  {$k$};
\node[squared] (w0) at ( 0,1)  {$1$};
\node [below] at (v0.south) {$0$};
\node [below] at (v1.south) {$1$};
\node [below] at (v2.south) {$2$};
\draw [-] (v0.east) -- (v1.west);
%\draw [-] (v1.east) -- (v2.west);
%\path (v1.east) -- node {$\Rrightarrow$} (v2.west);
\draw[triplearrow] (v1.east) -- (v2.west);
\draw[thirdline] (v1.east) -- (v2.west);
\draw [-] (w0.south) -- (v0.north);

\node[circled] (vv0) at ( 4,0)  {$k$};
\node[circled] (vv1) at ( 5,0)  {$2k$};
\node[circled] (vv2) at ( 6,1)  {$k$};
\node[circled] (vv3) at ( 6,0)  {$k$};
\node[circled] (vv4) at ( 6,-1)  {$k$};
\node[squared] (ww0) at ( 4,1)  {$1$};
\node [below] at (vv0.south) {$0$};
\node [below] at (vv1.south) {$1$};
\node [right] (21) at (vv2.east) {$2_1$};
\node [right] (23) at (vv4.east) {$2_3$}
   edge [->,bend right=60,dotted,thick] node[auto,swap] {$\sigma$} (21);
\draw [-] (vv0.east) -- (vv1.west);
\draw [-] (vv1) -- (vv2);
\draw [-] (vv1) -- (vv3);
\draw [-] (vv1) -- (vv4);
\draw [-] (ww0.south) -- (vv0.north);
\node [right] (22) at (vv3.east) {$2_2$}
  edge [<-,dotted,thick] (21)
  edge [->,dotted,thick] (23);
\end{tikzpicture}
\caption{$\mathcal M_C$ : $G_2$, $D_4$ $k$-instantons on $\RR^4$ and 
  folding.}
    \label{fig:G2}
\end{figure}

Let $G = \GL(k)\times\GL(2k)\times \GL(k)$, product of general linear
groups for circled vertices as usual. We take a triple $(\la^0, \la^1,
\la^2)$ of coweights of $\GL(k)$, $\GL(2k)$, $\GL(k)$. Let us denote a
triple by $\la$, considered as a coweight of $G$. Let $Y$ be the
coweight lattice of $G$, and $W$ the Weyl group of $G$. The monopole
formula in \eqref{eq:19} says the Hilbert series of the Coulomb branch is
\begin{equation*}
    \sum_{\la\in Y/W} t^{2\Delta(\la)} P_G(t;\la).
\end{equation*}
The definition of $P_G(t;\la) = P_{\GL(k)}(t;\la^0)
P_{\GL(2k)}(t;\la^1) P_{\GL(k)}(t;\la^2)$ is the same as usual.
The term $\Delta(\la)$ has two parts (see \eqref{eq:18}). The first
part is the pairing between $\la$ and positive roots of $G$. This
needs no modification.
The second part, in this example, comes from bi-fundamental
representations on edges. For a usual edge, the contribution is
given by the pairing between its weight with coweights of groups at
two ends. Concretely we write $\la^i = (\la^i_1,\dots,\la^i_k)$
($i=0,2$), $\la^1=(\la^1_1,\dots,\la^1_{2k})$, 
the edge between vertices $0$ and $1$ gives the
contribution \( |\la^0_a - \la^1_b| \) for $a = 1,\dots,k$,
$b=1,\dots,2k$. On the squared vertex, one should put the coweight
$0$, hence there is also \( |\la^0_a| \) for $a=1,\dots,k$.

A modification of the rule is required only for the edge between $1$
and $2$. The rule introduced in \cite{Cremonesi:2014xha} is
\(
   |3\la^1_b - \la^2_c|
\)
for $b=1,\dots,2k$, $c = 1,\dots,k$. Thus
\begin{multline*}
    2\Delta(\la) =
    - 2 \sum_{a\neq a'}|\la^0_a - \la^0_{a'}|
    - 2 \sum_{b\neq b'}|\la^1_b - \la^1_{b'}|
    - 2 \sum_{c\neq c'}|\la^2_c - \la^2_{c'}|
\\
    +
    \sum_{a=1}^k |\la^0_a| + \sum_{a=1}^k\sum_{b=1}^{2k}
    |\la^0_a - \la^1_b| + \sum_{b=1}^{2k}\sum_{c=1}^k |3\la^1_b - \la^2_c|.
\end{multline*}

Now let us explain how to modify our definition of the Coulomb branch
to recover this \emph{twisted monopole formula}.
\begin{NB}
    Corrected on Oct. 15 by HN.
\end{NB}%

We consider the unfolding of our affine Dynkin diagram as
in~\cite[14.1.5(f)]{Lu-book}.\footnote{It should be noted that the non
  simply-laced Lie algebra obtained by the folding in \cite{Lu-book}
  is the Langlands dual of what we get.} It is a $D_4^{(1)}$ affine
Dynkin graph with circled vertices $0,1,2_1,2_2,2_3$ (and $0$ is
connected to a squared vertex).  See Figure~\ref{fig:G2} right.  The
corresponding vector spaces are of dimensions $k,2k,k,k,k$ (and 1).
We orient all the edges from the vertex 1 (and from the squared
vertex).  We consider an automorphism $\sigma$ rotating cyclically the
vertices $2_1,2_2,2_3$. We set
$\bN=V_0\oplus\Hom(V_1,V_0)\oplus\Hom(V_1,V_{2_1})
\oplus\Hom(V_1,V_{2_2})\oplus\Hom(V_1,V_{2_3})$ (a representation of
$\hat{G}:=\GL(V_0)\times\GL(V_1)\times\GL(V_{2_1})\times\GL(V_{2_2})
\times\GL(V_{2_3})$). Then $\sigma$ acts naturally on $\hat G$ and on
$\bN$, hence on $\cM_C(\hat{G},\bN)$. We consider the fixed point set
$\cM_C(\hat{G},\bN)^\sigma$. We have a surjection $\varphi\colon
H^{\hat{G}_\CO}(\CR_{\hat{G},\bN})\twoheadrightarrow\BC[\cM_C(\hat{G},\bN)^\sigma]$.
Thus the grading of $H^{\hat{G}_\CO}(\CR_{\hat{G},\bN})$ whose Hilbert
series is given by the monopole formula induces a grading of
$\BC[\cM_C(\hat{G},\bN)^\sigma]$.

The above formulation and the following proposition work for any
quiver gauge theory with diagram automorphisms. In particular, they
work for quiver gauge theories studied in \secref{QGT}, where their
Coulomb branches are moduli spaces of bundles of an $ADE$ group over
$\proj^1$ with additional structures. The fixed point subscheme
$\mathcal M_C(\hat{G},\bN)^\sigma$ is identified with a moduli space
of bundles of a non simply-laced group.
\begin{NB}
    Added on Oct. 15 by HN
\end{NB}%

\begin{Proposition}
\label{prop:twimon}
The Hilbert series of the induced grading on $\BC[\cM_C(\hat{G},\bN)^\sigma]$
is given by the twisted monopole formula.
\end{Proposition}

\begin{proof}
Recall the multifiltration on $\BC[\cM_C(\hat{G},\bN)]$ introduced 
in~\ref{sec:filtra}. Let us denote the spectrum of the associated
graded algebra by $\ol{\cM}_C(\hat{G},\bN)$. The filtration is 
$\sigma$-invariant, so we have the induced automorphism $\sigma$ of
$\ol{\cM}_C(\hat{G},\bN)$. Moreover, the associated graded of the ideal
$I_\sigma\subset\BC[\cM_C(\hat{G},\bN)]$ of functions vanishing on 
$\cM_C(\hat{G},\bN)^\sigma$ is the ideal 
$\ol{I}_\sigma\subset\on{gr}\BC[\cM_C(\hat{G},\bN)]$ of functions vanishing
on $\ol{\cM}_C(\hat{G},\bN)^\sigma$. Hence it suffices to prove that the
Hilbert series of the induced monopole grading on 
$\BC[\ol{\cM}_C(\hat{G},\bN)^\sigma]$ is given by the twisted monopole formula.

We fix a $\sigma$-invariant Cartan torus $\hat{T}\subset\hat{G}$ corresponding 
to a $\sigma$-invariant decomposition
of $V_i$ into a direct sum of lines. 
%Then $H^{\hat{T}_\CO}(\CR_{\hat{G},\bN})=H^{\hat{G}_\CO}(\CR_{\hat{G},\bN})
%\otimes_{H^{\hat{G}(\on{pt})}}H^{\hat{T}(\on{pt})}$ inherits a grading
%from the monopole grading of $H^{\hat{G}_\CO}(\CR_{\hat{G},\bN})$. Note that
%$H^{\hat{T}_\CO}(\CR_{\hat{G},\bN})$ is a free 
%$H^{\hat{T}(\on{pt})}$-module, and its monopole grading is compatible with 
%the standard grading of $H^{\hat{T}(\on{pt})}$. A natural basis of 
%$H^{\hat{T}_\CO}(\CR_{\hat{G},\bN})$ over $H^{\hat{T}(\on{pt})}$ is
%given by the fundamental classes of preimages in $\CR_{\hat{G},\bN}$ of 
%Iwahori orbits in $\Gr_{\hat{G}}$. These orbits are numbered by their 
%$\hat{T}$-fixed points, that is by the coweight lattice $\hat{Y}$ of $\hat{T}$.
%The fundamental class corresponding to $\hat{\lambda}\in\hat{Y}$ will be
%denoted by $c_{\hat\lambda}$. Then the monopole degree 
%$\deg c_{\hat\lambda}=2\Delta(\hat\lambda^+)+2\ell(\hat\lambda)$ where
%$\hat\lambda^+$ is the unique dominant coweight in the Weyl group orbit
%$\hat{W}\hat\lambda$, and $\ell(\hat\lambda)$ is the length of 
%$\hat{w}\in\hat{W}$ such that $\hat{w}\hat\lambda=\hat\lambda^+$.
We have $\hat\ft^\sigma=\ft$ (the Lie algebra of the Cartan torus 
$T=\hat{T}^\sigma\subset\hat{G}^\sigma=G$). 
%and $\cM_C(\hat{G},\bN)^\sigma\times_{\ft/W}\ft=
%\left(\cM_C(\hat{G},\bN)\times_{\hat\ft/\hat{W}}\hat\ft\right)^\sigma$.
Let us specify a vector subspace $E$ of 
$\on{gr}H^{\hat{G}_\CO}(\CR_{\hat{G},\bN})$
such that the restriction $\ol{\varphi}|_E$ is an isomorphism onto
$\BC[\ol{\cM}_C(\hat{G},\bN)^\sigma]$.
%\times_{\hat\ft/\hat{W}}\hat\ft\right)^\sigma]$.
% \begin{NB}
%     It is not clear to HN what is $\ol{\varphi}|_E$ if $E$ is merely
%     subquotient.
% \end{NB}%
Recall from~\ref{sec:filtra} that $\on{gr}H^{\hat{G}_\CO}(\CR_{\hat{G},\bN})=
\bigoplus_{\hat\lambda\in\hat{Y}^+}\BC[\hat\ft]^{W_{\hat\lambda}}[\CR_{\hat\lambda}]$.
Here $\hat{Y}^+\subset\hat{Y}$ is the cone of dominant coweights of 
$\hat{T},\ \hat{Y}^+\iso\hat{Y}/\hat{W}$.
%Namely, $E$ is a free $\BC[\ft]=\BC[\hat\ft^\sigma]$-submodule with basis
%$\{c_{\hat\lambda}\}_{\hat\lambda\in\hat{Y}'}$ 
We define $\hat{Y}'\subset\hat{Y}^+$ as the set of collections 
$(\hat\lambda^0,\hat\lambda^1,\hat\lambda^{2_1},\hat\lambda^{2_2},\hat\lambda^{2_3})$ 
such that $\hat\lambda^{2_1}_c\geq\hat\lambda^{2_2}_c\geq\hat\lambda^{2_3}_c\geq
\hat\lambda^{2_1}_c-1$ for any $c=1,\ldots,k$. There is a bijection
$\psi\colon \hat{Y}'\iso Y^+$ (the dominant weights of $T$):
$(\hat\lambda^0,\hat\lambda^1,\hat\lambda^{2_1},\hat\lambda^{2_2},\hat\lambda^{2_3})
\mapsto(\lambda^0,\lambda^1,\lambda^2):=(\hat\lambda^0,\hat\lambda^1,  
\hat\lambda^{2_1}+\hat\lambda^{2_2}+\hat\lambda^{2_3})$. Note that for 
$\hat\lambda\in\hat{Y}'$ we have 
$\Delta(\hat\lambda)=\Delta(\psi\hat\lambda)$ (the RHS $\Delta$ is the
twisted one). Finally note that 
%under the natural projection 
%$\BC[\hat\ft]\to\BC[\hat\ft^\sigma]=\BC[\ft],\ \BC[\hat\ft]^{W_{\hat\lambda}}$
%goes onto $\BC[\ft]^{W_{\psi\hat\lambda}}$. 
$W_{\psi\hat\lambda}=W_{\lambda^0}\times W_{\lambda^1}\times W_{\lambda^2}$, and
$W_{\lambda^0}=W_{\hat\lambda^0},\ W_{\lambda^1}=W_{\hat\lambda^1},\
W_{\lambda^2}= W_{\hat\lambda^{2_1}}\cap W_{\hat\lambda^{2_2}}\cap W_{\hat\lambda^{2_3}}$.
\begin{NB} Old wrong version:
$W_{\lambda^2}=W_{\hat\lambda^{2_1}}$ (but $W_{\hat\lambda^{2_2}}$ and $W_{\hat\lambda^{2_3}}$
may be bigger than $W_{\lambda^2}$). 
\end{NB}
\begin{NB}
    HN. Oct.\ 17. 

    Suppose $k=2$. If $\la^2 = (1\ge 0)$, then $\la^{2_1} = (1\ge 0)$,
    $\la^{2_2} = (0\ge 0)$, $\la^{2_3} = (0\ge 0)$. On the other hand,
    if $\la^2 = (2\ge 1)$, then $\la^{2_1} = (1\ge 1)$, $\la^{2_2} =
    (1\ge 0)$, $\la^{2_3} = (0\ge 0)$. If $\la^2 = (3\ge 2)$,
    $\la^{2_1} = (1\ge 1)$, $\la^{2_2} = (1\ge 1)$, $\la^{2_3} = (1\ge
    0)$. Therefore we cannot determine which $2_1$, $2_2$, $2_3$ has
    the same stabilizer.
\end{NB}%
The diagonal embedding 
$\ft^2\hookrightarrow\hat\ft^{2_1}\oplus\hat\ft^{2_2}\oplus\hat\ft^{2_3}$
induces a surjection $\BC[\hat\ft^{2_1}]^{W_{\hat\lambda^{2_1}}}\otimes
\BC[\hat\ft^{2_2}]^{W_{\hat\lambda^{2_2}}}\otimes\BC[\hat\ft^{2_3}]^{W_{\hat\lambda^{2_3}}}
\twoheadrightarrow\BC[\ft^2]^{W_{\lambda^2}}$.
We choose a homogeneous section $\varepsilon$ of this surjection, and 
denote by $E_{\lambda^2}\subset\BC[\hat\ft^{2_1}]^{W_{\hat\lambda^{2_1}}}\otimes
\BC[\hat\ft^{2_2}]^{W_{\hat\lambda^{2_2}}}\otimes\BC[\hat\ft^{2_3}]^{W_{\hat\lambda^{2_3}}}$
the image of $\varepsilon$.
Now we define 
$E:=\bigoplus_{\hat\lambda\in\hat{Y}'}\BC[\ft]^{W_{\psi\hat\lambda}}[\CR_{\hat\lambda}]$
where $\BC[\ft]^{W_{\psi\hat\lambda}}=\BC[\ft^0]^{W_{\lambda^0}}
\otimes\BC[\ft^1]^{W_{\lambda^1}}\otimes E_{\lambda^2}$ is embedded into 
$\BC[\hat\ft]^{W_{\hat\lambda}}=\BC[\hat\ft^0]^{W_{\hat\lambda^0}}
\otimes\BC[\hat\ft^1]^{W_{\hat\lambda^1}}\otimes\BC[\hat\ft^{2_1}]^{W_{\hat\lambda^{2_1}}}
\otimes\BC[\hat\ft^{2_2}]^{W_{\hat\lambda^{2_2}}}
\otimes\BC[\hat\ft^{2_3}]^{W_{\hat\lambda^{2_3}}}$. 
%by $f^0\otimes f^1\otimes f^2\mapsto f^0\otimes f^1\otimes 
%f^2\otimes1\otimes1$.
%
The character of $E$ is given by the twisted monopole formula.
% \begin{NB}
%     Added on Oct. 15 by HN.
% \end{NB}%

It remains to check that 
$\ol{\varphi}\colon E\iso\BC[\ol{\cM}_C(\hat{G},\bN)^\sigma]$.
%We first check surjectivity. Given $\hat\lambda\in\hat{Y}^+$ we will find
%$\hat\mu\in\hat{Y}'$ such that 
%$[\CR_{\hat\lambda}]-[\CR_{\hat\mu}]\in I_\sigma$.
%Unfortunately, we can do this only in the toric case... 
First we consider the similar problem for $\bN=0$. Namely, let
$\widetilde{I}_\sigma\subset\on{gr}H^{\hat{G}_\CO}(\Gr_{\hat{G}})$ be the
ideal generated by the expressions $f-\sigma^*f$, and let 
$\widetilde{E}\subset\on{gr}H^{\hat{G}_\CO}(\Gr_{\hat{G}})$ be defined the same way
as $E$. We will prove 
$\on{gr}H^{\hat{G}_\CO}(\Gr_{\hat{G}})=\widetilde{E}\oplus\widetilde{I}_\sigma$.
To check the surjectivity of 
$\widetilde{E}\to\on{gr}H^{\hat{G}_\CO}(\Gr_{\hat{G}})/\widetilde{I}_\sigma$ we
will find for any $\hat\lambda\in\hat{Y}^+$ a coweight $\hat\mu\in\hat{Y}'$ 
such that $[\Gr_{\hat{G}}^{\hat\lambda}]-[\Gr_{\hat{G}}^{\hat\mu}]\in\widetilde{I}_\sigma$.
In effect, if the maximum of $|\hat\lambda^{2_1}_1-\hat\lambda^{2_2}_1|,\
|\hat\lambda^{2_2}_1-\hat\lambda^{2_3}_1|,\ 
|\hat\lambda^{2_3}_1-\hat\lambda^{2_1}_1|$ is bigger
than 1, and is equal to say $\hat\lambda^{2_2}_1-\hat\lambda^{2_1}_1$, then
we have 
\begin{multline*}
[\Gr_{\hat{G}}^{(\hat\lambda^0,\hat\lambda^1,\hat\lambda^{2_1},\hat\lambda^{2_2},
\hat\lambda^{2_3})}]-[\Gr_{\hat{G}}^{(\hat\lambda^0,\hat\lambda^1,\hat\lambda^{2_1}+(1,\ldots,1),
\hat\lambda^{2_2}-(1,\ldots,1),\hat\lambda^{2_3})}]
\\
=[\Gr_{\hat{G}}^{(\hat\lambda^0,\hat\lambda^1,\hat\lambda^{2_1}+(1,\ldots,1),
\hat\lambda^{2_2},\hat\lambda^{2_3})}]\cdot([\Gr_{\hat{G}}^{(0,0,-(1,\ldots,1),0,0)}]-
[\Gr_{\hat{G}}^{(0,0,0,-(1,\ldots,1),0)}])\in\widetilde{I}_\sigma.
\end{multline*}
Proceeding like this we can replace the initial $[\Gr_{\hat{G}}^{\hat\lambda}]$ with 
the one that is equal to it modulo $\widetilde{I}_\sigma$ but has the absolute
value of differences $\hat\lambda^{2_i}_1-\hat\lambda^{2_j}_1$ at most 1. 
Now if say $\hat\lambda^{2_2}_1-\hat\lambda^{2_3}_1=-1$ we repeat the above
replacement once more to swap $\hat\lambda^{2_2}_1$ and $\hat\lambda^{2_3}_1$
and make sure $\hat\lambda^{2_2}_1-\hat\lambda^{2_3}_1=1$. This way we replace
the initial $[\Gr_{\hat{G}}^{\hat\lambda}]$ with the one that is equal to it modulo 
$\widetilde{I}_\sigma$, and has $\hat\lambda^{2_1}_1\geq\hat\lambda^{2_2}_1\geq
\hat\lambda^{2_3}_1\geq\hat\lambda^{2_1}_1-1$. To take care of the second
coordinate $\hat\lambda^{2_i}_2$, instead of $-(1,\ldots,1)$ above we use
$-(0,1,\ldots,1)$ (that does not change the first coordinate 
$\hat\lambda^{2_i}_1$) in the above replacement procedure. Proceeding like this
we arrive at the desired coweight $\hat\mu\in\hat{Y}'$ 
such that $[\Gr_{\hat{G}}^{\hat\lambda}]-[\Gr_{\hat{G}}^{\hat\mu}]\in\widetilde{I}_\sigma$.
The surjectivity of 
$\widetilde{E}\to\on{gr}H^{\hat{G}_\CO}(\Gr_{\hat{G}})/\widetilde{I}_\sigma$ 
is proved.

Since
$\on{gr}H^{\hat{G}_\CO}(\Gr_{\hat{G}})=\bigotimes_i\on{gr}H^{\GL(V_i)_\CO}(\Gr_{\GL(V_i)})$
(the product over $i=0,1,2_1,2_2,2_3$), and $\sigma$ rotates cyclically the
last three factors, we see that 
$$\on{gr}H^{\hat{G}_\CO}(\Gr_{\hat{G}})/\widetilde{I}_\sigma\simeq
\on{gr}H^{\GL(V_0)_\CO}(\Gr_{\GL(V_0)})\otimes\on{gr}H^{\GL(V_1)_\CO}(\Gr_{\GL(V_1)})
\otimes\on{gr}H^{\GL(V_2)_\CO}(\Gr_{\GL(V_2)}),$$ and the graded dimension of the RHS
coincides with the one of $\widetilde{E}$. Here the grading is by the
cone of dominant coweights of $G$ times ${\mathbb Z}$ (the homological grading).
Hence the surjectivity established in the previous paragraph implies the
isomorphism
$\widetilde{E}\iso\on{gr}H^{\hat{G}_\CO}(\Gr_{\hat{G}})/\widetilde{I}_\sigma$.

We return to the proof of isomorphism 
$\ol{\varphi}\colon E\iso\BC[\ol{\cM}_C(\hat{G},\bN)^\sigma]$.
Consider the following commutative diagram:
$$\begin{CD}
\on{gr}H^{\hat{G}_\CO}(\CR_{\hat{G},\bN}) @>>> 
\on{gr}H^{\hat{G}_\CO}(\CR_{\hat{G},\bN})/\ol{I}_\sigma\\
@VV{\bz^*}V  @VV{\bz^*}V\\
\on{gr}H^{\hat{G}_\CO}(\Gr_{\hat{G}}) @>>>
\on{gr}H^{\hat{G}_\CO}(\Gr_{\hat{G}})/\widetilde{I}_\sigma
\end{CD}$$
Here 
%$\widetilde{I}_\sigma\subset\on{gr}H^{\hat{G}_\CO}(\Gr_{\hat{G}})$ is the
%ideal generated by the expressions $f-\sigma^*f$, and 
$\bz^*$ is the restriction
to the zero section (see~\ref{sec:z}). Thus $\bz^*$ is injective, 
$\bz^*E\subset\widetilde{E}$, and
$\ol{I}_\sigma=(\bz^*)^{-1}\widetilde{I}_\sigma$, and the right vertical arrow
is injective as well. Hence the injectivity 
$\widetilde{E}\hookrightarrow\on{gr}H^{\hat{G}_\CO}(\CR_{\hat{G},\bN})/\ol{I}_\sigma$.

To prove the surjectivity, recall the setup of~\ref{sec:another}.
We use the flavor symmetry group $\BC^\times$, and instead of
$E\subset\on{gr}H^{\hat{G}_\CO}(\CR_{\hat{G},\bN})\supset\ol{I}_\sigma$ we consider 
the similarly defined subspace and ideal 
$E'\subset\on{gr}H^{\BC^\times\times\hat{G}_\CO}(\CR_{\hat{G},\bN})\supset\ol{I}'_\sigma$.
It suffices to prove the surjectivity $\BC(\bt)\otimes_{\BC[\bt]}E'
\twoheadrightarrow\BC(\bt)\otimes_{\BC[\bt]}
\on{gr}H^{\BC^\times\times\hat{G}_\CO}(\CR_{\hat{G},\bN})/\ol{I}'_\sigma$
because the $\bt$-deformation of $\on{gr}H^{\hat{G}_\CO}(\CR_{\hat{G},\bN})$
is trivial due to~\ref{rem:W-cover}(2).
This generic surjectivity follows from~\ref{prop:another} and the 
surjectivity at $\bt=\infty$
which was already established earlier during the proof.
\begin{NB} Old correspondence:
Now $\bz^*[\CR_{\hat\lambda}]=f_{\hat\lambda}[\Gr_{\hat{G}}^{\hat\lambda}]$ for certain
$f_{\hat\lambda}\in\BC[\hat\ft]^{W_{\hat\lambda}}$. So it suffices to check that
$\bigoplus_{\hat\lambda\in\hat{Y}'}\BC[\ft]^{W_{\psi\hat\lambda}}[\Gr_{\hat{G}}^{\hat\lambda}]$
projects isomorphically onto 
$\on{gr}H^{\hat{G}_\CO}(\Gr_{\hat{G}})/\widetilde{I}_\sigma$. 
\begin{NB2}
    $E$ is not mapped to this space as the factor in
    $\CC[\hat\ft^{2_2}]$ is changed from $1$ to the Euler class of the
    normal bundle. Therefore the argument does not work.....
\end{NB2}%
This is clear since
$\on{gr}H^{\hat{G}_\CO}(\Gr_{\hat{G}})=\bigotimes_i\on{gr}H^{\GL(V_i)_\CO}(\Gr_{\GL(V_i)})$
(the product over $i=0,1,2_1,2_2,2_3$), and $\sigma$ rotates cyclically the
last three factors.
\begin{NB2}
    It is something like $\sigma f(t_1) F(x_1)\otimes g(t_2) G(x_2)
    \otimes h(t_3) H(x_3) = h(t_1) H(x_1)\otimes f(t_2) F(x_2) \otimes
    g(t_3) G(x_3)$. The quotient is given by $t_1 = t_2 = t_3$, $x_1 =
    x_2 = x_3$.
    We take $E = \{ f(t_1) x_1^{\la^1} \otimes x_2^{\la^2}\otimes
    x_3^{\la^3}\mid \la^1\ge\la^2\ge\la^3\ge\la^1-1\}$.
\end{NB2}%
\end{NB}%
\end{proof}

\begin{Remark}
\label{C_n}
The proof of~\propref{prop:twimon} works for all the twisted cases whose
unfolding has no cycles (because we use~\ref{rem:W-cover}(2)). This 
excludes the $C_n^{(1)}$ case whose unfolding is the cyclic quiver 
$A_{2n-1}^{(1)}$. In this case the fixed point set of the automorphism $\sigma$
of $A_{2n-1}^{(1)}$ consists of two points, and we choose a $\sigma$-invariant
orientation from the first one to the second one. Then the dilatation action
of $\BC^\times$ on $\bN$ factors through $\hat G$ again, and the proof 
of~\propref{prop:twimon} goes through as well.\footnote{We are grateful
to L.~Rybnikov for this observation.}
\end{Remark}

%%% Local Variables:
%%% mode: latex
%%% TeX-master: "blowup_pre"
%%% End:

%\input{determinant}

\specialsection*{{\large Appendices}
\\
    By 
    Alexander~Braverman,
    Michael~Finkelberg,
    Joel~Kamnitzer,
    Ryosuke~Kodera,
    Hiraku~Nakajima,
    Ben~Webster,
and    
    Alex~Weekes
}
\appendix

% !TEX root = blowup_pre.tex
In the first appendix we write certain elements of quantized Coulomb
branches $\cAh$ as explicit difference operators. These elements are
homology classes lived on closed $G_\cO$-orbits, i.e., orbits
$\Gr^\la_G$ for \emph{minuscule} coweights $\la$, and their slight
generalization corresponding to \emph{quasi-minuscule} and \emph{small
  fundamental} coweights. The first class is called minuscule monopole
operators in physics literature (see \ref{rem:BDG3}).

Examples of explicit difference operators include Macdonald operators
\cite[Chap.~VI, \S3]{Mac} and ones in representations of Yangian in
the work of Gerasimov-Kharchev-Lebedev-Oblezin \cite{GKLO} and its
generalization \cite{kwy}.

We hope that these elements, together with $H^*_G(\mathrm{pt})$,
generate quantized Coulomb branches $\cAh$ in many situations, possibly after
inverting $\hbar$ (and variables for flavor symmetry groups).
If this would happen, it identifies $\cAh$ as a subalgebra in the ring
of difference operators, generated by explicit elements. It gives a
purely algebraic characterization of $\cAh$.
We will show that this happens for quiver gauge theories of Jordan and
$ADE$ types. In particular, we will show that the quantized Coulomb
branches for quiver gauge theories of type $ADE$ are isomorphic to
truncated shifted Yangian under the dominance condition in the second
appendix. (See \corref{cor:shiftedYangian}.)

\section{Minuscule monopole operators as difference operators
% \\
%     By 
%     Alexander~Braverman,
%     Michael~Finkelberg,
%     Joel~Kamnitzer,
%     Ryosuke~Kodera,
%     Hiraku~Nakajima,
%     Ben~Webster,
% and    
%     Alex~Weekes
}\label{sec:minuscule}

\begin{NB}
    Here are comparison of notations in \cite{kwy} with ours:
    \begin{enumerate}
          \item $m_i = \langle\la-\mu,\omega_{i^*}\rangle$ in \cite{kwy}
        is our $a_i = \dim V_i$.
          \item $\la_i = \langle\la,\alpha_{i^*}\rangle$ in \cite{kwy}
        is our $l_i=\dim W_i$.
          \item $z_{i,k}$ in \cite{kwy} is our $w_{i,r}$.
        \item $\beta_{i,k}^\pm$ in \cite{kwy} is our $\sfu_{i,r}^\pm$.
        \item $c_i^{(r)}$ in \cite{kwy} is the $r$th elementary
      symmetric function in our $z_{s}$ ($i_s=s$).
    \end{enumerate}
\end{NB}%

\subsection{Embedding to the ring of difference operators}
\label{embed_diff_op}
Let us return back to general notation conventions in \cite{main}.
Let $(G,\bN)$ be a pair of a complex reductive group and its
representation. Let $T$ be a maximal torus of $G$ and $\bNT$ the
restriction of $\bN$ to $T$ as usual. Let $W$ denote the Weyl group.
Let us consider the quantized Coulomb branches for $(G,\bN)$,
$(T, \bNT)$ and $(T,0)$.
If we simply write $\cAh$, it means the quantized Coulomb branch for
$(G,\bN)$. We indicate a group and its representation for other
two.
We will add flavor symmetry groups in examples below, but we omit them
for brevity now.

Recall we have an embedding $\cAh\hookrightarrow\cAh[T,0][\hbar^{-1},
(\text{root}+m\hbar)^{-1}]_{m\in{\mathbb Z}}$ in \ref{rem:BDG1}. 
\begin{NB}
We have a ring homomorphism $\iota_*\colon \cAh[T,\bNT]^{W}\to\cAh$, which
becomes an isomorphism if we invert the expressions $\hbar, \alpha+m\hbar$ 
where $\alpha$ is a root of $G$ and $m$ is an integer, considered
as elements in $H^*_{T\times{\mathbb C}^\times}(\mathrm{pt})$ (\ref{sec:bimodule}).
(Here $(\cAh[T,\bNT], \hbar, (\text{roots}+{\mathbb Z}\hbar))$ 
satisfies the Ore condition as
remarked in \ref{rem:BDG1}. Therefore the localization as a
$\CC[\ft]$-module has an algebra structure.)
Indeed, it suffices to check that for any dominant coweight $\lambda$,
the homomorphism of $(H^*_{T\times\BC^\times}(\on{pt}))^W$-modules
$\iota_*\colon (H_*^{T_\CO\rtimes\BC^\times}(\cR_{T,\bN_T})_{W\lambda})^W\to
H_*^{G_\CO\rtimes\BC^\times}(\cR_\lambda)$ becomes an isomorphism after inverting
the expressions $\hbar, \alpha+m\hbar$. To this end note that at a $T$-fixed
point $w\lambda\in\cR_\lambda,\ w\in W$, the quotient of the tangent space to 
$\cR_\lambda$ modulo the tangent space to $\cR_{T,\bN_T}$ equals the tangent space
to $\Gr^\lambda_G$, and the weights of $T\times\BC^\times$ in the latter tangent
space are all of the form $\hbar, \alpha+m\hbar$. 

We further have a ring homomorphism $\bz^*\colon \cAh[T,\bNT]\to
\cAh[T,0]$ (\ref{sec:z}).
By \ref{prop:integrable} $\cAh[T,0]$ is a $\CC[\hbar]$-algebra
generated by $w_{r}$, $\sfu^{\pm 1}_{r}$ ($1\le r\le \dim T$) with
relations $[\sfu_{r}^{\pm1}, w_{s}] = \pm\delta_{r,s} \hbar
\sfu_{r}^{\pm1}$. (Here we take coordinates of $T^\vee$ and the induced
coordinates on $\ft$.)
\end{NB}%
We thus have an algebra embedding
\begin{equation*}
    \bz^* (\iota_*)^{-1}\colon \cAh\hookrightarrow
    \begin{aligned}[t]
      \tilde\cAh &\defeq \cAh[T,0][\hbar^{-1}, (\text{root}+m\hbar)^{-1}]_{m\in{\mathbb Z}} \\
      &\overset{\phantom{\operatorname{\scriptstyle def.}}}{=} \CC[\hbar] \langle w_{r}, \sfu^{\pm 1}_{r}, \hbar^{-1},
      (\alpha+m\hbar)^{-1} \rangle\quad (\text{$\alpha$:\ root},\
      m\in{\mathbb Z}).
    \end{aligned}
\end{equation*}
We consider $\tilde\cAh$ as the localized ring of $\hbar$-difference operators
on $\ft$: $\sfu_{r}^{\pm 1}$ is the operator
\begin{equation*}
    (\sfu_{r}^{\pm 1} f)(\dots,w_{s},\dots)%_{j\in I,1\le s\le a_i})\right)
    = f(\dots,w_{s} \pm \hbar \delta_{r,s},\dots).
\end{equation*}
\begin{NB}
\begin{equation*}
   [\sfu_{r}^{\pm 1}, w_{s}] f(w_{t})
   = (w_{s} \pm \hbar \delta_{rs} - w_{s}) 
   	f(w_{t} + \hbar \delta_{rt})
   	= \pm \delta_{rs} \hbar (\sfu_{r}^{\pm 1} f)(w_{t}).
\end{equation*}
\end{NB}%

\begin{Remark}
    We could also consider $\cAh[T,0]$ as the ring of differential
    operators on $T^\vee$: $\sfu_r^{\pm1}$ is a coordinate of $T^\vee$,
    and $w_s$ is $-\hbar \sfu_s \partial/\partial \sfu_s$. But it is
    natural for us to consider difference operators on $\ft$, as we
    invert roots.
\end{Remark}

In general, we do not know how to characterize the image of $\cAh$ in
$\tilde\cAh$ explicitly. Nevertheless, the image of a homology class
associated with a closed $G_\cO$-orbit $\Gr_G^\la$ can be explicitly
written down. (See 
\ref{prop:useful} for
$(\iota_*)^{-1}$ and \ref{sec:chang-repr} for $\bz^*$.)
% \namecref{Coulomb2-prop:useful}II-\oldref{Coulomb2-prop:useful} for
% $(\iota_*)^{-1}$ and \namecref{Coulomb2-sec:chang-repr}II-\oldref{Coulomb2-sec:chang-repr} for $\bz^*$.)
%\cite[\namecref{Coulomb2-prop:useful}\oldref{Coulomb2-prop:useful},
%\namecref{Coulomb2-sec:chang-repr}\oldref{Coulomb2-sec:chang-repr}]{main}.)
%Such a class is called a minuscule monopole operator in physics
%literature (see \ref{rem:BDG3}).
Note that $\Gr_G^\la$ is closed if and only if $\la$ is minuscule.
(Since $\ol\Gr{}^\la_G\supset\Gr^\mu$ if and only if $\la\geq\mu$, and the
minuscule coweights are minimal in this order.)
\begin{NB}
    Reference ?
\end{NB}%

\begin{Proposition}\label{prop:minuscule}
    Let $\la$ be a minuscule dominant coweight and $W_\la$ its
    stabilizer in $W$. Let $f\in\CC[\ft]^{W_\la}$. Let
    $\cR_\la = \pi^{-1}(\Gr_G^\la)$, where $\pi\colon \cR\to\Gr_G$ is
    the projection. Then
    \begin{equation*}
        \bz^*(\iota_*)^{-1}f[\cR_\la] = 
        \sum_{\la'=w\la\in W\la} \frac{wf \times
          e\left( z^{\la'}\bN_\cO/z^{\la'}\bN_\cO\cap\bN_\cO\right)}
        {e(T_{\la'}\Gr^\la_G)} \sfu_{\la'},
    \end{equation*}
    where $T_{\la'}\Gr^\la_G$ is the tangent space of $\Gr^\la_G$ at
    the point $z^{\la'}$ and $\sfu_{\la'}$ is the shift operator
    corresponding to $\la'$, i.e., $(\sfu_{\la'} f)(\bullet) =
    f(\bullet+\hbar\la')$ for $f\in\CC[\ft]$.
\end{Proposition}

\subsection{Quiver gauge theories}
\label{Qgt}

\renewcommand{\bNT}{\bN_{\TV}}

Let us return back to the notational convention in this paper.

Let $(\GV,\bN)$ be a quiver gauge theory, which is not
necessarily of either finite $ADE$ or affine type.
\begin{NB}
We slightly change the notation %from a general situation above 
to avoid a conflict with the group $G$ for zastava.
\end{NB}%
Let $\TV$ be a maximal torus of $\GV$, and $\bNT$ is the restriction
of $\bN$ to $\TV$.
We add the flavor symmetry group $\TW$ as in \cref{defo}. Thus we mean
$\cAh = H^{(\GV\times \TW)_\cO\rtimes\CC^\times}_*(\cR_{\GV,\bN})$,
and $\cAh[\TV,\bNT]$, $\cAh[\TV,0]$ are similar.

When there are several loops in the underlying graph (e.g., the Jordan
quiver or an affine quiver of type $A$), we should also add additional flavor symmetries rescaling entries in $\bN$ in loops. But we omit them for brevity except in \subsecref{subsec:jordan-quiver}.

Recall $w^*_{i,r}$ is the cocharacter of $\GV = \prod \GL(V_i)$,
which is equal to $0$ except at the vertex $i$, and is
$(0,\dots,0,1,0,\dots,0)$ at $i$. Here $1$ is at the $r$th entry
($r=1,\dots,a_i = \dim V_i$).
We take corresponding coordinates $w_{i,r}$, $\sfu_{i,r}$ ($i\in I$,
$1\le r\le a_i$) of $\operatorname{Lie}\TV$ and $\TV^\vee$.
The roots are $w_{i,r} - w_{i,s}$ ($r\neq s$).
Furthermore, $\cAh[\TV,0]$ is a $\CC[\hbar,z_1,\dots,z_N]$-algebra
generated by $w_{i,r}$, $\sfu^{\pm 1}_{i,r}$ ($i\in I$, $1\le r\le
a_i$) with relations $[\sfu_{j,s}^{\pm1},w_{i,r}] =
\pm\delta_{i,j}\delta_{r,s} \hbar \sfu_{i,r}^{\pm1}$. We thus have an algebra
embedding
\begin{equation*}
    \cAh\hookrightarrow
     \tilde\cAh \defeq %\cAh[T,0][\text{root}^{-1}] =
    \CC[\hbar,z_1,\dots, z_N]
    \langle w_{i,r}, \sfu^{\pm 1}_{i,r}, \hbar^{-1}, (w_{i,r}-w_{i,s}+m\hbar)^{-1}
    (r\neq s,\ m\in{\mathbb Z})\rangle.
\end{equation*}
We consider $\tilde\cAh$ as the localized ring of $\hbar$-difference operators
on $\operatorname{Lie}\TV$ as above, and $z_1$, \dots, $z_N$ are parameters.

% \subsection{Operators for
% closed \texorpdfstring{$G_\cO$}{G\textunderscore O}-orbits}

% Recall that the image under $(\iota_*)^{-1}\colon \cAh\to
% \cAh[\TV,\bNT][\text{root}^{-1}]$ is explicitly written down for a
% homology class associated with a closed $G_\cO$-orbit
% (\ref{prop:useful}).
% %
% It is called a minuscule monopole operator in physics literature
% (see \ref{rem:BDG3}).
% %
% Furthermore, the embedding $\bz^*\colon \cAh[\TV,\bNT]\to\cAh[\TV,0]$ is
% explicitly computed in \ref{sec:chang-repr}.
% %
% We write down the formula in this subsection.

Let $\varpi_{i,n}$ be the $n$th fundamental coweight of the factor
$\GL(V_i)$, i.e., $(1,\dots,1,0,\dots,0) = w^*_{i,1}+\dots+w^*_{i,n}$,
where $1$ appears $n$ times ($1\le n\le a_i$). Then
$\Gr_{\GV}^{\varpi_{i,n}}$ is closed and isomorphic to the
Grassmannian $\Gr(V_i,n)$ of $n$-dimensional quotients of $V_i$.
In fact, $\Gr^{\varpi_{i,n}}_{\GV}$ is identified with the moduli space
of $\cO$-modules $L$ such that
\begin{equation*}
     z\cO\otimes V_i \subset L \subset \cO\otimes V_i, \qquad
     	\dim_\CC %\nicefrac
		{\cO\otimes V_i}/{L} = n,
\end{equation*}
hence $\cO\otimes V_i/L$ is the corresponding quotient space
of $\cO\otimes V_i/ z\cO\otimes V_i\cong V_i$.
\begin{NB}
    $\varpi_{i,n}$ corresponds to $L = \operatorname{diag}(z,\dots,z,1,\dots,1)
    (\cO\otimes V_i)$.
\end{NB}%

Let $\uQ_i$ be the vector bundle over $\Gr_{\GV}^{\varpi_{i,n}}$ whose
fiber at $L$ is $\cO\otimes V_i/L$. It is the universal rank $n$
quotient bundle of $\Gr(V_i,n)\times V_i$. Its pull-back to
$\cR_{\varpi_{i,n}}$ is denoted also by $\uQ_i$ for brevity. Let
$c_p(\uQ_i)$ denote its $p$th Chern class.
More generally we can consider a class $f(\uQ_i)$ for a symmetric
function $f$ in $n$ variables so that $c_p(\uQ_i)$ corresponds
to the $p$th elementary symmetric polynomial.

The $\TV$ fixed points in $\Gr_{\GV}^{\varpi_{i,n}}$ are in bijection to subsets
$I\subset \{ 1,\dots, a_i\}$ with $\# I = n$.  
The bijection is given by assigning a cocharacter $\la_I \defeq \sum_{r\in I} w^*_{i,r}$ of $\GL(V_i)$ to $I$. The fixed point formula implies
\begin{equation*}
    (\iota_*)^{-1} f(\uQ_i)\cap [\cR_{\varpi_{i,n}}]
    = %\frac1{p!}
    \sum_{\substack{I\subset \{1,\dots, a_i\} \\ \# I = n}} 
    f(w_{i,I}) \prod_{r\in I, s\notin I}
    \frac{r^{\la_I}}{w_{i,r} - w_{i,s}},
\end{equation*} 
where $f(w_{i,I})$ means that we substitute $(w_{i,r})_{r\in I}$ to
the symmetric function $f$, and
$r^{\la_I}$ denote the fundamental class of the fiber of $\cR_{\TV,\bNT}\to\Gr_{\TV}$ at $\la_I$.
In view of \propref{prop:minuscule},
the $\TV$-fixed point set is the Weyl group orbit
$\Weyl\varpi_{i,n}$, and $\prod (w_{i,r}-w_{i,s})$ is the
equivariant Euler class $e(T_{\la_I} \Gr_{\GV}^{\varpi_{i,n}})$ of the tangent space of $\Gr_{\GV}^{\varpi_{i,n}}$ at the fixed point $\la_I$.
\begin{NB}
In fact, $T_{\la_I}\Gr_{\GV}^{\varpi_{i,n}} = \bigoplus_{\alpha\in\Delta}
\bigoplus_{n=0}^{\max(0,\langle\alpha,\la_I\rangle)-1} \gl(V)_\alpha z^n$ as 
is mentioned in the proof of \ref{lem:orbit}, where $\Delta$ is the set of roots
of $\gl(V)$ and $\gl(V)_\alpha$ is the root subspace for $\alpha\in\Delta$.
And $\langle\alpha,\la_I\rangle > 0$ if and only if $\alpha = w_{i,r} - w_{i,s}$
with $r\in I$, $s\notin I$.
\end{NB}%

Furthermore % the pull-back homomorphism
% $\bz^*\colon \cAh[\TV,\bNT]\to \cAh[\TV,0]$ in
% \cite[\namecref{Coulomb2-sec:chang-repr}\oldref{Coulomb2-sec:chang-repr}
% and \namecref{Coulomb2-sec:z}\oldref{Coulomb2-sec:z}]{main} is given by
% the multiplication of
\begin{equation*}
   e\left( z^{\la_I}\bN_\cO/ z^{\la_I}\bN_\cO\cap\bN_\cO\right) =
   \prod_{\substack{h\in\Qo: \vout{h}=i\\ r\in I}}\prod_{\substack{s=1 \\ \text{$\vin{h}\neq i$ or $s\notin I$}}}^{a_j} 
    (-w_{i,r} + w_{\vin{h},s} - \nicefrac{\hbar}2)
\end{equation*}
followed by the replacement $r^{\la_I}$ by $\prod_{r\in I}\sfu_{i,r}$.
Here `$\vin{h}\neq i$ or $s\notin I$' means that the product excludes $s\in I$ if
$h$ %or $j\to i$
is an edge loop.
We thus get
\begin{multline}
    \label{eq:82}
        \bz^* (\iota_*)^{-1} f(\uQ_i)\cap[\cR_{\varpi_{i,n}}] \\
    = %\frac1{p!}
    \sum_{\substack{I\subset \{1,\dots, a_i\} \\ \# I = n}} 
    f(w_{i,I})
    \frac{\displaystyle\prod_{\substack{h\in\Qo:\vout{h}=i\\ r\in I}}
      \prod_{\substack{s=1 \\ \text{$\vin{h}\neq i$ or $s\notin I$}}}^{a_j} 
      (-w_{i,r} + w_{\vin{h},s} - \nicefrac{\hbar}2)}
    {\displaystyle\prod_{r\in I, s\notin I} (w_{i,r} - w_{i,s})}
    \prod_{r\in I} \sfu_{i,r}.
\end{multline}

Instead of $f(\uQ_i)$, we can also consider the class $f(\uS_i)$, a
polynomial in Chern classes of the universal subbundle $\uS_i$ over
$\Gr_{\GV}^{\varpi_{i,n}}$. Then variables $(w_{i,r})_{r\in I}$ in
$f(w_{i,I})$ are replaced by $(w_{i,s})_{s\notin I}$. We will consider
symmetric functions in the full variables $w_{i,r}$ ($r=1,\dots, a_i$)
later, so the difference between $\uQ_i$ and $\uS_i$ are not
essential: the algebra generated by \eqref{eq:82} and one by
$f(\uS_i)$ are the same if we add symmetric functions in the full
$w_{i,r}$.

Let us recall the $\Delta$-degree, defined in
\eqref{eq:18}. Its value for $\varpi_{i,n}$ is
\begin{equation}\label{eq:Deltadeg}
  \Delta(\varpi_{i,n}) =
  (\# \{i\to i\in\Qo\} - 1)n(\dim V_i - n)
%  + \sum_{i\to i\in\Omega} (\dim V_i - 1)
  + \frac{n}{2} \sum_{\substack{h\in\Qo\sqcup\overline{\Qo}\\\vout{h}=i\\\vin{h}\neq i}}
  \dim V_{\vin{h}}
%  + \frac12 \sum_{\substack{i\to j\in\Qo \\i\neq j}} \dim V_j
  + \frac{n}{2} \dim W_i,
\end{equation}
where $\# \{i\to i\in\Qo\}$ is the number of edge loops at $i$.
For a finite quiver gauge theory of type $ADE$ and $n=1$, this is equal to
$1+\frac12\langle\mu,\alphavee_i\rangle$. For Jordan quiver and $n=1$, we have
$\frac12\dim W_i$.

Similarly we consider $\varpi_{i,n}^* = -w_0\varpi_{i,n}$, where the
corresponding orbit $\Gr^{\varpi_{i,n}^*}_{\GV}$ is also isomorphic to
the Grassmannian $\Gr(n,V_i)$ of $n$-planes in $V_i$. In fact,
$\Gr^{\varpi_{i,n}^*}_{\GV}$ is the moduli space of $\cO$-modules $L$
such that $\cO\otimes V_i\subset L\subset z^{-1}\cO\otimes V_i$ with
$\dim_\CC L/\cO\otimes V_i = n$.
\begin{NB}
    $\varpi_{i,n}^*$ corresponds to $L= \operatorname{diag}(z^{-1},\dots,z^{-1},1,\dots,1)(\cO\otimes V_i)$. Therefore $L/\cO\otimes V_i = (z^{-1}\cO/\cO)^{\oplus n} \oplus (\cO/\cO)^{\oplus a_i - n}$.
\end{NB}%
Let $\uS_i$ be the rank $n$ vector bundle over
$\Gr^{\varpi_{i,n}^*}_{\GV}$ whose fiber over $L$ is $L/\cO\otimes
V_i$. Its pull-back to $\cR_{\varpi_{i,n}^*}$ is also denoted by
$\uS_i$. Then
\begin{multline}
    \label{eq:83}
        \bz^* (\iota_*)^{-1} f(\uS_i)\cap[\cR_{\varpi_{i,n}^*}] \\
    = %\frac1{p!}
    \sum_{\substack{I\subset \{1,\dots, a_i\} \\ \# I = n}} 
    f(w_{i,I}-\hbar)
%    \sum_{r\in I} (w_{i,r}-\hbar)^p 
    \prod_{\substack{r\in I\\ k:i_k=i}} (w_{i,r} - z_k - \nicefrac{\hbar}2)
    \frac{\displaystyle\prod_{\substack{h\in\Qo: \vin{h}=i\\ r\in I}}
      \prod_{\substack{s=1 \\ \text{$\vout{h}\neq i$ or $s\notin I$}}}^{a_j} 
      (w_{i,r} - w_{\vout{h},s} - \nicefrac{\hbar}2)}
    {\displaystyle\prod_{r\in I, s\notin I} (-w_{i,r} + w_{i,s})}
    \prod_{r\in I} \sfu_{i,r}^{-1},
\end{multline}
where $f(w_{i,I}-\hbar)$ means that we substitute
$(w_{i,r}-\hbar)_{r\in I}$ to $f$. The extra factor $w_{i,r} - z_k -
\nicefrac{\hbar}2$ came from $\Hom(W_i,V_i)$.

Thus elements in the right hand sides of \textup{(\ref{eq:82},
  \ref{eq:83})} are in the image of $\bz^*(\iota_*)^{-1}\colon
\cAh\hookrightarrow \tilde\cAh%\cAh[\TV,0][\mathrm{root}^{-1}]
$.

\begin{Remark}
\label{Gaiotto-Witten}
We assume that $(\GV,\bN)$ is a quiver gauge theory of ADE type, so that
the Coulomb branch is isomorphic to a BD slice 
$\ol\CW{}^{\underline{\lambda}^*}_{\mu^*},\
\lambda-\mu=\alpha$. Recall the function $\chi^\lambda_{i,+}$ 
of~\cite[Theorem~1.6(5),~Theorem~6.4]{bdf}. It is a function on 
$\oZ^\alpha\times\BA^N$ measuring $\on{Ext}^1$ of certain line bundles on 
$\BP^1$ coming from the flags in $\oZ^\alpha$. More conceptually, it is the
crucial part (the 4-th summand of~\cite[(1.5)]{bdf}) of the {\em Gaiotto-Witten
superpotential}, or else the $i$-th summand of the {\em Whittaker function} 
(see e.g.~\cite[6.3]{bdf}). Composing $\chi^\lambda_{i,+}$ with the projection
$\ol\CW{}^{\underline{\lambda}^*}_{\mu^*}\to Z^\alpha\times\BA^N$ we can view
$\chi^\lambda_{i,+}$ as a rational function on 
$\ol\CW{}^{\underline{\lambda}^*}_{\mu^*}$. Now a direct comparison of 
formulas \textup{(\ref{eq:83})} and~\cite[(1.3)]{bdf}
shows that up to sign, $\chi^\lambda_{i,+}$ coincides with 
$\bz^*(\iota_*)^{-1}[\cR_{\varpi_{i,1}^*}]|_{\hbar=0}$; in particular, it is a 
{\em regular} function on $\ol\CW{}^{\underline{\lambda}^*}_{\mu^*}$.

Recall that the logarithmic part $\log F_\alpha$ of the Gaiotto-Witten 
superpotential (the 3-rd summand of~\cite[(1.5)]{bdf}) is also expressed in
terms of Coulomb branch, see~Remark~\ref{Ober}.
\end{Remark}

\subsection{Jordan quiver}\label{subsec:jordan-quiver}

Consider the case of Jordan quiver. We omit the index $i$ as we only
have one vertex. For example, let $\dim V = a$, $\dim W = l$. Hence
$\GV = \GL(a)$. As we mentioned before, we add the dilatation on $\bN$
as the flavor symmetry $\CC^\times$. Let us denote the corresponding
equivariant variable by $\bt$. 
By \propref{prop:ad_taut}, the Coulomb branch $\cA = \CC[\mathcal
M_C]$ with $\hbar = \bt = z_k = 0$ is $\operatorname{Sym}^a\cS_l$
where $\cS_l$ is the hypersurface $xy=w^l$ in $\BA^3$.

Note that the equivariant variable $\bt$ will be added each factor in
the numerator of (\ref{eq:82}, \ref{eq:83}). Since it always appears
with $-\hbar/2$, let us absorb $-\hbar/2$ to $\bt$. Also we replace
$z_k$ by $z_k + \hbar + \bt$ so that $w_r - z_k + \bt$ becomes $w_r -
\hbar - z_k$.
\begin{NB}
On the other hand, the action of the diagonal $\CC^\times\subset \TW$
can be absorbed into the action of $\GV$. Therefore we may assume
$z_1+\dots+z_l = 0$. (In this case $N=l$.)

This is probably possible, but let us not to assume, as we shift
equivariant variables.
\end{NB}%

Then (\ref{eq:82}, \ref{eq:83}) become
\begin{equation}\label{eq:79}
    \begin{split}
    & E_{n}[f] \defeq
    %\frac1{p!}
    \sum_{\substack{I\subset \{1,\dots,a\}\\ \# I =
        n}} % C_i(w_{i,r} + \hbar)
    f(w_I)
%    \sum_{r\in I} w_r^p
    \prod_{r\in I, s\notin I}
    \frac{%\prod_{h: \vout{h} = i}
      w_{r} - w_{s} - \bt}
    {w_{r}-w_{s}}
    \prod_{r\in I} \sfu_{r},
\\
    & 
        F_{n}[f] \defeq
%    \frac1{p!}
    \sum_{\substack{I\subset \{1,\dots,a\}\\ \# I = n}}
%    \sum_{r\in I} (w_r-\hbar)^p
    f(w_I - \hbar)
    \prod_{r\in I, s\notin I}
    \frac{%\prod_{h: \vout{h} = i}
      w_{r}  - w_{s} + \bt}
    {w_{r}-w_{s}}
    \prod_{r\in I} \left(
    \prod_{k=1}^l (w_{r} - \hbar - z_k)
    \cdot\sfu_{r}^{-1}\right),
    \end{split}
\end{equation}
where $f(w_I)$, $f(w_I - \hbar)$ are $f\left((w_r)_{r\in I}\right)$,
$f\left((w_r-\hbar)_{r\in I}\right)$ respectively. We also multiply
$(-1)^{n(a-n)}$ to omit the sign.

\begin{NB}
This is the record of an old version of operators.
\begin{equation*}
    E_{n}[f] \defeq
    %\frac1{p!}
    \sum_{\substack{I\subset \{1,\dots,a\}\\ \# I = n}} % C_i(w_{i,r} + \hbar)
    f\left( (w_r)_{r\in I}\right)
%    \sum_{r\in I} w_r^p
    \prod_{r\in I, s\notin I}
    \frac{%\prod_{h: \vout{h} = i}
      - w_{r} + w_{s} - \nicefrac{\hbar}2 + \mathbf t}
    {w_{r}-w_{s}}
    \prod_{r\in I} \sfu_{r}.
\end{equation*}
\begin{equation*}
    F_{n}[f] \defeq
%    \frac1{p!}
    \sum_{\substack{I\subset \{1,\dots,a\}\\ \# I = n}}
%    \sum_{r\in I} (w_r-\hbar)^p
    f\left((w_r-\hbar)_{r\in I}\right)
    \prod_{r\in I, s\notin I}
    \frac{%\prod_{h: \vout{h} = i}
      w_{r}  - w_{s} - \nicefrac{\hbar}2 + \mathbf t}
    {-w_{r}+w_{s}}
    \prod_{r\in I} \left(
    \prod_{k=1}^l (w_{r} - z_k - \nicefrac{\hbar}2)
    \cdot\sfu_{r}^{-1}\right).
\end{equation*}
I probably forgot to add $\bt$ to $(w_{r} - z_k - \nicefrac{\hbar}2)$,
though it is harmless as it can be absorbed to $z_k$ anyway.
\end{NB}%

If $f\equiv 1$, $E_n[1]$ is a rational version of the $n$th Macdonald
operator, once $\sfu_r$ is understood as the $\hbar$-difference
operator: $f(w_1,\dots,w_a) \mapsto f(w_1,\dots, w_r+\hbar, \dots,
w_a)$.
A little more precisely, the Macdonald operator is
\begin{equation*}
    \sum_{\substack{I\subset \{1,\dots,a\}\\ \# I = n}} % C_i(w_{i,r} + \hbar)
    \prod_{r\in I, s\notin I}
    \frac{t x_r - x_s}{x_{r}-x_{s}}
    \prod_{r\in I} T_{r},
\end{equation*}
for $(T_{r}f)(x_1,\dots,x_a) = f(x_1,\dots,q x_r, \dots, x_a)$.
(See e.g., \cite[Chap.~VI,~\S3]{Mac}.) We recover $E_n[1]$ if we set
$x_r = \exp(\boldsymbol\beta w_r)$, $q = \exp(\boldsymbol\beta
\hbar)$, $t = \exp(-\boldsymbol\beta\bt)$ and take the limit
$\boldsymbol\beta\to 0$.
\begin{NB}
Since the operator for $f\equiv 1$ is important, let us simply denote
it by $E_n$. We also set $F_n = F_n[1]$.
\end{NB}%

\begin{Remark}
    Let us consider the operator $E_n$ for the $K$-theoretic version
    of the quantized Coulomb branch. The computation is the same, we
    just replace Euler classes by $K$-theoretic ones, e.g., $w_r -
    w_s$ by $1 - x_s x_r^{-1} = 1 - \exp(-(w_r-w_s))$ under the
    identification $x_r = \exp w_r$. Then $(- w_{r} + w_{s} + \bt)/
    ({w_{r}-w_{s}})$ is replaced by
    \begin{equation*}
        \frac{1 - x_r x_s^{-1} \exp(-\bt)}
        {1 - x_s x_r^{-1}}
        = - \exp(-\bt) \frac{x_r}{x_s}
        \frac{x_r - x_s \exp\bt}
        {x_r - x_s}.
    \end{equation*}
    If we compare this with the Macdonald operator, we see the extra
    factor $x_r/x_s$. It can be regarded as the canonical bundle of
    $\Gr_{\GL(a)}^{\varpi_n}$, and absorbed into the symmetric
    function $f$ for our purpose. However if we want to check the
    commutativity $\left[E_m, E_n\right] = 0$, it is true for
    the homology case, and need to put the extra factor for the
    $K$-theory.

    By the way, we do not see {\it a priori\/} geometric reason why we
    have $\left[E_m, E_n\right] = 0$.
\end{Remark}

% Similarly we consider the orbit for $\la = -\varpi_n$. Then
% $\bz^* (\iota_*)^{-1}f(\uS)\cap[\cR_{-\varpi_n}]$ in \eqref{eq:83} is
% \begin{equation}\label{eq:80}
%     F_{n}[f] \defeq
% %    \frac1{p!}
%     \sum_{\substack{I\subset \{1,\dots,a\}\\ \# I = n}}
% %    \sum_{r\in I} (w_r-\hbar)^p
%     f\left((w_r-\hbar)_{r\in I}\right)
%     \prod_{r\in I, s\notin I}
%     \frac{%\prod_{h: \vout{h} = i}
%       w_{r}  - w_{s} + \bt}
%     {-w_{r}+w_{s}}
%     \prod_{r\in I} \left(
%     \prod_{k=1}^l (w_{r} - z_k)
%     \cdot\sfu_{r}^{-1}\right).
% \end{equation}

\begin{Theorem}\label{thm:generate-Jordan}
    Operators $E_{n}[f]$, $F_{n}[f]$ \textup($1\le n\le a$, $f$:
    a symmetric function in $n$ variables\textup) in \eqref{eq:79}
    together with symmetric functions in $w_s$ generate the quantized
    Coulomb branch $\cAh$ over $\CC[\hbar,\bt, z_1,\dots,z_l]$.
\end{Theorem}

This result identifies $\cAh$ as a subalgebra in $\tilde\cAh$,
generated by explicit elements, as we have remarked after
\propref{prop:minuscule}. It is purely an algebraic problem to
identify this subalgebra with the spherical part of cyclotomic
rational Cherednik algebra. We will return back to this problem in
\cite{2016arXiv160800875K,2016arXiv161110216B}

Let us use the vector notation $\vec{w}$ for
$(w_1,\dots,w_a)$. Therefore symmetric functions $f$ in the full $w_r$
are denoted by $f(\vec{w})$. On the other hand, symmetric functions in
less variables as still denoted like in \eqref{eq:79}.

\begin{proof}
    At $\bt = \hbar = z_k = 0$, $E_{n}[f]$, $F_{n}[f]$ are
    specialized to
\begin{equation}\label{eq:81}
%\begin{gather*}
%    \frac1{p!}
    \sum_{\substack{I\subset \{1,\dots,a\}\\ \# I = n}} % C_i(w_{i,r} + \hbar)
%    \sum_{r\in I} w_r^p
    f(w_I)
%    \prod_{r\in I, s\notin I}
    \prod_{r\in I} \sfu_{r},%\\
\qquad
%    \frac1{p!}
    \sum_{\substack{I\subset \{1,\dots,a\}\\ \# I = n}}
%    \sum_{r\in I} w_r^p
    f(w_I)
    \prod_{r\in I} w_{r}^l \sfu_{r}^{-1}.
\end{equation}
%\end{gather*}
It is enough to show that these elements together with symmetric
functions in $\vec{w}$ generate $\cA$ at $\bt = z_k = 0$ by
graded Nakayama lemma.

If $a=1$, i.e., $\dim V = 1$, we have $\mathcal M_C = \cS_l = \{ xy =
w^l\}\subset\BA^3$, where $w = w_1$, $x = \sfu_1$,
$y=w_1^l\sfu_1^{-1}$ in the current notation. (See \ref{abel}.)
The above elements with $f=1$ are $x=\sfu_1$, $y=w_1^l\sfu_1^{-1}$
respectively. Therefore they together with $w$ generate $\cA =
\CC[\mathcal M_C]$.
% The assertion for $\dim V = 1$ follows by graded Nakayama lemma.

Let us write $x_r = \sfu_r$, $y_r = w_r^l\sfu_r^{-1}$. Then we have a
surjective homomorphism 
\[
\CC[\mathcal M_C] = \operatorname{Sym}^a \cS_l\twoheadleftarrow
\CC[\vec x,\vec w]^{S_a}\otimes
\CC[\vec y,\vec w]^{S_a},
\]
where $\vec x = (x_1,\dots, x_a)$ and $\vec y$, $\vec w$ are similar.
\begin{NB}
    They are generated by $\sum_r x_r^s y_r^q w_r^p$, but $x_r y_r = w_r^l$.
\end{NB}%
It is a classical result that the left and right elements in
\eqref{eq:81} and symmetric polynomials in $\vec{w}$ generate
$\CC[\vec x,\vec w]^{S_a}$ and $\CC[\vec y,\vec w]^{S_a}$
respectively. (See \cite[\S2.2]{MR1488158}.)
\end{proof}

\begin{NB}
\newenvironment{NB3}{\color{purple}{\bf NB3}. \footnotesize
}{}

Consider the commutator of the class $E_{n}[f]$ and the $k$th power
sum $p_k(\vec{w}) = \sum_{t=1}^a w_t^k$. It is
%\begin{multline*}
\begin{equation*}
    \left[ p_k(\vec{w}), E_{n}[f]\right]
%\\
    = - %\left[ \sum_{t=1}^a w_t^p, 
      \sum_{\substack{I\subset \{1,\dots,a\}\\ \# I = n}}
      \left( p_k((w_r + \hbar)_{r\in I})
        - p_k(w_I)\right)
      f(w_I)
      \prod_{r\in I, s\notin I}
      \frac{%\prod_{h: \vout{h} = i}
        w_{r} - w_{s} - \bt}
      {w_{r}-w_{s}}
      \prod_{r\in I} \sfu_{r}, %\right]
%\end{multline*}
\end{equation*}
thanks to the commutation relation
\(
   [\prod_{r\in I} \sfu_r,w_t] = \hbar \prod_{r\in I} \sfu_r
\)
if $t\in I$ and $0$ otherwise.
Therefore we can generate elements $E_n[f]$ for general $f$
inductively from $E_n$ and symmetric functions in $w_t$ by taking
commutators divided by $\hbar$. The same is true for $F_n[f]$.
(In particular, $\cA$ is generated by $E_n$, $F_n$ and symmetric
functions in $\vec{w}$ \emph{as a Poisson algebra}.)

\begin{Corollary}
    Suppose $l=0$. Then the quantized Coulomb branch $\cAh$ is
    isomorphic to the spherical part $\SDAHA_a$ of the graded
    Cherednik algebra \textup(alias trigonometric DAHA\textup) for
    $\GL(a)$.
\end{Corollary}

\begin{proof}
    Consider the embedding of $\SDAHA_a$ to the ring of rational
    difference operators on $\ft = \CC^a$, obtained as the
    trigonometric degeneration of the usual embedding of the spherical
    part of the DAHA. Symmetric functions in $w_t$ are considered as
    functions on $\ft$. See \secref{sec:DAHA}.

    Since $E_n$, $F_n$ are trigonometric degeneration of Macdonald
    operators, they are contained in $\SDAHA_a$.
    (In the notation in \secref{sec:DAHA}, $E_n$ and $F_n$ correspond
    to $S(e_n(X))S$ and $S(e_n(X^{-1}))S$ respectively, where $S$ is
    the symmetrizer. cf.\ Example~\ref{ex:gl2})
    \begin{NB2}
        The assertion for $E_n$ is well-known. For $F_n$, it is
        probably possible to compute directly $S(e_n(X^{-1}))S$, but I
        consider as follows. Note $F_n$ is diagonalized by Macdonald
        polynomials as $P_\lambda(q,t) =
        P_\lambda(q^{-1},t^{-1})$. Also the eigenvalues are
        eigenvalues of $E_n$ with $q\mapsto q^{-1}$, $t\mapsto
        t^{-1}$. This means the assertion.
    \begin{NB3}
    For $F_n$, observe that $F_n$ is equal to $E_n$ if we replace
    $w_r$ by $-w_r$. In other words, $F_n = \omega E_n \omega$, where
    $(\omega f)(\vec{w}) = f(-\vec{w})$.
        Or, it is also true that $F_n = E_n|_{\substack{\hbar\mapsto
            -\hbar\\ \bt\mapsto -\bt}}$.
    \end{NB3}%
    \end{NB2}%

    By the explanation preceding the corollary, all operators
    $E_n[f]$, $F_n[f]$ in \eqref{eq:79} are also contained $\SDAHA_a$.

    By \thmref{thm:generate-Jordan}, the image of $\cAh$ is contained
    in $\SDAHA_a$. Therefore we have an injective homomorphism
    $\cAh\to \SDAHA_a$. We know that both $\cAh$, $\SDAHA_a$
    degenerate to $\operatorname{Sym}^a (\CC\times\CC^\times)$ at
    $\hbar=0=\bt$, we are done.
\end{proof}

The same argument applies to $l>0$ if we would know that the
cyclotomic rational Cherednik algebra has an embedding to the ring of
difference operators such that its image contains \eqref{eq:79} with
$f=1$ (and symmetric functions).
\end{NB}%

\subsection{Adjoint}\label{subsec:adjoint}

We consider the case $\bN = \g$, the adjoint representation of a
reductive group $G$. When $G=\GL(a)$, it corresponds to the case
studied in the previous subsection with $W=0$.

We add the flavor symmetry $\CC^\times$, the
dilatation on $\bN$, and denote the corresponding equivariant variable
by $\bt$.

% Let $\la$ be a dominant minuscule coweight. It is well-known that
% $\Gr_G^\la$ is a closed $G_\cO$-orbit.

\subsubsection{Minuscule coweights}
The minuscule monopole operator in \propref{prop:minuscule} for the
adjoint is given by
\begin{Proposition}\label{prop:ad_minuscule}
\begin{equation*}
    \sum_{w\la\in\Weyl\la} wf \times
    \prod_{\substack{\alphavee\in\Delta^\vee \\ \langle\alphavee,w\la\rangle = 1}}
    \frac{-\alphavee-\nicefrac{\hbar}2+\bt}{\alphavee}
    \sfu_{w\la}.
\end{equation*}
\end{Proposition}
This is a rational version of the Macdonald operator for a minuscule
coweight for $f\equiv 1$. (See e.g., \cite{MR1441642}.)

\begin{proof}
    Let $\la' = w\la\in \Weyl\la$. As above $(\iota_*)^{-1}$ is given
    by the equivariant Euler class $e(T_{\la'}\Gr^\la_G)$ of the
    tangent space at $\la'$. It is given by
\begin{equation*}
    e(T_{\la'}\Gr^\la_G) =
    \prod_{\substack{\alphavee\in\Delta^\vee \\ \langle\alphavee,\la'\rangle = 1}}
    \alphavee.
\end{equation*}
In fact, $T_{\la'}\Gr_{G}^{\la} = \bigoplus_{\alphavee\in\Delta^\vee}
\bigoplus_{n=0}^{\max(0,\langle\alphavee,\la'\rangle)-1} \g_\alphavee z^n$
as is mentioned in the proof of \ref{lem:orbit}. Since $\la$ is
minuscule, $\langle\alphavee,\la'\rangle = 0$, $\pm 1$. Therefore only
roots with $\langle\alphavee,\la'\rangle = 1$ contribute.

Next consider $\bz^*$. It is the multiplication of the equivariant Euler
class of $z^{\la'}\bN_\cO/z^{\la'}\bN_\cO\cap\bN_\cO$ by
\ref{sec:chang-repr} as before. We consider the decomposition
$\bN=\g=\bigoplus_{\alphavee\in\Delta} \g_\alphavee \oplus \ft$, and
conclude that roots $\alphavee$ with $\langle\alphavee,\la'\rangle = -1$
contribute. It gives the numerator $-\alphavee-\hbar/2+\bt$ in the
formula.
\end{proof}

\subsubsection{Quasi-minuscule coweights}
We consider a generalization of \propref{prop:minuscule} to the case
when $\la$ is a \emph{quasi-minuscule} coweight, i.e.,
$\la=\alpha_0$ where $\alphavee_0$ is the highest root.
\begin{NB}
    In the literature, a quasi-minuscule \emph{weight} is usually
    defined so that $\alphavee_0$ is the highest weight of the dual
    root system.
\end{NB}%
Then $\langle\alphavee,\la\rangle \le 2$ for any positive root
$\alphavee\in\Delta_+^\vee$, and the equality holds if and only if $\alphavee =
\alphavee_0$. Therefore
\begin{equation}\label{eq:85}
    \begin{split}
        & e(T_{\la}\Gr_G^{\la}) = (\alphavee_0 + \hbar)
        \prod_{\substack{\alphavee\in\Delta^\vee \\
            \langle\alphavee,\la\rangle > 0}} \alphavee,
        \\
        & e(z^\la\bN_\cO/z^\la\bN_\cO\cap \bN_\cO)
        = %(-\alphavee_0 - \frac{\hbar}2 + \mathsf t)
        (-\alphavee_0 - \frac{3\hbar}2 + \mathsf t)
        \prod_{\substack{\alphavee\in\Delta^\vee \\
            \langle\alphavee,\la\rangle > 0}} (- \alphavee -
        \frac{\hbar}2 + \mathsf t).
    \end{split}
\end{equation}
In fact, $\Gr^\la_G$ is a line bundle $L$ over $G/P_\la$. The factor
$\alphavee_0+\hbar$ corresponds to the tangent direction to the
fiber.
The space $z^\la\bN_\cO/z^\la\bN_\cO\cap \bN_\cO$ is the fiber of the
quotient $\cT/\cR$ at $z^\la\in \Gr_G$. For $\bN = \g$, the quotient
is the cotangent bundle of $\Gr^\la_G$. Therefore the second formula
in \eqref{eq:85} is obtained from the first one by changing the sign,
and then adding $-\hbar/2+\mathsf t$ for each factor, which
corresponds to the action on fibers.

The closure $\overline{\Gr}^\la_G = \Gr^\la_G \sqcup \Gr^0_G$ has a
singularity at $1 = \Gr^0_G$ (isomorphic to the singularity of the
closure of the minimal nilpotent orbit in $\g$ at $0$), but it has a
resolution $\CP(\shfO\oplus L)$ the projective bundle associated with
$L$. (See \cite[Lemma~7.3]{MR1832331}.)
Also the vector bundle $\cT/\cR$ over $\Gr^\la_G$ extends to
$\CP(\shfO\oplus L)$ as it is the cotangent bundle. 
More precisely, let us denote by $p\colon \BP(\CO\oplus L)\to\ol\Gr{}^\lambda_G$
the above resolution. Then the vector subbundle $\cR_\lambda\subset\cT_\lambda$
over $\Gr^\lambda_G$ extends to a vector subbundle in $p^*\cT$ over the whole
of $\BP(\CO\oplus L)$, to be denoted $\widetilde\cR_{\leq\lambda}$, such that
$p^*\cT/\widetilde\cR_{\leq\lambda}$ is the cotangent bundle
$T^*\BP(\CO\oplus L)$. We have a proper projection
$p\colon \widetilde\cR_{\le\la}\to\cT$ with the image lying in
$\cR_{\le\la}$. By base change we can compute $\bz^*(\iota_*)^{-1}$ of a
class $p_*(f [\widetilde\cR_{\le\la}])$ over $\widetilde\cR_{\le\la}$,
where $f\in \CC[\ft]^{\Weyl_\la}$ viewed as a class in
$H^*_{\Stab_G(\la)}(\mathrm{pt}) \cong H^*_G(G/P_\la)$ pull-backed to
$\widetilde\cR_{\le\la}$.
\begin{NB}
    It seems that it does not give new elements even if we consider
    $p_*(f c_1(H)^k [\widetilde\cR_{\le\la}])$, where $H$ is the
    hyperplane bundle of $\CP(\shfO\oplus L)$. First of all, it is
    enough to consider $k=0,1$, as $c_1(H)^2$ can be put into $f$. The
    difference for $k=0$, $1$ seems to be just constant (without
    shift).
\end{NB}%

The torus fixed points in $\widetilde\cR_{\le\la}$ come in pairs, $0$ and
$\infty$ in $\proj^1$ for each $T$-fixed point in $G/P_\la$, i.e., a
point in the orbit $\Weyl\la$. Let us denote them by $0_{\la'}$,
$\infty_{\la'}$ for $\la'\in\Weyl\la$. The points $0_{\la'}$ are in
$\Gr^\la_G$, hence the Euler classes are given by the formula
\eqref{eq:85}, after applying $w$ with $\la' = w\la$.
At $\infty_{\la'}$, the Euler class of the tangent space
$e(T_{\infty_{\la'}}\CP(\shfO\oplus L))$ is almost the same as
$e(T_{\la'}\Gr^\la_G)$, but the factor $\alpha_0+\hbar$ corresponding
to the fiber of the projective bundle changes the sign. The second
Euler class $e((\cT/\cR)_{\infty_{\la'}})$ is obtained from
$e(T_{\infty_{\la'}}\CP(\shfO\oplus L))$ by the same process as
before. We thus get

\begin{Theorem}\label{thm:quasiminuscule}
    Let $\la = \alpha_0$, the quasi-minuscule coroot. Then
    \begin{equation*}
        \bz^* (\iota_*)^{-1} p_* (f%c_1(H)^k
        [\widetilde\cR_{\le\la}]) =
        \begin{aligned}[t]
            \sum_{w\la\in\Weyl\la} wf \times & \Big( \frac{-w\alphavee_0 -
              \nicefrac{3\hbar}2 + \mathsf t}{w\alphavee_0 + \hbar}
            \prod_{\substack{\alphavee\in\Delta\\
                \langle\alphavee,w\la\rangle > 0}}
            \frac{-\alphavee-\nicefrac{\hbar}2 + \mathsf t}{\alphavee}
            \sfu_{w\la}
            \\
            & \quad + \frac{w\alphavee_0 + \nicefrac{\hbar}2 + \mathsf
              t}{-w\alphavee_0 - \hbar} \prod_{\substack{\alphavee\in\Delta\\
                \langle\alphavee,w\la\rangle > 0}}
            \frac{-\alphavee-\nicefrac{\hbar}2 + \mathsf t}{\alphavee} \Big).
        \end{aligned}
    \end{equation*}
\end{Theorem}

When $f=1$, this is a rational version of the Macdonald operator for a
quasi-minuscule weight \cite{MR1817334} up to constant in
$\CC[\ft]^{\Weyl_\la}$. The constant vanishes if we use the form in
\cite{MR2822182}.

\begin{Remark}
   For general $\bN$, we are not certain whether we have a resolution
   $\tilde\cR_{\le\la}$ of $\cR_{\le\la}$ for which we can calculate
   $\bz^*(\iota_*)^{-1}$. Nevertheless it is clearly possible for
   $\bN=0$: we have a resolution $\CP(\shfO\oplus L)$ of
   $\overline{\Gr}_G^\la$. In this case, we get a formula as in
   \thmref{thm:quasiminuscule}, where the numerator is replaced by
   $1$. Its proof is contained in one in \thmref{thm:quasiminuscule}.

   % This should be compared with \emph{explicit} operators in Toda
   % system. (See ????)
\end{Remark}

\subsubsection{Small fundamental}
\label{small}
Recall that apart from type $A$, the quasi-minuscule coweight is fundamental.
More generally, we consider a \emph{small} fundamental coweight $\omega$, i.e.\
$\langle\alphavee,\omega\rangle\leq2$ for any $\alphavee\in\Delta^\vee$ (see e.g.~\cite{MR2822182}).\footnote{Some authors use another definition of small
coweights: $\omega$ is small if in the corresponding irreducible representation
$V^\omega$ of $G^\vee$ the zero weight has a nonzero multiplicity, but the weight
$2\alpha$ has zero multiplicity for any $\alpha\in\Delta$.} 
According to~\cite[Table~1]{MR2822182}, any dominant coweight $\mu\leq\omega$ is
also small fundamental, and all such coweights are totally ordered:
$\omega^{(0)}<\omega^{(1)}<\ldots<\omega^{(n)}=\omega$, and $\omega^{(0)}$ is either
minuscule or zero. Moreover, there is a chain of connected subdiagrams of the
Dynkin diagram of $G\colon D^{(1)}\supset D^{(2)}\supset\ldots\supset D^{(n)}$
with the corresponding Levi subgroups $G\supset L^{(1)}\supset\ldots\supset L^{(n)}\supset T$
such that $\omega^{(i)}-\omega^{(i-1)}$ is the quasiminuscule coweight $\alpha_0^{(i)}$
of $L^{(i)}$. Note that $\alpha_0^{(n)}$ is a fundamental coweight of $L^{(n)}$,
and we define $D^{(n+1)}$ as the complement in $D^{(n)}$ of the corresponding vertex;
$L^{(n+1)}\subset L^{(n)}$ is the corresponding Levi subgroup.
According to~\cite[Lemma~3.1.1]{mov}, there is a natural isomorphism
of slices $\oW^{\omega^{(i)}}_{G,\omega^{(i-1)}}\simeq\oW^{\alpha_0^{(i)}}_{L^{(i)},0}$.
Hence the above resolution $\widetilde\Gr{}^{\alpha_0}_G$ admits the following generalization:
a resolution $\widetilde\Gr{}^\omega_G\to\ol\Gr{}^\omega_G$ constructed as an
iterated blowup. We first take the blowup $\on{Bl}^{(1)}:=
\on{Bl}_{\ol\Gr{}^{\omega^{(0)}}_G}\ol\Gr{}^\omega_G$ at the closed
$G_\CO$-orbit $\ol\Gr{}^{\omega^{(0)}}_G$. The strict transform of
$\ol\Gr{}^{\omega^{(1)}}_G\subset\ol\Gr{}^\omega_G$ is a resolution
$\widetilde\Gr{}^{\omega^{(1)}}_G\subset\on{Bl}^{(1)}$. The preimage of
$\ol\Gr{}^{\omega^{(0)}}_G\subset\ol\Gr{}^\omega_G$ fibers over
$\ol\Gr{}^{\omega^{(0)}}_G$ with fibers isomorphic to the partial flag variety
$L^{(1)}/P^{(1)}_{\alpha_0^{(1)}}$ of the Levi group $L^{(1)}$. We define
$\on{Bl}^{(2)}:=\on{Bl}_{\widetilde\Gr{}^{\omega^{(1)}}_G}\on{Bl}^{(1)}$.
The strict transform of
$\ol\Gr{}^{\omega^{(2)}}_G\subset\ol\Gr{}^\omega_G$ is a resolution
$\widetilde\Gr{}^{\omega^{(2)}}_G\subset\on{Bl}^{(2)}$. The preimage of
$\widetilde\Gr{}^{\omega^{(1)}}_G\subset\on{Bl}^{(1)}$ fibers over
$\widetilde\Gr{}^{\omega^{(1)}}_G$ with fibers isomorphic to the partial flag
variety $L^{(2)}/P^{(2)}_{\alpha_0^{(2)}}$.
We continue like this till we arrive at $\widetilde\Gr{}^\omega_G:=\on{Bl}^{(n)}:=
\on{Bl}_{\widetilde\Gr{}^{\omega^{(n-1)}}_G}\on{Bl}^{(n-1)}\stackrel{p}{\to}\ol\Gr{}^\omega_G$. 
The preimage of
$\widetilde\Gr{}^{\omega^{(n-1)}}_G\subset\on{Bl}^{(n-1)}$ fibers over
$\widetilde\Gr{}^{\omega^{(n-1)}}_G$ with fibers isomorphic to the partial flag
variety $L^{(n)}/P^{(n)}_{\alpha_0^{(n)}}$. The vector subbundle $\cR_\omega\subset\cT_\omega$
over $\Gr^\omega_G$ extends to a vector subbundle in $p^*\cT$ over the whole of
$\widetilde\Gr{}^\omega_G$, to be denoted $\widetilde\cR_{\leq\omega}$,
such that $p^*\cT/\widetilde\cR_{\leq\omega}=T^*\widetilde\Gr{}^\omega_G$. 
We have a proper projection
$p\colon \widetilde\cR_{\leq\omega}\to\cT$ with the image lying in
$\cR_{\le\omega}$.

Since $\ol\Gr{}_G^\omega=\bigsqcup_{0\leq i\leq n}\Gr_G^{\omega^{(i)}}$, the $T$-fixed points
set $(\ol\Gr{}_G^\omega)^T$ decomposes into a disjoint union
$\bigsqcup_{0\leq i\leq n}(\Gr_G^{\omega^{(i)}})^T=\bigsqcup_{0\leq i\leq n}\Weyl\omega^{(i)}$.
From the above description of the resolutions we get
$(\widetilde\cR_{\leq\omega})^T=(\widetilde\Gr{}^\omega_G)^T=
\bigsqcup_{0\leq i\leq n}\Weyl\omega^{(i)}\times\Weyl^{(i+1)}/\Weyl^{(i+2)}\times\ldots\times
\Weyl^{(n-1)}/\Weyl^{(n)}\times\Weyl^{(n)}/\Weyl^{(n+1)}$.
We rewrite the RHS in the following form:
$(\widetilde\cR_{\leq\omega})^T=(\widetilde\Gr{}^\omega_G)^T=
\bigsqcup_{0\leq i\leq n}\{(\nu\in\Weyl\omega^{(i)},\ \eta^{(i+1)}\in
w_i\Weyl^{(i+1)}\omega^{(i+1)},\ldots,\ \eta^{(n)}\in w_{n-1}\Weyl^{(n)}\omega^{(n)})\}$
where $w_i$ is an element of $\Weyl$ such that $\nu=w_i\omega^{(i)},\
w_{i+1}$ is an element of $\Weyl$ such that $\eta^{(i+1)}=w_{i+1}\omega^{(i+1)}$,
and so on, and finally $w_{n-1}$ is an element of $\Weyl$ such that
$\eta^{(n-1)}=w_{n-1}\omega^{(n-1)}$.

The Euler class of the tangent space
$e(T_{(\nu,\eta^{(i+1)},\ldots,\eta^{(n)})}\widetilde\Gr{}^\omega_G)$ equals
$$\prod_{\substack{\alphavee\in\Delta^\vee\\ \langle\nu,\alphavee\rangle=2}}(\alphavee+\hbar)\cdot
\prod_{\substack{\alphavee\in\Delta^\vee\\ \langle\nu,\alphavee\rangle>0}}\alphavee\cdot
\prod_{i+1\leq j\leq n}\left(\left(w_{j-1}(\alpha_0^{(j)})^{\!\scriptscriptstyle\vee}+\hbar\right)
\prod_{\substack{\alphavee\in w_{j-1}\Delta^\vee_{(j)}\\
\langle\eta^{(j)},\alphavee\rangle>0}}\alphavee\right).$$
Here the second product arises from the tangent bundle to the
partial flag variety $(\Gr_G^{\omega^{(i)}})^{\BC^\times}$ (fixed points set of the
loop rotations); the first product arises from the normal bundle
${\mathcal N}_{(\Gr_G^{\omega^{(i)}})^{\BC^\times}/\Gr_G^{\omega^{(i)}}}$;
the last product arises from the tangent bundle to the fiber of blowup,
and its prefactor arises from the normal bundle to the fiber of blowup.
The Euler class of the cotangent bundle fiber at
$(\nu,\eta^{(i+1)},\ldots,\eta^{(n)})\in\widetilde\Gr{}^\omega_G$ is obtained from
the above one by changing the sign of each factor and then adding $-\hbar+\sft$
to each factor. The result is
$$\prod_{\substack{\alphavee\in\Delta^\vee\\
\langle\nu,\alphavee\rangle=2}}(\sft-\alphavee-\frac{3\hbar}{2})\times
\prod_{\substack{\alphavee\in\Delta^\vee\\
\langle\nu,\alphavee\rangle>0}}(\sft-\alphavee-\frac{\hbar}{2})\times
\prod_{i+1\leq j\leq n}\left(\left(\sft-w_{j-1}
(\alpha_0^{(j)})^{\!\scriptscriptstyle\vee}-\frac{3\hbar}{2}\right)
\prod_{\substack{\alphavee\in w_{j-1}\Delta^\vee_{(j)}\\
\langle\eta^{(j)},\alphavee\rangle>0}}(\sft-\alphavee-\frac{\hbar}{2})\right).$$
We thus get (cf.~\cite[Section~3]{MR2822182})

\begin{Theorem}\label{thm:small}
    Let $\omega$ be a small fundamental coweight. Then
    \begin{equation*}
        \bz^* (\iota_*)^{-1}p_*
        [\widetilde\cR_{\leq\omega}] =
        \begin{aligned}[t]
\sum_{0\leq i\leq n} & \sum_{(\nu,\eta^{(i+1)},\ldots,\eta^{(n)})}
\prod_{\substack{\alphavee\in\Delta^\vee\\
\langle\nu,\alphavee\rangle=2}}\frac{\sft-\alphavee-\nicefrac{3\hbar}2}{\alphavee+\hbar}\times
\prod_{\substack{\alphavee\in\Delta^\vee\\
\langle\nu,\alphavee\rangle>0}}\frac{\sft-\alphavee-\nicefrac{\hbar}2}{\alphavee}
\\
& \times\prod_{i+1\leq j\leq n}\left(\frac{\sft-w_{j-1}
(\alpha_0^{(j)})^{\!\scriptscriptstyle\vee}-\nicefrac{3\hbar}2}
{w_{j-1}(\alpha_0^{(j)})^{\!\scriptscriptstyle\vee}+\hbar}
\prod_{\substack{\alphavee\in w_{j-1}\Delta^\vee_{(j)}\\
\langle\eta^{(j)},\alphavee\rangle>0}}
\frac{\sft-\alphavee-\nicefrac{\hbar}2}{\alphavee}\right)
\sfu_\nu.
        \end{aligned}
    \end{equation*}
\end{Theorem}

\begin{Remark}
Note that the fundamental class $[\widetilde\cR_{\leq\omega}]$ does not have a 
coefficient $f\in\BC[\ft]^{\Weyl_\omega}$ as opposed 
to~\thmref{thm:quasiminuscule}, because there is no projection
$[\widetilde\cR_{\leq\omega}]\to G/P_\omega$ for arbitrary small fundamental 
$\omega$, so we do not have a way to produce natural homology classes on
$\widetilde\cR_{\leq\omega}$ except its fundamental class.
\end{Remark}

\begin{Question}
    \begin{NB}
(1)    Consider $\hbar = 0 = \mathbf t$. We get
    \begin{equation*}
        \sum_{w\la\in\Weyl\la} wf \times \sfu_{w\la} \quad
        (f\in \CC[\ft]^{\Weyl_\la})
    \end{equation*}
    for $\la$ a minuscule or quasi-minuscule coweight \textup(up to
    $\CC[\ft]^\Weyl$ for the latter\textup). Together with
    $\CC[\ft]^\Weyl$, do these elements generate $\CC[\mathcal M_C]
    \cong \CC[\ft\times T^\vee]^\Weyl$ as an associative algebra ?

    (2) We do not have many (quasi-)minuscule weights in
    general. Therefore let us just set $\mathbf t = \hbar/2$. Then we get
    \begin{equation*}
        \sum_{w\la\in\Weyl\la} wf \times \sfu_{w\la}  \quad
        (f\in \CC[\ft]^{\Weyl_\la})
    \end{equation*}
    up to sign. (In the quasi-minuscule, we get an additional term
    $\sum_{w\la\in\Weyl\la} wf$, which is in $\CC[\ft]^\Weyl$, hence
    can be omitted.) Here $\sfu_{w\la}$ is an $\hbar$-difference
    operator in this case. Together with $\CC[\ft]^\Weyl$, do these
    elements generate $\mathcal D_\hbar[\ft]^\Weyl$ at $\hbar\neq 0$ ?
    Here $\mathcal D_\hbar[\ft]$ is the ring of $\hbar$-difference
    operators on $\ft$, and $\mathcal D_\hbar[\ft]^\Weyl$ is its
    $\Weyl$-invariant part.

    (3) 
        
    \end{NB}%
    We know that $\CC[\mathcal M_C]\cong \CC[\ft\times T^\vee]^\Weyl$
    as a Poisson algebra. We do not know elements in
    \propref{prop:ad_minuscule}, Theorems~\ref{thm:quasiminuscule},
    \ref{thm:small} with $\CC[\ft]^\Weyl$ generate $\CC[\ft\times
    T^\vee]^\Weyl$ as a Poisson algebra at $\hbar = \bt = 0$, or they
    generate $\cAh$ if we invert $\hbar$. Recall (see
%    \ref{subsub:adjoint}
    \ref{subsec:prev}\oldref{Coulomb2-subsub:adjoint})
    that it is conjectured that $\cAh$ is
    isomorphic to the spherical part of the graded Cherednik
    algebra. We do not know the corresponding statement for the
    spherical part either. These are true for type $A$, as we will show
    in a separate publication.
\end{Question}

\begin{NB}
Consider type $A$ with $\mathbf t = \hbar/2$. We have
\begin{equation*}
    \left[\sum_i \sfu_i^p, \sum_j w_j \sfu_j\right]
    = p \hbar \sum_i \sfu_i^{p+1}.
\end{equation*}
Therefore starting from $p=1$, we can inductively obtain $\sum_i \sfu_i^p$ for $p=2,3,\dots$.

On the other hand,
\begin{equation*}
    \left[\sum_i \sfu_i, \sum_j w_j(w_j-\hbar) \right]
    = 2\hbar \sum_i w_i \sfu_i.
\end{equation*}
\end{NB}%

%%% Local Variables:
%%% mode: latex
%%% TeX-master: "blowup_pre"
%%% End:

\begin{NB}
The old manuscript is kept in \verb+GKLO_temp.tex+.
\renewenvironment{NB}{
\color{blue}{\bf NB2}. \footnotesize
}{}
\renewenvironment{NB2}{
\color{purple}{\bf NB3}. \footnotesize
}{}
\end{NB}%

%%% Local Variables:
%%% mode: latex
%%% TeX-master: "blowup_pre"
%%% End:

%\documentclass[A4,11pt]{amsart}
%\usepackage{amsmath,amsfonts,amssymb,amsthm,color,hyperref}
% \usepackage[all]{xy}

% \addtolength{\hoffset}{-1cm}
% \addtolength{\textwidth}{2cm}

% \newtheorem{Theorem}{Theorem}[section]
%\newtheorem{Proposition}[Theorem]{Proposition}
% \newtheorem{Lemma}[Theorem]{Lemma}
% \newtheorem{Question}[Theorem]{Question}
% \newtheorem{Corollary}[Theorem]{Corollary}
% \newtheorem{Conjecture}[Theorem]{Conjecture}
% \newtheorem{Definition}[Theorem]{Definition}
% \newtheorem{Example}[Theorem]{Example}
% \newtheorem{Remark}[Theorem]{Remark}

\newcommand{\nc}{\newcommand}
\newcommand{\renc}{\renewcommand}

% Fraktur font
%\nc{\fg}{\mathfrak g}
\nc{\fk}{\mathfrak k}
\nc{\fh}{\mathfrak h}
% Blackboard bold
%\nc{\C}{\mathbb{C}}
%\nc{\CW}{\mathcal{W}}
%\nc{\CC}{\C}
%\nc{\GV}{G_V}
%\nc{\bN}{N}
%\nc{\cAh}{A_h}
%\nc{\sfu}{\mathsf{u}}
%\nc{\uS}{\mathbf{S}}
%\nc{\uQ}{\mathbf{Q}}
%\nc{\cR}{\mathcal{R}}
%\nc{\scP}{\mathcal{P}}
%\nc{\BP}{\mathbb{P}}
%\nc{\unl}{\underline}
%\nc{\oW}{\overline{\mathcal{W}}}
%\nc{\BC}{\C}
%\nc{\iso}{\cong}
%\nc{\BA}{\mathbb{A}}
%\let\C=\CC
\nc{\Z}{\mathbb{Z}}
\nc{\R}{\mathbb{R}}
\nc{\Q}{\mathbb{Q}}
\nc{\N}{\mathbb{N}}
\nc{\Pone}{\mathbb{P}^1}
% Calligraphic font
%\nc{\cO}{\mathcal{O}}
\nc{\cC}{\mathcal{C}}
\nc{\cB}{\mathcal{B}}
%\nc{\cA}{\mathcal{A}}
%\renewcommand{\cA}{\mathcal{A}}
%\nc{\cW}{\mathcal{W}}
% Bold font
%\nc{\br}{\mathbf{r}}
\nc{\bd}{\mathbf{d}}
\nc{\bc}{\mathbf{c}}
\nc{\bb}{\mathbf{b}}
\nc{\ba}{\mathbf{a}}
% Operator names
\nc{\Sym}{\operatorname{Sym}}
%\nc{\Hom}{\operatorname{Hom}}
\nc{\fmod}{\operatorname{-fmod}}
% Miscellaneous
\nc{\con}{-}
%\nc{\hbar}{\hslash}
%\nc{\Gr}{\mathsf{Gr}}

\nc{\om}{\varpi}
\nc{\lv}{{\gamma}}
\nc{\Rees}{{\mathop{\operatorname{Rees}}}}
\nc{\bY}{{\mathbf{Y}}}

\newcommand{\jcom}{\color{red} Joel: }
\newcommand{\acom}{\color{green} Alex: }
%\begin{document}

\section{Shifted Yangians and quantization of generalized slices
% \\
%     By
%     Alexander~Braverman,
%     Michael~Finkelberg,
%     Joel~Kamnitzer,
%     Ryosuke~Kodera,
%     Hiraku~Nakajima,
%     Ben~Webster,
% and
%     Alex~Weekes
}\label{sec:quantization}

%\paragraph{Finite type case}
In this section, we study quiver gauge theory coming from the Dynkin diagram of a simple algebraic group $ G$.  As usual we fix an orientation of the Dynkin diagram and we fix  a dominant coweight $\lambda$ and a coweight $ \mu $ such that $ \lambda - \mu = \sum a_i \alpha_i $ with $ a_i \in \bN$.  We also fix a sequence of fundamental coweights $ \unl\lambda = (\omega_{i_1}, \dots, \omega_{i_N}) $ such that $ \sum_{s=1}^N \omega_{i_s} = \lambda $.  We will relate the quantized Coulomb branch to a generalization of the truncated shifted Yangians from \cite{kwy}.

\subsection{Shifted Yangians}

In this section, we will work with filtered algebras.  We begin by recalling some basic facts about filtered algebras and the Rees construction.

Let $ A $ be a $\BC$-algebra and let $ F^\bullet A = \dots \subseteq F^{-1} A \subseteq F^0 A \subseteq F^1 A \subseteq \dots$ be a separated and exhaustive filtration, meaning that $ \cap_k F^k A = 0 $ and $\cup_k F^k A = A $.  We assume that this filtration is compatible with the algebra structure in the sense that $ F^k A \cdot F^l A \subset F^{k+l} A $ and $1 \in F^0 A$.

In this case, we define the {\em Rees algebra} of $ A$ to be the graded $ \BC[\hbar]$--algebra $ \Rees^F A := \oplus_k \hbar^k F^k A $, viewed as a subalgebra of $ A[\hbar,\hbar^{-1}]$. We also define the associated graded of $ A $ to be the graded algebra $ \gr^F A := \bigoplus F^k A / F^{k-1} A $.  Note that we have a canonical isomorphism of graded algebras $ \Rees^F A/ \hbar \Rees^F A \cong \gr^F A  $.

We say that the filtered algebra $ A $ is {\em almost commutative} if $ \gr^F A $ is commutative.  In this case, for any $ a \in F^k A $, $ b \in F^l A$, we have $ ab - ba \in F^{k+l - 1} A $.  Thus in $ \Rees^F$, we can define a Poisson bracket by $ \{ a, b\} := \frac{1}{\hbar}(ab - ba) $.

\begin{Definition}
We define the ``Cartan doubled Yangian'' $ Y_{\infty} $ to be the $ \BC $-algebra with generators $ E_i^{(q)}, F_i^{(q)}, H_i^{(p)} $ for $ q > 0 $ and $ p \in \Z $ and $ i \in \II $
\begin{align*}
[H_i^{(p)}, H_j^{(p)}] &= 0,  \\
[E_i^{(p)}, F_j^{(q)}] &=   \delta_{ij} H_i^{(p+q-1)}, \\
[H_i^{(p+1)},E_j^{(q)}] - [H_i^{(p)}, E_j^{(q+1)}] &= \frac{ \alpha_i \cdot \alpha_j}{2} (H_i^{(p)} E_j^{(q)} + E_j^{(q)} H_i^{(p)}) , \\
[H_i^{(p+1)},F_j^{(q)}] - [H_i^{(p)}, F_j^{(q+1)}] &= -\frac{ \alpha_i \cdot \alpha_j}{2} (H_i^{(p)} F_j^{(q)} + F_j^{(q)} H_i^{(p)}) , \\
[E_i^{(p+1)}, E_j^{(q)}] - [E_i^{(p)}, E_j^{(q+1)}] &= \frac{ \alpha_i \cdot \alpha_j}{2} (E_i^{(p)} E_j^{(q)} + E_j^{(q)} E_i^{(p)}), \\
[F_i^{(p+1)}, F_j^{(q)}] - [F_i^{(p)}, F_j^{(q+1)}] &= -\frac{ \alpha_i \cdot \alpha_j}{2} (F_i^{(p)} F_j^{(q)} + F_j^{(p)} F_i^{(q)}),\\
i \neq j, N = 1 - \alpha_i \cdot \alpha_j \Rightarrow
\operatorname{sym} &[E_i^{(p_1)}, [E_i^{(p_2)}, \cdots [E_i^{(p_N)}, E_j^{(q)}]\cdots]] = 0, \\
i \neq j, N = 1 - \alpha_i \cdot \alpha_j \Rightarrow
\operatorname{sym} &[F_i^{(p_1)}, [F_i^{(p_2)}, \cdots [F_i^{(p_N)}, F_j^{(q)}]\cdots]] = 0, 
\end{align*}
\end{Definition}

\begin{Definition}
The {\em shifted Yangian} $Y_\mu$ is the quotient of $ Y_{\infty} $ by the relations $ H_i^{(p)} = 0 $ for $ p < -\langle \mu, \alpha_i^\vee \rangle $ and $ H_i^{(-\langle \mu, \alpha_i^\vee \rangle)} = 1 $.
\end{Definition}

\begin{Remark}
When $ \mu = 0 $, then it is easy to see that  $Y = Y_0$ coincides with the Yangian, as defined in~\cite[Section~3.4]{kwy}.  On the other hand, suppose that $ \mu $ is dominant.  Then the map $ Y_\mu \rightarrow Y $ defined by
\begin{equation*}
H_i^{(s)} \mapsto H_i^{(s + \langle \mu, \alpha_i^\vee \rangle)}, \ E_i^{(s)} \mapsto E_i^{(s)}, \ F_i^{(s)} \mapsto F_i^{(s + \langle \mu, \alpha_i^\vee \rangle)}
\end{equation*}
gives an isomorphism between $Y_\mu $ and the subalgebra of $ Y $ which is also denoted $Y_\mu $ in \cite{kwy}.

To be a bit more precise, in \cite{kwy}, we worked with the corresponding graded $\BC[\hbar]$-algebras. In fact, we made a mistake concerning these presentations of these algebras; \cite[Theorem~3.5]{kwy} is incorrect.  We claimed to give a presentation of $ (U_{\hbar} \fg[z])'$, using generators $ E_\alpha^{(p)}, H_i^{(p)}, F_\alpha^{(p)}$, but we are definitely missing relations involving the $ E_\alpha^{(p)}, F_\beta^{(q)} $ for $ \alpha, \beta $ not simple roots.  At this time, we do not know an explicit description of all the relations.  In this paper, we will work with Rees algebras to avoid this problem.
\end{Remark}

Denote the generators of the shifted Yangian $Y_\mu$ by $E_i^{(r)}, H_i^{(r)}, F_i^{(r)}$, and form their respective generating series
$$ E_i(z) = \sum_{r> 0 } E_i^{(r)} z^{-r}, \quad H_i(z) = z^{\langle\mu, \alpha_i^\vee\rangle} + \sum_{r>-\langle \mu, \alpha_i^\vee\rangle} H_i^{(r)} z^{-r}, \quad F_i(z) = \sum_{r>0} F_i^{(r)} z^{-r}.$$

The relations for $Y_\mu$ can be written as identities of formal series.  First, given a series $ X(z) = \sum_{r \in \mathbb Z} X_i^{(r)} z^{-r} $, we write $\underline{X(z)} = \sum_{r > 0} X_i^{(r)} z^{-r} $ for the principal part.

Then for all $i,j\in \II$ we have relations
\begin{align}
[H_i(z), H_j(y)] &= 0, \label{eq: Y relation H and H} \\
(z - y - a ) H_i(z) E_j(y) &= (z-y+a) E_j(y) H_i(z) - 2a E_j(z-a) H_i(z), \label{eq: Y relation H and E}\\
(z-y-a) E_i(z) E_j(y) &= (z-y+a)E_j(y) E_i(z) + [E_i^{(1)}, E_j(y)] - [E_i(z), E_j^{(1)}], \label{eq: Y relation E and E}\\
(z - y + a ) H_i(z) F_j(y) &= (z-y-a) F_j(y) H_i(z) + 2a F_j(z+a) H_i(z), \\
(z-y+a) F_i(z) F_j(y) &= (z-y-a)F_j(y) F_i(z) + [F_i^{(1)}, F_j(y)] - [F_i(z), F_j^{(1)}], \\
(z-y)[E_i(z), F_j(y)] &= \delta_{i,j} \left( \underline{H_i(y)} - \underline{H_i(z)} \right), \label{eq: Y relation E and F}
\end{align}
where we denote $a = \tfrac{1}{2} \alpha_i\cdot \alpha_j$.
We also have the Serre relations.  First when $ a_{ij} = 0 $. we have
\begin{align}
[E_i(z), E_j(y)] &= 0  \\
[F_i(z), F_j(y)] &= 0
\end{align}
and for $a_{ij}=-1$ we have
\begin{align}
[E_i(z_1), [E_i(z_2), E_j(y)]] + [E_i(z_2), [E_i(z_1), E_j(y)]] &= 0, \label{eq: Y relation Serre}\\
[F_i(z_1), [F_i(z_2), F_j(y)]] + [F_i(z_2),[ F_i(z_1), F_j(y)]] &= 0.
\label{eq: Y relation Serre F}
\end{align}

Let $ \mu_1, \mu_2 $ be two coweights such that $\mu_1 + \mu_2 = \mu $.  In \cite{fkprw}, we defined filtrations $ F_{\mu_1, \mu_2} Y_\mu $ of $ Y_\mu $.  In this filtration, the degrees of the generators are given
$$ \deg E_i^{(r)} = \langle \mu_1, \alpha_i^\vee \rangle + r, \ \deg F_i^{(r)} = \langle \mu_2, \alpha_i^\vee \rangle + r, \ \deg H_i^{(r)} = \langle \mu, \alpha_i^\vee \rangle + r 
$$
However, we note that these degrees do not determine the filtration because we also specify the degrees of certain PBW variables, see \cite[section 5.4]{fkprw} for more details.

In \cite{fkprw}, we proved that $Y_\mu $ is almost commutative with this filtration.  We also proved that for any pair $ \mu_1, \mu_2 $ as above, the Rees algebras $ \Rees^{F_{\mu_1, \mu_2}} Y_\mu $ are canonically isomorphic (as $ \BC[\hbar]$-algebras). 

For the purposes of this paper, we will choose $ \mu_1, \mu_2 $ as follows
$$
\langle \mu_1, \alpha_i^\vee \rangle = \langle \lambda, \alpha_i^\vee \rangle - a_i +  \sum_{h:\vin{h}=i} a_{\vout{h}}, \ \langle \mu_2, \alpha_i^\vee \rangle = -a_i + \sum_{h:\vout{h}=i} a_{\vin{h}}
$$
where the sums are taken over all arrows $h$ to $i$ or from $i$ respectively.
We write $ \bY_\mu := \Rees^{F_{\mu_1, \mu_2}} Y_\mu $ for this Rees algebra (with the induced grading).  

\subsection{A representation using difference operators}
We will work with the larger algebra $Y_\mu[z_1,\ldots,z_N] = Y_\mu \otimes \CC[z_1,\ldots,z_N]$.  We extend the filtration $ F_{\mu_1, \mu_2} $ to $ Y_\mu[z_1,\ldots,z_N] $ by placing all generators in degree 1.

Denote
$$Z_i(z) = \prod_{k: i_k = i} (z - z_k-\tfrac{1}{2}),$$
and define new ``Cartan'' elements $A_i^{(p)}$ for $p>0$ by
\begin{equation} \label{eq: H and A}
H_i(z) = Z_i(z) \frac{\prod_{\substack{h\in\Qo\sqcup\overline\Qo\\ \vout{h}= i}} (z-\tfrac{1}{2})^{a_{\vin{h}}}}{z^{a_i} (z-1)^{a_i}} \frac{\prod_{\substack{h\in\Qo\sqcup\overline\Qo\\ \vout{h}= i}} A_{\vin{h}}(z-\tfrac{1}{2} )}{A_i(z) A_i(z-1)}.
\end{equation}

Consider also the $\CC$-algebra $\tilde\cA\defeq 
    \CC[z_1,\dots, z_N]
    \langle w_{i,r}, \sfu^{\pm 1}_{i,r}, (w_{i,r}-w_{i,s}+m)^{-1}
    (r\neq s,\ m\in{\mathbb Z})\rangle$, 
with relations $[\sfu_{i,r}^\pm, w_{j,s}] = 
\pm \delta_{i,j}\delta_{r,s}  \sfu_{i,r}^\pm$.  
Denote
$$W_i (z) = \prod_{r=1}^{a_i}(z - w_{i,r}) \ \ \text{ and } \ \ W_{i,r}(z) = \prod_{\substack{s=1 \\ s\neq r}}^{a_i} (z - w_{i,s}).$$

We define a filtration on $ \tilde \cA $ by setting by setting the degree of each $ w_{i,r} $ to be $1$ and the degree of $ \sfu_{i,r}^\pm $ to be $0 $.  The filtration degree of each $(w_{i,r}-w_{i,s}+m)^{-1}$ is also set to be $-1$.

Note that $ \tilde \cA $ is almost commutative and we have $ \Rees \tilde \cA = \tilde \cA_\hbar$, the algebra defined in~\ref{embed_diff_op}.

The following result generalizes~\cite[Theorem~4.5]{kwy} which was a generalization of a construction of Gerasimov-Kharchev-Lebedev-Oblezin \cite{GKLO}.

\begin{Theorem} \label{th:filGKLO} There is a homomorphism of filtered $\BC$-algebras $\Phi_\mu^\lambda\colon Y_\mu[z_1,\ldots,z_N] \longrightarrow\tilde\cA$, defined by
\begin{align*}
A_i(z) & \mapsto z^{-a_i} W_i(z), \\
E_i(z) & \mapsto -\sum_{r=1}^{a_i}  \frac{Z_i(w_{i,r}) \prod_{h\in\Qo:\vin{h}=i} W_{\vout{h}}(w_{i,r} - \tfrac{1}{2})}{(z-w_{i,r}) W_{i,r}(w_{i,r})} \sfu_{i,r}^{-1}, \\
F_i(z) & \mapsto \sum_{r=1}^{a_i} \frac{\prod_{h\in\Qo:\vout{h}=i} W_{\vin{h}}(w_{i,r} + \tfrac{1}{2} )}{(z-w_{i,r} - 1) W_{i,r}(w_{i,r})} \sfu_{i,r}.
\end{align*}
\end{Theorem}

\begin{proof}
The argument is basically the same as in~\cite[Theorem~4.5]{kwy}.  The proof in \cite{kwy} should be considered incomplete, since we didn't have a complete presentation.

We verify the relations~(\ref{eq: Y relation H and H})--(\ref{eq: Y relation Serre F})
%from Lemma  \ref{lemma: series relations for Y_mu}
which involve $E_i(z)$, those involving $F_i(z)$ being similar (they can also 
be deduced from the $E_i(z)$ cases by using certain involutions of $Y_\mu$ and 
$\tilde\cA$).

Note that by (\ref{eq: H and A}), under $\Phi_\mu^\lambda$ we have
$$ H_i(z) \mapsto \frac{ Z_i(z) \prod_{\substack{h\in\Qo\sqcup\overline\Qo\\ \vout{h}= i}} W_{\vin{h}}(z-\tfrac{1}{2} )}{W_i(z) W_i(z-1)},$$
and these images clearly satisfy equation (\ref{eq: Y relation H and H}).

%%%%%%%
\subsection{Relation (\ref{eq: Y relation H and E}) between
\texorpdfstring{$H_i(z)$}{Hi(z)} and
\texorpdfstring{$E_j(y)$}{Ej(y)}}

\subsubsection{ The case \texorpdfstring{$a_{ij} = 0$}{aij=0}:} Equation (\ref{eq: Y relation H and E}) simply says that $H_i(z)$ and $E_j(y)$ commute.  It is clear that this holds true for their images under $\Phi_\mu^\lambda$.

\subsubsection{The case \texorpdfstring{$a_{ij} = -1$}{aij=-1}:} Here, equation (\ref{eq: Y relation H and E}) reads
$$ (z-y + \tfrac{1}{2}) H_i(z) E_j(y) = (z-y-\tfrac{1}{2}) E_j(y) H_i(z) +  E_j(z+\tfrac{1}{2}) H_i(z).$$
This is an $\mathfrak{sl}_3$ relation, and we can assume that $\II = \{i,j\}$.  We may also assume that $\Qo$ consists of a single arrow $j\rightarrow i$.

Then, the image of the left-hand side under $\Phi_\mu^\lambda$ is
$$- (z - y +\tfrac{1}{2} ) \cdot \frac{ Z_i(z) W_j(z-\tfrac{1}{2} )}{W_i(z) W_i(z-1)} \cdot \sum_{r=1}^{a_j}  \frac{Z_j(w_{j,r})}{(y-w_{j,r}) W_{j,r}(w_{j,r})} \sfu_{j,r}^{-1} $$
$$ =-\sum_{r=1}^{a_j} \frac{(z-y+\tfrac{1}{2})(z- w_{j,r}-\tfrac{1}{2})}{y-w_{j,r}} \frac{Z_i(z) W_{j,r}(z-\tfrac{1}{2})}{W_i(z) W_i(z-1)}\frac{Z_j(w_{j,r})}{W_{j,r}(w_{j,r})} \sfu_{j,r}^{-1}.$$
On the other hand, the image of the right-hand side under $\Phi_\mu^\lambda$ is
$$  -\sum_{r=1}^{a_j}  \left( \frac{z-y-\tfrac{1}{2}}{y- w_{j,r}} + \frac{1}{z-w_{j,r}+\tfrac{1}{2}}\right) \frac{Z_j(w_{j,r})}{W_{j,r}(w_{j,r})} \sfu_{j,r}^{-1} \cdot \frac{ Z_i(z) W_j(z-\tfrac{1}{2} )}{W_i(z) W_i(z-1)}.$$
Commuting the $\Phi_\mu^\lambda(H_i(z))$ factor to the left, this is equal to
$$ -\sum_{r=1}^{a_j} \left( \frac{z-y-\tfrac{1}{2}}{y- w_{j,r}} + \frac{1}{z-w_{j,r}+\tfrac{1}{2}}\right) (z-w_{j,r}+\tfrac{1}{2}) \frac{Z_i(z) W_{j,r}(z-\tfrac{1}{2})}{W_i(z) W_i(z-1)} \frac{Z_j(w_{j,r})}{W_{j,r}(w_{j,r})} \sfu_{j,r}^{-1}.$$
So, the relation (\ref{eq: Y relation H and E}) follows in this case from an equality of rational functions:
$$ \frac{(z-y+\tfrac{1}{2})(z- w_{j,r}-\tfrac{1}{2})}{y-w_{j,r}} =  \left( \frac{z-y-\tfrac{1}{2}}{y- w_{j,r}} + \frac{1}{z-w_{j,r}+\tfrac{1}{2}}\right) (z-w_{j,r}+\tfrac{1}{2}).$$

\subsubsection{The case \texorpdfstring{$i = j$}{i=j}:}
Here, equation (\ref{eq: Y relation H and E}) says that
$$ (z-y -1) H_i(z) E_i(y) = (z-y+1) E_i(y) H_i(z) - 2  E_i(z-1) H_i(z).$$
In this case we may assume that $\fg = \mathfrak{sl}_2$, and so we will temporarily drop the index $i$ from our notation.

The image of the left-hand side under $\Phi_\mu^\lambda$ is then
$$ -(z-y-1) \cdot \frac{Z(z)}{W(z) W(z-1)} \cdot \sum_{r=1}^{a} \frac{Z(w_r)}{(y-w_r) W_r(w_r)} \sfu_r^{-1} $$
$$ = -\sum_{r=1}^a \frac{(z-y - 1)}{(y-w_r) (z-w_r)(z-w_r-1)}\frac{Z(z)}{W_r(z) W_r(z-1)} \frac{Z(w_r)}{W_r(w_r)} \sfu_r^{-1},$$
while the image of the right-hand side under $\Phi_\mu^\lambda$ is
$$ -\sum_{r=1}^a \left( \frac{z-y+1}{y- w_r} + \frac{-2}{z-w-1} \right) \frac{Z(w_r)}{W_r(w_r)} \sfu_r^{-1} \cdot \frac{Z(z)}{W(z) W(z-1)} $$
$$ = -\sum_{r=1}^a \left( \frac{z-y+1}{y- w_r} + \frac{-2}{z-w-1} \right) \frac{1}{(z-w_r + 1) (z- w_r)} \frac{Z(z)}{W_r(z) W_r(z-1)} \frac{Z(w_r)}{W_r(w_r)} \sfu_r^{-1}.$$
So, the relation follows from the equality
$$ \frac{z-y-1}{(y - w_r)(z - w_r)(z-w_r-1)} = \left( \frac{z-y+1}{y- w_r} + \frac{-2}{z-w-1} \right) \frac{1}{(z-w_r + 1) (z- w_r)}.$$

%%%%%%%
\subsection{Relation (\ref{eq: Y relation E and E}) between
\texorpdfstring{$E_i(z)$}{Ei(z)} and
\texorpdfstring{$E_j(y)$}{Ej(y)}}
We will verify that for all $i,j$ we have
$$ (z- y - a) E_i(z) E_j(y) + E_i(z) E_j^{(1)} - E_i^{(1)} E_j(y) $$
\begin{equation} \label{eq: E and E}
= (z-y+a) E_j(y) E_i(z) +  E_j^{(1)} E_i(z)- E_j(y) E_i^{(1)},
\end{equation}
where $a = \tfrac{1}{2}  a_{ij}$.

\subsubsection{The case \texorpdfstring{$a_{ij}=0$}{aij=0}:} We need to check that $[\Phi_\mu^\lambda(E_i(z)), \Phi_\mu^\lambda(E_j(y))]=0$, which is clear.

\subsubsection{The case \texorpdfstring{$a_{ij} = -1$}{aij=-1}:} In this case we can assume that $\II = \{i, j\}$, and by the symmetry of (\ref{eq: E and E}) we may also assume that $\Qo$ consists of a single arrow $i\rightarrow j$.  After collecting terms, the image of the left-hand side of (\ref{eq: E and E}) is
$$ \sum_{r=1}^{a_i} \sum_{s=1}^{a_j} \left( \frac{z-y+\tfrac{1}{2}}{(z-w_{i,r})(y-w_{j,s})} + \frac{1}{z-w_{i,r}}- \frac{1}{y-w_{j,s}}  \right) \frac{Z_i(w_{i,r})}{W_{i,r}(w_{i,r})} \sfu_{i,r}^{-1} \frac{Z_j(w_{j,s}) W_i(w_{j,s}-\tfrac{1}{2})}{W_{j,s}(w_{j,s})} \sfu_{j,s}^{-1} $$
$$ = \sum_{r,s} \frac{(w_{i,r}-w_{j,s}+\tfrac{1}{2})(w_{j,s} - w_{i,r}+\tfrac{1}{2}) }{(z-w_{i,r})(y-w_{j,s})} \frac{Z_i(w_{i,r}) Z_j(w_{j,s}) W_{i,r}(w_{j,s}-\tfrac{1}{2})}{W_{i,r}(w_{i,r}) W_{j,s}(w_{j,s})} \sfu_{i,r}^{-1} \sfu_{j,s}^{-1}.$$
The image of the right-hand side of (\ref{eq: E and E}) reduces to the same expression, so the relation holds.

\subsubsection{The case \texorpdfstring{$i = j$}{i=j}:}
Here we may assume that $\fg = \mathfrak{sl}_2$, and we will drop the index $i$ from our notation for this calculation.  In this case, the left-hand side of (\ref{eq: E and E}) is
$$ \sum_{r=1}^a \sum_{s=1}^a \left( (z-y-1) \frac{Z(w_r)}{(z-w_r) W_r(w_r)}\sfu_r^{-1} \frac{Z(w_s)}{(y-w_s) W_s(w_s)}\sfu_s^{-1} \right.  $$
$$\left. + \frac{Z(w_r)}{(z-w_r) W_r(w_r)}\sfu_r^{-1}\frac{Z(w_s)}{W_s(w_s) }\sfu_s^{-1} - \frac{Z(w_r)}{W_r(w_r)}\sfu_r^{-1} \frac{Z(w_s)}{(y-w_s) W_s(w_s)}\sfu_s^{-1} \right).$$
Collecting terms, we express this as two sums:
\begin{align*}&\sum_r  \left( \frac{z-y-1}{(z-w_r)(y-w_r+1)} +\frac{1}{z-w_r} - \frac{1}{y-w_r+1}\right) \frac{Z(w_r) Z(w_r-1)}{W_r(w_r) W_r(w_r-1)} \sfu_r^{-2} \\
&+ \sum_{r\neq s}\left(\frac{z-y-1}{(z-w_r)(y-w_s)} +\frac{1}{z-w_r}-  \frac{1}{y-w_s} \right) \frac{Z(w_r) Z(w_s)}{W_r(w_r) (w_s - w_r+1) W_{rs}(w_s)} \sfu_r^{-1} \sfu_s^{-1}.
\end{align*}
The term in brackets in the first sum is zero, while the second sum is
$$ - \sum_{r\neq s} \frac{1}{(z-w_r)(y-w_s)}\frac{Z(w_r) Z(w_s)}{W_r(w_r) W_{rs}(w_s)} \sfu_r^{-1} \sfu_s^{-1},$$
where $W_{rs}(z) = \sum_{t\neq r,s} (z-w_t)$.  We get the same expression for the right-hand side of (\ref{eq: E and E}), so this relation holds.

%%%%%%%
\subsection{Relation (\ref{eq: Y relation E and F}) between
\texorpdfstring{$E_i(z)$}{Ei(z)} and
\texorpdfstring{$F_j(y)$}{Fj(y)}}

\subsubsection{The case \texorpdfstring{$i\neq j$}{i≠j}:}
Here, we must check that $[\Phi^\lambda_\mu(E_i(z)), \Phi^\lambda_\mu(F_j(y))] =0$. We may assume that $\II = \{i, j\}$. The only interesting case is when $a_{ij} = -1$ and $j\rightarrow i$.  Then $[\Phi^\lambda_\mu(E_i(z)), \Phi^\lambda_\mu(F_j(y))]$ is equal to
$$ - \sum_{r=1}^{a_i}\sum_{s=1}^{a_j} \left[\frac{Z_i(w_{i,r})W_j(w_{i,r}-\tfrac{1}{2})}{(z-w_{i,r}) W_{i,r}(w_{i,r})} \sfu_{i,r}^{-1},  \frac{W_i(w_{j,s}+\tfrac{1}{2})}{(y-w_{j,s}-1)W_{j,s}(w_{j,s})} \sfu_{j,s} \right] $$
$$ =- \sum_{r=1}^{a_i}\sum_{s=1}^{a_j} \frac{Z_i(w_{i,r})}{(z-w_{i,r}) W_{i,r}(w_{i,r})}  \frac{1}{(y-w_{j,s}-1)W_{j,s}(w_{j,s})}\left[W_j(w_{i,r}-\tfrac{1}{2} ) \sfu_{i,r}^{-1}, W_i(w_{j,s}+\tfrac{1}{2}) \sfu_{j,s}\right].$$
This is indeed zero, as the commutator in each summand is zero.

\subsubsection{The case \texorpdfstring{$i = j$}{i=j}:}
The proof in this case is almost identical to that of \cite[Theorem~4.5]{kwy}. Recall that $(z-y)[\Phi_\mu^\lambda(E_i(z)), \Phi_\mu^\lambda(F_i(y))]$ is equal to
$$ (z-y ) \left[ -\sum_{r=1}^{a_i} \frac{Z_i(w_{i,r}) \prod_{\substack{h\in\Qo\\\vin{h}=i}} W_{\vout{h}}(w_{i,r} -\tfrac{1}{2})}{(z-w_{i,r}) W_{i,r}(w_{i,r})} \sfu_{i,r}^{-1},  \sum_{s=1}^{a_i}\frac{ \prod_{\substack{h\in\Qo\\ \vout{h}=i}} Z_{\vin{h}}(w_{i,s}+\tfrac{1}{2})}{(y-w_{i,s}-1) W_{i,s}(w_{i,s})} \sfu_{i,s} \right].$$
All terms where $r\neq s$ vanish, and what remains can be rewritten as
$$ \sum_{r=1}^{a_i}\big( (L_{i,r}(y) - R_{i,r}(y) ) - (L_{i,r}(z)-R_{i,r}(z)) \big),$$
where
\begin{align*}
L_{i,r}(y) & = \frac{Z_i(w_{i,r}+1) \prod_{\substack{h\in\Qo\sqcup\overline\Qo\\ \vout{h}= i}} W_{\vin{h}}(w_{i,r}+\tfrac{1}{2})}{(y - w_{i,r}-1) W_{i,r}(w_{i,r}+1) W_{i,r}(w_{i,r})},\\
R_{i,r}(y) & = \frac{Z_i(w_{i,r}) \prod_{\substack{h\in\Qo\sqcup\overline\Qo\\ \vout{h}= i}} W_{\vin{h}}(w_{i,r}-\tfrac{1}{2})}{(y-w_{i,r}) W_{i,r}(w_{i,r}) W_{i,r}(w_{i,r}-1)}.
\end{align*}
Therefore, it remains to verify that
$$ \sum_{r=1}^{a_i}( L_{i,r}(y) - R_{i,r}(y) ) =  H_{i,+}(y).$$
As in \cite{kwy}, this is done by comparing coefficients at all $y^{-k}$ for $k>0$ between the left-hand side and $ H_i(y)$, using partial fractions to compute the case of $ H_i(y)$.

%%%%%%%
\subsection{The Serre relations}
When $ a_{ij} = 0 $, the relation is immediate, so we concentrate on the $ a_{ij} = -1 $ case and in particular, the version with $ E$s, see (\ref{eq: Y relation Serre}) above.  The proof of this relation is sketched out in \cite{GKLO}.  Following their notation, let us denote
$$ \chi_{i,r} = - \frac{Z_i(w_{i,r}) \prod_{h\in Q_1: \vin{h}=i} W_{\vout{h}}(w_{i,r}-\tfrac{1}{2}) }{W_{i,r}(w_{i,r})} \sfu_{i,r}^{-1},$$
so that $\Phi_\mu^\lambda(E_i(y)) = \sum_{r=1}^{a_i} \frac{1}{y-w_{i,r}} \chi_{i,r}$.

These elements satisfy the relations $[\chi_{i,r}, w_{i,s}] = - \delta_{r,s} \chi_{i,r}$ and
\begin{align*}
(w_{i,r} - w_{i,s} -1) \chi_{i,r} \chi_{i,s} &= (w_{i,r} - w_{i,s} +1) \chi_{i,s} \chi_{i,r}, \qquad \text{ for } r\neq s,\\
(w_{i,r} - w_{j,t} +\tfrac{1}{2}) \chi_{i,r} \chi_{j,t} &= (w_{i,r} - w_{j,t}-\tfrac{1}{2}) \chi_{j,t} \chi_{i,r}.
\end{align*}
Using the above relations, we find that
$$[\Phi_\mu^\lambda(E_i(y_1)), [ \Phi_\mu^\lambda(E_i(y_2)), \Phi_\mu^\lambda(E_j(z))]]$$
\begin{align*}
&= \left[ \sum_{r=1}^{a_i} \frac{1}{y_1-w_{i,r}} \chi_{i,r}, \sum_{s=1}^{a_i}\sum_{t=1}^{a_j} \frac{1}{(y_2-w_{i,s})(z-w_{j,t})} \frac{-1}{w_{i,s}-w_{j,t}-\tfrac{1}{2}} \chi_{i,s} \chi_{j,t} \right] \\
&= \sum_r\sum_t 
  \begin{multlined}[t]
\left( \frac{1}{(y_1-w_{i,r})(y_2-w_{i,r}+1)} - \frac{1}{(y_1-w_{i,r}+1)(y_2-w_{i,r})}\right)
\\
\qquad
\times
\frac{1}{z-w_{j,t}} \frac{-1}{w_{i,r}-w_{j,t}-\tfrac{3}{2}} \chi_{i,r}^2 \chi_{j,t}      
\end{multlined}
 \\
&+ \sum_{r\neq s} 
  \begin{multlined}[t]
\sum_t\frac{1}{(y_1-w_{i,r})(y_2-w_{i,s})(z-w_{j,t})}\frac{-1}{(w_{i,s} - w_{j,t}-\tfrac{1}{2})}
\\
\qquad\times \frac{w_{i,r}+w_{i,s}-2w_{j,t}}{(w_{i,r}-w_{j,t}-\tfrac12)(w_{i,r}-w_{i,s}+1)} \chi_{i,r} \chi_{i,s} \chi_{j,t}.
  \end{multlined}
\end{align*}
The first sum is clearly skew-symmetric in $y_1, y_2$.  The second sum is as well, which one can see by applying the above relation between $\chi_{i,r}$ and $\chi_{i,s}$. This proves the Serre relation along with the theorem.

\subsection{The filtration}
We are left to verify the claim that the filtrations match.  To do this, it suffices to check that each PBW variable $ E_\beta^{(p)}, H_i^{(q)}, F_\beta^{(p)} $ (see \cite[Remark 3.4]{fkprw} for their definition) is sent to the correct filtered degree.  When $ \beta $ is a simple root, this is immediate.

Now suppose that $ \beta $ is not a simple root.  Then $E_\beta^{(p)} $ is defined by commutators.  Since $ \tilde \cA $ is almost commutative, this immediately implies that $ E_\beta^{(p)} $ is mapped into the correct filtered piece.

\end{proof}

Applying the Rees functor to Theorem~\ref{th:filGKLO}, we deduce the following result. 
\begin{Corollary} \label{cor:GKLO}
 There exists a unique graded $ \BC[\hbar,z_1, \dots, z_N]$-algebra homomorphism $ \bY_\mu \rightarrow \tilde \cA_\hbar $, such that
\begin{align*}
A_i(z) & \mapsto z^{-a_i} W_i(z), \\
E_i(z) & \mapsto -\sum_{r=1}^{a_i}  \frac{Z_i(w_{i,r}) \prod\limits_{h\in\Qo:\vin{h}=i} W_{\vout{h}}(w_{i,r} - \tfrac{1}{2}\hbar)}{(z-w_{i,r}) W_{i,r}(w_{i,r})} \sfu_{i,r}^{-1}, \\
F_i(z) & \mapsto \sum_{r=1}^{a_i} \frac{\prod\limits_{h\in\Qo:\vout{h}=i} W_{\vin{h}}(w_{i,r} + \tfrac{1}{2} \hbar)}{(z-w_{i,r} - \hbar) W_{i,r}(w_{i,r})} \sfu_{i,r}.
\end{align*}
\end{Corollary}

In the above corollary, we are using a slight abuse of notation.  For a
generator $ x $ (such as $ E_i^{(p)} $ or $ w_{i,r} $) of the algebra $ Y_\mu $
or $\tilde\cA$ which lives in filtered degree $ k $ (but not in filtered
degree $ k-1 $) we write $ x $ for the element $ \hbar^k x \in \Rees Y_\mu $
or $ \Rees\tilde\cA$.

\subsection{Relation to quantization of Coulomb branch}
Recall the setup of~\ref{Qgt}: we have $\cAh = 
H^{(\GV\times \TW)_\cO\rtimes\CC^\times}_*(\cR_{\GV,\bN})\hookrightarrow\tilde\cAh$,
the quantized Coulomb branch 
associated to the pair $(\GL(V),\bN)$ with flavor symmetry.  This inclusion takes the homological grading on $ \cAh $ (not the $ \Delta$-grading) to the above grading on $ \tilde \cAh$.
%and we have the embedding 
%$\bz^*(\iota_*)^{-1}\colon \cAh\hookrightarrow \widetilde{\cA}_\hbar$. 
%We also have a flavor deformation $\cA_\hbar^F:=
%\cAh(\tilde{G},\bN)^{T(W)^\vee}$
%where $\tilde{G}=\GL(V)\times T(W)$. We choose a basis $z_1,\ldots,z_N$
%of the character lattice of $T(W)$ as in~\ref{defo}. We obtain a 
%subalgebra $H^*_{T(W)}(\on{pt})=\BC[z_1,\ldots,z_N]\subset\cA_\hbar^F$.
%We have an embedding $\bz^*(\iota_*)^{-1}\colon 
%\cA_\hbar^F\hookrightarrow \widetilde{\cA}_\hbar[z_1,\ldots,z_N]$.

\begin{Theorem} \label{th:PhiBar}
There exists a unique graded $ \BC[\hbar, z_1, \dots, z_N] $-algebra homomorphism
$$ \overline{\Phi}_\mu^\lambda\colon \bY_\mu[z_1, \ldots, z_N] \rightarrow 
\cA_\hbar,$$ such that
%\begin{NB} Old version:
%\begin{align*}
%A_i^{(p)} &\mapsto (-1)^p e_p(\{ w_{i,r} \}), \\
%E_i^{(p)} &\mapsto  (-1)^{\sum_{j \leftarrow i} a_j} c_1(\uS_i)^{p-1}\cap[\cR_{\varpi_{i,1}}], \\
%F_i^{(p)} &\mapsto (-1)^{a_i + 1} c_1(\uQ_i)^{p-1} \cap [\cR_{\varpi^*_{i,1}}].
%\end{align*}
%\end{NB}%
\begin{align*}
A_i^{(p)} &\mapsto (-1)^p e_p(\{ w_{i,r} \}), \\
F_i^{(p)} &\mapsto  (-1)^{\sum_{\vout{h}=i} a_{\vin{h}}} (c_1(\uQ_i)+\hbar)^{p-1}\cap[\cR_{\varpi_{i,1}}], \\
E_i^{(p)} &\mapsto (-1)^{a_i} (c_1(\uS_i)+\hbar)^{p-1} \cap [\cR_{\varpi^*_{i,1}}].
\end{align*}
%where $\uS'_i$ is $\uS_i$ tensored with the trivial line bundle equipped
%with the action of the loop rotation $\BC^\times$ via the character with
%differential $-\hbar$, and
%$\uQ'_i$ is $\uQ_i$ tensored with the trivial line bundle equipped
%with the action of the loop rotation $\BC^\times$ via the character with
%differential $\hbar$.
%This gives rise to an inclusion 
%$Y^\lambda_\mu \hookrightarrow \cA_\hbar$.
\end{Theorem}

\begin{Remark}
This homomorphism is analogous to (and was inspired by) the action of the Yangian of $\mathfrak{gl}_n $ on the cohomology of Laumon spaces, constructed by Feigin-Finkelberg-Negut-Rybnikov \cite{FFNR}.
\end{Remark}

\begin{Remark}\label{rem:Delta-grading}
In the above Theorem, we use the $ (\mu_1, \mu_2) $-grading on $ \bY_\mu[z_1, \ldots, z_N] $ (where $ \mu_1, \mu_2 $ are defined above) and the homological grading on $ \cA_\hbar $.  On the other hand, if we want to use the $\Delta$-grading on $ \cAh $ (as defined in~\ref{discrepancy}(2)), then we should use the $ (\mu/2, \mu/2) $-grading on $ \bY_\mu[z_1, \ldots, z_N] $.

Recall that the $(\mu/2,\mu/2)$-grading is defined so that PBW variables
$E_\beta^{(p)}$, $F_\beta^{(p)}$, $H_i^{(q)}$  have degree
\begin{equation*}
    \deg E_\beta^{(p)} = \frac12 \langle\mu,\beta\rangle + p,\quad
    \deg F_\beta^{(p)} = \frac12 \langle\mu,\beta\rangle + p,\quad
    \deg H_i^{(q)} = \langle\mu,\alpha_i\rangle + q,
\end{equation*}
where $\beta$ is a positive root. See \cite[section
5.4]{fkprw}. Therefore $\bY_\mu[z_1, \ldots, z_N]$ is
$\ZZ_{\ge 0}$-graded and the degree $0$ part consists only of scalars
(with respect to the $(\mu/2,\mu/2)$-grading) if and only if
$\langle\mu,\beta\rangle \ge -1$ for any positive root $\beta$.
Note that $\cA$ is called \emph{good or ugly} when its
$\Delta$-grading satisfies the same property. See \cite[Remark
4.2]{2015arXiv150303676N}. One of the authors show that
$\langle\mu,\beta\rangle \ge -1$ if $\cA$ is good or ugly. See
\cite[Proof of Prop.~5.9]{2015arXiv151003908N}. The converse is also
true if $\overline{\Phi}_\mu^\lambda$ is surjective.
\end{Remark}

\begin{proof}
We have the graded $ \BC[\hbar, z_1, \dots, z_N] $-algebra homomorphisms $ \Phi_\mu^\lambda : \bY_\mu[z_1, \ldots, z_N] \rightarrow 
\cA_\hbar $ and $  \bz^* (\iota_*)^{-1} : \cAh \rightarrow \tilde\cAh $, the second of which is injective.  So we just need to verify that the image $ \Phi_\mu^\lambda $ is contained in the image of $ \bz^* (\iota_*)^{-1} $.

It follows immediately from equations (\ref{eq:82}) and (\ref{eq:83}) that $ \bz^* (\iota_*)^{-1}(\overline{\Phi}_\mu^\lambda(X_i^{(s)})) = \Phi_\mu^\lambda(X_i^{(s)}) $ for $ X = A, E, F $.  Now, the elements $ A_i^{(s)}, E_i^{(1)}, F_i^{(1)} $ generate $ \bY_\mu \otimes \BC[z_1, \dots, z_N] $ as a $ \BC[\hbar,z_1, \dots, z_N]$ Poisson algebra (where the Poisson bracket is $ \{ a, b \} = \tfrac{1}{\hbar}(ab - ba) $).   Since $\cAh$ is almost commutative, $\cAh$ is closed under the Poisson bracket and so the image of $ \Phi_\mu^\lambda$ is contained in $ \bz^* (\iota_*)^{-1}(\cAh) $.
\end{proof}

The image of $ \overline{\Phi}^\lambda_\mu $ is called the truncated shifted Yangian and is denoted $\bY^\lambda_\mu $.

\begin{Remark}
It is easy to see that the elements $ A_i^{(p)} $ for $ p > a_i $  are sent to 0 under $\overline{\Phi}^\lambda_\mu$.  We conjecture that these elements generate the kernel of $ \overline{\Phi}^\lambda_\mu $ and thus we get a presentation of $Y^\lambda_\mu $ (the $ \hbar = 1 $ specialization of $ \bY^\lambda_\mu $).
\end{Remark}

\subsection{Specialization to the dominant case}
Now, let us assume that $ \mu $ is dominant.

\subsubsection{The scheme \texorpdfstring{$ \CG_\mu$}{Gμ}}

Consider the scheme $ \CW_{\mu} $ defined as the locus 
$ G_1[[z^{-1}]] z^{\mu} \subset G((z^{-1}))/G[z]$.  It is the moduli space of 
pairs $ (\scP, \sigma) $ where  $\scP$ is a $ G$-bundle on $\BP^1$ of 
isomorphism type $ \mu $ and $ \sigma $ is a trivialization in the formal 
neighbourhood of $ \infty $, such that $ \scP $ has isomorphism type $ \mu $ 
and such that the Harder-Narasimhan flag of $ \scP $ at $ \infty $ is 
compatible with $ B_- \subset G $ (under $\sigma$).  For any $ \unl\lambda $, 
we have a morphism $ \oW^{\unl\lambda}_{\mu} \rightarrow \CW_{\mu} $ and a closed 
embedding $ \oW^{\unl\lambda}_{\mu} \hookrightarrow \CW_{\mu} \times \BA^N $.

For any $ \unl\lambda $ and any point $\unl{z}\in \BA^N $, let
$ \oW^{\unl\lambda, \unl{z}}_{\mu} $ be the fibre of
$ \oW^{\unl\lambda}_{\mu} \rightarrow \BA^N $ over the point $\unl{z}$.
The open locus $ \CW^{\unl\lambda, \unl{z}}_{\mu} $ embeds into $\CW_{\mu}$ as the
intersection $\CW_{\mu}\cap G[z] z^{\unl\lambda,\unl{z}} \subset G((z^{-1}))/G[z]$,
where $ z^{\unl\lambda,\unl{z}} = \prod_{s=1}^N (z-z_s)^{\omega_{i_s}} $.

In \cite{kwy}, we constructed a Poisson structure on $ \CW_{\mu} $.  Now~\cite[Theorem~2.5]{kwy} generalizes immediately to show that
$ \CW^{\unl\lambda,\unl{z}}_{\mu} $ is a symplectic leaf of $ \CW_{\mu} $.
(In the case when $ G = SL_n$, this is closely related
to~\cite[Theorem~2.2]{arXiv1405.3909}).

Now, consider the subgroup of $ G_1[[z^{-1}]] $ defined as
$$
\CG_\mu = \{ g \in G_1[[z^{-1}]]\mid z^{-\mu} g z^\mu \in G_1[[z^{-1}]] \}.
$$

The natural map $ g \mapsto gz^\mu $ provides an isomorphism $ \CG_\mu \cong  \CW_\mu$.

The following result is~\cite[Theorem~3.12]{kwy}.

\begin{Theorem}
\label{previous}
There is an isomorphism of Poisson algebras $ \Psi\colon \bY_\mu / \hbar \bY_\mu \rightarrow \BC[\CG_{\mu^*}] $ given by
\begin{align*}
H_i(z) &\mapsto z^{\langle \mu, \alpha_i \rangle} \prod_{\substack{h\in\Qo\sqcup\overline\Qo\\ \vout{h}= i}} \Delta_{w_0 \omega^\vee_{\vin{h}}, w_0 \omega^\vee_{\vin{h}}}(z) \Delta_{w_0\omega^\vee_i, w_0\omega^\vee_i}(z)^{-2}, \\
F_i(z) &\mapsto \Delta_{w_0s_i \omega^\vee_i, w_0\omega^\vee_i}(z) \Delta_{w_0\omega^\vee_i, w_0\omega^\vee_i}(z)^{-1}, \\
E_i(z) &\mapsto  z^{\langle \mu, \alpha_i \rangle} \Delta_{w_0\omega^\vee_i, w_0s_i\omega^\vee_i}(z) \Delta_{w_0\omega^\vee_i, w_0\omega^\vee_i}(z)^{-1}.
\end{align*}
\end{Theorem}

Here $ \Delta_{w_0\omega^\vee_i, w_0\omega^\vee_i} $, etc. are generalized minors
(see~\cite[Section~2]{kwy} for more explanation) and we define $ \Delta_{w_0\omega^\vee_i, w_0\omega^\vee_i}(z) \in \BC[\CG_{\mu^*}]((z^{-1})) $ by
$$
\Delta_{w_0\omega^\vee_i, w_0\omega^\vee_i}(z)(g) =
\Delta_{w_0\omega^\vee_i, w_0\omega^\vee_i}(g).
$$

\subsubsection{Involutions}
Let $ G \rightarrow G $, $ g \mapsto g^t $ denote the transpose involution (it is an antiautomorphism which corresponds to the Lie algebra antiautomorphism given $ E_i \mapsto F_i, F_i \mapsto E_i, H_i \mapsto H_i $).  Also, let
$ \varkappa_{-1}\colon G_1[[z^{-1}]] \rightarrow G_1[[z^{-1}]]$ be the involution given by $ z \mapsto -z $.

If $ g \in \CG_{\mu*} $, then $ z^{-\mu^*} g z^{\mu^*} \in G_1[[z^{-1}]] $ and so $ (z^{-\mu^*} g z^{\mu^*})^t = z^{\mu^*} g^t z^{-\mu^*} \in \CG_{\mu^*} $.

We define an involution $ \mathbf i \colon \CG_{\mu^*} \rightarrow \CG_{\mu^*} $ by $ \mathbf i(g) = z^{-\mu^*} \varkappa_{-1}(g^{t}) z^{\mu^*} $.  We can extend $ \mathbf i $ to $ \CG_{\mu^*} \times \mathbf \BA^N $ by acting by multiplication by $ -1 $ on the second factor.

Following \cref{Joel}, we consider an involution
$ \mathbf i\colon \oW^{\unl\lambda^*}_{\mu^*} \rightarrow \oW^{\unl\lambda^*}_{\mu^*} $ as the composition
of $ \iota^\lambda_\mu $ and $ \varkappa_{-1} $ and the action of $ \beta(-1) $, where $ \beta$ is the coweight defined by
$$
\beta = \sum_i (a_i - \sum_{h\in\Qo:\vout{h}=i} a_{\vin{h}}) \omega_i.
$$

Let us write $ \mho\colon \oW^{\unl\lambda^*}_{\mu^*} \rightarrow \CW_{\mu^*} \times \BA^N \cong \CG_{\mu^*} \times \BA^N $ for the natural composition.  The following result is immediate.

\begin{Lemma} \label{lem:Inv1}
Up to $ \beta(-1) $, the involutions are compatible with $ \mho $.  More precisely,
$$ \mho \circ a(\beta(-1)) \circ \mathbf i = \mathbf i \circ \mho.$$
\end{Lemma}

We also can define an involution on $ \mathbf i\colon \bY_\mu \rightarrow \bY_\mu $ by
$$
E_i^{(p)} \mapsto (-1)^p F_i^{(p)}, \ H_i^{(p)} \mapsto (-1)^{p+\langle \mu, \alpha_i \rangle} H_i^{(p)}, \ F_i^{(p)} \mapsto (-1)^{p+ \langle \mu, \alpha_i \rangle } E_i^{(p)}.
$$

Above we defined the map $ \Psi\colon \bY_\mu /\hbar \bY_\mu \rightarrow \BC[\CG_{\mu^*}] $.
A simple computation shows the following result.

\begin{Lemma} \label{lem:Inv2}
The involutions are compatible with $ \Psi $.  More precisely,
$$ \Psi \circ \mathbf i = \mathbf i \circ \Psi.$$
\end{Lemma}

Finally, we also have the involution
$\fri^{\unl\lambda}_{\mu*}\colon \cA_0 \rightarrow \cA_0 $ (where $ \cA_0 = \cAh/\hbar \cAh$)
defined as in \cref{Joel}, which comes from the isomorphism of varieties
$\fri^{\unl\lambda}_\mu\colon \cR_{\GL(V),\bN^\lambda_\mu} \iso\cR_{\GL(V^*),\bN^\lambda_\mu}$.
Note that the
$ \bN^\lambda_\mu $ on the right hand side is computed with respect to the opposite orientation.

In \cref{th:PhiBar}, we defined a homomorphism
$ \overline{\Phi}\colon \bY_\mu[z_1,\dots, z_N] \rightarrow \cAh$ and thus a homomorphism $ \bY_\mu[z_1,\dots, z_N]/\hbar \bY_\mu[z_1, \dots, z_N] \rightarrow \cA_0 $.  We extend the involution $ \mathbf i $ from $ \bY_\mu $ to $ \bY_\mu[z_1, \dots, z_N] $ by setting $ \mathbf i(z_k) = -z_k $ for all $k $.

\begin{Lemma} \label{lem:Inv3}
Up to $ \beta(-1) $, the involutions are compatible with $ \overline{\Phi} $.  More precisely,
$$ \overline{\Phi} \circ  \mathbf i = a(\beta(-1)) \circ
\fri^{\unl\lambda}_{\mu*} \circ \overline{\Phi}$$
as maps $ \bY_\mu[z_1,\dots, z_N]/\hbar \bY_\mu[z_1, \dots, z_N] \rightarrow \cA_0 $.
\end{Lemma}
\begin{proof}
Note that we have $\fri^{\unl\lambda}_{\mu*}([\cR_{\varpi_{i,1}}]) = [\cR_{\varpi^*_{i,1}}]$
and $\fri^{\unl\lambda}_{\mu*}([\cR_{\varpi^*_{i,1}}]) = [\cR_{\varpi_{i,1}}]$.
Also $\fri^{\unl\lambda}_{\mu*}(c_1(\uS_i)) = - c_1(\uQ_i) $ since under the
isomorphism $\fri^{\unl\lambda}_\mu\colon \CR_{\GL(V),\bN^\lambda_\mu}
\iso\CR_{\GL(V^*),\bN^\lambda_\mu}$
from \cref{Joel}, we have that $\fri^{\unl\lambda}_{\mu*} (\uS_i) = \uQ_i^*$.
Finally, we have that $\fri^{\unl\lambda}_{\mu*}(w_{i,r}) = -w_{i,r} $.

Hence examining the formulas for $ \overline{\Phi} $ given
in \cref{th:PhiBar}, the result follows.
\end{proof}

\begin{Remark}
The involution $\fri^{\unl\lambda}_{\mu*}\colon \cA_0 \rightarrow \cA_0 $ extends
to an involution $\fri^{\unl\lambda}_{\mu*}\colon \cAh \rightarrow \cAh$.
However, it is easy to see that the map
$ \overline{\Phi}\colon \bY_\mu[z_1, \dots, z_N] \rightarrow \cAh $ is not compatible with this involution (not even up to sign).  It is possible to modify the involution of $ \bY_\mu $ to make it compatible up to sign, but it will be given by a bit more complicated formulae (for example $ E_i(z) \mapsto -F_i(-z+\hbar) $).  However, we will not need compatibility at the non-commutative level in this paper.
\end{Remark}

\subsubsection{Commutativity}
We have a surjection
$$ \mho \circ \Psi\colon \bY_\mu[z_1, \dots, z_N] / \hbar \bY_\mu [z_1, \dots, z_N] \iso \BC[\CG_{\mu^*} \times \BA^N] \rightarrow \BC[\oW^{\unl\lambda^*}_{\mu^*}].$$

Recall that in the previous section, we constructed a map $ \overline{\Phi}\colon \bY_\mu[z_1, \dots, z_N] \rightarrow \cAh$.  On the other hand,
in \cref{Coulomb_defo}, we have constructed an isomorphism $ \Xi\colon \BC[\oW^{\unl\lambda^*}_{\mu^*}]\iso \cAh / \hbar \cAh$.

\begin{Lemma}
The composition $ \Xi^{-1} \circ \overline{\Phi} $ equals $ \mho \circ \Psi $ as morphisms $$ \bY_\mu[z_1, \dots, z_N] / \hbar \bY_\mu[z_1, \dots, z_N] \rightarrow \BC[\oW^{\unl\lambda^*}_{\mu^*}].$$
\end{Lemma}

\begin{proof}
Since all the morphisms involved are Poisson $\BC[z_1, \dots, z_N]$-algebra morphisms, it suffices to check the statement on the generators $ E_i^{(s)}, A_i^{(s)}, F_i^{(s)} $ of $ \bY_\mu $.

Recall that from~\subsecref{subsec:bd-slices} we have the morphism
$ s^{\unl\lambda^*}_{\mu^*} \colon \oW^{\unl\lambda^*}_{\mu^*} \rightarrow Z^\alpha \times \BA^N $, where $ \alpha = \lambda-\mu $.  Given a point $ ([gz^{\mu^*}], (z_1, \dots, z_N)) \in \oW^{\unl\lambda^*}_{\mu^*} $, the corresponding principal $ G$-bundle $\scP $ has associated vector bundle $ \CV^{\lambda^\vee}_{\scP} = gz^{\mu^*} (V^{\lambda^\vee} \otimes \CO_{\BP^1}) $ and invertible subsheaf $ \CL_{\lambda^\vee} = gz^{\mu^*}(V^{\lambda^\vee}_{w_0 \lambda^\vee}\otimes \CO_{\BP^1}) $.  Thus the image of $   ([gz^{\mu^*}], (z_1, \dots, z_N)) $ under $ s^{\unl\lambda^*}_{\mu^*} $ gives the collection of invertible subsheaves
$$ gz^{\mu^*}(V^{\lambda^\vee}_{w_0 \lambda^\vee}\otimes \CO_{\BP^1})\big(\sum_{s=1}^N \langle w_0 \omega_{i^*_s}, \lambda^\vee \rangle \cdot z_s\bigr) \subset V^{\lambda^\vee} \otimes \CO_{\BP^1}.
$$
Now, we specialize to $ \lambda^\vee = \omega_i^\vee $.  Then the invertible subsheaf is generated over $ \CO_{\BP^1}$ by $$ (\prod_{s=1}^N (z-z_s)^{-\langle w_0 \omega_{i^*_s}, \omega_i^\vee\rangle})gz^{\mu^*}(v_{w_0 \omega_i}) =  Q_i(z) v_{w_0 \omega_i^\vee} + P_i(z) v_{w_0 s_i \omega_i^\vee} + \cdots $$ where
$$
Q_i(z) = \Delta_{w_0\omega^\vee_i, w_0\omega^\vee_i}(\prod_{s=1}^N (z-z_s)^{-\langle w_0 \omega_{i^*_s}, \omega_i^\vee\rangle} g z^{\mu^*}) = z^{-\langle \mu, \omega_i^\vee\rangle}   \prod_{s=1}^N (z-z_s)^{\langle \omega_{i_s}, \omega_i^\vee \rangle}  \Delta_{w_0\omega^\vee_i, w_0\omega^\vee_i}(g),
$$
and
$$
P_i(z) = \Delta_{w_0s_i\omega^\vee_i, w_0\omega^\vee_i}(\prod_{s=1}^N (z-z_s)^{-\langle w_0 \omega_{i^*_s}, \omega_i^\vee\rangle} g z^{\mu^*}) = z^{-\langle \mu, \omega_i^\vee\rangle}   \prod_{s=1}^N (z-z_s)^{\langle \omega_{i_s}, \omega_i^\vee \rangle}  \Delta_{w_0s_i\omega^\vee_i, w_0\omega^\vee_i}(g).
$$
By definition (see~\cite[2.2]{bdf}), $ Q_i(z), P_i(z) $ are related to the coordinates $ (w_{i,r}, y_{i,r}) $ by $ Q_i(w_{i,r}) = 0, P_i(w_{i,r}) = y_{i,r} $.
% This follows from the commutativity
%of the diagram in~\cite[Theorem~2.8.(1)]{bf14}.

Now using the definition of $ \Psi(H_i(z)) $ given in \cref{previous}
and the definition of $ A_i(z) $ given in \eqref{eq: H and A}, we deduce that
$$
\Psi(A_i(z)) = z^{-a_i} \prod_{s=1}^N (z-z_s)^{\langle \omega_{i_s}, \omega_i^\vee \rangle} z^{-\langle \mu, \omega^\vee_i \rangle} \Delta_{w_0\omega^\vee_i, w_0\omega^\vee_i}(z),
$$
and so $ \Psi(A_i(z)) = z^{-a_i}Q_i(z) $ which agrees with $\overline{\Phi} $.

Next, we consider $F_i(z)$.  First, we have that $$ \Psi(F_i(z)) = \Delta_{w_0s_i \omega^\vee_i, w_0\omega^\vee_i}(z) \Delta_{w_0\omega^\vee_i, w_0\omega^\vee_i}(z)^{-1} = \frac{P_i(z)}{Q_i(z)}$$
by the above analysis.

We also have that
$$
\Phi_\mu^\lambda(F_i(z)) =  \sum_{r=1}^{a_i}
    \frac{\displaystyle\prod_{h\in\Qo:\vout{h}=i}
      \prod_{s=1}^{a_j}
      (w_{i,r} - w_{\vin{h},s})}
    {(z - w_{i,r}) \displaystyle\prod_{s \ne r} (w_{i,r} - w_{i,s})}
     \sfu_{i,r}.
$$
On the other hand, we have $ \Xi^{-1}(\displaystyle\prod_{h\in\Qo:\vout{h}= i}\displaystyle\prod_{s=1}^{a_j} (-w_{i,r} + w_{\vin{h},s}) \sfu_{i,r} ) = y_{i,r} $
(see \cref{Co_bra,quivar}) and thus
$$
\Xi^{-1}(\overline{\Phi}(F_i(z)) =  \sum_{r=1}^{a_i}
    \frac{ y_{i,r} }{(z-w_{i,r}) \displaystyle\prod_{s \ne r} (w_{i,r} - w_{i,s})}  = \frac{P_i(z)}{Q_i(z)},
$$
where the last equality is obtained by Lagrange interpolation.

Thus, the statement holds for $ F_i^{(p)}$.

Finally, we wish to show that $ \Xi^{-1} \circ \overline{\Phi}(E_i^{(p)}) = \mho \circ \Psi(E_i^{(p)})$.  It suffices to prove that this equation holds after applying $ \mathbf i $.

Applying \cref{Joel} and \cref{lem:Inv3}, we deduce that
$$
\mathbf i(\Xi^{-1} \circ \overline{\Phi}(E_i^{(p)})) = (-1)^{b_i}\Xi^{-1} \circ \overline{\Phi}(\mathbf i(E_i^{(p)})),
$$
where $ b_i = a_i - \sum_{i \rightarrow j} a_j$.
However $ \mathbf i(E_i^{(p)}) = (-1)^p F_i^{(p)}$ and above we proved that $\Xi^{-1} \circ \overline{\Phi}( F_i^{(p)}) = \mho \circ \Psi(F_i^{(p)}) $.  Thus we conclude that
$$
\mathbf i(\Xi^{-1} \circ \overline{\Phi}(E_i^{(p)})) = (-1)^{b_i}\mho \circ \Psi(\mathbf i(E_i^{(p)})).
$$

Now, applying Lemmas \ref{lem:Inv1} and \ref{lem:Inv2}, we deduce that
$$ \mho \circ \Psi(\mathbf i(E_i^{(p)})) = (-1)^{b_i} \mathbf i( \mho \circ \Psi(E_i^{(p)})).
$$

Thus, we conclude that $$ \mathbf i( \Xi^{-1} \circ \overline{\Phi}(E_i^{(p)})) = \mathbf i (\mho\circ\Psi(E_i^{(p)})),$$
and hence $ \Xi^{-1} \circ \overline{\Phi}(E_i^{(p)}) = \mho\circ\Psi(E_i^{(p)})$ as desired.
\end{proof}

\begin{Corollary}\label{cor:shiftedYangian}
We have an equality $ \bY^\lambda_\mu = \cAh$ and in particular, we have an isomorphism $ \bY^\lambda_\mu / \hbar \bY^\lambda_\mu \cong \BC[\oW^{\unl\lambda^*}_{\mu^*}]$.
\end{Corollary}

\begin{proof}
The above theorem shows that the inclusion $ \bY^\lambda_\mu \hookrightarrow \cAh $ gives as isomorphism $ \bY^\lambda_\mu / \hbar \bY^\lambda_\mu \cong \cAh / h \cAh $.

Thus each element of $ \cAh $ admits a lift modulo $ \hbar $ to $ \bY^\lambda_\mu $.  Since $\bY^\lambda_\mu, \cAh $ are graded and the grading is bounded below, this proves the equality.
\end{proof}

\begin{Remark}
The isomorphism $ \bY^\lambda_\mu / \hbar \bY^\lambda_\mu \cong \BC[\oW^{\unl\lambda^*}_{\mu^*}] $ was conjectured in \cite{kwy}.  More precisely, in \cite{kwy}, we proved that the map $ \Psi $ descended to a surjection $ \bY^\lambda_\mu / (\hbar, z_1, \dots, z_N) \rightarrow \BC[\oW^{\lambda^*}_{\mu^*}] $ which was an isomorphism modulo nilpotents.  The above Corollary shows that this map is an isomorphism.  For other points $\unl{z}\in \BA^N$, this also proves that the corresponding quotient in $ \bY^\lambda_\mu $ is isomorphic to the corresponding fibre of $ \oW^{\unl\lambda^*}_{\mu^*} $.  In \cite{kwy}, we made a mistake on this point (we stated that this would always quantize the central fibre).
\end{Remark}

\begin{Remark}
If we take $ \mu^* $ not dominant, then some of the results of this section continue to hold.  In this case, we defined a version of $ \CW_{\mu^*}$ and we directly constructed the isomorphism $ \bY / \hbar \bY \cong \BC[\CW_{\mu^*}] $ in \cite[Theorem 5.15]{fkprw} .  However, we do not know how to see that $ \CW_{\mu^*} $ has an intrinsic Poisson structure, nor have we proven the surjectivity of $ \bY_\mu \rightarrow \cAh $ in this situation.
\end{Remark}

%\end{document}

%--------------------
% \begin{thebibliography}{E-G-S}
% \bibitem[BDG]{BDG}

% \bibitem[KWWY]{KWWY}
% \end{thebibliography}

%%% Local Variables:
%%% mode: latex
%%% TeX-master: "blowup_pre"
%%% End:

\makeatletter
\renewcommand{\cA}[1][{}]{%
  \@ifmtarg{#1}%
  {\mathcal A}% if #1 is empty
  {\mathcal A(#1)}% if #1 is not empty
}
\makeatother

\begin{NB}
    Sections for the affine Yangian of $\gl_1$ and Cherednik algebras
    are comment out as they are not completed yet.

\renewenvironment{NB}{
\color{blue}{\bf NB2}. \footnotesize
}{}
\renewenvironment{NB2}{
\color{purple}{\bf NB3}. \footnotesize
}{}
\end{NB}

\bibliographystyle{myamsalpha}
\bibliography{nakajima,mybib,coulomb}

\newcommand{\etalchar}[1]{$^{#1}$}
\def\cprime{$'$} \def\cprime{$'$} \def\cprime{$'$} \def\cprime{$'$}
  \def\cprime{$'$}
  \providecommand{\noopsort}[1]{}\def\cftil#1{\ifmmode\setbox7\hbox{$\accent"5E#1$}\else
  \setbox7\hbox{\accent"5E#1}\penalty 10000\relax\fi\raise 1\ht7
  \hbox{\lower1.15ex\hbox to 1\wd7{\hss\accent"7E\hss}}\penalty 10000
  \hskip-1\wd7\penalty 10000\box7}
\providecommand{\bysame}{\leavevmode\hbox to3em{\hrulefill}\thinspace}
\providecommand{\MR}{\relax\ifhmode\unskip\space\fi MR }
\providecommand{\MRhref}[2]{%
  \href{http://www.ams.org/mathscinet-getitem?mr=#1}{#2}
}
\providecommand{\href}[2]{#2}
\begin{thebibliography}{KWWY14}

\bibitem[Part II]{main}
A.~{\noopsort{01}}~Braverman, M.~Finkelberg, and H.~Nakajima, \emph{{Towards a
  mathematical definition of Coulomb branches of $3$-dimensional $\mathcal N=4$
  gauge theories, II}}, ArXiv e-prints (2016),
  \href{http://arxiv.org/abs/1601.03586}{{\ttfamily arXiv:1601.03586
  [math.RT]}}.

\bibitem[Affine]{affine}
\bysame, \emph{{Ring objects in the equivariant derived Satake category arising
  from Coulomb branches}}, ArXiv e-prints (2017),
  \href{http://arxiv.org/abs/1706.02112}{{\ttfamily arXiv:1706.02112
  [math.RT]}}.

\bibitem[AH88]{MR934202}
M.~Atiyah and N.~Hitchin, \emph{The geometry and dynamics of magnetic
  monopoles}, M. B. Porter Lectures, Princeton University Press, Princeton, NJ,
  1988. \MR{934202 (89k:53067)}

\bibitem[BC11]{MR2748801}
C.~D.~A. Blair and S.~A. Cherkis, \emph{Singular monopoles from {C}heshire
  bows}, Nuclear Phys. B \textbf{845} (2011), no.~1, 140--164. \MR{2748801
  (2012e:53038)}

\bibitem[BDF16]{bdf}
A.~{Braverman}, G.~{Dobrovolska}, and M.~{Finkelberg}, \emph{{Gaiotto-Witten
  superpotential and Whittaker D-modules on monopoles}}, Adv. Math.
  \textbf{300} (2016), 451--472,
  \href{http://arxiv.org/abs/1406.6671}{{\ttfamily arXiv:1406.6671 [math.AG]}}.

\bibitem[BDG15]{2015arXiv150304817B}
M.~{Bullimore}, T.~{Dimofte}, and D.~{Gaiotto}, \emph{{The Coulomb Branch of 3d
  $\mathcal N=4$ Theories}}, ArXiv e-prints (2015),
  \href{http://arxiv.org/abs/1503.04817}{{\ttfamily arXiv:1503.04817
  [hep-th]}}.

\bibitem[BEF16]{2016arXiv161110216B}
A.~{Braverman}, P.~{Etingof}, and M.~{Finkelberg}, \emph{{Cyclotomic double
  affine Hecke algebras (with an appendix by Hiraku Nakajima and Daisuke
  Yamakawa)}}, ArXiv e-prints (2016),
  \href{http://arxiv.org/abs/1611.10216}{{\ttfamily arXiv:1611.10216
  [math.RT]}}.

\bibitem[BF10]{braverman-2007}
A.~Braverman and M.~Finkelberg, \emph{Pursuing the double affine {G}rassmannian
  {I}: transversal slices via instantons on ${A}_k$-singularities}, Duke Math.
  J. \textbf{152} (2010), no.~2, 175--206. \MR{MR2656088 (2011i:14024)}

\bibitem[BF12]{MR2900549}
\bysame, \emph{Pursuing the double affine {G}rassmannian {II}: {C}onvolution},
  Adv. Math. \textbf{230} (2012), no.~1, 414--432. \MR{2900549}

\bibitem[BF13]{MR3134906}
\bysame, \emph{Pursuing the double affine {G}rassmannian {III}: {C}onvolution
  with affine zastava}, Mosc. Math. J. \textbf{13} (2013), no.~2, 233--265,
  363. \MR{3134906}

\bibitem[BF14a]{bf14}
\bysame, \emph{Semi-infinite {S}chubert varieties and quantum {K}-theory of
  flag manifolds}, J. Amer. Math. Soc. \textbf{27} (2014), no.~4, 1147--1168.

\bibitem[BF14b]{brfi}
\bysame, \emph{{Twisted zastava and q-Whittaker functions}}, ArXiv e-prints
  (2014), \href{http://arxiv.org/abs/1410.2365}{{\ttfamily arXiv:1410.2365
  [math.AG]}}.

\bibitem[BFG06]{BFG}
A.~Braverman, M.~Finkelberg, and D.~Gaitsgory, \emph{Uhlenbeck spaces via
  affine {L}ie algebras}, The unity of mathematics, Progr. Math., vol. 244,
  Birkh\"auser Boston, Boston, MA, 2006, see
  \url{http://arxiv.org/abs/math/0301176} for erratum, pp.~17--135. \MR{2181803
  (2007f:14008)}

\bibitem[BFGM02]{bfgm}
A.~Braverman, M.~Finkelberg, D.~Gaitsgory, and I.~Mirkovi{\'c},
  \emph{Intersection cohomology of {D}rinfeld's compactifications}, Selecta
  Math. (N.S.) \textbf{8} (2002), no.~3, 381--418, see
  \url{http://arxiv.org/abs/math/0012129v3} or Selecta Math. (N.S.) {\bf 10}
  (2004), 429--430, for erratum.

\bibitem[BFM05]{MR2135527}
R.~Bezrukavnikov, M.~Finkelberg, and I.~Mirkovi{\'c}, \emph{Equivariant
  homology and {$K$}-theory of affine {G}rassmannians and {T}oda lattices},
  Compos. Math. \textbf{141} (2005), no.~3, 746--768. \MR{2135527
  (2006e:19005)}

\bibitem[BFN16]{2014arXiv1406.2381B}
A.~{Braverman}, M.~{Finkelberg}, and H.~{Nakajima}, \emph{{Instanton moduli
  spaces and $\mathscr W$-algebras}}, Ast\'erisque (2016), no.~385, vii+128,
  \href{http://arxiv.org/abs/1406.2381}{{\ttfamily arXiv:1406.2381 [math.QA]}}.
  \MR{3592485}

\bibitem[BG01]{bragai}
A.~Braverman and D.~Gaitsgory, \emph{Crystals via the affine {G}rassmannian},
  Duke Math. J. \textbf{107} (2001), 561--575.

\bibitem[BG08]{baga}
P.~Baumann and S.~Gaussent, \emph{On {M}irkovi{\'c}-{V}ilonen cycles and
  crystal combinatorics}, Represent. Theory \textbf{12} (2008), 83--130.

\bibitem[BLPW16]{2014arXiv1407.0964B}
T.~Braden, A.~Licata, N.~Proudfoot, and B.~Webster, \emph{Quantizations of
  conical symplectic resolutions {II}: category {$\mathcal O$} and symplectic
  duality}, Ast\'erisque (2016), no.~384, 75--179, with an appendix by I.
  Losev, \href{http://arxiv.org/abs/1407.0964}{{\ttfamily arXiv:1407.0964
  [math.RT]}}. \MR{3594665}

\bibitem[BM17]{2017arXiv170903004B}
T.~{Braden} and C.~{Mautner}, \emph{{Ringel duality for perverse sheaves on
  hypertoric varieties}}, ArXiv e-prints (2017),
  \href{http://arxiv.org/abs/1709.03004}{{\ttfamily arXiv:1709.03004
  [math.AG]}}.

\bibitem[Bra03]{Braden}
T.~Braden, \emph{Hyperbolic localization of intersection cohomology},
  Transform. Groups \textbf{8} (2003), no.~3, 209--216. \MR{1996415
  (2004f:14037)}

\bibitem[CB01]{CB}
W.~Crawley-Boevey, \emph{Geometry of the moment map for representations of
  quivers}, Compositio Math. \textbf{126} (2001), no.~3, 257--293.
  \MR{MR1834739 (2002g:16021)}

\bibitem[CFHM14]{Cremonesi:2014xha}
S.~Cremonesi, G.~Ferlito, A.~Hanany, and N.~Mekareeya, \emph{{Coulomb branch
  and the moduli space of instantons}}, JHEP \textbf{1412} (2014), 103,
  \href{http://arxiv.org/abs/1408.6835}{{\ttfamily arXiv:1408.6835 [hep-th]}}.

\bibitem[CH08]{MR2395473}
B.~Charbonneau and J.~Hurtubise, \emph{Calorons, {N}ahm's equations on {$S^1$}
  and bundles over {$\Bbb P^1\times\Bbb P^1$}}, Comm. Math. Phys. \textbf{280}
  (2008), no.~2, 315--349. \MR{2395473 (2009f:53032)}

\bibitem[CH10]{MR2681686}
\bysame, \emph{The {N}ahm transform for calorons}, The many facets of geometry,
  Oxford Univ. Press, Oxford, 2010, pp.~34--70. \MR{2681686 (2011g:53044)}

\bibitem[Che09]{MR2525636}
S.~A. Cherkis, \emph{Moduli spaces of instantons on the {T}aub-{NUT} space},
  Comm. Math. Phys. \textbf{290} (2009), no.~2, 719--736. \MR{2525636
  (2010j:53034)}

\bibitem[Che10]{MR2721657}
\bysame, \emph{Instantons on the {T}aub-{NUT} space}, Adv. Theor. Math. Phys.
  \textbf{14} (2010), no.~2, 609--641. \MR{2721657 (2011m:53040)}

\bibitem[CK98]{MR1636383}
S.~A. Cherkis and A.~Kapustin, \emph{Singular monopoles and supersymmetric
  gauge theories in three dimensions}, Nuclear Phys. B \textbf{525} (1998),
  no.~1-2, 215--234. \MR{1636383 (99f:81193)}

\bibitem[dBHO{\etalchar{+}}97]{MR1454292}
J.~de~Boer, K.~Hori, H.~Ooguri, Y.~Oz, and Z.~Yin, \emph{Mirror symmetry in
  three-dimensional gauge theories, {${\rm SL}(2,\bold Z)$} and {D}-brane
  moduli spaces}, Nuclear Phys. B \textbf{493} (1997), no.~1-2, 148--176.
  \MR{1454292 (99c:81231)}

\bibitem[dBHOO97]{MR1454291}
J.~de~Boer, K.~Hori, H.~Ooguri, and Y.~Oz, \emph{Mirror symmetry in
  three-dimensional gauge theories, quivers and {D}-branes}, Nuclear Phys. B
  \textbf{493} (1997), no.~1-2, 101--147. \MR{1454291 (99c:81230)}

\bibitem[Del87]{Deligne}
P.~Deligne, \emph{Le d{\'e}terminant de la cohomologie}, Contemp. Math.
  \textbf{67} (1987), 93--177.

\bibitem[Don84]{MR769355}
S.~K. Donaldson, \emph{Nahm's equations and the classification of monopoles},
  Comm. Math. Phys. \textbf{96} (1984), no.~3, 387--407. \MR{769355
  (86c:58039)}

\bibitem[Fal03]{Faltings}
G.~Faltings, \emph{Algebraic loop groups and moduli spaces of bundles}, J.
  European Math. Soc. \textbf{5} (2003), no.~1, 41--68.

\bibitem[FFNR11]{FFNR}
B.~Feigin, M.~Finkelberg, A.~Negut, and L.~Rybnikov, \emph{Yangians and
  cohomology rings of {L}aumon spaces}, Selecta Math. (N.S.) \textbf{17}
  (2011), no.~3, 573--607. \MR{2827177}

\bibitem[Fin18]{fnkl_icm}
M.~Finkelberg, \emph{{Double affine {G}rassmannians and {C}oulomb branches of
  $3d$ $\mathcal N=4$ quiver gauge theories}}, Proceedings of the International
  Congress of Mathematicians, 2018, to appear.

\bibitem[FKMM99]{fkmm}
M.~Finkelberg, A.~Kuznetsov, N.~Markarian, and I.~Mirkovi{\'c}, \emph{A note on
  the symplectic structure on the space of $g$-monopoles}, Comm. Math. Phys.
  \textbf{201} (1999), no.~2, 411--421, see
  \url{http://arxiv.org/abs/math/9803124v6} or Comm. Math. Phys. {\bf 334}
  (2015), no. 2, 1153--1155, for erratum.

\bibitem[FKP{\etalchar{+}}16]{fkprw}
M.~Finkelberg, J.~Kamnitzer, K.~Pham, L.~Rybnikov, and A.~Weekes,
  \emph{{Comultiplication for shifted Yangians and quantum open Toda lattice}},
  ArXiv e-prints (2016), \href{http://arxiv.org/abs/1608.03331}{{\ttfamily
  arXiv:1608.03331 [math.RT]}}.

\bibitem[FKR15]{fkr}
M.~{Finkelberg}, A.~{Kuznetsov}, and L.~{Rybnikov}, \emph{{Towards a cluster
  structure on trigonometric zastava (with appendix by Galyna Dobrovolska)}},
  ArXiv e-prints (2015), \href{http://arxiv.org/abs/1504.05605}{{\ttfamily
  arXiv:1504.05605 [math.AG]}}.

\bibitem[FL06]{fl}
G.~Fourier and P.~Littelmann, \emph{Tensor product structure of affine
  {D}emazure modules and limit constructions}, Nagoya Math. J. \textbf{182}
  (2006), 171--198.

\bibitem[FM99]{mf}
M.~Finkelberg and I.~Mirkovi{\'c}, \emph{Semi-infinite flags. {I}. {C}ase of
  global curve {${\mathbb P}^1$}}, Amer. Math. Soc. Transl. Ser. 2 \textbf{194}
  (1999), 81--112.

\bibitem[FR14]{fra}
M.~Finkelberg and L.~Rybnikov, \emph{Quantization of {D}rinfeld zastava in type
  {A}}, J. Eur. Math. Soc. \textbf{16} (2014), no.~2, 235--271.

\bibitem[Gai08]{gait}
D.~Gaitsgory, \emph{Twisted {W}hittaker model and factorizable sheaves},
  Selecta Math. (N.S.) \textbf{13} (2008), no.~4, 617--659.

\bibitem[{Gin}95]{1995alg.geom.11007G}
V.~{Ginzburg}, \emph{{Perverse sheaves on a Loop group and Langlands'
  duality}}, ArXiv e-prints (1995),
  \href{http://arxiv.org/abs/alg-geom/9511007}{{\ttfamily
  arXiv:alg-geom/9511007 [alg-geom]}}.

\bibitem[GKLO05]{GKLO}
A.~Gerasimov, S.~Kharchev, D.~Lebedev, and S.~Oblezin, \emph{On a class of
  representations of the {Y}angian and moduli space of monopoles}, Comm. Math.
  Phys. \textbf{260} (2005), no.~3, 511--525. \MR{2182434}

\bibitem[Hit83]{MR709461}
N.~J. Hitchin, \emph{On the construction of monopoles}, Comm. Math. Phys.
  \textbf{89} (1983), no.~2, 145--190. \MR{709461 (84m:53076)}

\bibitem[HM89]{MR994495}
J.~Hurtubise and M.~K. Murray, \emph{On the construction of monopoles for the
  classical groups}, Comm. Math. Phys. \textbf{122} (1989), no.~1, 35--89.
  \MR{994495 (91d:58037)}

\bibitem[Hur85]{MR804459}
J.~Hurtubise, \emph{Monopoles and rational maps: a note on a theorem of
  {D}onaldson}, Comm. Math. Phys. \textbf{100} (1985), no.~2, 191--196.
  \MR{804459}

\bibitem[Hur89]{MR987771}
\bysame, \emph{The classification of monopoles for the classical groups}, Comm.
  Math. Phys. \textbf{120} (1989), no.~4, 613--641. \MR{987771 (90c:53182)}

\bibitem[HW97]{MR1451054}
A.~Hanany and E.~Witten, \emph{Type {IIB} superstrings, {BPS} monopoles, and
  three-dimensional gauge dynamics}, Nuclear Phys. B \textbf{492} (1997),
  no.~1-2, 152--190. \MR{1451054 (98h:81096)}

\bibitem[Ill71]{illusie}
L.~Illusie, \emph{Complexe cotangent et d\'eformations. {I}}, Lecture Notes in
  Mathematics, 239, Springer-Verlag, Berlin -- New York, 1971.

\bibitem[Jar98]{MR1625475}
S.~Jarvis, \emph{Euclidean monopoles and rational maps}, Proc. London Math.
  Soc. (3) \textbf{77} (1998), no.~1, 170--192. \MR{1625475 (99h:58024)}

\bibitem[Kas95]{kash-crys}
M.~Kashiwara, \emph{On crystal bases}, CMS Conf. Proc. \textbf{16} (1995),
  155--197.

\bibitem[Kir97]{MR1441642}
A.~A. Kirillov, Jr., \emph{Lectures on affine {H}ecke algebras and
  {M}acdonald's conjectures}, Bull. Amer. Math. Soc. (N.S.) \textbf{34} (1997),
  no.~3, 251--292. \MR{1441642 (99c:17015)}

\bibitem[KN16]{2016arXiv160800875K}
R.~{Kodera} and H.~{Nakajima}, \emph{{Quantized Coulomb branches of Jordan
  quiver gauge theories and cyclotomic rational Cherednik algebras}}, ArXiv
  e-prints (2016), \href{http://arxiv.org/abs/1608.00875}{{\ttfamily
  arXiv:1608.00875 [math.RT]}}.

\bibitem[Kro85]{Kronheimer-msc}
P.~B. Kronheimer, \emph{Monopoles and {T}aub-{N}{U}{T} metrics}, Master's
  thesis, Oxford, 1985.

\bibitem[{Kry}17]{2017arXiv170900391K}
V.~{Krylov}, \emph{{Integrable crystals and restriction to Levi via generalized
  slices in the affine Grassmannian}}, ArXiv e-prints (2017),
  \href{http://arxiv.org/abs/1709.00391}{{\ttfamily arXiv:1709.00391
  [math.RT]}}.

\bibitem[KS90]{KaSha}
M.~Kashiwara and P.~Schapira, \emph{Sheaves on manifolds}, Grundlehren der
  Mathematischen Wissenschaften [Fundamental Principles of Mathematical
  Sciences], vol. 292, Springer-Verlag, Berlin, 1990, With a chapter in French
  by Christian Houzel. \MR{MR1074006 (92a:58132)}

\bibitem[KW07]{MR2306566}
A.~Kapustin and E.~Witten, \emph{Electric-magnetic duality and the geometric
  {L}anglands program}, Commun. Number Theory Phys. \textbf{1} (2007), no.~1,
  1--236. \MR{2306566 (2008g:14018)}

\bibitem[KWWY14]{kwy}
J.~Kamnitzer, B.~Webster, A.~Weekes, and O.~Yacobi, \emph{Yangians and
  quantizations of slices in the affine {G}rassmannian}, Algebra and Number
  Theory \textbf{8} (2014), no.~4, 857--893.

\bibitem[Lus83]{Lus-ast}
G.~Lusztig, \emph{Singularities, character formulas, and a q-analog of weight
  multiplicities}, Ast\'erisque \textbf{101-102} (1983), 208--229.

\bibitem[Lus93]{Lu-book}
G.~Lusztig, \emph{Introduction to quantum groups}, Progress in Mathematics,
  vol. 110, Birkh\"auser Boston Inc., Boston, MA, 1993. \MR{MR1227098
  (94m:17016)}

\bibitem[Mac01]{MR1817334}
I.~G. Macdonald, \emph{Orthogonal polynomials associated with root systems},
  S\'em. Lothar. Combin. \textbf{45} (2000/01), Art.\ B45a, 40. \MR{1817334
  (2002a:33021)}

\bibitem[Mac95]{Mac}
\bysame, \emph{Symmetric functions and {H}all polynomials}, second ed., Oxford
  Mathematical Monographs, The Clarendon Press Oxford University Press, New
  York, 1995, With contributions by A. Zelevinsky, Oxford Science Publications.
  \MR{1354144 (96h:05207)}

\bibitem[Mat86]{matsumura}
H.~Matsumura, \emph{Commutative ring theory}, Cambridge Studies in Advanced
  Mathematics, 8, Cambridge University Press, Cambridge, 1986.

\bibitem[Mek15]{Mekareeya2015}
N.~Mekareeya, \emph{The moduli space of instantons on an {ALE} space from 3d
  $\mathcal{N}=4$ field theories}, Journal of High Energy Physics \textbf{2015}
  (2015), no.~12, 1--30.

\bibitem[Mir14]{mirk}
I.~Mirkovi{\'c}, \emph{Loop {G}rassmannians in the framework of local spaces
  over a curve}, Contemp. Math. \textbf{623} (2014), 215--226.

\bibitem[MO12]{2012arXiv1211.1287M}
D.~{Maulik} and A.~{Okounkov}, \emph{{Quantum Groups and Quantum Cohomology}},
  ArXiv e-prints (2012), \href{http://arxiv.org/abs/1211.1287}{{\ttfamily
  arXiv:1211.1287 [math.AG]}}.

\bibitem[MOV05]{mov}
A.~Malkin, V.~Ostrik, and M.~Vybornov, \emph{The minimal degeneration
  singularities in the affine {G}rassmannians}, Duke Math. J. \textbf{126}
  (2005), no.~2, 233--249.

\bibitem[MV07]{MV2}
I.~Mirkovi{\'c} and K.~Vilonen, \emph{Geometric {L}anglands duality and
  representations of algebraic groups over commutative rings}, Ann. of Math.
  (2) \textbf{166} (2007), no.~1, 95--143. \MR{2342692 (2008m:22027)}

\bibitem[Nak93]{MR1215288}
H.~Nakajima, \emph{Monopoles and {N}ahm's equations}, Einstein metrics and
  Yang-Mills connections (Sanda, 1990), Lecture Notes in Pure and Appl. Math.,
  vol. 145, Dekker, New York, 1993, pp.~193--211. \MR{MR1215288 (94b:53055)}

\bibitem[Nak94]{Na-quiver}
\bysame, \emph{Instantons on {ALE} spaces, quiver varieties, and {K}ac-{M}oody
  algebras}, Duke Math. J. \textbf{76} (1994), no.~2, 365--416. \MR{MR1302318
  (95i:53051)}

\bibitem[Nak98]{Na-alg}
\bysame, \emph{Quiver varieties and {K}ac-{M}oody algebras}, Duke Math. J.
  \textbf{91} (1998), no.~3, 515--560. \MR{MR1604167 (99b:17033)}

\bibitem[Nak01]{Na-Tensor}
\bysame, \emph{Quiver varieties and tensor products}, Invent. Math.
  \textbf{146} (2001), no.~2, 399--449. \MR{MR1865400 (2003e:17023)}

\bibitem[Nak15]{2015arXiv151003908N}
\bysame, \emph{{Questions on provisional Coulomb branches of $3$-dimensional
  $\mathcal N=4$ gauge theories}}, S\=urikaisekikenky\=usho K\=oky\=uroku
  (2015), no.~1977, 57--76, \href{http://arxiv.org/abs/1510.03908}{{\ttfamily
  arXiv:1510.03908 [math-ph]}}.

\bibitem[Nak16a]{2015arXiv150303676N}
\bysame, \emph{{Towards a mathematical definition of Coulomb branches of
  $3$-dimensional $\mathcal N=4$ gauge theories, I}}, Adv. Theor. Math. Phys.
  \textbf{20} (2016), no.~3, 595--669,
  \href{http://arxiv.org/abs/1503.03676}{{\ttfamily arXiv:1503.03676
  [math-ph]}}.

\bibitem[Nak16b]{2016arXiv160406316N}
\bysame, \emph{{Lectures on perverse sheaves on instanton moduli spaces}},
  ArXiv e-prints (2016), \href{http://arxiv.org/abs/1604.06316}{{\ttfamily
  arXiv:1604.06316 [math.RT]}}.

\bibitem[NP01]{MR1832331}
B.~C. Ng{\^o} and P.~Polo, \emph{R\'esolutions de {D}emazure affines et formule
  de {C}asselman-{S}halika g\'eom\'etrique}, J. Algebraic Geom. \textbf{10}
  (2001), no.~3, 515--547. \MR{1832331 (2002f:14032)}

\bibitem[NP12]{2012arXiv1211.2240N}
N.~{Nekrasov} and V.~{Pestun}, \emph{{Seiberg-Witten geometry of four
  dimensional N=2 quiver gauge theories}}, ArXiv e-prints (2012),
  \href{http://arxiv.org/abs/1211.2240}{{\ttfamily arXiv:1211.2240 [hep-th]}}.

\bibitem[NS00]{MR1789949}
T.~M.~W. Nye and M.~A. Singer, \emph{An {$L^2$}-index theorem for {D}irac
  operators on {$S^1\times\bold R^3$}}, J. Funct. Anal. \textbf{177} (2000),
  no.~1, 203--218. \MR{1789949 (2002a:58020)}

\bibitem[NT17]{2016arXiv160602002N}
H.~Nakajima and Y.~Takayama, \emph{{Cherkis bow varieties and {C}oulomb
  branches of quiver gauge theories of affine type {$A$}}}, Selecta Mathematica
  \textbf{23} (2017), no.~4, 2553--2633,
  \href{http://arxiv.org/abs/1606.02002}{{\ttfamily arXiv:1606.02002
  [math.RT]}}.

\bibitem[Nye01]{Nye-thesis}
T.~M.~W. Nye, \emph{The geometry of calorons}, Ph.D. thesis, U. of Edinburgh,
  2001.

\bibitem[{Sch}16]{sch}
S.~{Schieder}, \emph{{Geometric Bernstein asymptotics and the
  Drinfeld-Lafforgue-Vinberg degeneration for arbitrary reductive groups}},
  ArXiv e-prints (2016), \href{http://arxiv.org/abs/1607.00586}{{\ttfamily
  arXiv:1607.00586 [math.AG]}}.

\bibitem[{Sha}14]{arXiv1405.3909}
A.~{Shapiro}, \emph{{Poisson Geometry of Monic Matrix Polynomials}}, ArXiv
  e-prints (2014), \href{http://arxiv.org/abs/1405.3909}{{\ttfamily
  arXiv:1405.3909 [math-ph]}}.

\bibitem[Slo12]{MR2855083}
W.~Slofstra, \emph{A {B}rylinski filtration for affine {K}ac-{M}oody algebras},
  Adv. Math. \textbf{229} (2012), no.~2, 968--983. \MR{2855083}

\bibitem[Tak16]{Takayama}
Y.~Takayama, \emph{Nahm's equations, quiver varieties and parabolic sheaves},
  Publ. Res. Inst. Math. Sci. \textbf{52} (2016), no.~1, 1--41.

\bibitem[Ton99]{MR1677752}
D.~Tong, \emph{Three-dimensional gauge theories and {$ADE$} monopoles}, Phys.
  Lett. B \textbf{448} (1999), no.~1-2, 33--36. \MR{1677752 (2000e:81190)}

\bibitem[vDE11]{MR2822182}
J.~F. van Diejen and E.~Emsiz, \emph{A generalized {M}acdonald operator}, Int.
  Math. Res. Not. IMRN (2011), no.~15, 3560--3574. \MR{2822182 (2012f:33030)}

\bibitem[Wey97]{MR1488158}
H.~Weyl, \emph{The classical groups, their invariants and representations},
  Princeton Landmarks in Mathematics, Princeton University Press, Princeton,
  NJ, 1997, Fifteenth printing, Princeton Paperbacks. \MR{1488158 (98k:01049)}

\end{thebibliography}

\end{document}